 \documentclass[a4paper,11pt]{article}
\usepackage{pdfsync}
 \usepackage[plainpages=false]{hyperref}
 \usepackage{amsfonts,amsmath,amssymb,amsthm}
 \usepackage{latexsym,lscape,rawfonts,mathrsfs}
 \usepackage{mathtools}
 \usepackage[utf8]{inputenc}
 \usepackage[overload]{empheq}
 \usepackage{cases}

 \usepackage[dvips]{color}
 \usepackage{multicol}

 \usepackage[all]{xy}
 \usepackage{eufrak}
 \usepackage{makeidx}
 \usepackage{graphicx,psfrag}
 \usepackage{pstool}
 \usepackage{float}

 \usepackage{array,tabularx}

 \usepackage{setspace}

 \usepackage{appendix}
 
\usepackage{txfonts}

 \newcommand{\C}{\ensuremath{\mathbb{C}}}
 \newcommand{\D}[2]{\ensuremath{ \frac{\partial{#1}}{\partial{#2}}}}

 \newcommand{\R}{\ensuremath{\mathbb{R}}}
 
 \newcommand{\CP}{\ensuremath{\mathbb{CP}}}

 \newcommand{\ba}{\begin{align*}}
 \newcommand{\ea}{\end{align*}}

 \DeclareMathOperator{\diam}{diam}

 \newcommand{\norm}[2]{{ \ensuremath{\left\|} #1 \ensuremath{\right\|}}_{#2}}

 \makeatletter
 \def\ExtendSymbol#1#2#3#4#5{\ext@arrow 0099{\arrowfill@#1#2#3}{#4}{#5}}
 
 \makeatother

 \makeatletter
 \def\ExtendSymbol#1#2#3#4#5{\ext@arrow 0099{\arrowfill@#1#2#3}{#4}{#5}}
 \newcommand\longright[2][]{\ExtendSymbol{-}{-}{\rightarrow}{#1}{#2}}
 \makeatother

 \definecolor{orange}{rgb}{1,0.5,0}
 \definecolor{brown}{rgb}{0.48,0.33,0.19}
 \definecolor{miao}{cmyk}{0.5,0,0.2,0.2}
 \definecolor{qiao}{gray}{0.96}

\newtheorem{prop}{Proposition}[section]

\newtheorem{proposition}[prop]{Proposition}

\newtheorem{theorem}[prop]{Theorem}

\newtheorem{lemma}[prop]{Lemma}
\newtheorem{claim}[prop]{Claim}

\newtheorem{corollary}[prop]{Corollary}

\newtheorem{remark}[prop]{Remark}

\newtheorem{definition}[prop]{Definition}

\newtheorem{conjecture}[prop]{Conjecture}
\newtheorem{question}[prop]{Question}

\numberwithin{equation}{section}

\title{The local entropy along Ricci flow\\ \large---Part B: the pseudo-locality theorems}
  
\author{Bing Wang}
\date{}

\begin{document}
\maketitle

\begin{abstract}
 We localize the entropy functionals of G. Perelman and generalize his no-local-collapsing theorem and pseudo-locality theorem.
 Our generalization is technically inspired by further development of Li-Yau estimates along the Ricci flow. 
 It has various applications, including to show the continuous dependence of the Ricci flow with respect to the initial metric 
 in Gromov-Hausdorff topology with Ricci curvature bounded below,  and to show the compactness of the moduli of K\"ahler manifolds
 with bounded scalar curvature and a rough locally almost Euclidean condition.  
 \end{abstract}

\tableofcontents

\section{Introduction}

In the celebrated work~\cite{Pe1}, Perelman introduced his famous functionals $\boldsymbol{\mu}$ and $\boldsymbol{\nu}$, which are monotone along the Ricci flow.
These functionals can be naturally localized(cf.~\cite{ROS},~\cite{QZhbook} and~\cite{BWang17local}). 
Inspired by further developing the seminal estimates of Li-Yau~\cite{LiYau},
the author~\cite{BWang17local} investigated the evolution of localized functionals along the Ricci flow and derived effective monotonicity formulas for them. 
Based on these formulas, the author improved the no-local-collapsing theorem of Perelman to the following version. 

\begin{theorem}[\textbf{Improved version of no-local-collapsing}, cf. Theorem 1.1 of~\cite{BWang17local}] 
  For every $A>1$ there exists $\kappa=\kappa(m,A)>0$ with the following property. 
  Suppose $\{(M^{m}, g(t)), 0 \leq t \leq  r_0^2\}$ is a Ricci flow solution such that 
  \begin{align}
   r_0^2 |Rm|(x,t) \leq m^{-1}, \quad \forall \; x \in B_{g(0)}(x_0, r_0),  \; 0 \leq t \leq r_0^2;   \quad   r_0^{-m} \left| B_{g(0)}(x_0, r_0) \right|_{dv_{g(0)}} \geq A^{-1}.  \label{eqn:ML27_3}
  \end{align}
  Then we have
  \begin{align}
    r^{-m} \left|B_{g(t_0)}(x,r) \right|_{dv_{g(t_0)}} \geq \kappa      \label{eqn:ML30_1}
  \end{align}
  whenever $A^{-1}r_0^2 \leq t_0 \leq r_0^2$, $0<r \leq r_0$, and $B_{g(t_0)}(x,r) \subset B_{g(t_0)}(x_0,Ar_0)$ is a geodesic ball satisfying $r^2R(\cdot, t_0) \leq 1$. 
\label{thmin:ML14_1}
\end{theorem}

Theorem~\ref{thmin:ML14_1} improves the no-local-collapsing theorem of Perelman(cf. Theorem 8.2 of~\cite{Pe1}).
Here we replace the requirement of space-time condition by a time-slice condition around $(x,t)$. 
Namely, we only need $R(\cdot, t) \leq r^{-2}$ in the ball $B_{g(t)}(x,r)$ to conclude non-collapsing of $B_{g(t)}(x,r)$ whenever $(x,t)$ is not very far away from $(x_0, 0)$. 
However, we still require space-time condition (\ref{eqn:ML27_3}) around the base $(x_0, 0)$.
The condition (\ref{eqn:ML27_3}) is not optimal and can be replaced by Ricci curvature bound.  In fact, in Theorem 7.2 of~\cite{BWang17local}, we assumed a much weaker condition
\begin{align}
  t |Rc|(x,t)<(m-1)A,  \quad \forall \; x \in B_{g(t)}(x,\sqrt{t}), \quad 0<t<T   \label{eqn:RH17_6}
\end{align}
to play the same role as (\ref{eqn:ML27_3}). 
Although the condition (\ref{eqn:ML27_3}) can be obtained in the K\"ahler setting under  many circumstances, it looks artificial in the Riemannian setting. 
A condition depending only on the initial metric $g(0)$ should be more natural.  This leads us to develop the following pseudo-locality theorem, 
which improves the known pseudo-locality theorems(cf. Theorem 10.1 and Corollary 10.3 of Perelman~\cite{Pe1}, also Proposition 3.1 and Theorem 3.1 of Tian-Wang~\cite{TiWa}).

 \begin{figure}[H]
 \begin{center}
 \psfrag{A}[c][c]{\color{green}{Locally almost Euclidean initial time slice}}
 \psfrag{A1}[c][c]{\color{green}{$B_{g(0)}(x_0, r_0)$}}
 \psfrag{B}[l][l]{\color{blue}{Nearby space-time with bounded geometry}}
 \psfrag{C}[c][c]{$M$}
 \psfrag{D}[c][c]{$t$}
 \psfrag{E}[c][c]{$t=(\delta r_0)^2$}
 \psfrag{F}[c][c]{$(x_0, 0)$}
 \includegraphics[width=0.5 \columnwidth]{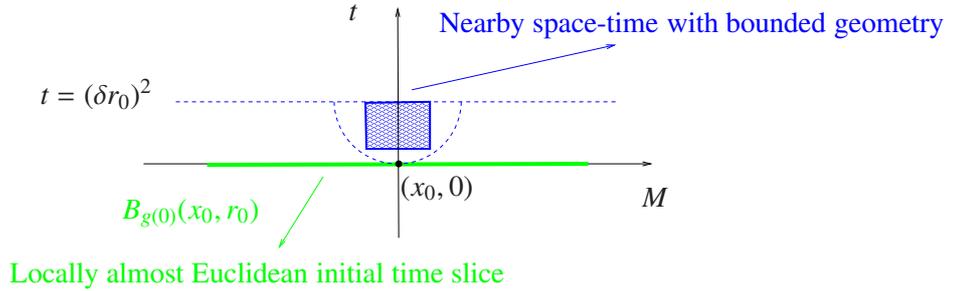}
 \caption{Existence of space-time with bounded geometry}
 \label{fig:step2}
 \end{center}
 \end{figure}

\begin{theorem}[\textbf{Improved version of pseudo-locality}]

For each $\alpha \in (0, \frac{1}{100m})$,  there exists $\delta=\delta(\alpha,m)$ with the following properties. 

Suppose $\left\{(M^m, g(t)),  0 \leq t \leq T\right\}$ is a Ricci flow solution satisfying
\begin{align}
 \inf_{0<t \leq T} \boldsymbol{\mu}\left(B_{g(0)}\left(x_0, \delta^{-1} \sqrt{t} \right), g(0), t \right) \geq -\delta^2. 
\label{eqn:MK30_1} 
\end{align}
Then for each $t \in (0,T]$ and $x \in B_{g(t)}\left(x_0, \alpha^{-1}\sqrt{t} \right)$, we have
\begin{align}
 & t|Rm|(x,t) \leq \alpha,   \label{eqn:ML28_1}\\
 & \inf_{\rho \in (0, \alpha^{-1} \sqrt{t})} \rho^{-m}\left|B_{g(t)}\left(x, \rho \right) \right|_{dv_{g(t)}} \geq (1-\alpha) \omega_m ,  \label{eqn:ML28_2}\\
 & t^{-\frac{1}{2}} \cdot  inj(x,t) \geq \alpha^{-1}. \label{eqn:RH27_4}   
\end{align}
In particular, (\ref{eqn:ML28_1}),(\ref{eqn:ML28_2}) and (\ref{eqn:RH27_4}) hold whenever the following condition is satisfied: 
\begin{align}
  \boldsymbol{\nu}\left( B_{g(0)}\left(x_0, \delta^{-1}\sqrt{T} \right), g(0), T \right) \geq -\delta^{2}.
  \label{eqn:RH17_4}
\end{align}
\label{thmin:ML14_2}
\end{theorem}

The inequality  (\ref{eqn:MK30_1}) can be interpreted  as a new ``almost Euclidean" condition for the domain $B_{g(0)}\left(x_0, \delta^{-1} \sqrt{T} \right)$ equipped with initial metric $g(0)$. 
The functionals $\boldsymbol{\mu}$ and $\boldsymbol{\nu}$ are the localized functionals of Perelman(cf.(\ref{eqn:MJ16_a})-(\ref{eqn:MJ16_2}) for precise definitions).  
The left hand side of (\ref{eqn:MK30_1}) is completely determined by the geometry of  $B_{g(0)}\left(x_0, \delta^{-1} \sqrt{T} \right)$ with initial metric $g(0)$. 
In particular, (\ref{eqn:MK30_1}) holds whenever the ball $B_{g(0)}\left(x_0, \delta^{-1} \sqrt{T} \right)$ satisfies one of the following conditions:
either the ball  has small scalar curvature lower bound and almost Euclidean isoperimetric constant, or the ball has small Ricci curvature lower bound and almost Euclidean volume ratio. 
Therefore, it is not hard to see(cf. Corollary~\ref{cly:MJ23_1} and Corollary~\ref{cly:MJ23_2}) that (\ref{eqn:MK30_1}) is weaker than the conditions in the original version and in Tian-Wang's version of
Perelman's pseudo-locality theorem.  
While it is known now that almost non-negative Ricci curvature and almost Euclidean volume ratio together imply almost Euclidean isoperimetric constant~\cite{CaMo},
the proof is complicated and involves the deep theory of RCD spaces and optimal transportation. 
Therefore, it is still worthwhile to provide a more self-contained and elementary proof to unify the pseudo-locality theorems of Perelman~\cite{Pe1} and Tian-Wang~\cite{TiWa}.  
More importantly, the local entropy has an obvious advantage over the isoperimetric constant: the local entropy is almost monotone along the Ricci flow(cf. Theorem~\ref{thm:CA02_3}), however, 
the monotonicity of isoperimetric constant along the Ricci flow is not clear.  This advantage makes Theorem~\ref{thmin:ML14_2} an effective tool to estimate the lifespan of the Ricci 
flow solution(cf. Corollary~\ref{cly:CK21_1} and Corollary~\ref{cly:RH27_1}). 

In Theorem~\ref{thmin:ML14_2}, if we assume $r_0=\delta^{-1} \sqrt{T}$ and the ball $B_{g(0)}(x_0,r_0)$ is ``almost Euclidean", then we obtain curvature and volume estimate around base point $x_0$
on the time period $(0, (\delta r_0)^2]$, as illustrated in Figure~\ref{fig:step2}. Consequently, the condition (\ref{eqn:RH17_6}) is naturally implied by the almost Euclidean initial condition.  
If one allows time-shifting and rescaling, then even the much stronger condition (\ref{eqn:ML27_3}) can be obtained by a rough almost Euclidean initial condition like (\ref{eqn:MK30_1}).
Therefore, even in the Riemannian setting, the condition in Theorem~\ref{thmin:ML14_1} can be deduced from local information of the initial metric.

Theorem~\ref{thmin:ML14_2} can be regarded as a bridge between the Ricci flow theory and the structure theory of Cheeger-Colding concerning 
the Riemannian manifolds with Ricci curvature bounded below.
Both theories are of fundamental importance and have essentially changed the landscape of Riemannian geometry in last two decades.
In the resolution of the Hamilton-Tian conjecture on Fano K\"ahler Ricci flow(cf.~\cite{CW17A},~\cite{CW17B} and~\cite{CW19}) by Chen and the author,  
it is a key point to further develop the Cheeger-Colding theory to obtain a compactness theorem for non-collapsed Calabi-Yau spaces with mild singularities.
Such a compactness result then played a crucial role in showing the convergence of the K\"ahler Ricci flow,  via the Bergman Kernel and Perelman's reduced volume.
The compactness theory for the moduli space of gradient shrinking solitons-the singularity model spaces for Ricci flows-has also been developed in a series of
works~\cite{LiLiWang18},~\cite{HuangLiWang18} and~\cite{YLBW19}. 
Along the same line,  we shall in this paper discuss further applications of the Cheeger-Colding theory in the study of Ricci flows.
On the other hand, we shall also see how one can exploit Ricci flow techniques to improve the understanding of the Cheeger-Colding theory.  
The following theorem directly relates the Ricci flow to the Cheeger-Colding theory. 

\begin{theorem}[\textbf{Continuous dependence on initial data in the Gromov-Hausdorff topology}]
  Suppose that  $\mathcal{N}=\{(N^{m}, h(t)), 0 \leq t \leq T\}$ is a (unnormalized, volume normalized, or $c$-normalized) Ricci flow solution initiated from $h(0)=h$, 
  a smooth Riemannian metric on a closed manifold $N^{m}$. 
  For each $\epsilon$ small, there exists $\delta=\delta(\mathcal{N}, \epsilon)$ with the following properties. 

  Suppose $(M^{m}, g)$ is a Riemannian manifold satisfying
  \begin{align}
    Rc>-(m-1)\epsilon^{-1},  \quad d_{GH}\left\{ (M,g), (N,h) \right\}<\delta. 
    \label{eqn:RG05_4}
  \end{align}
  Then the (unnormalized, volume normalized, or $c$-normalized) Ricci flow initiated from $(M,g)$ exists on $[0, T]$. 
  Furthermore, there exists a family of diffeomorphisms $\left\{ \Phi_{t}: N \to M, \; \epsilon \leq t \leq T \right\}$ such that
 \begin{align}
   \sup_{t \in [\epsilon, T]} \norm{\Phi_{t}^{*}g(t)-h(t)}{C^{[\epsilon^{-1}]}(N,h(t))} < \epsilon.   \label{eqn:RG05_2}
 \end{align}
 \label{thm:RG05_3}
\end{theorem}

As in~\cite{BWang17local}, we call a solution of the following equation as normalized Ricci flow: 
\begin{align}
    \frac{\partial}{\partial t} g= -2 \left\{ Rc + \lambda(t)g \right\}.    \label{eqn:ML27_2} 
\end{align}
If $\lambda(t)=\frac{1}{m} \fint_{M} R dv$, then the volume is fixed and the flow is called volume normalized.
If $\lambda(t)=c$, then the flow is called $c$-normalized. 
Theorem~\ref{thm:RG05_3} claims that the Ricci flow depends on the initial data continuously in the Gromov-Hausdorff topology, when the Ricci curvature is bounded from below.
The diffeomorphism $\Phi_{t}:N \to M$ in Theorem~\ref{thm:RG05_3} can be constructed explicitly.
Consequently, Theorem~\ref{thm:RG05_3} implies that if $M_i$ converges to $N$ in the Gromov-Hausdorff topology with
Ricci bounded below, then $M_i$ is diffeomorphic to $N$ for large $i$.
Therefore, we provide an alternative approach to construct diffeomorphisms, independent of the Reifenberg method, as done in~\cite{CC}.
In the presence of the local isometry group actions, the diffeomorphisms constructed from the Ricci flows seem natural and respect
the group actions. For further discussions, see~\cite{HuangWang20A},~\cite{HuangWang20B} and~\cite{HRW20A}. 
The conditions in Theorem~\ref{thm:RG05_3} are not optimal and could be modified in various applications. 
In~\cite{HRW20B}, we will develop a version of Theorem~\ref{thm:RG05_3} for certain complete non-compact initial data, and use this to show that
certain torus bundles with almost non-negative Ricci curvature are trivial. 

Theorem~\ref{thm:RG05_3} implies the commutativity of taking Gromov-Hausdorff limit and running the Ricci flow, under the condition that the initial Ricci curvatures
are uniformly bounded below. The proof of Theorem~\ref{thm:RG05_3} relies heavily on the improved estimates (\ref{eqn:ML28_1}), (\ref{eqn:ML28_2}) and (\ref{eqn:RH27_4}) in Theorem~\ref{thmin:ML14_2}. 
Note that the application of pseudo-locality type theorems are heavily studied in the Ricci flow literature(cf.~\cite{CTY},~\cite{SimTop},~\cite{BCRW},~\cite{LaiYi},~\cite{MT2},~\cite{LeeTam},~\cite{DeScSi},~\cite{YuLi20}, etc.).
Typically, the other (than Perelman's) pseudo-locality type theorems rely on stronger curvature or dimensional 
assumptions(e.g. $BK \geq -C$, $PIC_{1} \geq -C$, $\dim \leq 3$, etc.) and the analysis of ancient solutions with non-negativity of the strong curvature assumptions, which are preserved by the Ricci flow. 
A key difference in Theorem~\ref{thm:RG05_3} is that the non-negativity of Ricci curvature is not preserved in dimension higher than three(cf.~\cite{BoWil07},~\cite{Maximo}).

Theorem~\ref{thm:RG05_3} can be used to show the global convergence of Ricci flows near(in the Gromov-Hausdorff topology)stable Einstein manifolds(cf.~Theorem~\ref{thm:RF12_2}). 
Although the Einstein constant we consider could be negative, zero or positive, as discussed in Section~\ref{sec:globalstability}, we shall focus only on the case of positive Einstein constant
in this introduction.  Among other things,  we have the following global convergence theorem for normalized Ricci flows. 

\begin{theorem}[\textbf{Stability of normalized Ricci flow near sphere}]
  For each $m \geq 3$, there is a constant $\delta_{0}=\delta_{0}(m)$ with the following property.
  Suppose $(M^{m}, g)$ is a Riemannian manifold satisfying 
  \begin{align}
    \begin{cases}
    &Rc \geq (m-1)(1-\delta), \\
    &|M|_{dv_{g}}=(m+1)\omega_{m+1}, 
    \end{cases}
    \label{eqn:RG26_1}  
  \end{align}
  where $\omega_{m+1}$ is the volume of unit ball in $\R^{m+1}$, and $\delta\in (0, \delta_{0})$. 
  Then the normalized Ricci flow initiated from $(M^{m}, g)$ exists immortally and converges to a round metric $g_{\infty}$ exponentially fast.
  The limit metric $(M, g_{\infty})$ has constant sectional curvature $1$. 

  Furthermore, for each pair of points $x,y \in M$ satisfying $d_{g}(x,y) \leq 1$, the following distance bi-H\"older estimate holds:
 \begin{align}
   e^{-\psi} d_{g}^{1+\psi}(x,y) \leq d_{g_{\infty}}(x,y) \leq e^{\psi} d_{g}^{1-\psi}(x,y),    \label{eqn:RG26_2} 
 \end{align}
 where $\psi=\psi(\delta|m)$ such that $\displaystyle \lim_{\delta \to 0} \psi(\delta|m)=0$.    
  \label{thm:RG26_3}
\end{theorem}

Theorem~\ref{thm:RG26_3} corresponds to certain rigidity phenomenon.  
If $\delta=0$ in (\ref{eqn:RG26_1}), the proof of Bishop-Gromov volume comparison implies that $(M,g)$ is isometric to the round sphere $(S^{m}, g_{round})$ of sectional curvature $1$. 
Consequently, the normalized Ricci flow initiated from $(M,g)$ is static, which means $g_{\infty}=g$.
Therefore, Theorem~\ref{thm:RG26_3} can be regarded as an effective version of the aforementioned rigidity phenomenon. 
In K\"ahler geometry, there is a similar rigidity phenomenon with the model space $(\CP^{n}, g_{FS})$, where  $n=\frac{m}{2}$ is the complex dimension.
Note that the Ricci curvature of $(\CP^{n}, g_{FS})$ is $2(n+1)$. 
In the same vein, replacing $(S^{m},g_{round})$ by $(\CP^{n}, g_{FS})$, one can expect a K\"ahler version of Theorem~\ref{thm:RG26_3}, which is realized by the following theorem.

\begin{theorem}[\textbf{Stability of normalized Ricci flow near projective space}]
  For each integer $n \geq 1$, there exists a constant $\delta_{0}=\delta_{0}(n)$ with the following properties.

Suppose $(M^{n},g,J)$ is a Fano manifold whose metric form is proportional to  $c_1(M,J)$ and satisfies
 \begin{align}
    \begin{cases}
    &Rc \geq 2(n+1)(1-\delta), \\
    &|M|_{dv_{g}}=\Omega_n,
    \end{cases}
    \label{eqn:RG26_4}  
  \end{align}
where $\Omega_{n}$ is the volume of $(\CP^{n}, g_{FS})$, and $\delta \in (0, \delta_{0})$. 
Then the normalized Ricci flow initiated from $(M, g)$ exists immortally and converges smoothly(without further diffeomorphisms) to a metric $g_{\infty}$  such that
$(M, g_{\infty}, J)$ is biholomorphic-isometric to $(\CP^{n}, g_{FS}, J_{FS})$. Namely, there exists a diffeomorphism $\Phi:\CP^{n} \to M$ such that
$\Phi^{*} g_{\infty}=g_{FS}$ and $\Phi^{*}(J)=J_{FS}$. 
Furthermore, for each pair of points $x,y \in M$ satisfying $d_{g}(x,y) \leq 1$, the distance bi-H\"older estimate (\ref{eqn:RG26_2}) holds. 
\label{thm:RG26_5}
\end{theorem}

Note that the bi-H\"older distance estimate in Theorem~\ref{thm:RG26_5} implies the following conclusion(cf. Corollary~\ref{cly:RF09_1}).
Under the condition $Rc \geq 2(n+1)$ and $g$ is in a K\"ahler class
proportional to $c_{1}(M,J)$,
the K\"ahler manifold $(M,g,J)$ is close to $(\CP^{n},g_{FS},J_{FS})$ in Gromov-Hausdorff topology if and only
if the volume of $(M,g,J)$ is close to $\Omega_{n}$. This result is the K\"ahler version of the theorems of Colding(cf.~\cite{Co1} and~\cite{Co2}). 
Theorem~\ref{thm:RG26_3} also has a non-compact version.  In fact, from Bishop-Gromov volume comparison, it is clear that a complete Riemannian manifold $(M^{m}, g)$ with 
non-negative Ricci curvature and asymptotic volume ratio $AVR(M,g)=1$ must be isometric to Euclidean space $(\R^{m}, g_{E})$. Here 
\begin{align*}
  AVR(M,g) \coloneqq  \lim_{r \to \infty} \omega_{m}^{-1} r^{-m} |B(x_0,r)|
\end{align*}
for some $x_0 \in M$. It is not hard to see that the above definition is independent of the choice of base point $x_0 \in M$.
The Ricci flow started from $(M,g)$ is then static and thus is always isometric to the Euclidean space.  
This rigidity phenomenon is the non-compact counterpart of the one underlying Theorem~\ref{thm:RG26_3}.
Therefore, the ``almost rigidity'' shall be some effective estimate of the Ricci flow, which is realized by the following theorem.

\begin{theorem}[\textbf{Almost Euclidean immortal Ricci flow}]
  For each $m \geq 3$, there is a constant $\delta_{0}=\delta_{0}(m)$ with the following property.
  Suppose $(M^{m}, g)$ is a complete non-compact  Riemannian manifold satisfying 
  \begin{align}
    \begin{cases}
    &Rc \geq 0, \\
    &AVR(M,g) \geq 1-\delta,  
  \end{cases}
    \label{eqn:RH19_1}  
  \end{align}
  for some $\delta \in (0, \delta_{0})$. 
  Then there exists an immortal Ricci flow solution initiated from $(M^{m}, g)$ satisfying
  \begin{align}
    &t |Rm|(x,t) \leq \psi(\delta|m), \label{eqn:RH19_2}\\
    &t^{-\frac{1}{2}} \cdot inj(x,t) \geq \psi^{-1}(\delta|m),  \label{eqn:RH19_3}
  \end{align}
  for every $x \in M$ and $t>0$.   Furthermore, for each pair of points $x,y \in M$  and $t>0$, the following distance distortion estimate holds:
  \begin{align}
    \left| \log \frac{d_{g(t)}(x,y)}{d_{g(0)}(x,y)} \right|<\psi(\delta|m) \cdot \left\{ 1+ \log_{+} \frac{\sqrt{t}}{d_{g(0)}(x,y)} \right\}.  
    \label{eqn:RH19_4}
  \end{align}
  \label{thm:RH19_5}
\end{theorem}

If we set $\delta=0$ in Theorem~\ref{thm:RH19_5}, then we return to the rigidity phenomenon discussed before. 
Theorem~\ref{thm:RH19_5} has a topological application.  Based on the curvature estimate (\ref{eqn:RH19_2}), the injectivity radius estimate (\ref{eqn:RH19_3}) 
and the distance distortion estimate (\ref{eqn:RH19_4}), 
we can construct(cf. Theorem~\ref{thm:CK21_2}) a canonical diffeomorphism from $M$ to the Euclidean space $\R^m$. 
Again, the existence of the diffeomorphism from $M$ to $\R^m$ was already known by Cheeger-Colding~\cite{CC}, 
via the Reifenberg method.  We expect our construction here to provide new perspectives on the Reifenberg method, which has far-reaching influences in geometric analysis
and is still quickly evolving right now(cf.~\cite{CJN} and~\cite{HuangHZ}). 

Comparing Theorem~\ref{thm:RH19_5} with its compact version Theorem~\ref{thm:RG26_3}, one immediately realizes that we miss the discussion of exponential convergence of the flow in 
Theorem~\ref{thm:RH19_5}.  By the curvature and injectivity radius estimates (\ref{eqn:RH19_2}) and (\ref{eqn:RH19_3}), we have the convergence of
$(M, x_0, g(t_{i}))$ to Euclidean space $(\R^{m}, 0, g_{E})$ in the pointed-Cheeger-Gromov topology, by taking subsequence if necessary. 
However, in Cheeger-Gromov convergence, one needs further modifications by diffeomorphisms, which make the convergence weaker than the convergence of the metric tensors in fixed coordinates,
as achieved by Theorem~\ref{thm:RG26_3}. 
Actually, in fixed coordinate charts, the immortal Ricci flow in Theorem~\ref{thm:RH19_5} must diverge unless it is the static Euclidean flow.  
Detailed analysis on the immortal Ricci flow(cf. Proposition~\ref{prn:RE14_12}) in Theorem~\ref{thm:RH19_5} shows that
\begin{align}
  \liminf_{t \to \infty} R(x,t) t \geq \psi(\delta|m),   \label{eqn:RI24_1}
\end{align}
and the above limit is independent of the choice of $x \in M$. Therefore, $dv_{g(t)}$ will degenerate as $t \to \infty$. 
The scalar curvature lower bound in (\ref{eqn:RI24_1}) is basically caused by the fact that the tangent cone of $(M,g)$ at
infinity is not isometric to $(\R^{m}, g_{E})$.  If we relax the non-negative Ricci curvature condition to non-negative scalar curvature condition,
then it is possible that the tangent cone of $(M,g)$ at infinity is isometric to $(\R^{m}, g_{E})$ and $(M,g)$ itself is non-flat. 
For example, when $(M,g)$ is an AE manifold with non-negative scalar curvature, the almost non-negativity of $\boldsymbol{\nu}$-functional can be combined with the work
of Y. Li~\cite{YuLi} to obtain the global existence and convergence of the flow in fixed coordinate charts.

The failure of the convergence of the immortal flow in Theorem~\ref{thm:RH19_5} suggests that the metric distortion estimate (\ref{eqn:RH19_4}) is optimal.
Note that if the right hand side of (\ref{eqn:RH19_4}) is uniformly bounded, then we have uniform metric bi-Lipschitz equivalence along the flow, which cannot hold by 
our previous discussion.
In fact, even on closed K\"ahler manifolds with Ricci curvature bounded from below and almost Euclidean initial condition on a fixed scale, 
we explicitly construct examples which show the failure of the bi-Lipschitz equivalence.
For more details, see Proposition~\ref{prn:RG12_5} and related discussion in Section~\ref{sec:globalstability}. 
Therefore, the single condition of the Ricci curvature lower bound is too weak for the metric bi-Lipschitz equivalence to hold along the Ricci flow.
Inspired by the fundamental works on complex Monge-Amp\`{e}re equation by Yau~\cite{Yau} and Chen-Cheng~\cite{ChenCheng1}, 
we are able to show that the  metric bi-Lipschitz equivalence holds under the almost Euclidean condition,
if we replace the lower bound of the Ricci curvature by the two-sided bound of the scalar curvature.
Based on the metric bi-Lipschitz equivalence, we can essentially improve the pseudo-locality theorem in the K\"ahler setting.

\begin{theorem}[\textbf{Pseudo-locality in the K\"ahler Ricci flow}]
  For each  $\alpha \in (0,1)$ and $Q>1$, there is a constant $\epsilon=\epsilon(n,\alpha,Q)$ with the following properties. 

  Suppose $\left\{ (M^{n}, g(t), J), 0 \leq t \leq 1 \right\}$ is a Ricci flow solution, $x_0 \in M$ and $r \in (0,1)$.
  Suppose 
    \begin{align}
       \begin{cases}
	 &\boldsymbol{\bar{\nu}}(B_{g(0)}(x_0, r), g(0), r^2) \geq -\epsilon;\\
	 &|R(x,0)| \leq r^{-2}, \quad \forall\; x \in B_{g(0)}(x_0, r).    
       \end{cases}
   \label{eqn:RG26_9}
   \end{align}
   Then there is a biholomorphic map $\Phi: B_{g(0)}(x_0, \epsilon r) \to U \subset  \C^n$ such that
  \begin{align}
    & B(0, Q^{-1}) \subset U \subset B(0, Q);  \label{eqn:RG26_6} \\
    & \norm{(\epsilon r)^{-2}\Phi_{*} g(t)}{C^{1,\alpha}(U)} < Q,  \;  \forall \; t \in [0, \epsilon^2 r^2];   \label{eqn:RG26_7}\\
    & |R|(x,t) < Q (\epsilon r)^{-2}, \quad \forall \; x \in U, \; t \in [0, \epsilon^2 r^2].    \label{eqn:RG27_1}
  \end{align}
\label{thm:RG26_8}
\end{theorem}

Note that the estimate (\ref{eqn:RG26_7}) implies that the metric $g(t)$ and $g(0)$ are uniformly equivalent in the $C^{1,\alpha}$-topology nearby $(x_0, 0)$.
In particular, we have the bi-Lipschitz equivalence
\begin{align*}
  \frac{1}{Q} g(y,0) \leq g(y,t) \leq Q g(y,0), \quad \forall\; y \in B_{g(0)}(x_0, \epsilon r), \quad t \in [0, \epsilon^2 r^{2}]. 
\end{align*}
Theorem~\ref{thm:RG26_8} provides us a method to construct uniform holomorphic charts from a very weak assumption (\ref{eqn:RG26_9}).
Note that the lower bound of $\boldsymbol{\bar{\nu}}$ in (\ref{eqn:RG26_9}) can be achieved if the underlying local isoperimetric constant is very close to the Euclidean one.
Under such a condition and Ricci curvature lower bound, similar construction of holomorphic charts was developed in Liu-Sz\'{e}kelyhidi~\cite{LiuSze}.
The key improvements of Theorem~\ref{thm:RG26_8} are the estimates (\ref{eqn:RG26_7}) and (\ref{eqn:RG27_1}), which imply the following theorem. 

\begin{theorem}[\textbf{Preservation of two-sided scalar curvature bound}]
  Suppose $(M_i^{n}, g_i, J_{i})$ is a sequence of K\"ahler manifolds with uniformly bounded scalar curvatures and diameters. 
  Suppose this sequence converges to a smooth Riemannian manifold $(N, g)$ in $C^{0}$-topology in the sense that there exist diffeomorphisms
  $f_{i}: N \to M_i$ such that 
  \begin{align}
    \lim_{i \to \infty} \norm{f_{i}^{*}g_i-g}{C^{0}(g)}=0. \label{eqn:RH16_1} 
  \end{align}
  Then $N$ admits a natural smooth K\"ahler structure $J$. 
  Furthermore, both the upper and lower bounds of the scalar curvature are continuous under the convergence. 
  Namely, if the scalar curvature satisfies  $\sigma_{a,i} \leq R_{g_{i}} \leq \sigma_{b,i}$ for each $i$, then we have
  \begin{align}
    \sigma_{a} \leq R_{g} \leq \sigma_{b},    \label{eqn:RH16_2}
  \end{align}
  where
  \begin{align}
    \sigma_{a} \coloneqq \limsup_{i \to \infty} \sigma_{a,i}, \quad \sigma_{b} \coloneqq \liminf_{i \to \infty} \sigma_{b,i}.  \label{eqn:RH16_3}
  \end{align}
  If $\sigma_{a}=\sigma_{b}$, then $(N,g,J)$ is a smooth cscK(constant scalar curvature K\"ahler) manifold. 
  \label{thm:RH16_1}
\end{theorem}

The preservation of scalar curvature lower bound along convergence has been well-studied, for example, in Lott~\cite{Lott}, Honda~\cite{Honda}, etc. 
See also Sormani~\cite{Sormani16} for a survey.   
Without any curvature assumption, the preservation of scalar curvature lower bound under $C^{0}$-convergence was first proved by Gromov~\cite{Gmv} in the Riemannian setting.   
An alternative proof was later provided by Bamler~\cite{Bamler}, using the Ricci flow. Theorem~\ref{thm:RH16_1} indicates that in the K\"ahler setting, both the lower and upper bounds of the scalar curvature 
are preserved under the $C^{0}$-convergence. This is reasonable since the scalar curvature in K\"ahler geometry carries much more information than that in Riemannian geometry.
Calabi~\cite{Calabi} initiated the study of cscK metrics and introduced the Calabi flow to search for such metrics. 
In~\cite{Don04}, Donaldson pointed out that in the space of all complex structures compatible with a given metric form, the scalar curvature function can be regarded as a moment map.
The famous Yau-Tian-Donaldson conjecture predicts that the existence of cscK metrics is equivalent to the stability of complex structures.
For more information in this direction, see the most recent breakthrough of Chen-Cheng(cf.~\cite{ChenCheng1},~\cite{ChenCheng2},~\cite{ChenCheng3} and the references therein).   
We hope Theorem~\ref{thm:RH16_1} could shed some light on the problem of finding cscK metrics.\\

We now briefly describe how to prove the aforementioned theorems. 

\begin{proof}[Outline proof of Theorem~\ref{thmin:ML14_2}:] 
  We shall basically follow the ideas of Perelman~\cite{Pe1}, with some technical differences.  One key difference is the size of the domain where the curvature and volume estimates hold.
  Similar to~\cite{TiWa}, the balls we concern here are $B_{g(t)}(x_0, \sqrt{t})$,  rather than $B_{g(t)}(x_0, \epsilon r_0)$ as in Theorem 10.1 of Perelman~\cite{Pe1}. 
  Such difference is crucial for the subsequent blowup argument under the condition (\ref{eqn:MK30_1}) and is very helpful to simplify the steps of the proof(cf. Remark~\ref{rmk:RJ04_1}).

  We argue by contradiction.
  We first prove (\ref{eqn:ML28_1}) for $\alpha=A$ sufficiently large.  By continuity, (\ref{eqn:ML28_1}) holds for very small $t$. Let $t_0$ be the first time that they start to fail.  
  So there exists a point $(y_0, t_0)$ with $|Rm|(y_0,t_0)=\alpha$ and $y_0 \in B_{g(t)}(x_0,\sqrt{t_0})$.   Via an auxiliary function, we then obtain a local(in the space-time) maximum curvature point
  $(z_0, s_0)$ such that $s_0 \in [0.5 t_0, t_0]$,  $z_0 \in B_{g(s_0)}(x_0, 10 \sqrt{s_0})$. 
  According to the local functional monotonicity, the almost Euclidean $\boldsymbol{\mu}$-functional at time $t=0$ implies an almost Euclidean $\boldsymbol{\nu}$-functional value
  of the proper neighborhood containing  $(z_0, s_0)$. 
  After a proper rescaling, the point $(z_0, s_0)$ has $|Rm|$ being equal to $1$ and nearby curvature all bounded by $2$.
  By scaling invariance of $\boldsymbol{\nu}$, the neighborhood of $(z_0, s_0)$ has $\boldsymbol{\nu}$ value very close to $0$ also. 
  Recall that the lower bound of $\boldsymbol{\nu}$ implies the lower bound of volume ratio.   Therefore we can take a limit of the neighborhoods around $(z_0, s_0)$ and obtain a complete Riemannian manifold with 
  $\boldsymbol{\nu} \geq 0$, which must be  Euclidean. However, the limit cannot be flat since its base point has curvature $1$. Contradiction.  

  We then proceed to prove (\ref{eqn:ML28_1}) for arbitrary small $\alpha \in (0, \frac{1}{100m})$. 
  Note that $\boldsymbol{\mu}$ is monotone for the containment relationship.  Applying the first step to all balls inside the given initial ball, for each point $(x,t)$ satisfying  
  $d_{g(t)}(x, \partial B_{g(0)}(x_0, r_0))> A \sqrt{t}$, we obtain rough curvature estimates.  Another blowup argument implies that the curvature bound $t|Rm|(\cdot, t)<A$ can be
  improved to $t|Rm|(\cdot, t)<\alpha$ when $d_{g(t)}(x,x_0)<\sqrt{t}$. Since the blowup limit is Euclidean,  the argument also implies the volume estimate (\ref{eqn:ML28_2}) and
  the injectivity radius estimate (\ref{eqn:RH27_4}). 

  For full details of the proof, see Section~\ref{sec:rigidity} and Section~\ref{sec:pseudo}. 
\end{proof}

\begin{proof}[Outline proof of Theorem~\ref{thm:RG05_3}:]
  By the volume continuity of Colding~\cite{Co}, we know that for large $i$, there is a uniform scale $r_0$ such that each geodesic ball under this scale has almost Euclidean volume ratio.
  Then one can apply Theorem~\ref{thmin:ML14_2} to obtain a uniform existence time of the flow, say on $[0, 1]$.
  Fix a small positive number $\xi_0$.  
  By the curvature estimate (\ref{eqn:ML28_1}) and volume estimate (\ref{eqn:ML28_2}), it is clear that each unit geodesic ball under $\xi_{0}^{-1} g(\xi_0)$ 
  is very close to the unit ball in the Euclidean space.
  The curvature condition also forces the distance to be almost expanding along the flow.  
  Together with the almost non-negativity of the scalar curvature, we see that each unit ball of 
  the metric $\xi_{0}^{-1}g(0)$ has almost Euclidean volume ratio and almost non-negative Ricci curvature.    These bounds imply a local almost Einstein condition in the sense of Tian-Wang~\cite{TiWa}.
  Following the argument in Section 4 of~\cite{TiWa}, we obtain a distance distortion estimate(cf. Chen-Rong-Xu~\cite{CRX} also) between $\xi_{0}^{-1}g(\xi_{0})$ and $\xi_{0}^{-1}g(0)$.    
  Note that $(M, \xi_{0}^{-1}g(0))$ converges to $(N, \xi_{0}^{-1} h(0))$. So we have Gromov-Hausdorff approximation $\hat{F}: (M, \xi_{0}^{-1}g(0)) \to (N, \xi_{0}^{-1} h(0))$
  and $\hat{G}: (N, \xi_{0}^{-1} h(0)) \to  (M, \xi_{0}^{-1}g(0))$.
  By the metric distortion estimate on these two Ricci flows, they induce Gromov-Hausdorff approximations $\tilde{F}$ and $\tilde{G}$ between $(M, \xi_{0}^{-1}g(\xi_0))$ and $(N, \xi_{0}^{-1}h(\xi_{0}))$,
  both of which are almost flat and have almost Euclidean volume ratio on scale $10$. Then standard center of mass argument can be applied to deform $\tilde{G}$ to $G$, which is a diffeomorphism
  from $N$ to $M$ such that $G^{*}\left( \xi_{0}^{-1} g(\xi_0) \right)$ is very close to $\xi_{0}^{-1} h(\xi_{0})$ in the $C^{0}$-sense.  
  In particular, $G^{*}\left(  g(\xi_0) \right)$ is close to $h(\xi_{0})$ in $C^{0}$-norm. 
  Since the flow on $N$ is smooth, we know that $h(\xi_0)$ has to be very close to $h(0)$ in $C^{0}$-sense also.  Therefore, we obtain the $C^{0}$-closeness between $g(\xi_0)$ and $h(0)$.
  Then we run the Ricci-DeTurck flow from $g(\xi_0)$ and $h(0)$ respectively, it follows from the strict parabolicity of the Ricci-DeTurck flow that such $C^0$-closeness is preserved until time $T$, 
  up to a multiplicative constant.
  Finally, we construct diffeomorphisms which pull back the Ricci-Deturk flows to the Ricci flows. This construction is standard in the proof of the uniqueness of
  the Ricci flow solution.   The proximity of Ricci-DeTurck flows then implies the closeness of the Ricci flows in $C^{0}$-sense. 
  The high-order derivatives approximation of the Ricci flow then follows from Shi's estimate(cf.~\cite{Shi}).

  The complete details of the proof can be found in Section~\ref{sec:cdepend}. 
\end{proof}

\begin{proof}[Outline proof of Theorem~\ref{thm:RG26_3}:] 

  It follows from the work of Colding(cf.~\cite{Co1} and ~\cite{Co2}) that the condition (\ref{eqn:RG26_1}) implies that $d_{GH}\left( (M^{m},g), (S^{m}, g_{round}) \right)$ is very small. 
  Applying Theorem~\ref{thmin:ML14_2} nearby the static normalized Ricci flow initiated from $(S^{m}, g_{round})$, we see that $g(1)$ is very close to $g_{round}$ in
  the $C^{5}$-topology. In particular, the curvature operator of $g(1)$ is sufficiently pinched. Therefore, one can apply Hamilton's work(cf.~\cite{Ha82},~\cite{Ha86} if $\dim \leq 4$)
  and Huisken's work(cf.~\cite{Huisken85} if $\dim \geq 5$) to obtain the exponential convergence of the normalized Ricci flow initiated from $g(1)$.
  Alternatively, one can apply other related convergence results(cf.~B\"ohm-Wilking~\cite{BoWil} and Brendle-Schoen~\cite{BrSch}, etc.) here to draw the same conclusion. 

  The same argument works for every $g(\xi)$ for each fixed $\xi$. Then the metric $g(\xi)$ and $g(\infty)$ are bi-Lipschitz with bi-Lipschitz constant almost $1$.
  The metric distortion estimate (\ref{eqn:RG26_2}) is basically induced by
  the distortion between $g(\xi)$ and $g=g(0)$, which can be estimated by the pseudo-locality theorem and the local almost Einstein condition implied by the given condition. 

  The detailed proof is contained in Section~\ref{sec:globalstability}. 
\end{proof}

\begin{proof}[Outline proof of Theorem~\ref{thm:RG26_5}:]  
  Note that condition (\ref{eqn:RG26_4}) implies that $(M^{n}, J)$ is biholomorphic to $\CP^{n}$(cf.~Zhang~\cite{KZhang}). As the metric class is proportional to $2\pi c_1(M, J)$, the volume condition 
  forces $g$ and $g_{FS}$ to locate in the same K\"ahler class.  In the current situation, the normalized Ricci flow initiated from $g$ is the standard Fano K\"ahler Ricci flow and thus  
  it exists immortally(cf. Cao~\cite{HDC}). The difficulty here is to obtain the metric distortion estimate (\ref{eqn:RG26_2}).
  Compared with the Riemannian version Theorem~\ref{thm:RG26_3}, an equivalence theorem similar to Colding's(cf.~\cite{Co1} and~\cite{Co2}) is absent here.
  We shall apply the techniques in Tian-Wang~\cite{TiWa} and Chen-Wang~\cite{CW17B} to overcome this difficulty.   In fact, by an improvement of Perelman's estimate(cf. Jiang~\cite{Jiang}), we can apply
  Chen-Wang's compactness theorem for the flow space time $M \times [0.5, 2]$.  As $\delta \to 0$, we know this flow space-time converges to a static(possibly singular) Einstein space-time, which can only be 
  the standard Fubini-Study space-time on $\CP^{n}$ by the partial-$C^{0}$-estimate and the stability condition satisfied by $\CP^{n}$(cf.~\cite{CDS1},~\cite{CDS2},~\cite{CDS3} and~\cite{CSW}).
  Consequently, the degeneration must be smooth. Therefore, if $\delta$ is sufficiently small,  the flow has positive bisectional curvature  on $M \times [0.5, 2]$.  
  Then the global exponential convergence of the flow follows from Chen-Tian~\cite{ChenTian}(See also Phong-Song-Weinkove~\cite{PhSoWe}).
  As $\delta \to 0$, we can obtain an almost Einstein sequence in the sense that $\int_{0}^{1}\int_{M} |Rc-2(n+1)g|dvdt \to 0$.  At each time $t_0 \in (0,1)$, we know from the previous argument that $g(t_0)$ converges
  smoothly to $g_{FS}$. For time $t=0$, we know $g(0)$ converges to a possibly singular Einstein manifold by Tian-Wang~\cite{TiWa}. A gap argument then shows that the limit of $g(0)$ must be smooth and
  be isometric to $g_{FS}$ also. This forces the almost Euclidean condition (\ref{eqn:RH17_4}) to be satisfied on a fixed scale. 
  Then we can apply pseudo-locality, i.e., Theorem~\ref{thmin:ML14_2}, and the almost Einstein condition
  to obtain the distance distortion estimate, in the spirit of section 4 of~\cite{TiWa}. 

  The detailed proof can be found in Section~\ref{sec:globalstability}. 
\end{proof}

\begin{proof}[Outline proof of Theorem~\ref{thm:RH19_5}:] 
  This basically follows from the application of the pseudo-locality theorem and the Hochard's construction~\cite{Hochard}.   Similar results were already obtained by F. He~\cite{HeFei}.  
  The Ricci flow is known to have a solution existing for a fixed short period of time, say on $[0, 1]$. The key point now is to extend this solution to an immortal solution. 
  We realize this by taking advantage of our version of pseudo-locality in terms of $\boldsymbol{\nu}$,  condition (\ref{eqn:RH17_4}). 
  Note that by our local functional monotonicity proved in~\cite{BWang17local}, it is not hard to see that $\boldsymbol{\nu}(M, g(t))$ is non-decreasing.  
  Therefore, at time $t=1$, we still have $\boldsymbol{\nu}(M, g(1)) \geq -\epsilon$, which guarantees the flow to exist for another unit time peoriod. 
  Repeating such steps, we obtain an immortal Ricci flow solution.   Then the curvature estimate (\ref{eqn:RH19_2}) and the injectivity radius estimate (\ref{eqn:RH19_3}) 
  follow naturally from (\ref{eqn:ML28_1}) and (\ref{eqn:RH27_4}). 
  We remark that the constants in (\ref{eqn:RH19_2}) and (\ref{eqn:RH19_3}) being sufficiently small is very important for our construction of the canonical diffeomorphism from $M$ to $\R^{m}$. 

  The readers are referred to Section~\ref{sec:topstable} for the complete proof and related further discussions. 
\end{proof} 

\begin{proof}[Outline proof of Theorem~\ref{thm:RG26_8}:]
The scalar curvature bound basically follows from the maximum principle and the evolution equation of $R$ and $R+|\nabla \dot{\varphi}|^2$ where $\dot{\varphi}$ is the local Ricci potential.
The $C^{1,\alpha}$-estimate of metrics follows from the scalar curvature bound together with the metric uniform bi-Lipschitz bound. 
The key difficulty is to show the uniform metric bi-Lipschitz equivalence in a short time period near the center of the ball.  
This equivalence is established through the combination of distance distortion estimates originated from Tian-Wang~\cite{TiWa} and the regularity estimates for K\"ahler metrics with bounded scalar
curvature in Chen-Cheng~\cite{ChenCheng1}. We shall focus on this key difficulty in this outline.  

 By the pseudo-locality theorem in the Riemannian setting, one should be able to construct uniform holomorphic charts(cf.~\cite{LiuSze}) for $t$ small. 
 We need to show at this time the metric $g(t)$ and $g(0)$ are uniformly bi-Lipschitz near the center part.  
 The proof follows from a contradiction argument. Suppose not, we can choose $t_0$ to be the first time such that the metric $g(t_0)$ is $(1+\xi)$-bi-Lipschitz equivalent to $g(0)$ at $(x_0, t_0)$ 
 for some fixed small $\xi$. Then we rescale to let $t_0=1$ and let $\epsilon \to 0$. The blowup limit of the initial balls will be a Euclidean space, in light of the assumed metric equivalence, 
 and the elliptic PDE estimates for bounded scalar curvature metric in a given holomorphic coordinate chart.   The elliptic PDE estimates together with the metric distortion induced from the almost Einstein 
 condition forces the identity map  $(M, g(1)) \mapsto (M, g(0))$ to converge to an isometry.  On the one hand, the metric $g(1)$ converges in smooth topology 
 to $\delta_{k\bar{l}}$, the standard Euclidean metric on $\C^{n}$.  On the other hand,  the metric $g(0)$ converges in $C^{1,\alpha}$-topology to a limit metric $\hat{g}$ on $\C^{n}$  which satisfies
 $\det \hat{g}=1$ on $\C^{n}$. The Evans-Krylov estimate for complex Monge-Amp\`{e}re equation then implies that $\hat{g}$ must be a constant matrix valued function on $\C^{n}$. 
 Applying the metric distortion estimate again, one obtains that $\hat{g}$ must also be $\delta_{k \bar{l}}$.  Since the convergence topology is $C^{1,\alpha}$,  we obtain that 
 the equivalence constant between $g(x_0, 0)$ and $g(x_0, 1)$ can be arbitrarily close to $1$, which contradicts the assumption. 

 The complete proof is written in more detail in Section~\ref{sec:kpseudo}.  
\end{proof}

\begin{proof}[Outline proof of Theorem~\ref{thm:RH16_1}:] 
  The $C^{0}$-convergence implies that under a fixed scale $r_0$, all the manifolds we consider have almost Euclidean isoperimetric constants.  
  In particular, we are able to apply the pseudo-locality theorem to run the K\"ahler Ricci flow initiated from $(M_i, g_i)$ for a fixed period of time $[0, \xi_0]$.
  Applying Theorem~\ref{thm:RG26_8}, we know $(M_i, g_i(\xi_0), J_i)$ converges to a smooth K\"ahler manifold $(M_{\infty}, g_{\infty}(\xi_{0}), J_{\infty})$ and
  $g_{i}$ converges in $C^{1,\alpha}$-topology to a limit metric $g_{\infty}(0)$, which is at least $C^{1,\alpha}$.  The Riemannian manifold $(M_{\infty}, g_{\infty}(0))$
  is isometric to the smooth Riemannian manifold $(N, g)$.  
  Since the holomorphic charts are automatically harmonic and the metrics have the best regularity under the harmonic charts, we know that $(M_{\infty}, g_{\infty}(0), J_{\infty})$ is a smooth K\"ahler manifold. 
  There is a smooth diffeomorphism $F: N \to M_{\infty}$ such that $J=F^{*} J_{\infty}$, $g=F^{*} g_{\infty}(0)$.

  For each $t \in [0, \xi_0]$, we show by the maximum principle that $\sigma_{a}-C t^{\beta} \leq R_{i} \leq \sigma_{b}+C t^{\beta}$ for some uniform constant $C$ and $\beta \in (0,1)$.
  Then the scalar curvature bounds naturally follow from the commutativity of taking $C^{0}$-limits with running the K\"ahler Ricci flow, and the aforementioned scalar curvature bounds along the evolution. 

  For further details of the proof and other related results, see Section~\ref{sec:kpseudo}. 
\end{proof}

Some notations frequently used in this paper are listed below with pointers to their definitions. \\

\noindent 
\textbf{List of important notations}

\begin{itemize}
\item $\boldsymbol{\mu}$, $\boldsymbol{\nu}$: local functionals with scalar curvature. Defined in (\ref{eqn:MJ16_1}).
\item $\bar{\boldsymbol{\mu}}$, $\bar{\boldsymbol{\nu}}$: local functionals without scalar curvature. Defined in (\ref{eqn:MJ16_2}). 
\item $\mathbf{er}_{\epsilon_0}$, $\mathbf{Ir}_{\epsilon_0}$, $\mathbf{vr}_{\epsilon_0}$: entropy radius, isoperimetric radius and volume radius. Defined in Definition~\ref{dfn:RH26_3}. 
\item $\mathbf{I}$: isoperimetric constant.  First appears in Lemma~\ref{lma:MJ25_1}.  
\item $m$: the default  real dimension of a Riemannian manifold $(M, g)$.
\item $\omega_m$: the volume of standard unit ball in $\R^m$.  Defined in Theorem~\ref{thm:RG26_3}. 
\item $\alpha_{m}$: a small dimensional constant defined in (\ref{eqn:RJ13_3}). 
\item $n$: the default complex dimension of a K\"ahler manifold $(M, g, J)$.   
\item $\Omega_{n}$: the volume of $(\CP^{n}, g_{FS})$.  Defined in Theorem~\ref{thm:RG26_3}. 
\item $\psi(\epsilon|m,  \gamma_{1}, \cdots, \gamma_{k})$: a small constant depending on parameters $m, \gamma_{1}, \cdots, \gamma_{k}$ and $\epsilon$ 
  such that $\displaystyle \lim_{\epsilon \to 0}\psi(\epsilon|m,  \gamma_{1}, \cdots, \gamma_{k})=0$.  First appears in Theorem~\ref{thm:RG26_3}.  
\item $f_{+}$: $\max\left\{ f,0 \right\}$.  Similar notation is defined for other functions. Defined in Theorem~\ref{thm:RD10_1}. 
\item $\log_{+}$: $\max \left\{ \log, 0 \right\}$. Defined in Proposition~\ref{prn:RG08_1}. 
\item $\square$:  Heat operator $\partial_t-\Delta$. Defined in Theorem~\ref{thm:RD10_1}.  
\item $cscK$: constant scalar curvature K\"ahler.  Defined in Theorem~\ref{thm:RH16_1}.  
\item $\mathcal{M}_{cscK}(n,r_0,\bar{\epsilon}, \sigma, D)$: moduli of cscK manifolds with proper bounds.  Defined in Theorem~\ref{thm:RG23_1}. 
\end{itemize}

 This paper is organized as follows. 
 In Section~\ref{sec:pre}, we list some elementary estimates of the Ricci flow and many properties of the local functionals.  
 Most of these properties were already proved in~\cite{BWang17local} except for some minor improvements. 
 In Section~\ref{sec:rigidity}, we prove several rigidity theorems as the preparation for the pseudo-locality theorem.   In particular, we generalize Anderson's gap theorem for Ricci-flat metrics to gap theorems for
 various critical metrics, including Einstein metrics and cscK metrics, in terms of the $\boldsymbol{\nu}$-functionals.  In Section~\ref{sec:pseudo}, we prove Theorem~\ref{thmin:ML14_2} and
 discuss its relationship with the known pseudo-locality theorems. 
 In Section~\ref{sec:topstable}, we show distance distortion estimate and topological stability, based on the almost-Einstein techniques in~\cite{TiWa}. 
 We also discuss the asymptotic behavior of the immortal Ricci flow in Theorem~\ref{thm:RH19_5} in detail.  
 In Section~\ref{sec:cdepend}, we prove the continuous dependence of the Ricci flow with respect to the initial metric in Gromov-Hausdorff topology, i.e., Theorem~\ref{thm:RG05_3}.
 In Section~\ref{sec:globalstability}, we extend the continuous dependence to the immortal Ricci flow solutions, with which we prove Theorem~\ref{thm:RG26_3} and Theorem~\ref{thm:RG26_5}.  We also 
 construct examples where the bi-H\"older distance estimate cannot be improved to bi-Lipschitz distance estimate.
 Then in Section~\ref{sec:kpseudo}, we prove the K\"ahler version of the pseudo-locality theorem, i.e., Theorem~\ref{thm:RG26_8}, and apply it to prove Theorem~\ref{thm:RH16_1}.
 Finally, we discuss the future study and make several conjectures in Section~\ref{sec:further}.\\

{\bf Acknowledgements}:  
This paper is partially supported by NSFC-11971452 and research fund from USTC.
The author would like to thank Weiyong He,  Shaosai Huang, Haozhao Li, Yu Li and Kai Zheng for helpful discussions.
Particular thanks go to Shaosai Huang and Yu Li for carefully reading the early versions of this manuscript and
providing many helpful suggestions to improve the exposition.

\section{Preliminaries}
\label{sec:pre}

In this paper, we study Ricci flows on Riemannian manifolds of dimension $m \geq 2$. For convenience of notation, we define
\begin{align}
 \alpha_{m} \coloneqq \frac{1}{100m}.   \label{eqn:RJ13_3}
\end{align}
Each Ricci flow $\left\{ (M, g(t)), 0 <t <T \right\}$ has bounded $|Rm|$ on every compact subset $I$ of $(0,T)$. 
Namely, we always assume 
\begin{align}
  \sup_{x \in M, t \in I} |Rm|(x,t)<\infty, \label{eqn:RJ18_1} 
\end{align}
where $|Rm|(x,t)$ means $|Rm|_{g(t)}(x)$.

The local entropy functionals are the basic tools we use.  
For the convenience of the readers, we shall quickly review their definitions and elementary properties in this section.
Let $(M, g)$ be a complete Riemannian manifold of dimension $m$, and $\Omega$ be a connected, open subset of $M$ with smooth boundary.
Then we can regard $(\Omega, \partial \Omega, g)$ as a smooth manifold with boundary.  Let $\boldsymbol{a}$ be a smooth function on $\bar{\Omega}$, and $\tau$ be a positive constant.
Then we define
\begin{align}
  &\mathscr{S}(\Omega) \coloneqq \left\{ \varphi \left| \varphi \in W_0^{1,2}(\Omega), \quad \varphi \geq 0,  \quad \int_{\Omega} \varphi^2 dv=1 \right. \right\}, \label{eqn:MJ16_a}\\
  &\mathcal{W}^{(\boldsymbol{a})}(\Omega, g, \varphi, \tau) \coloneqq -m -\frac{m}{2} \log (4\pi \tau) + \int_{\Omega} \left\{ \tau\left(\boldsymbol{a}\varphi^2+4\left|\nabla \varphi \right|^2 \right)-2\varphi^2 \log \varphi \right\}dv, \label{eqn:MJ16_b} \\
  &\boldsymbol{\mu}^{(\boldsymbol{a})} \left( \Omega, g, \tau \right) \coloneqq \inf_{\varphi \in \mathscr{S}(\Omega)} \mathcal{W}^{(\boldsymbol{a})}(\Omega, g, \varphi, \tau), \label{eqn:MJ16_c}\\
  &\boldsymbol{\nu}^{(\boldsymbol{a})} \left( \Omega, g, \tau \right) \coloneqq \inf_{s \in (0, \tau]} \boldsymbol{\mu}^{(\boldsymbol{a})} \left( \Omega, g, s \right),  \label{eqn:MJ16_d}\\
  &\boldsymbol{\nu}^{(\boldsymbol{a})} \left( \Omega, g\right) \coloneqq \inf_{\tau \in (0, \infty)} \boldsymbol{\mu}^{(\boldsymbol{a})} \left( \Omega, g, \tau \right).  \label{eqn:MJ16_e}
\end{align} 
For each smooth function $\boldsymbol{a}$ and positive number $\tau>0$, 
$\boldsymbol{\mu}^{(\boldsymbol{a})}(\Omega, g, \tau)$ is achieved by a function $\varphi \in W_{0}^{1,2}(\Omega)$ whenever $\Omega$ is bounded.    
Moreover, $\varphi$ is positive and smooth in $\Omega$, and $\varphi$ satisfies the following Euler-Lagrangian equation
\begin{align}
  -4\tau \Delta \varphi  + \tau \boldsymbol{a} \varphi -2\varphi \log \varphi - \left( \boldsymbol{\mu}^{(\boldsymbol{a})} +m+\frac{m}{2}\log (4\pi \tau) \right) \varphi=0.
  \label{eqn:GH01_1}  
\end{align}
We call $\varphi$ as the minimizer function of $\boldsymbol{\mu}^{(\boldsymbol{a})}(\Omega, g, \tau)$.  
The cases $\boldsymbol{a}=0$ and $R$ are the most important cases for the application of $\boldsymbol{\mu}^{(\boldsymbol{a})}$ in the study of the Ricci flow.
For simplicity of notation, we define
\begin{align}
   &\boldsymbol{\mu} \coloneqq \boldsymbol{\mu}^{(R)}, \quad \boldsymbol{\nu} \coloneqq  \boldsymbol{\nu}^{(R)};   \label{eqn:MJ16_1}\\
   &\bar{\boldsymbol{\mu}} \coloneqq  \boldsymbol{\mu}^{(0)}, \quad \bar{\boldsymbol{\nu}} \coloneqq  \boldsymbol{\nu}^{(0)}.  \label{eqn:MJ16_2}
\end{align}

Now we list the elementary properties of the local functionals, which will be repeatedly used in this paper.
Except Proposition~\ref{prn:RG08_1}, part of Proposition~\ref{prn:RH31_1} and part of Theorem~\ref{thm:CA02_3},  
all the results listed in this section were already known or proved in~\cite{BWang17local}.

\begin{proposition}[\textbf{Monotonicity induced by inclusion}, cf. Proposition 2.1 of~\cite{BWang17local}]
  Suppose $\Omega_1, \Omega_2$ are bounded domains of $M$ satisfying $\Omega_1 \subsetneq \Omega_2$.  Then we have
  \begin{align}
   &\boldsymbol{\mu}^{(\boldsymbol{a})}(\Omega_1, \tau) > \boldsymbol{\mu}^{(\boldsymbol{a})}(\Omega_2, \tau),  \label{eqn:CA01_1}\\
   &\boldsymbol{\nu}^{(\boldsymbol{a})}(\Omega_1, \tau) \geq \boldsymbol{\nu}^{(\boldsymbol{a})}(\Omega_2, \tau).  \label{eqn:CA01_2}
  \end{align}
\label{prn:CA01_2}  
\end{proposition}

\begin{proposition}[\textbf{Non-positivity of $\boldsymbol{\nu}^{(\boldsymbol{a})}$}, cf. Proposition 2.2 of~\cite{BWang17local}]
For each bounded domain $\Omega \subset M$, we have
\begin{align}
  \boldsymbol{\nu}^{(\boldsymbol{a})}(\Omega, \tau) \leq 0.     \label{eqn:CA01_3}
\end{align}
\label{prn:CA01_3}
\end{proposition}

\begin{proposition}[\textbf{Continuity of $\boldsymbol{\mu}^{(\boldsymbol{a})}$}, cf. Proposition 2.3 of~\cite{BWang17local}]
 Suppose $D$ is a possibly unbounded domain of $M$ with an exhaustion $D=\cup_{i=1}^{\infty} \Omega_i$ by bounded domains. 
 In other words, we have 
 \begin{align*}
    \Omega_1 \subset \Omega_2 \subset \cdots \subset \Omega_k \subset \cdots \subset D
 \end{align*}
 and each $\Omega_i$ is a bounded domain.  Then for each $\tau>0$ we have 
\begin{align}
  \boldsymbol{\mu}^{(\boldsymbol{a})}(D, \tau)= \lim_{i \to \infty} \boldsymbol{\mu}^{(\boldsymbol{a})}(\Omega_i, \tau).   \label{eqn:CA02_3}
\end{align}
\label{prn:CA02_1}
\end{proposition}

\begin{corollary}[\textbf{Continuity of local-$\boldsymbol{\mu}$-functional under $C^{\infty}$-Cheeger-Gromov convergence}, cf. Corollary 2.5 of~\cite{BWang17local}]
 Suppose $(M_i^m, x_i, g_i)$ is a sequence of pointed Riemannian manifolds such that
 \begin{align}
  (M_i^m, x_i, g_i)  \longright{C^{\infty}-Cheeger-Gromov} (M_{\infty}^m, x_{\infty}, g_{\infty}).     \label{eqn:CB19_4}
 \end{align}
 Suppose $\Omega_{\infty}$ is a bounded open set in $M_{\infty}$ with smooth boundary, $\Omega_i \subset M_i$ are the sets converging to $\Omega_{\infty}$
 along the convergence (\ref{eqn:CB19_4}).  Then for each $\tau>0$ we have
 \begin{align}
    \lim_{i \to \infty} \boldsymbol{\mu}(\Omega_i, g_i, \tau)=\boldsymbol{\mu}(\Omega_{\infty}, g_{\infty}, \tau).  \label{eqn:CB19_6}
 \end{align}
\label{cly:CB19_2} 
\end{corollary}

\begin{lemma}[cf. Lemma 3.1 of~\cite{BWang17local}]
  Suppose $-\underline{\Lambda} \leq R \leq \bar{\Lambda}$ on $\Omega$. Then for each $\tau>0$ we have
 \begin{align}
  &\boldsymbol{\mu}(\Omega,\tau) - \bar{\Lambda} \tau \leq \bar{\boldsymbol{\mu}}(\Omega,\tau) \leq \boldsymbol{\mu}(\Omega,\tau) +\underline{\Lambda} \tau,  \label{eqn:MJ16_6}\\
  &\boldsymbol{\nu}(\Omega,\tau) -\bar{\Lambda}  \tau \leq \bar{\boldsymbol{\nu}}(\Omega,\tau) \leq \boldsymbol{\nu}(\Omega,\tau)+ \underline{\Lambda} \tau.  \label{eqn:MJ16_7}
 \end{align}
\label{lma:MJ16_1} 
\end{lemma}

\begin{proposition}[cf. Proposition 3.2 of~\cite{BWang17local}]
    Consequently, for each $0<\tau_1<\tau_2$,  we have
\begin{align}
  &\bar{\boldsymbol{\mu}}(\Omega, \tau_1) \leq \bar{\boldsymbol{\mu}}(\Omega, \tau_2) + \frac{m}{2} \log \frac{\tau_2}{\tau_1}, \label{eqn:RH31_2} \\
  &\boldsymbol{\mu}(\Omega, \tau_1) \leq \boldsymbol{\mu}(\Omega, \tau_2) + \frac{m}{2} \log \frac{\tau_2}{\tau_1} +\bar{\Lambda} \tau_1+\underline{\Lambda} \tau_2, \label{eqn:RH31_3} \\
  &\bar{\boldsymbol{\nu}}(\Omega, \tau_2) \leq  \bar{\boldsymbol{\nu}}(\Omega, \tau_1) \leq \bar{\boldsymbol{\nu}}(\Omega, \tau_2) + \frac{m}{2} \log \frac{\tau_2}{\tau_1},  \label{eqn:CF10_1}\\     
  &\boldsymbol{\nu}(\Omega, \tau_2) \leq  \boldsymbol{\nu}(\Omega, \tau_1) \leq \boldsymbol{\nu}(\Omega, \tau_2) + \frac{m}{2} \log \frac{\tau_2}{\tau_1} +\bar{\Lambda} \tau_1+\underline{\Lambda} \tau_2.  \label{eqn:CF10_3}   
\end{align} 
\label{prn:RH31_1}  
\end{proposition}

\begin{proof}
  Note that (\ref{eqn:CF10_1}) and (\ref{eqn:CF10_3}) were already proved in Proposition 3.2 of~\cite{BWang17local}. The inequalities (\ref{eqn:RH31_2}) and (\ref{eqn:RH31_3})
  were implied by the proof of Proposition 3.2~\cite{BWang17local}. We provide the proof here for the convenience of the reader.  

  We shall first prove (\ref{eqn:RH31_2}), which means that the quantity $\bar{\boldsymbol{\mu}}(\Omega, \tau) + \frac{m}{2} \log \tau$ is a monotonically non-decreasing function of $\tau$. 
  Let $\varphi$ be the minimizer function of $\bar{\boldsymbol{\mu}}(\Omega, \tau_2)$.  Then the Euler-Lagrange equation (\ref{eqn:GH01_1}) reads as
  \begin{align*}
   \left( \bar{\boldsymbol{\mu}}(\Omega, \tau_2) +m+\frac{m}{2}\log (4\pi \tau_2) \right) \varphi=-4\tau_2 \Delta \varphi -2\varphi \log \varphi. 
  \end{align*}
  Multiplying both sides of the above equation by $\varphi$ and integrating on $\Omega$, we obtain
  \begin{align*}
     \bar{\boldsymbol{\mu}}(\Omega, \tau_2) +m+\frac{m}{2}\log (4\pi \tau_2) &=4\tau_2 \int_{\Omega} |\nabla \varphi|^2 dv  - \int_{\Omega} \varphi^2 \log \varphi^2 dv\\
     &=4\tau_1 \int_{\Omega} |\nabla \varphi|^2 dv  - \int_{\Omega} \varphi^2 \log \varphi^2 dv + 4(\tau_2-\tau_1) \int_{\Omega} |\nabla \varphi|^2 dv\\
         &\geq \bar{\boldsymbol{\mu}}(\Omega, \tau_1) +m+\frac{m}{2}\log (4\pi \tau_1)+ 4(\tau_2-\tau_1) \int_{\Omega} |\nabla \varphi|^2 dv,
  \end{align*}
  where we used the definition equation (\ref{eqn:MJ16_c}) in the last step. In particular, we have
  \begin{align}
  \bar{\boldsymbol{\mu}}(\Omega, \tau_2) +\frac{m}{2}\log \tau_2 
  \geq \bar{\boldsymbol{\mu}}(\Omega, \tau_1) +\frac{m}{2}\log  \tau_1 + 4(\tau_2-\tau_1) \int_{\Omega} |\nabla \varphi|^2 dv.
  \label{eqn:RI10_1}
  \end{align}
  Then (\ref{eqn:RH31_2}) follows immediately from the above inequality. 
  Note that (\ref{eqn:RH31_3}) follows from the combination of (\ref{eqn:MJ16_6}) and (\ref{eqn:RH31_2}).  The proof of the proposition is complete. 
\end{proof}

\begin{proposition}
  Suppose $\Omega \subset M$ is a bounded set, $g$ and $h$ are two Riemannian metrics on $M$ satisfying
  \begin{align}
    \lambda  g(x) \leq  h(x) \leq \lambda^{-1} g(x), \quad \forall \; x \in \Omega,  \label{eqn:RI08_2}
  \end{align}
  for some $\lambda \in (0,1)$.  Then for each interval $[\tau_{1}, \tau_{2}] \subset \R^{+}$, we have
  \begin{align} 
    \inf_{s \in [\tau_{1}, \tau_{2}]} \bar{\boldsymbol{\mu}}(\Omega, h, s) 
    \geq \lambda^{m}  \bar{\boldsymbol{\nu}}(\Omega, g, \tau_{2})
    +(1-\lambda^{m}) \Gamma(\Omega,h,\tau_{1},\tau_{2})
  +\frac{m}{2}\lambda^{m} \log \lambda^{m+2}.
  \label{eqn:RI10_7}
  \end{align}
  In the above inequality we denote
  \begin{align}
    \Gamma(\Omega,h,\tau_{1}, \tau_{2}) \coloneqq 2\bar{\boldsymbol{\nu}}(\Omega,h,\tau_{2})-2\log_{+} \frac{|\Omega|_{dv_{h}}}{\tau_{2}^{\frac{m}{2}}} +m\log \frac{\tau_1}{\tau_2}+\frac{m}{2} \log (\pi e^2)-\frac{2}{e},   
   \label{eqn:RI10_5}  
  \end{align}
  where $\log_{+}x=\max\left\{ \log x, 0 \right\}$. 
  \label{prn:RG08_1}
\end{proposition}

\begin{proof}
  Fix $\tau>0$. Suppose $\varphi$ is a minimizer for $\bar{\boldsymbol{\mu}}(\Omega, h, \tau)$. It follows from definition(cf. (\ref{eqn:GH01_1})) that  
  \begin{align}
    \bar{\boldsymbol{\mu}}(\Omega,h,\tau) +m+\frac{m}{2}\log (4\pi)
    =\underbrace{4\tau \int_{\Omega} |\nabla \varphi|_{h}^2 dv_{h}}_{I} - \underbrace{\int_{\Omega} \varphi^2  \left( \log \varphi^{2} + \frac{m}{2} \log \tau \right)  dv_{h}}_{II}.
    \label{eqn:RI11_1}  
  \end{align}

  We first estimate $I$ in (\ref{eqn:RI11_1}).   
  Applying (\ref{eqn:RI10_1}) with $\tau_2=\tau$ and $\tau_1= 0.5 \tau$, we have
  \begin{align*}
    2\tau \int_{\Omega} |\nabla \varphi|_{h}^2 dv_{h} \leq \bar{\boldsymbol{\mu}}(\Omega, h, \tau) -\bar{\boldsymbol{\mu}}(\Omega, h, 0.5 \tau) +\frac{m}{2}\log 2. 
  \end{align*}
  Since $\bar{\boldsymbol{\mu}}(\Omega,h, 0.5\tau) \geq \bar{\boldsymbol{\nu}}(\Omega,h,\tau)$ by definition(cf. (\ref{eqn:MJ16_d}) and (\ref{eqn:MJ16_2})), it follows from the above inequality that 
  \begin{align}
   0 \leq I=4\tau \int_{\Omega} |\nabla \varphi|_{h}^2 dv_{h} \leq 2\left( \bar{\boldsymbol{\mu}}(\Omega,h, \tau) -\bar{\boldsymbol{\nu}}(\Omega,h, \tau)  \right) + m \log 2. 
    \label{eqn:RI10_2}  
  \end{align}
  Then we estimate $II$.  Note that
  \begin{align*}
    II=\int_{\Omega} \varphi^2  \left( \log \varphi^{2} + \frac{m}{2} \log \tau \right)  dv_{h}
    =\int_{\Omega} \varphi^{2} \log \left( \varphi^{2} \tau^{\frac{m}{2}} \right) dv_{h}
    =\tau^{-\frac{m}{2}}\int_{\Omega} \left( \varphi^{2} \tau^{\frac{m}{2}} \right) \log \left( \varphi^{2} \tau^{\frac{m}{2}} \right) dv_{h}. 
  \end{align*}
  We decompose $\Omega$ as
  \begin{align*}
    \Omega_{+} \coloneqq \left\{ x \in \Omega| \varphi(x)> \tau^{-\frac{m}{4}} \right\}, \quad \Omega_{-} \coloneqq \left\{ x \in \Omega |\varphi(x) \leq \tau^{-\frac{m}{4}} \right\}. 
  \end{align*}
  Since $x\log x$ is a convex function, it follows from Jensen's inequality that  
  \begin{align*}
    0 &\geq \int_{\Omega_{-}} \varphi^{2} \log \left( \varphi^{2} \tau^{\frac{m}{2}} \right) dv_{h}
    =\frac{|\Omega_{-}|_{dv_{h}}}{\tau^{\frac{m}{2}}}\int_{\Omega_{-}} \left( \varphi^{2} \tau^{\frac{m}{2}} \right) \log \left( \varphi^{2} \tau^{\frac{m}{2}} \right) \frac{dv_{h}}{|\Omega_{-}|_{dv_{h}}}
    \notag\\
    &\geq \frac{|\Omega_{-}|_{dv_{h}}}{\tau^{\frac{m}{2}}} \left\{ \int_{\Omega_{-}} \left( \varphi^{2} \tau^{\frac{m}{2}} \right) \frac{dv_{h}}{|\Omega_{-}|_{dv_{h}}} \right\} 
   \log \left\{  \int_{\Omega_{-}} \left( \varphi^{2} \tau^{\frac{m}{2}} \right) \frac{dv_{h}}{|\Omega_{-}|_{dv_{h}}} \right\} \notag\\
   &=\left\{ \int_{\Omega_{-}} \varphi^{2} dv_{h} \right\} \cdot \left\{ \log \left(\int_{\Omega_{-}} \varphi^{2} dv_{h} \right) -\log \frac{|\Omega_{-}|_{dv_{h}}}{\tau^{\frac{m}{2}}} \right\}. 
  \end{align*}
  Let $a^2 \coloneqq \int_{\Omega_{-}} \varphi^{2} dv_{h}$. Then $0<a^2 \leq 1$. Since $x\log x \geq -e^{-1}$ on $(0,1]$, the above inequality implies that
  \begin{align}
    0 &\geq \int_{\Omega_{-}} \varphi^{2} \log \left( \varphi^{2} \tau^{\frac{m}{2}} \right) dv_{h} \geq a^{2} \log a^{2} -a^{2} \log \frac{|\Omega_{-}|_{dv_{h}}}{\tau^{\frac{m}{2}}}
    \geq -e^{-1} -\log_{+} \frac{|\Omega|_{dv_{h}}}{\tau^{\frac{m}{2}}}. 
  \label{eqn:RI10_3}  
  \end{align}
  Putting (\ref{eqn:RI10_2}) into (\ref{eqn:RI11_1}), we obtain 
  \begin{align}
      II=I-\bar{\boldsymbol{\mu}}(\Omega,h,\tau) -m-\frac{m}{2}\log (4\pi) 
      \leq \bar{\boldsymbol{\mu}}(\Omega,h,\tau) -2\bar{\boldsymbol{\nu}}(\Omega,h,\tau)-m-\frac{m}{2} \log \pi. 
  \label{eqn:RI10_4}  
  \end{align}
  Consequently, by the combination of (\ref{eqn:RI10_3}) and (\ref{eqn:RI10_4}), we have  
   \begin{align}
    |II| &\leq \int_{\Omega} \varphi^{2} \left| \log \left( \varphi^{2} \tau^{\frac{m}{2}} \right) \right| dv_{h} \notag\\
    &=\int_{\Omega_{+}} \varphi^2 \log \left( \varphi^{2} \tau^{\frac{m}{2}} \right)  dv_{h}-\int_{\Omega_{-}} \varphi^2 \log \left( \varphi^{2} \tau^{\frac{m}{2}} \right) dv_{h} \notag\\
    &=\int_{\Omega} \varphi^2 \log \left( \varphi^{2} \tau^{\frac{m}{2}} \right)  dv_{h}-2\int_{\Omega_{-}} \varphi^2 \log \left( \varphi^{2} \tau^{\frac{m}{2}} \right) dv_{h} \notag\\
    &\leq \frac{2}{e} + 2\log \frac{|\Omega|_{dv_{h}}}{\tau^{\frac{m}{2}}} + \bar{\boldsymbol{\mu}}(\Omega,h,\tau) -2\bar{\boldsymbol{\nu}}(\Omega,h,\tau)-\frac{m}{2} \log (\pi e^2) \notag\\
    &\leq 2\log \frac{|\Omega|_{dv_{h}}}{\tau^{\frac{m}{2}}} + \bar{\boldsymbol{\mu}}(\Omega,h,\tau) -2\bar{\boldsymbol{\nu}}(\Omega,h,\tau) -c_{m},
    \label{eqn:RI10_6}  
  \end{align}
  where
  \begin{align}
   c_{m} \coloneqq \frac{m}{2} \log (\pi e^2)-\frac{2}{e}. 
  \label{eqn:RI23_4}
  \end{align}

 We proceed to combine the estimates in (\ref{eqn:RI10_2}) and (\ref{eqn:RI10_6}) to relate the value of  $I-II$ with $\bar{\boldsymbol{\mu}}(\Omega,g, \cdot)$. 
 Recall that  $\int_{\Omega} \varphi^2 dv_{h}=1$. By (\ref{eqn:RI08_2}), it is clear that
 \begin{align}
  \lambda^{\frac{m}{2}} dv_{g} \leq dv_{h} \leq \lambda^{-\frac{m}{2}} dv_{g}. 
  \label{eqn:RI08_3}
 \end{align}
  Consequently, we have  
  \begin{align}
    \lambda^{\frac{m}{2}} \leq  c^{2} \coloneqq \int_{\Omega} \varphi^{2} dv_{g}  \leq \lambda^{-\frac{m}{2}}.   \label{eqn:RI08_4} 
  \end{align}
  Define  $\tilde{\varphi} \coloneqq c^{-1} \varphi$. Then $\tilde{\varphi}$ satisfies the normalization condition $ \int_{\Omega} \tilde{\varphi}^2 dv_{g}=1$ under the metric $g$.  
  Now we have
  \begin{align*}
    &I=4\tau \int_{\Omega} |\nabla \varphi|_{h}^2 dv_{h}=4\tau c^2 \int_{\Omega} |\nabla \tilde{\varphi}|_{h}^2 dv_{h} \geq 4\tau \lambda^{m+1} \int_{\Omega} |\nabla \tilde{\varphi}|_{g}^{2} dv_{g},\\
    &II=\int_{\Omega} \varphi^{2} \log \left( \varphi^{2} \tau^{\frac{m}{2}} \right) dv_{h}. 
  \end{align*}
  Subtracting $II$ from $I$ yields that 
  \begin{align}
    &\quad I-II \notag\\
    &\geq 4\tau \lambda^{m+1} \int_{\Omega} |\nabla \tilde{\varphi}|_{g}^{2} dv_{g}-\int_{\Omega} \tilde{\varphi}^2 \log \tilde{\varphi}^2 dv_{g} 
    +\int_{\Omega} \tilde{\varphi}^2 \log \tilde{\varphi}^2 dv_{g}
    -\int_{\Omega} \varphi^{2} \log \left( \varphi^{2} \tau^{\frac{m}{2}} \right) dv_{h} \notag\\
    &\geq \bar{\boldsymbol{\mu}}(\Omega,g,\tau \lambda^{m+1})+m+\frac{m}{2} \log (4\pi \lambda^{m+1})
    +\int_{\Omega} \tilde{\varphi}^2 \log \tilde{\varphi}^2 dv_{g}
    -\int_{\Omega} \varphi^{2} \log  \varphi^{2} dv_{h}.  
    \label{eqn:RI11_2} 
  \end{align}
  Recall from (\ref{eqn:RI08_4}) that
  \begin{align*}
  &\int_{\Omega} \tilde{\varphi}^2 \log \tilde{\varphi}^2 dv_{g}-\int_{\Omega} \varphi^{2} \log  \varphi^{2} dv_{h}\\
  &=\int_{\Omega} \tilde{\varphi}^2 \log \left( \tilde{\varphi}^2 \tau^{\frac{m}{2}}\right) dv_{g}-\int_{\Omega} \varphi^{2} \log  \left( \varphi^{2} \tau^{\frac{m}{2}} \right) dv_{h}\\
  &=\log c^{-2} + c^{-2}\int_{\Omega} \varphi^2 \log \left(\varphi^2  \tau^{\frac{m}{2}}\right) dv_{g}
   -\int_{\Omega} \varphi^{2} \log  \left( \varphi^{2} \tau^{\frac{m}{2}} \right) dv_{h}\\
   &=\log c^{-2} + \int_{\Omega} \varphi^2 \log \left(\varphi^2  \tau^{\frac{m}{2}}\right) \left\{ c^{-2}\frac{dv_{g}}{dv_{h}}-1 \right\} dv_{h}. 
  \end{align*}
  In light of  (\ref{eqn:RI08_3}) and (\ref{eqn:RI08_4}), we can estimate 
  \begin{align*}
    &\quad \left|\int_{\Omega} \varphi^2 \log \left(\varphi^2  \tau^{\frac{m}{2}}\right) \left\{ c^{-2}\frac{dv_{g}}{dv_{h}}-1 \right\} dv_{h}\right|\\
    &\leq \int_{\Omega} \left| \varphi^2 \log \left( \varphi^2 \tau^{\frac{m}{2}} \right) \right| \cdot \left| c^{-2} \frac{dv_{g}}{dv_{h}} - 1 \right| dv_{h}
     \leq \left\{ \lambda^{-m}-1 \right\} \cdot \int_{\Omega} \left| \varphi^2 \log \left(\varphi^2 \tau^{\frac{m}{2}} \right) \right| dv_{h}\\
     &\leq  \left\{ \lambda^{-m}-1 \right\} \cdot \left\{ 2\log_{+} \frac{|\Omega|_{dv_{h}}}{\tau^{\frac{m}{2}}} + \bar{\boldsymbol{\mu}}(\Omega,h,\tau) -2\bar{\boldsymbol{\nu}}(\Omega,h,\tau) -c_{m} \right\}. 
  \end{align*}
  Noting that $\log c^{-2} \geq \log \lambda^{\frac{m}{2}}$ by (\ref{eqn:RI08_4}). 
  Plugging these inequalities into (\ref{eqn:RI11_1}) and (\ref{eqn:RI11_2}), and using (\ref{eqn:RI10_6}), we obtain 
  \begin{align*}
    &\quad \bar{\boldsymbol{\mu}}(\Omega, h, \tau)-\bar{\boldsymbol{\mu}}(\Omega,g,\tau \lambda^{m+1})\\
    &\geq \frac{m}{2} \log \lambda^{m+2}-\left\{ \lambda^{-m}-1 \right\} \cdot 
    \left\{ 2\log_{+} \frac{|\Omega|_{dv_{h}}}{\tau^{\frac{m}{2}}} + \bar{\boldsymbol{\mu}}(\Omega,h,\tau) -2\bar{\boldsymbol{\nu}}(\Omega,h,\tau) -c_{m} \right\}. 
  \end{align*}
  It follows that
  \begin{align*}
  \lambda^{-m} \bar{\boldsymbol{\mu}}(\Omega, h, \tau) \geq \bar{\boldsymbol{\mu}}(\Omega,g,\tau \lambda^{m+1}) + \frac{m}{2} \log \lambda^{m+2}
    -\left\{ \lambda^{-m}-1 \right\} \cdot  \left\{ 2\log_{+} \frac{|\Omega|_{dv_{h}}}{\tau^{\frac{m}{2}}}-2\bar{\boldsymbol{\nu}}(\Omega,h,\tau) -c_{m} \right\},
  \end{align*}
  whence we have  
  \begin{align*}
  \bar{\boldsymbol{\mu}}(\Omega, h, \tau) \geq 
  \lambda^{m}  \bar{\boldsymbol{\mu}}(\Omega, g, \lambda^{m+1} \tau)
  -(1-\lambda^{m}) \left\{ 2\log_{+} \frac{|\Omega|_{dv_{h}}}{\tau^{\frac{m}{2}}} -2\bar{\boldsymbol{\nu}}(\Omega,h,\tau) -c_{m} \right\} 
  +\frac{m}{2}\lambda^{m} \log \lambda^{m+2}.
  \end{align*}
Since $\lambda \in (0,1)$, it is clear that 
$\displaystyle \bar{\boldsymbol{\mu}}(\Omega, g, \lambda^{m+1} \tau) \geq \inf_{s \in (0,\tau)} \bar{\boldsymbol{\mu}}(\Omega, g, \tau)=\bar{\boldsymbol{\nu}}(\Omega,g,\tau)$.
Plugging this fact and (\ref{eqn:RI10_5}) into the above inequality, for each $\delta \in (0, 1]$, we arrive at 
  \begin{align*} 
    &\quad \inf_{s \in [\delta \tau, \tau]} \bar{\boldsymbol{\mu}}(\Omega, h, s) \\
    &\geq \lambda^{m}  \bar{\boldsymbol{\nu}}(\Omega, g, \tau)
    +(1-\lambda^{m}) \left\{2\bar{\boldsymbol{\nu}}(\Omega,h,\tau) +c_{m}-2\log_{+} \frac{|\Omega|_{dv_{h}}}{\tau^{\frac{m}{2}}} +m\log \delta \right\} 
    +\frac{m}{2}\lambda^{m} \log \lambda^{m+2}.
  \end{align*}
  Replacing $[\delta \tau, \tau]$ by $[\tau_{1}, \tau_{2}]$ and using the definition of $c_{m}$ (\ref{eqn:RI23_4}), we obtain (\ref{eqn:RI10_7}) directly from the above inequality and (\ref{eqn:RI10_5}).  
\end{proof}

\begin{theorem}[\textbf{Lower bound of volume ratio in terms of $\boldsymbol{\nu}$ and scalar curvature}, cf. Theorem 3.3 of~\cite{BWang17local}]
Suppose $B=B(x_0,r_0) \subset \Omega$ is a geodesic ball, then we have
\begin{align}
 &\frac{|B|}{\omega_m r_0^m} \geq  e^{\bar{\boldsymbol{\nu}}-2^{m+7}} \geq  e^{\boldsymbol{\nu}-2^{m+7}-\bar{\Lambda}r_0^2}
 \label{eqn:GC28_2}   
\end{align}
where $\bar{\boldsymbol{\nu}}=\bar{\boldsymbol{\nu}}(B,r_0^2), \; \boldsymbol{\nu}=\boldsymbol{\nu}(B,r_0^2)$. 
\label{thm:CF21_3}
\end{theorem}

\begin{lemma}[\textbf{Estimate of functionals by isoperimetric constant and scalar lower bound}, cf. Lemma 3.5 of~\cite{BWang17local}]
Suppose $\Omega$ is a bounded domain in $(M^{m}, g)$, $\tilde{\Omega}$ is a ball in $(\R^{m}, g_{E})$ such that $|\tilde{\Omega}|=|\Omega|$.
Define
\begin{align}
  \lambda \coloneqq \frac{\mathbf{I}(\Omega)}{\mathbf{I}_m}.   
  \label{eqn:CF20_4}
\end{align}
Then we have
\begin{align}
   \bar{\boldsymbol{\mu}}(\Omega, g, \tau) \geq  \bar{\boldsymbol{\mu}} \left(\tilde{\Omega}, g_{E}, \tau \lambda^2 \right) + m\log \lambda.   \label{eqn:MJ26_5}
\end{align}
Consequently, we have
\begin{align}
     &\bar{\boldsymbol{\nu}}(\Omega, g, \tau) \geq   m\log \lambda,   \label{eqn:MJ26_6}\\
     &\boldsymbol{\mu}(\Omega, g, \tau) \geq \bar{\boldsymbol{\mu}} \left(\tilde{\Omega}, g_{E}, \tau \lambda^2 \right) + m\log \lambda - \underline{\Lambda} \tau, \label{eqn:MJ26_8}\\
     &\boldsymbol{\nu}(\Omega, g, \tau) \geq m\log \lambda -\underline{\Lambda} \tau. \label{eqn:MJ26_9}
\end{align}
\label{lma:MJ25_1}
\end{lemma}

\begin{theorem}[\textbf{Almost monotonicity of local-$\boldsymbol{\mu}$-functional}, cf. Theorem 5.4 of~\cite{BWang17local}]
Let $A \geq 1000m$ be a large constant. 
Let $\{(M^m, g(t)), 0 \leq t \leq T\}$ be a Ricci flow solution satisfying
\begin{align}
    t \cdot Rc(x,t) \leq (m-1)\alpha_{m} A, \quad \forall \; x \in B_{g(t)} \left(x_0, \sqrt{t} \right), \; t \in (0, T].   \label{eqn:MJ17_11}
\end{align}
Then we have
\begin{align}
  \boldsymbol{\mu}(\Omega_T', g(T), \tau_T) -\boldsymbol{\mu}(\Omega_0, g(0), \tau_{T}+T) \geq -A^{-2},  \quad \forall \;  \tau_T \in (0, A^2 T),  \label{eqn:MJ17_9}
\end{align} 
where $\Omega_T'=B_{g(T)}\left(x_0,  8A \sqrt{T} \right)$ and $\Omega_0=B_{g(0)}\left(x_0, 20A\sqrt{T} \right)$.   In particular, we have
\begin{align}
  \boldsymbol{\nu}(\Omega_T', g(T), \tau_T) \geq -A^{-2} +\inf_{\tau \in [T, \tau_T+T]} \boldsymbol{\mu}(\Omega_0, g(0), \tau) \geq -A^{-2} + \boldsymbol{\nu}(\Omega_0, g(0), \tau_{T}+T). 
  \label{eqn:MJ17_10}
\end{align}
Furthermore, if $-\underline{\Lambda} \leq R \leq \bar{\Lambda}$ on $\Omega_{0}$, then we have
\begin{align}
  \boldsymbol{\nu}(\Omega_T', g(T), \tau_T) \geq -A^{-2} + \boldsymbol{\mu}(\Omega_0, g(0), T) -\frac{m}{2} \log \frac{\tau_{T}+T}{T} -\bar{\Lambda}T - \underline{\Lambda}(\tau_{T}+T).  
  \label{eqn:RG03_1}
\end{align}
\label{thm:CA02_3}
\end{theorem}

\begin{proof}
Note that the condition (\ref{eqn:MJ17_11}) is slightly modified to be compatible with later applications. 
It is clear that (\ref{eqn:MJ17_9}) and (\ref{eqn:MJ17_10}) are already proved in Theorem 5.4 of~\cite{BWang17local}, we only need to prove (\ref{eqn:RG03_1}).
  In (\ref{eqn:RH31_3}), by setting $\tau_1=T$ annd $\tau_2$ running through $[T, \tau_{T}+T]$, we obtain (\ref{eqn:RG03_1}). 
\end{proof}

\section{Rigidity and gap theorems}  
\label{sec:rigidity}

In this section, we first list some important properties of the Ricci flows(cf. Proposition~\ref{prn:RH24_1}, Proposition~\ref{prn:RB06_1}, Corollary~\ref{cly:RH30_1} and Proposition~\ref{prn:MJ24_1}). 
They are implied or inspired by the proof of pseudo-locality theorem of Perelman~\cite{Pe1}, with further improvements. 
Then in Theorem~\ref{thm:RH27_10} and Theorem~\ref{thm:RH30_18},  we use $\boldsymbol{\nu}$-entropy as a tool to generalize Anderson's gap theorem and $\epsilon$-regularity theorem. 
We observe that Einstein condition is not essentially used and we are able to generalize the gap theorem and the $\epsilon$-regularity theorem for cscK metrics, which is achieved in Theorem~\ref{thm:PA08_1}
and Theorem~\ref{thm:PA05_2}. Finally, we obtain a compactness property for some moduli of cscK metrics in Theorem~\ref{thm:RG23_1}.\\ 

Let us reformulate Proposition 4.8  and Proposition 4.9 of~\cite{BWang17local} as the following rigidity properties.
\begin{proposition}[\textbf{Rigidity of Ricci flow in terms of $\boldsymbol{\mu}$}]
  Suppose $\{(M^{m}, g(t)), 0 \leq t \leq T\}$ is a Ricci flow solution satisfying
  \begin{align}
    \boldsymbol{\mu}(M, g(0), T) \geq 0. 
    \label{eqn:RH24_2}
  \end{align}
  Then the flow is the static flow on Euclidean space $(\R^{m}, g_{E})$. 
\label{prn:RH24_1}
\end{proposition}

\begin{proof}
  This is nothing but reformulation of Proposition 4.8 of~\cite{BWang17local}. 
\end{proof}

\begin{proposition}[\textbf{Rigidity of Riemannian manifolds in terms of $\boldsymbol{\nu}$}]
Suppose $(M^{m}, g)$ is a complete Riemannian manifold with bounded curvature satisfying  
\begin{align}
   \boldsymbol{\nu}(M,g,T)\geq 0  \label{eqn:CA22_1}
\end{align}
for some $T>0$. Then $(M, g)$ is isometric to the Euclidean space $(\R^{m}, g_{E})$.
\label{prn:RB06_1}
\end{proposition}

\begin{proof}
As $(M,g)$ has bounded curvature, we can run Ricci flow from $(M,g)$ for a short time period $\delta_0$ for some $\delta_0 \in (0,T]$. By Proposition 4.9 of~\cite{BWang17local}, we have 
\begin{align}
  \boldsymbol{\nu}(M,g,\delta_0) \leq \boldsymbol{\mu}(M,g,\delta_0) \leq 0.    \label{eqn:CA05_1}
\end{align}
On the other hand,   since $\delta_0 \in (0,T]$, it follows from definition of $\boldsymbol{\nu}$ and (\ref{eqn:CA22_1}) that
\begin{align}
 0 \leq \boldsymbol{\nu}(M,g,T) \leq \boldsymbol{\nu}(M,g,\delta_0).  \label{eqn:RB06_1}
\end{align}
Combining (\ref{eqn:RB06_1}) with (\ref{eqn:CA05_1}), we obtain
\begin{align*}
  \boldsymbol{\nu}(M,g,\delta_0) =\boldsymbol{\mu}(M,g,\delta_0) =0.  
\end{align*}
As $\boldsymbol{\nu}(M, g(t), \delta_0-t)$ is monotonically nondecreasing along the flow, it is non-negative for each $t \in (0,\delta_0)$. 
Replacing $g$ by $g(t)$, $\delta_{0}$ by $\delta_{0}-t$ in the previous argument, we obtain 
\begin{align*}
   \boldsymbol{\mu}(M, g(t), \delta_0-t) =\boldsymbol{\nu}(M, g(t), \delta_0-t)=0, \quad \forall t \in (0, \delta_0). 
\end{align*}
In particular, $\frac{d}{dt}  \boldsymbol{\mu}(M, g(t), \delta_0-t) =0$.  Then we know that $\{(M, g(t)), 0 \leq t \leq \delta_0\}$ satisfies the shrinking gradient Ricci soliton equation
\begin{align*}
    R_{ij} + f(t)_{ij}-\frac{g_{ij}(t)}{\delta_0-t} \equiv 0.
\end{align*}
Then the bounded curvature condition of the flow implies that $\{(M, g(t)), 0 \leq t \leq \delta_0\}$ must be isometric to Euclidean space, as done in Proposition 4.8 and 4.9 of~\cite{BWang17local}. 
In particular, $(M,g)$ is isometric to Euclidean space. 
\end{proof}

By Lemma~\ref{lma:MJ16_1},  the following corollary immediately follows from Proposition~\ref{prn:RB06_1}.  

\begin{corollary}[\textbf{Rigidity of Riemannian manifolds in terms of $\overline{\boldsymbol{\nu}}$ and $R$}]
Suppose $(M^{m}, g)$ is a complete Riemannian manifold with bounded curvature satisfying  
\begin{align}
   \bar{\boldsymbol{\nu}}(M,g,T)\geq 0, \quad  R \geq 0,   \label{eqn:RB06_3}
\end{align}
for some $T>0$. Then $(M, g)$ is isometric to the Euclidean space $(\R^{m}, g_{E})$.
\label{cly:RH30_1}
\end{corollary}

A quantitative version of Proposition~\ref{prn:RH24_1} is the following gap property for local $\boldsymbol{\mu}$-functional. 

\begin{proposition}[\textbf{$\boldsymbol{\mu}$-gap of Ricci flow with bounded curvature}]
For each $\alpha \in (0,1]$, $H \in [1, \infty)$, $L \in [1, \infty)$,  there is a constant $\delta=\delta(\alpha,m,H,L)$ with the following property.

Suppose $\{(M^m, x_0, g(t)),  0 \leq t \leq 1\}$ is a Ricci flow solution satisfying
\begin{align}
\begin{cases}
&|Rm|(x,t) \leq H, \quad \forall \; x \in B_{g(t)}(x_0, \delta^{-1}), \; t \in [0,1];\\
&\boldsymbol{\mu}(B_{g(0)}(x_0, \delta^{-1}), g(0), 1) \geq -\delta. 
\end{cases}
\label{eqn:MJ24_0}
\end{align}
Then we have
\begin{align}
 &\sup_{x \in B_{g(1)}(x_0,L)}|Rm|(x, 1) <\alpha,   \label{eqn:MJ24_1}\\
 &\inf_{x \in B_{g(1)}(x_0,L)} |B_{g(1)}(x, 1)| \geq (1-\alpha) \omega_{m},  \label{eqn:MJ24_5}\\
 &\inf_{x \in B_{g(1)}(x_0, L)} inj(x,1) > \alpha^{-1}.    \label{eqn:RH27_6} 
\end{align}
\label{prn:MJ24_1}
\end{proposition}

\begin{proof}
 Before we start the main stream of the proof, let us do a technical preparation.  
 We observe that under the condition (\ref{eqn:MJ24_0}), the lower bound of $\boldsymbol{\mu}$-functional implies a rough injectivity radius estimate
 for points near $x_0$ at positive time.  
 In fact, by the local monotonicity formula (\ref{eqn:RG03_1}) with $T=\tau_{T}=1$ and proper choice of $\Omega_{1}'$ and $\Omega_{0}$, we obtain  
 \begin{align*}
   \boldsymbol{\nu}(B_{g(1)}(x_0, 0.2 \delta^{-1}), g(1), 1) &\geq -100 \delta^2 -\delta-\frac{m}{2}\log 2  -3m(m-1)H \geq -3m^2 H. 
 \end{align*}
 Exploiting the curvature condition in (\ref{eqn:MJ24_0}), it follows from Theorem~\ref{thm:CF21_3} that
 \begin{align*}
 r^{-m}|B(x,r)|_{g(1)} \geq c(m,H), \quad \forall \; r \in (0,1],  \; x \in B_{g(1)}(x_0, 0.1 \delta^{-1}).
 \end{align*}
 Therefore, by Cheeger-Gromov-Taylor~\cite{ChGrTa}, we obtain the desired rough injectivity radius estimate
 \begin{align}
   inj(x, 1) \geq c(m,H), \quad \forall \; x \in B_{g(1)}(x_0, 0.1 \delta^{-1}). \label{eqn:RG01_5} 
 \end{align}

Now we start the main stream of proof. 
We first show (\ref{eqn:MJ24_1}) via a contradiction argument.
If (\ref{eqn:MJ24_1}) were wrong, we could find a sequence of Ricci flow solutions $\left\{ (M_i, x_i, g_i(t)), 0 \leq t \leq 1\right\}$ satisfying the condition (\ref{eqn:MJ24_0}) for number $\delta_i \to 0$
and violating the (\ref{eqn:MJ24_1}). In other words, for some points $z_i \in B_{g_i(1)}(x_i, L)$, we have
\begin{align}
  H \geq |Rm|(z_i, 1) \geq \alpha.   \label{eqn:MJ24_7}
\end{align}
By the curvature estimate (\ref{eqn:MJ24_7}) and the rough injectivity radius estimate (\ref{eqn:RG01_5}), we can apply Hamilton's compactness theorem(cf.~\cite{Ha95c}) to obtain
 \begin{align}
   & \{(M_i, x_i, g_i(t)),  0 < t \leq 1\}  \longright{C^{\infty}-Cheeger-Gromov}  \{(\bar{M}, \bar{x}, \bar{g}(t)),  0 < t \leq 1\}.  \label{eqn:RG03_5} 
 \end{align}
 In light of the local monotonicity formula (\ref{eqn:MJ17_9}), the lower bound of $\boldsymbol{\mu}$ by $-\delta_i$ in (\ref{eqn:MJ24_0}) and the above convergence together
 imply that $\boldsymbol{\mu}(\bar{M}, \bar{g}(0.5), 0.5) \geq 0$.  Then it follows from Proposition~\ref{prn:RH24_1} that the limit flow solution on interval $[0.5, 1]$ 
 is a static flow on Euclidean space $(\R^m, g_{E})$. 
 In particular, we have
 \begin{align*}
   &\lim_{i \to \infty} |Rm|(z_i,1)=|Rm|(\bar{z},1)=0,
 \end{align*}
 where $\bar{z}$ is the limit of $z_i$. The existence of $\bar{z}$ is guaranteed by the fact that $d_{g_i(1)}(z_i,x_i) \leq L$ uniformly. 
 The above equations mean that
 \begin{align*}
    |Rm|(z_i,1)<\alpha
 \end{align*}
 for large $i$, which contradicts the assumption of the choice of $z_i$ in (\ref{eqn:MJ24_7}). 
 Therefore, we finish the proof of (\ref{eqn:MJ24_1}). 
 
 The proof of (\ref{eqn:MJ24_5}) follows the same route as that of (\ref{eqn:MJ24_1}), up to an application of Bishop-Gromov volume comparison, as almost-flat curvature bound is already provided by (\ref{eqn:MJ24_1}).
 Similarly, one can prove (\ref{eqn:RH27_6}). A key point is that when (\ref{eqn:MJ24_1}) holds(up to adjusting $L$ to $2L$ if necessary), in $B_{g(1)}(x_0, L)$,
 we have conjugate radius bounded from below by $2\alpha^{-1}$ under the condition (\ref{eqn:MJ24_0}). If (\ref{eqn:RH27_6})
 fails at some point $y_{i}$, we can find a smooth unit-speed geodesic lasso(cf. Lemma 5.6 of Cheeger-Ebin~\cite{ChEbin}) $\gamma_i: [0, \rho_i] \mapsto (M_i, g_{i}(1))$ with $\gamma_i(0)=\gamma_i(\rho_i)=y_{i}$
 and $\rho_i \leq 2\alpha^{-1}$.  Combining this inequality with (\ref{eqn:RG01_5}), we obtain that $c_0 \leq \rho_i \leq 2\alpha^{-1}$. 
 Passing to the limit along (\ref{eqn:RG03_5}), we obtain a unit-speed geodesic lasso on Euclidean space $(\R^{m}, g_{E})$ with length $\rho_{\infty} \in [c_0, 2\alpha^{-1}]$, which is absurd. 
 This contradiction establishes the proof of (\ref{eqn:RH27_6}).  The proof of the proposition is complete. 
 \end{proof}

 \begin{theorem}[\textbf{$\boldsymbol{\nu}$-gap of Einstein manifold}]
  For each non-flat Einstein manifold $\left(M^{m}, g\right)$, we have
  \begin{align}
    \boldsymbol{\nu}(M,g) \leq -\epsilon        \label{eqn:RH27_9}
  \end{align}
  for some universal small constant $\epsilon=\epsilon(m)$. 
  \label{thm:RH27_10}
\end{theorem}

\begin{proof}
  Let $(M^{m}, g)$ be a complete non-flat manifold such that $\boldsymbol{\nu}(M,g) \geq -\epsilon$.
  By the non-flatness, we can find a point $x_0 \in M$ such that $|Rm|(x_0) \neq 0$.
  Since $\boldsymbol{\nu}(M,g)$ is scaling invariant,  up to rescaling, we can assume $|Rm|(x_0)=1$.
  Fix $L >>1$.  In the ball $\Omega=B(x_{0},2L)$, let $y_0$ be a point where $|Rm|(\cdot) d^{2}(\cdot, \partial \Omega)$ achieves the maximum value $F^2$(See Figure~\ref{fig:pointselecting_ball}).
  Let $Q=|Rm|(y_{0})$.  It is clear that 
  \begin{align*}
    F^2=|Rm|(y_0) d^{2}(y_0, \partial \Omega) \geq |Rm|(x_0) d^{2}(x_0, \partial \Omega)= 4L^{2}. 
  \end{align*}
  Let $\tilde{g}=Q g$. Then we have $|Rm|_{\tilde{g}}(y_0)=1$ and $d_{\tilde{g}}(y_0, \partial \Omega)=F \geq 2L$. 
  For each point $x \in B_{\tilde{g}}(y_0, L)$, it follows from the triangle inequality that $d_{\tilde{g}}(x, \partial \Omega) \geq F-L$.  Consequently, by the scaling invariance of $|Rm|d^2$,
  we have
  \begin{align}
    |Rm|_{\tilde{g}}(x) \leq  \frac{|Rm|_{\tilde{g}}(y_0) d_{\tilde{g}}^{2}(y_0, \partial \Omega)}{d_{\tilde{g}}^{2}(x, \partial \Omega)} \leq \frac{F^2}{\left(F-L \right)^{2}}
    =\frac{1}{\left(1-\frac{L}{F} \right)^{2}} \leq 4, \label{eqn:RH30_2} 
  \end{align}
  for each point $x \in B_{\tilde{g}}(y_0, L)$.   Note  that 
  \begin{align}
    \boldsymbol{\nu}(B_{\tilde{g}}(y_0, L),\tilde{g}) \geq \boldsymbol{\nu}(M,\tilde{g})=\boldsymbol{\nu}(M,g) \geq -\epsilon. \label{eqn:RH30_3} 
  \end{align} 
  Combining (\ref{eqn:RH30_2}) and (\ref{eqn:RH30_3}), we can apply Theorem~\ref{thm:CF21_3} to obtain volume lower bound for each unit geodesic ball contained in $B_{\tilde{g}}(y_0, L)$. 
  Consequently, it follows from Cheeger-Gromov-Taylor~\cite{ChGrTa} that
  \begin{align}
    inj_{\tilde{g}}(x)> c_0(m), \quad \forall \; x \in B_{\tilde{g}}(y_0, L-1).   \label{eqn:RH30_4}
  \end{align}

  Now we finish the proof of the theorem through a blowup argument.
  If the theorem were wrong, we could find a sequence of non-flat Einstein manifolds $(M_i, g_i)$ satisfying $\boldsymbol{\nu}(M_i, g_i) \geq -\epsilon_i$ with $\epsilon_i \to 0$.
  For each $i$, we choose $L_i$ such that $L_i \to \infty$. By the above discussion, modulo rescaling if necessary, we can find $y_i \in M_i$ such that 
  \begin{align}
    & |Rm|_{g_{i}}(y_i)=1,      \label{eqn:RH30_5}
  \end{align}
  and 
  \begin{align}
     \sup_{x \in B(y_i, L_i)} |Rm|_{g_i}(x) \leq 4, 
      \quad \inf_{x \in B(y_i, L_i-1)} inj_{g_i}(x) \geq c_0,
      \quad \boldsymbol{\nu}(B(y_i, L_i), g_i) \geq -\epsilon_i.
      \label{eqn:RH30_15}    
  \end{align}
  Since each $g_i$ is an Einstein manifold, we can obtain the harmonic radius estimate by Anderson~\cite{Anderson90}.
  In each harmonic coordinate, the Einstein equation can be applied to obtain uniform higher order derivative estimates of the metrics. 
  Consequently, we have the following compactness property:
  \begin{align}
    (M_i, y_i, g_i) \longright{C^{\infty}-Cheeger-Gromov} (M_{\infty}, y_{\infty}, g_{\infty}),  \label{eqn:RH30_11}
  \end{align}
  where the limit is a smooth manifold. We claim that 
  \begin{align}
    \boldsymbol{\nu}(M_{\infty}, g_{\infty}) \geq 0. \label{eqn:RH30_7}
  \end{align}
  Actually, let $\cup_{k=1}^{\infty} \Omega_k$ be an exhaustion of $M_{\infty}$ by domains $\Omega_k$ with smooth boundaries $\partial \Omega_k$. 
  Fixing $k$ and $\tau>0$, we can apply the continuity of $\boldsymbol{\mu}$-functionals(cf. Corollary~\ref{cly:CB19_2}) to obtain
   \begin{align*}
     \boldsymbol{\mu}(\Omega_k, g_{\infty}, \tau)= \lim_{i \to \infty} \boldsymbol{\mu}(\Omega_k, g_i, \tau) 
     \geq  \lim_{i \to \infty} \boldsymbol{\nu}(\Omega_k, g_i, \tau) \geq  \lim_{i \to \infty} \boldsymbol{\nu}(M_{i}, g_i) \geq  0. 
   \end{align*}
   Applying Proposition~\ref{prn:CA02_1} on $M_{\infty}$, we obtain $\boldsymbol{\mu}(M_{\infty}, g_{\infty}, \tau) \geq 0$ for each positive $\tau$. 
   Consequently, it follows from definition that 
  \begin{align*}
   \boldsymbol{\nu}(M_{\infty}, g_{\infty})=\inf_{\tau \in (0, \infty)} \boldsymbol{\mu}(M_{\infty}, g_{\infty}, \tau)  \geq 0, 
  \end{align*}
  which is (\ref{eqn:RH30_7}).  In particular, $\boldsymbol{\nu}(M_{\infty}, g_{\infty},1) \geq \boldsymbol{\nu}(M_{\infty}, g_{\infty}) \geq 0$.
  Then it follows from Proposition~\ref{prn:RB06_1} that $(M_{\infty}, g_{\infty})$ is isometric to $(\R^{m}, g_{E})$. 
  In particular, we have $|Rm|_{g_{\infty}}(y_{\infty})=0$ which contradicts (\ref{eqn:RH30_5}) via the $C^{\infty}$-Cheeger-Gromov convergence. 
\end{proof}

 \begin{figure}[H]
 \begin{center}
 \psfrag{Omega}[c][c]{$\Omega=B(x_0, 2L)$}
 \psfrag{x}[l][l]{$y_{0}$}
 \psfrag{y}[l][l]{$x_{0}$}
 \psfrag{2L}[l][l]{$2L$}
 \psfrag{B}[l][l]{\textcolor{blue}{$B(y_0,LQ^{-\frac{1}{2}})$}}
 \psfrag{C}[l][l]{$|Rm|(x_{0})=1$}
 \psfrag{D}[l][l]{$|Rm|(y_{0})=Q$}
 \includegraphics[width=0.5 \columnwidth]{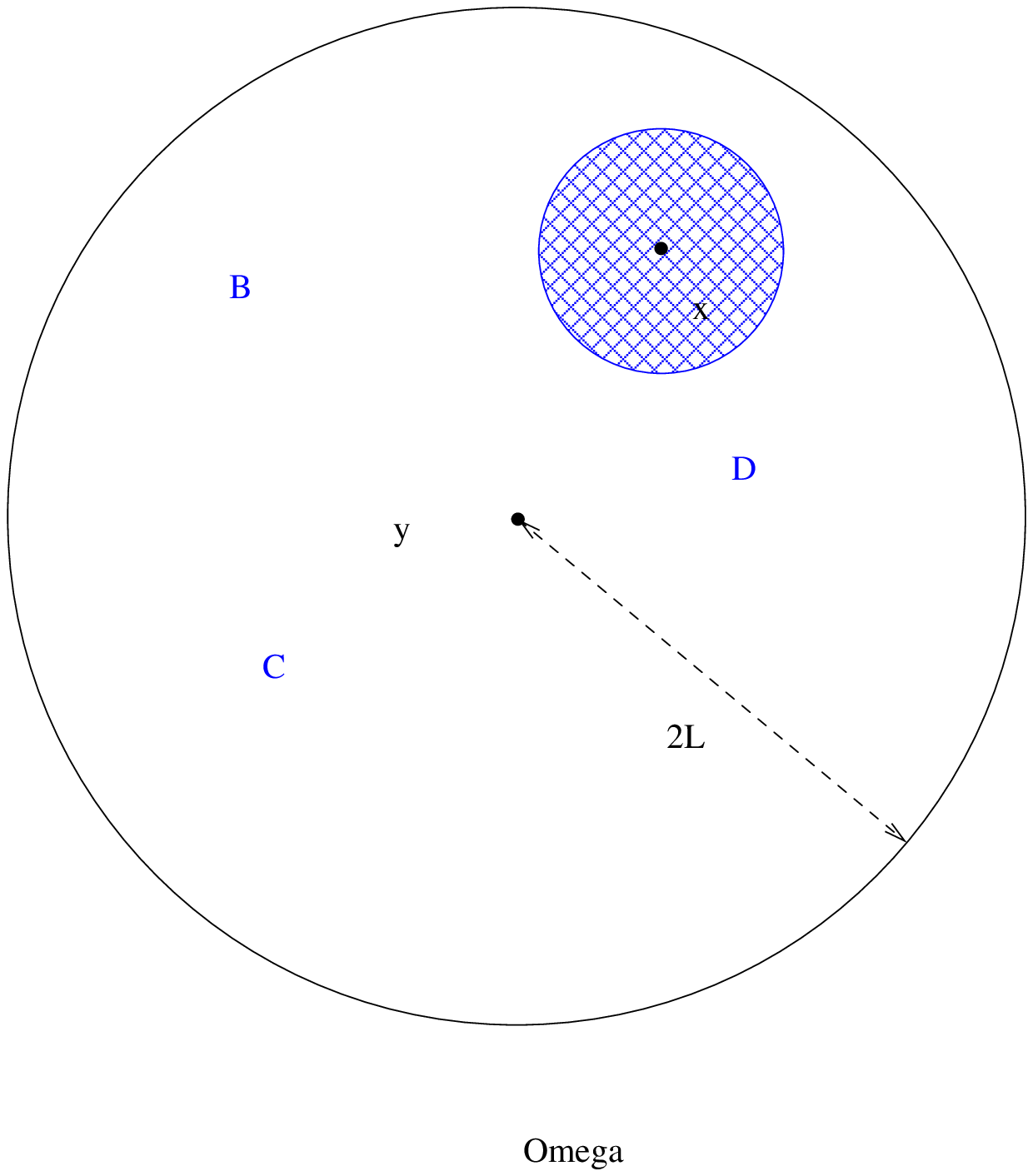}
 \caption{Existence of well-behaved domain}
 \label{fig:pointselecting_ball}
 \end{center}
 \end{figure}

 Theorem~\ref{thm:RH27_10} can be regarded as a generalization of the famous gap theorem of Anderson for Ricci-flat manifolds(cf.~Lemma 3.1 of~\cite{Anderson90}), in terms of asymptotical volume ratio. 
 A quantitative version of Theorem~\ref{thm:RH27_10} is the following $\epsilon$-regularity theorem. 

\begin{theorem}[\textbf{$\epsilon$-regularity of Einstein manifold}]
  Suppose $(M^m, x_0, g)$ is an Einstein manifold satisfying
   \begin{align}
       \begin{cases}
       &\boldsymbol{\bar{\nu}}(B(x_0, 1), g, 1) \geq -\epsilon;  \\
       &Rc(z)=c\cdot g(z), \quad \forall \; z \in B(x_0, 1).   
       \end{cases}
   \label{eqn:RH30_16}
   \end{align}
  Here $\epsilon$ is the small constant in Theorem~\ref{thm:RH27_10}, $c$ is a constant satisfying $|c| \leq m-1$.
  Then for each non-negative integer $k$, there exists a constant $C_k=C_k(m)$ such that
  \begin{align}
    \sup_{z \in B(x_0, 1)}  |\nabla^{k} Rm| d^{2+k}(z, \partial B(x_0, 1)) \leq C_k.  \label{eqn:RH30_17}
  \end{align}
  \label{thm:RH30_18}
\end{theorem}

\begin{proof}
  It suffices to prove (\ref{eqn:RH30_17}) for the case $k=0$, the higher order(i.e. $k \geq 1$) estimates follow from the estimate when $k=0$, 
  via elliptic regularity of Einstein manifolds in harmonic coordinate charts. 
  Therefore, we focus on the proof of
  \begin{align}
     \sup_{z \in B(x_0, 1)}  |Rm|(z) d^{2}(z, \partial B(x_0, 1)) \leq C.   \label{eqn:PA05_6}
  \end{align}
  Let $y_0$ be a point in $B(x_0,1)$ where the supreme of the left hand side of (\ref{eqn:PA05_6}), which we denote by $F_{0}^2$, is achieved.
  Let $L_{0} \coloneqq 0.5 F_0=0.5 \sqrt{|Rm|(y_0)} d(y_0, \partial B(x_0, 1))$ and let $Q_0 \coloneqq |Rm|(y_0)$.  

  We now prove (\ref{eqn:PA05_6}) by a contradiction argument. 
  If (\ref{eqn:PA05_6}) were wrong, we could find a sequence $(M_i, x_i, g_i)$ satisfying (\ref{eqn:RH30_16}) but violating (\ref{eqn:PA05_6}).
  Then we have
       \begin{align}
	 &Q_i \geq F_i^2=|Rm|(y_i) d^{2}(y_i, \partial B(x_i, 1))=4L_{i}^{2} \to \infty;  \label{eqn:RH30_19a}\\
       &\sup_{B\left(y_i, L_i Q_i^{-\frac{1}{2}}\right)}|Rm| \leq 4Q_i;  \label{eqn:RH30_19b}\\
       & B\left(y_i, L_i Q_i^{-\frac{1}{2}} \right) \subset B(x_i, 1).  \label{eqn:RH30_19c}
       \end{align}
   Let $\tilde{g}_i \coloneqq Q_i g_i$. 
   Notice that, in light of Proposition~\ref{prn:CA01_2}, the relationship (\ref{eqn:RH30_19c}) implies that
   \begin{align*}
     \boldsymbol{\bar{\nu}}(B_{\tilde{g}_i}(y_i, L_i), \tilde{g}_i, L_i^2) \geq \boldsymbol{\bar{\nu}}(B_{g_i}(y_i, L_iQ_i^{-\frac{1}{2}}), g_i, 1)
     \geq \boldsymbol{\bar{\nu}}(B_{g_i}(x_i, 1), g_i, 1) \geq  -\epsilon.
   \end{align*}
   In summary, we have
       \begin{align}
       &|Rm|_{\tilde{g}_i}(y_i)=1;  \label{eqn:PA05_7a}\\
       &|Rm|_{\tilde{g}_i}(z) \leq 4, \quad \forall \; z \in B_{\tilde{g}_i}(y_i, L_i);   \label{eqn:PA05_7b}\\
       &\boldsymbol{\bar{\nu}}\left( B_{\tilde{g}_i}(y_i, L_i), \tilde{g}_i, L_i^2 \right) \geq -\epsilon.   \label{eqn:PA05_7c}
       \end{align}
  Around $y_i$, the Ricci curvatures are tending to zero:
  \begin{align}
    |Rc|_{\tilde{g}_i}(z)=Q_i^{-1}c_i \to 0, \quad \forall \; z \in B_{\tilde{g}_i(0)}(y_i, L_i).    \label{eqn:PA08_4}
  \end{align}
  Similar to the proof of Theorem~\ref{thm:RH27_10}, the harmonic radius in concern is uniformly bounded from below.
  In harmonic coordinate, the Einstein condition implies that the metrics satisfy a strictly elliptic PDE.
  For each positive integer $k$, the standard elliptic PDE estimate implies that the $k$-th order curvature derivatives around $y_i$ are uniformly bounded.
  As discussed from (\ref{eqn:RH30_3}) to (\ref{eqn:RH30_4}), we know that (\ref{eqn:PA05_7c}) provides a uniform lower bound of injectivity radius.  So we have 
  \begin{align*}
    \left( M_i, y_i, \tilde{g}_i \right) \longright{C^{\infty}-Cheeger-Gromov} \left( M_{\infty}, y_{\infty}, \tilde{g}_{\infty} \right),
  \end{align*}
  where $M_{\infty}$ is a smooth manifold. It follows from the smooth convergence, (\ref{eqn:PA05_7a}) and (\ref{eqn:PA08_4}) that
   \begin{align}
       &|Rm|_{\tilde{g}_{\infty}}(y_{\infty})=1;  \label{eqn:PA08_5a}\\
       &Rc_{\tilde{g}_{\infty}} \equiv 0.  \label{eqn:PA08_5b}
   \end{align}
  Following the same argument as in the proof of Theorem~\ref{thm:RH27_10}, the condition (\ref{eqn:PA05_7c}) implies that
  \begin{align*}
    \bar{\boldsymbol{\nu}}\left( M_{\infty}, \tilde{g}_{\infty} \right) = \inf_{\tau \in (0,\infty)} \bar{\boldsymbol{\nu}}\left( M_{\infty}, \tilde{g}_{\infty}, \tau \right) \geq -\epsilon. 
  \end{align*}
  In view of this entropy lower bound and the Ricci-flatness (\ref{eqn:PA08_5b}), we can apply Theorem~\ref{thm:RH27_10} to obtain that $(M_{\infty}, \tilde{g}_{\infty})$ is actually
  isometric to the Euclidean space $\left( \R^{m}, g_{E} \right)$. In particular, we have $|Rm|_{\tilde{g}_{\infty}}(y_{\infty})=0$, which contradicts (\ref{eqn:PA08_5a}). 
  The proof of (\ref{eqn:PA05_6}) and consequently the proof of the theorem is complete. 
\end{proof}

Checking the proof of Theorem~\ref{thm:RH27_10} and Theorem~\ref{thm:RH30_18}, it is clear that the Einstein condition is only used for regularity improvement and bounding scalar curvature.
For critical metrics with these two properties, one can easily follow the same argument to obtain gap theorem and $\epsilon$-regularity theorem.
For example, we have the analogies of Theorem~\ref{thm:RH27_10} and Theorem~\ref{thm:RH30_18} for cscK manifolds. 

\begin{theorem}[\textbf{$\boldsymbol{\nu}$-gap of cscK manifold}]
  For each non-flat cscK manifold $\left( M^{n}, g, J \right)$, we have
  \begin{align}
    \boldsymbol{\nu}(M,g) \leq -\epsilon        \label{eqn:RB15_3}
  \end{align}
  for some universal small constant $\epsilon=\epsilon(n)$. 
  \label{thm:PA08_1}
\end{theorem}

\begin{proof}
  The strategy for the proof is the same as that in Theorem~\ref{thm:RH27_10}.  The only difference is the step to obtain the smooth convergence (\ref{eqn:RH30_11}).  
  We shall replace the harmonic coordinate estimate by holomorphic coordinate estimate in the current setting.

  As discussed in the proof of Theorem~\ref{thm:RH27_10}, for each non-flat Riemannian manifold $(M,g)$ satisfying $\boldsymbol{\nu}(M,g) > -\epsilon$ and each fixed constant $L>>1$, 
  up to rescaling, we can choose a base point
  $x_0 \in M$ such that 
  \begin{align}
     &\qquad |Rm|(x_0)=1,      \label{eqn:RH30_12}\\
     &\sup_{x \in B(x_0, L)} |Rm|(x) \leq 4, 
      \quad \inf_{x \in B(x_0, L-1)} inj(x) \geq c_0,
      \quad \boldsymbol{\nu}(B(x_0, L), g) \geq -\epsilon.
      \label{eqn:RH30_13}
  \end{align}
  Now we apply the condition that $(M,g,J)$ is cscK.  Note that in holomorphic coordinate chart, we have the equation 
  \begin{align}
       \begin{cases}
       &R_{\alpha \bar{\beta}}=-\partial_{\alpha} \partial_{\bar{\beta}} \log \det g_{\gamma \bar{\lambda}},\\
       &R=-\Delta_{g} \log \det g_{\gamma \bar{\lambda}}. 
       \end{cases}
      \label{eqn:RB12_1}
   \end{align}
  Since $\nabla R \equiv 0$, one can apply Schauder estimate to obtain a uniform size $c$ with the following properties:
  each ball $B(y,c) \subset B(x_0, L-1)$ admits a holomorphic chart $\{z_1, \cdots, z_n\}$ such that the metric $g$ in this chart have uniform $C^{2, \frac{1}{2}}$-norm(cf.~Proposition~1.2 of Tian-Yau~\cite{TianYau1}).
  By the first inequality in (\ref{eqn:RH30_13}), the scalar curvature in concern is a constant which is uniformly bounded by $C(n)$. 
  Therefore, applying the bootstrapping argument on (\ref{eqn:RB12_1}) then implies uniform $C^{k,\frac{1}{2}}$-norm estimate of $g$ for each $k$.  In particular, we have
   \begin{align}
     |\nabla^k Rm|(z) \leq C_k, \quad \forall \; z \in B(x_0, L-1).   \label{eqn:RH30_14}  
   \end{align}

  Then we are ready to finish the proof of this theorem. 
  If the theorem were wrong, we could find a sequence of non-flat cscK manifolds $(M_i, g_i, J_i)$ satisfying $\boldsymbol{\nu}(M_i, g_i) \geq -\epsilon_i$ with $\epsilon_i \to 0$.
  For each $i$, we choose $L_i$ such that $L_i \to \infty$ and (\ref{eqn:RH30_12}) and (\ref{eqn:RH30_13}) hold, if we replace $x_0$ by $x_i$, replace $L$ by $L_{i}$.  
  The cscK condition implies (\ref{eqn:RH30_14}) holds, if we replace $x_0$ by $x_i$. 
  Therefore, it follows from (\ref{eqn:RH30_12}), (\ref{eqn:RH30_13}) and (\ref{eqn:RH30_14}) that the $C^{\infty}$-Cheeger-Gromov convergence (\ref{eqn:RH30_11}) hold.
  Same as the proof in Theorem~\ref{thm:RH27_10}, we know $\boldsymbol{\nu}(M_{\infty}, g_{\infty})=0$ and consequently $(M_{\infty}, g_{\infty})$ is isometric to $(\R^{2n}, g_{E})$. 
  Therefore, it follows from (\ref{eqn:RH30_12}) and the previous discussion that 
  \begin{align*}
    0=|Rm|_{g_{\infty}}(x_{\infty})=\lim_{i \to \infty} |Rm|_{g_i}(x_i)=1. 
  \end{align*}
  This contradiction establishes the proof of the theorem. 
\end{proof}

\begin{theorem}[\textbf{$\epsilon$-regularity of cscK manifold}]
  Suppose $(M^n, x_0, g,J)$ is a K\"ahler manifold satisfying
   \begin{align}
       \begin{cases}
       &\boldsymbol{\bar{\nu}}(B(x_0, 1), g, 1) \geq -\epsilon;  \\
       &R(z)=c, \quad \forall \; z \in B(x_0, 1).   
       \end{cases}
   \label{eqn:PA05_4}
   \end{align}
  Here $\epsilon$ is the small constant in Theorem~\ref{thm:PA08_1}, $c$ is a constant satisfying $|c| \leq 2(n+1)$.
  Then for each non-negative integer $k$, there exists a constant $C_k=C_k(n)$ such that
  \begin{align}
    \sup_{z \in B(x_0, 1)}  |\nabla^{k} Rm| d^{2+k}(z, \partial B(x_0, 1)) \leq C_k.  \label{eqn:PA05_5}
  \end{align}
  \label{thm:PA05_2}
\end{theorem}

\begin{proof}
  Following the strategy of the proof of Theorem~\ref{thm:RH30_18}, we can reduce the estimate (\ref{eqn:PA05_5}) to Theorem~\ref{thm:PA08_1}. 
  The key point is the regularity improvement, which follows from the bootstrapping of the equation (\ref{eqn:RB12_1}) in holomorphic chart.
  Then the proof is almost the same as that in Theorem~\ref{thm:RH30_18}. We leave the details for interested readers. 
\end{proof}

As an application of Theorem~\ref{thm:PA05_2}, we have the following compactness property.

\begin{theorem}[\textbf{Compactness of a subset of cscK moduli}]
 There exists $\bar{\epsilon}=\bar{\epsilon}(n)$ sufficiently small such that the following properties hold. 

 Let $\mathcal{M}_{cscK}(n,r_0,\bar{\epsilon}, \sigma, D)$ be the moduli space of cscK metrics $(M^{n}, g, J)$ satisfying
  \begin{align}
    &\inf_{x \in M} \mathbf{I}(B(x,r_0),g) \geq  (1-\bar{\epsilon}) \mathbf{I}(\R^{2n}), \label{eqn:RG23_2}\\
    &\sup_{M} |R| \leq \sigma, \label{eqn:RG23_3}\\
    &\diam(M^{n},g)\leq D.  \label{eqn:RG23_4}
  \end{align}
  Here $\mathbf{I}$ means the isoperimetric constant.  
  Then the moduli $\mathcal{M}_{cscK}(n,r_0,\bar{\epsilon},\sigma,D)$ is compact with respect to the $C^{\infty}$-Cheeger-Gromov topology. 
  \label{thm:RG23_1}
\end{theorem}

\begin{proof}
 Define
 \begin{align*}
   \bar{\epsilon} \coloneqq e^{\frac{\epsilon}{4n}}-1, \quad \rho_0 \coloneqq \min \left\{ \sqrt{\frac{\epsilon}{2 \sigma}}, r_0 \right\}, 
 \end{align*}
 where $\epsilon$ is the small constant in Theorem~\ref{thm:PA08_1}.  Clearly, $\rho_{0}^{2} \leq r_0^2$. 
 It follows from Lemma~\ref{lma:MJ25_1} that  
 \begin{align*}
    \boldsymbol{\nu}\left( B(x,\rho_{0}^{2}), g, \rho_{0}^{2} \right) \geq \boldsymbol{\nu}\left( B(x,r_{0}), g, r_{0}^{2} \right) \geq -0.5 \epsilon - \sigma \rho_0^2 \geq -\epsilon.  
 \end{align*}
 In view of Theorem~\ref{thm:PA05_2}, we have   
  \begin{align*}
    &\sup_{x \in M} |\nabla^{k} Rm| \leq C_k \rho_{0}^{-k}, \quad \forall \; k \in \{0, 1, 2, 3, \cdots \};\\
    &\inf_{x \in M} inj(x) \geq c_0 \rho_{0}.  
  \end{align*}
 Then it follows from the compactness theorem of Hamilton that  $\mathcal{M}_{cscK}(n,r_0,\bar{\epsilon},\sigma,D)$ is compact with respect to the $C^{\infty}$-Cheeger-Gromov topology.
\end{proof}

By Theorem~\ref{thm:RG23_1}, if $(M_i,g_i,J_i) \in \mathcal{M}_{cscK}(n,r_0,\bar{\epsilon}, \sigma, D)$, then by taking subsequence if necessary, we have
\begin{align}
  \left( M_i, g_i, J_i \right) \longright{C^{\infty}-Cheeger-Gromov} (M_{\infty}, g_{\infty}, J_{\infty}).   \label{eqn:RG25_2}
\end{align}
Let $\sigma_{i}$ be the scalar curvature of $(M_i,g_i,J_{i})$, then $\sigma_{i} \to \sigma_{\infty}$. If we drop the constant scalar curvature condition, 
can we still have a compactness theorem? We shall finally achieve this in Theorem~\ref{thm:RB21_12},  in a weaker sense that the convergence topology is $C^{1,\alpha}$-Cheeger-Gromov topology, rather than
the $C^{\infty}$-Cheeger-Gromov topology in (\ref{eqn:RG25_2}).  Furthermore, if the limit metric is smooth, then the precise upper and lower bound of the scalar curvature
is continuous under such convergence. Note that in this case we do not have the regularity improvement property like Theorem~\ref{thm:RG23_1}, as the underlying metric need not to satisfy
a critical PDE.   Instead, we apply the Ricci flow to construct holomorphic coordinates and then develop a priori estimates.  All such estimates will be built on the pseudo-locality theorems with
several more technical steps.  In the proof of Theorem~\ref{thm:RH30_18} and Theorem~\ref{thm:PA05_2}, the rigidity property Proposition~\ref{prn:RB06_1} plays an important role.
Correspondingly,  the rigidity property Proposition~\ref{prn:RH24_1} will play a similar role in the proof of the pseudo-locality theorems. 
Therefore, we can regard Proposition~\ref{prn:RH24_1} and its quantitative version,  Proposition~\ref{prn:MJ24_1}, as the prelude of the pseudo-locality theorems in Ricci flow.
Then Theorem~\ref{thm:RH30_18} and Theorem~\ref{thm:PA05_2} are nothing but the baby models of the pseudo-locality theorems in the Riemannian and the K\"ahler Ricci flow.

\section{The pseudo-locality theorems in the Ricci flow}
\label{sec:pseudo}

In this section, we shall prove the improved pseudo-locality, Theorem~\ref{thmin:ML14_2}.
The basic idea of the proof follows from Perelman~\cite{Pe1}.  
However, as the conditions and consequences in Theorem~\ref{thmin:ML14_2} refine those of Perelman, 
we shall adjust the organizations of Perelman's proof to include more intermediate steps, which are important 
for the further generalization of the pseudo-locality theorem under the K\"ahler setting in Section~\ref{sec:kpseudo}.
Based on the gap constant chosen in Proposition~\ref{prn:MJ24_1}, we actually do not need further blowup argument in the proof of Theorem~\ref{thmin:ML14_2}.
Consequently, if the constant in Proposition~\ref{prn:MJ24_1} is explicitly calculated, then all the constants in Theorem~\ref{thmin:ML14_2} can also 
be calculated.  Therefore, our proof provides an approach for estimating the gap constant in Theorem~\ref{thmin:ML14_2} explicitly.

We briefly describe the structure of this section.  In Proposition~\ref{prn:ML23_1}, we apply traditional maximum principle to choose a well-controlled space-time
domain in the Ricci flow where Proposition~\ref{prn:MJ24_1} can be applied.  Based on the existence of such a domain, we then show in Proposition~\ref{prn:CA04_1} that
a rough bound of curvature and a delicate bound of entropy together imply curvature improvement.  Consequently,  this improvement can be combined with a maximum principle type
argument to show that the pseudo-locality theorem holds with rough coefficients, which is done in Theorem~\ref{thm:MJ22_1}.   Then we apply Proposition~\ref{prn:MJ24_1} again, with better coefficients,
to improve Theorem~\ref{thm:MJ22_1} to  Theorem~\ref{thm:MJ24_1}, Corollary~\ref{cly:RJ18_0} and Corollary~\ref{cly:RJ17_9}, which are refined versions of the pseudo-locality theorem.
Then we discuss how to relate our version of pseudo-locality theorem with the existing versions.  In Lemma~\ref{lma:MJ23_1}, 
we show that the condition (\ref{eqn:MK30_1}) is natural and can be obtained from Cheeger-Colding theory if Ricci curvature is bounded from below. 
Then we relate versions of Perelman and Tian-Wang to our current versions of pseudo-locality theorem in Corollary~\ref{cly:MJ23_1}
and~\ref{cly:MJ23_2}.  Combining all the discussions,  we finally conclude the proof of Theorem~\ref{thmin:ML14_2} at the end of this section.

\begin{proposition}[\textbf{Choice of well-behaved space-time domain}]
 Suppose $\left\{(M^m, g(t)),  0 \leq t \leq r_0^2 \right\}$ is a Ricci flow solution satisfying
 \begin{align}
   t \cdot Rc(x,t) \leq (m-1)\alpha_{m}A, \quad  \quad \forall \; x \in B_{g(t)}\left(x_0, \sqrt{t} \right),  \; t \in [0, r_0^2],  \label{eqn:MJ23_3}
 \end{align}
 for some large constant $A$ satisfying
 \begin{align}
   A \geq \alpha_{m}^{-5} K^{4},    
   \label{eqn:RJ13_2}
 \end{align}
 where $K \geq 1$ is a parameter. Define
 \begin{align}
   &\Omega_t \coloneqq B_{g(t)}\left(x_0, 2K\sqrt{t} \right), \label{eqn:MJ22_4}\\
   &\mathcal{D} \coloneqq \left\{ (x,t) \left| x \in \Omega_t, \quad  0 \leq t \leq r_0^2 \right. \right\}, \label{eqn:MJ20_0} \\
   &F(x,t) \coloneqq \sqrt{|Rm|(x,t)} \left( 2K \sqrt{t}-d_{g(t)}(x,x_{0}) \right).   \label{eqn:MJ19_0}
 \end{align}
 Let $(y_0,t_0)$ be a point where $F$ achieves the maximum value on $\mathcal{D}$.  Define
 \begin{align}
   Q_0 \coloneqq |Rm|(y_0, t_0), \quad F_0 \coloneqq F(y_0,t_0).  \label{eqn:MJ22_1}
 \end{align}
 If
 \begin{align}
   F_0^2 \geq \frac{\alpha_{m}A}{2}, 
  \label{eqn:RI06_3}
 \end{align}
 then for every $t \in \left[t_0-Q_0^{-1}, t_0 \right]$,  we have
 \begin{align}
   &B_{g(t)}\left(y_0, \frac{1}{5}F_0 Q_0^{-\frac{1}{2}} \right) \subset  B_{g(t)}\left(x_0, \left( 2-\frac{3F_{0}}{5\sqrt{Q_{0}t_{0}}} \right) \sqrt{t} \right) \subset \Omega_t.     \label{eqn:MJ19_3}
 \end{align}
 Consequently, we have the curvature estimate
 \begin{align}
  |Rm|(x,t) \leq 4 Q_0, \quad \forall \; x \in B_{g(t)} \left(y_0, \frac{1}{5}F_0 Q_0^{-\frac{1}{2}} \right), \; t \in \left[t_0-Q_0^{-1}, t_0 \right]. \label{eqn:MJ19_1}   
 \end{align} 
\label{prn:ML23_1} 
\end{proposition}

\begin{proof}
  The existence of $(y_{0}, t_{0})$ is guaranteed by the compactness of $\mathcal{D}$.
  The choice of $(y_0, t_0)$ is intuitively illustrated in Figure~\ref{fig:welldomain}, where
 \begin{align*}
   &\partial \Omega_t=\left\{(x,t)|d_{g(t)}(x, x_0)=2K\sqrt{t} \right\}, \\
   &\partial \Omega_t'=\left\{(x,t) \left|d_{g(t)}(x, x_0)= \left(2K-\frac{3F_{0}}{5\sqrt{Q_{0}t_{0}}} \right) \sqrt{t} \right. \right\},\\ 
  &\partial \Omega_t''=\left\{(x,t) \left|d_{g(t)}(x, x_0)=\sqrt{t}  \right.\right\},\\
  &\mathcal{H}=\left\{ (x,t) \left| d_{g(t)}(x,y_0) \leq \frac{1}{5} F_0Q_0^{-\frac{1}{2}},   t \in \left[t_0-Q_0^{-1}, t_0 \right] \right.\right\}. 
 \end{align*}

 \begin{figure}
 \begin{center}
 \psfrag{A}[c][c]{$M$}
 \psfrag{B}[c][c]{$t$}
 \psfrag{C}[c][c]{$t=0$}
 \psfrag{D}[c][c]{$t=r_0^2$}
 \psfrag{E}[c][c]{$\partial \Omega_t$}
 \psfrag{F}[c][c]{$\color{red}{\partial \Omega_t'}$}
 \psfrag{G}[c][c]{$(x_0, 0)$}
 \psfrag{H}[c][c]{$(y_0,t_0)$}
 \psfrag{I}[c][c]{$\color{blue}{\mathcal{H}}$}
 \psfrag{J}[c][c]{$\color{green}{\partial \Omega_t''}$}
 \includegraphics[width=0.5 \columnwidth]{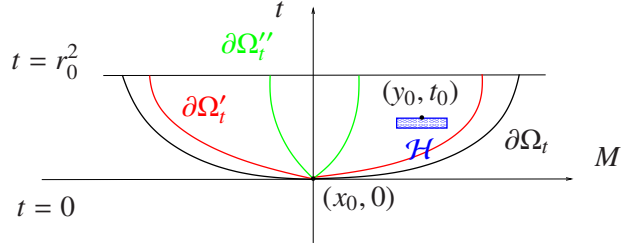}
 \caption{Existence of well-behaved space-time domain}
 \label{fig:welldomain}
 \end{center}
 \end{figure}
 
 \noindent
 By the choice of $(y_{0}, t_{0})$, it is clear that 
 \begin{align}
   2K\sqrt{t_{0}}-d_{g(t_0)}(y_0,x_0)=F_0Q_0^{-\frac{1}{2}}<2K\sqrt{t_{0}}.     \label{eqn:MJ22_2}
 \end{align}
 We shall show (\ref{eqn:MJ19_3}) and (\ref{eqn:MJ19_1}) under the assumption (\ref{eqn:RI06_3}).\\ 
 
 \noindent
 \textit{Step 1. (\ref{eqn:RI06_3}) implies  (\ref{eqn:MJ19_3}).}\\

 Since $(y_0,t_0)$ is the point where $F$ achieves the maximum value,  we have
 \begin{align*}
   F^2(x,t)&=|Rm|(x,t) \left(2\sqrt{t}-d_{g(t)}(x, x_{0}) \right)^{2} \leq F^2(y_0, t_0)=Q_0 \left(2\sqrt{t_{0}}-d_{g(t_{0})}(y_{0}, x_{0}) \right)^{2} 
 \end{align*}
 which implies that
 \begin{align}
   |Rm|(x,t) \leq \left\{ \frac{2\sqrt{t_{0}}-d_{g(t_{0})}(y_{0}, x_{0}) }{2\sqrt{t}-d_{g(t)}(x, x_{0}) } \right\}^{2} Q_0, \quad \forall \; (x,t) \in \mathcal{D}.     \label{eqn:MJ19_5}
 \end{align}
 For every $x \in B_{g(t_0)} \left(y_0, \frac{1}{5}F_0Q_0^{-\frac{1}{2}} \right)$, it follows from (\ref{eqn:MJ22_2}) and the triangle inequality that
\begin{align*}
  d_{g(t_{0})}(x,x_{0}) \leq d_{g(t_{0})}(x,y_{0}) + d_{g(t_{0})}(y_{0}, x_{0}) \leq 2K\sqrt{t_{0}} -\frac{4}{5}F_{0}Q_{0}^{-\frac{1}{2}}. 
\end{align*}
 Plugging the above inequalities into (\ref{eqn:MJ19_5}), we obtain
 \begin{align}
   |Rm|(x,t_0) \leq \frac{25}{16}Q_0 < 4Q_0, \quad \forall \; x \in B_{g(t_0)} \left(y_0, \frac{1}{5} F_0Q_0^{-\frac{1}{2}} \right).   \label{eqn:MJ19_2}
 \end{align}
 We now show (\ref{eqn:MJ19_3}) by contradiction argument. 
 Clearly, (\ref{eqn:MJ19_3}) holds at time $t=t_0$. If (\ref{eqn:MJ19_3}) fails for some $t \in \left[t_0-Q_0^{-1}, t_0 \right]$, we can choose an $s_0 \in \left[t_0-Q_0^{-1}, t_0 \right)$
 such that (\ref{eqn:MJ19_3}) holds for all $t \in [s_0, t_0]$ and starts to fail beyond $s_0$.   In particular, we can find a point 
 \begin{align*}
   z_0 \in \partial B_{g(s_0)}\left(y_0, \frac{1}{5}F_0Q_0^{-\frac{1}{2}} \right) \cap \partial B_{g(s_0)}\left(x_0, \left( 2K-\frac{3F_{0}}{5\sqrt{Q_{0}t_{0}}} \right) \sqrt{s_{0}} \right). 
 \end{align*}
 In other words, we have
 \begin{align*}
   d_{g(s_0)}(z_0,y_0)=\frac{1}{5}F_0Q_0^{-\frac{1}{2}}, \quad d_{g(s_0)}(z_0, x_0)=\left( 2K-\frac{3F_{0}}{5\sqrt{Q_{0}t_{0}}} \right) \sqrt{s_{0}}.  
 \end{align*}
 It follows from triangle inequality and the fact $0<s_{0}<t_{0}$ that
 \begin{align}
    d_{g(s_0)}(y_0, x_0) \geq d_{g(s_0)}(z_0, x_0)- d_{g(s_0)}(z_0,y_0)=2K\sqrt{s_{0}}-\frac{4}{5} F_0Q_0^{-\frac{1}{2}}.  \label{eqn:MJ19_4}
 \end{align}
 Since  (\ref{eqn:MJ19_3}) holds for each $t \in [s_0, t_0]$, it follows that  
 \begin{align}
   d_{g(t)}(x, x_{0}) \leq \left( 2-\frac{3F_{0}}{5\sqrt{Q_{0}t_{0}}} \right) \sqrt{t}   \label{eqn:MJ20_2}
 \end{align}
 for every $x \in B_{g(t)} \left( y_0, \frac{1}{5}F_0 Q_0^{-\frac{1}{2}}\right)$. 
 Plugging (\ref{eqn:MJ20_2}) into (\ref{eqn:MJ19_5}), we obtain
 \begin{align}
   |Rm|(x,t)  \leq \frac{25}{9} \cdot \frac{t_{0}}{t} \cdot Q_0, \quad \forall \; x \in B_{g(t)} \left( y_0, \frac{1}{5}F_0 Q_0^{-\frac{1}{2}}\right), \quad t \in [s_0, t_0].  
     \label{eqn:RG01_1}     
 \end{align}
 The term $\frac{t_{0}}{t}$ will be estimated in the following steps. 
 By (\ref{eqn:MJ22_2}), it is clear that 
 \begin{align}
   Q_{0}t_{0} \geq \frac{1}{4K^{2}}F_{0}^{2} \geq \frac{\alpha_{m}A}{8K^{2}}>1,
  \label{eqn:RJ15_5}
 \end{align}
 which yields
 \begin{align}
   \frac{t_0}{t} \leq \frac{t_0}{s_0} \leq \frac{t_0}{t_0-Q_0^{-1}}<\frac{Q_{0}t_{0}}{Q_{0}t_{0}-1} \leq \frac{ F_{0}^{2}}{F_{0}^{2}-4K^{2}} \leq \frac{\alpha_{m}A}{\alpha_{m}A-8K^{2}}<2 \label{eqn:MJ20_1}
 \end{align}
 for each $t \in [s_0, t_0] \subset [t_0-Q_0^{-1}, t_0]$.  It follows from the combination of (\ref{eqn:RJ15_5}) and (\ref{eqn:MJ20_1}) that
 \begin{align}
   Q_{0}t  \geq \frac{\alpha_{m}A-8K^{2}}{8K^{2}} >1.
   \label{eqn:RJ15_17}
 \end{align}
 Combining (\ref{eqn:MJ20_1}) with (\ref{eqn:RG01_1}) gives
 \begin{align}
   |Rm|(x,t)  \leq \frac{25}{9} \cdot \frac{\alpha_{m}A}{\alpha_{m}A-8K^{2}} \cdot Q_0 \leq  4Q_0, \quad \forall \; x \in B_{g(t)} \left( y_0, \frac{1}{5}F_0 Q_0^{-\frac{1}{2}}\right), \quad t \in [s_0, t_0]. 
     \label{eqn:RJ15_7}    
 \end{align}
 Thus we have the following Ricci curvature bound:
 \begin{align}
     |Rc|(x,t)   \leq  4(m-1)Q_0, \quad \forall \; x \in B_{g(t)} \left( y_0, \frac{1}{5}F_0 Q_0^{-\frac{1}{2}}\right), \quad t \in [s_0, t_0].   \label{eqn:MJ20_3}
 \end{align}
 On the other hand, we can also estimate the Ricci curvature around $(x_0, t_0)$. 
 Actually, in light of (\ref{eqn:RJ15_5}) and (\ref{eqn:MJ20_1}), the condition (\ref{eqn:MJ23_3}) gives 
 \begin{align}
   Rc(x,t)  &\leq \frac{(m-1)\alpha_{m}A}{t}=(m-1) \cdot \frac{\alpha_{m}A}{Q_0t_0} \cdot \frac{t_0}{t} \cdot Q_0 \leq  (m-1) \cdot 8K^{2} \cdot \frac{\alpha_{m}A}{\alpha_{m}A-8K^{2}} \cdot Q_0 \notag\\
   &\leq 16(m-1)K^{2}Q_{0},\label{eqn:RJ15_3}
 \end{align}
 for all $x \in B_{g(t)} \left( x_0, \sqrt{t}\right)$ and $t \in [s_0, t_0]$. In view of (\ref{eqn:RJ15_17}), it is clear that $\sqrt{t}>Q_{0}^{-\frac{1}{2}}$.
 Consequently, $B_{g(t)}\left(x_{0}, Q_{0}^{-\frac{1}{2}} \right) \subset B_{g(t)}\left(x_{0}, \sqrt{t} \right)$ for every $t \in [t_{0}-Q_{0}^{-1}, t_{0}]$.
 Thus, it follows from (\ref{eqn:RJ15_3}) that 
 \begin{align}
   Rc(x,t) \leq 16(m-1)K^{2}Q_{0}, \quad \forall \; x \in B_{g(t)} \left(x_{0}, Q_{0}^{-\frac{1}{2}} \right), \; t \in \left[t_{0}-Q_{0}^{-1}, t_{0} \right].  \label{eqn:MJ23_5}
 \end{align}
 Combining (\ref{eqn:MJ20_3}) and (\ref{eqn:MJ23_5}) gives
 \begin{align}
   Rc(x,t)  \leq 16(m-1)K^{2}Q_0,  \quad \forall \; x \in B_{g(t)} \left( x_0,  Q_{0}^{-\frac{1}{2}} \right) \cup B_{g(t)} \left( y_0,  Q_0^{-\frac{1}{2}}\right), \quad t \in [s_0, t_0].   \label{eqn:MJ23_6}
 \end{align}
 Therefore, we can apply the Hamilton-Perelman type distance estimate(cf. Lemma 8.3(b) of~\cite{Pe1}) to obtain
 \begin{align*}
   \frac{d}{dt} d_{g(t)}(x_0, y_0) \geq -2(m-1) \left\{ \frac{2}{3} \cdot 16 K^{2} Q_{0} \cdot Q_{0}^{-\frac{1}{2}} + Q_{0}^{\frac{1}{2}} \right\} \geq -40m K^{2} Q_{0}^{\frac{1}{2}}, 
 \end{align*}
 whose integration over time gives
\begin{align*}
  d_{g(t_0)}(y_0,x_0)-d_{g(s_0)}(y_0,x_0)=\int_{s_0}^{t_0} \frac{d}{dt} d_{g(t)}(x_0, y_0) dt\geq -40 m K^{2}Q_0^{\frac{1}{2}} (t_0-s_0) \geq  -40mK^{2}Q_0^{-\frac{1}{2}}.
 \end{align*}
 Plugging (\ref{eqn:MJ22_2}) into the above equation, we have 
 \begin{align*}
   d_{g(s_0)}(y_0, x_0) \leq d_{g(t_0)}(y_0,x_0) +40 m K^{2}Q_0^{-\frac{1}{2}} =2K\sqrt{t_{0}}-F_0 Q_0^{-\frac{1}{2}} +40 mK^{2}Q_0^{-\frac{1}{2}}.
 \end{align*}
 Combining with (\ref{eqn:MJ19_4}) and the assumption $s_{0} \in [t_{0}-Q_{0}^{-1}, t_{0}]$, the above inequality yields that
 \begin{align*}
   2K\sqrt{t_{0}-Q_{0}^{-1}}-\frac{4}{5} F_0Q_0^{-\frac{1}{2}} \leq 2K\sqrt{s_{0}}-\frac{4}{5} F_0Q_0^{-\frac{1}{2}} \leq  2K\sqrt{t_{0}}-F_0 Q_0^{-\frac{1}{2}} +40 mK^{2}Q_0^{-\frac{1}{2}}. 
 \end{align*}
 Since $\sqrt{1-(Q_{0}t_{0})^{-1}} \geq 1- (Q_{0}t_{0})^{-1}$, by elementary calculation, the above inequality can be simplified as 
 \begin{align}
   F_{0} \leq 200m K^{2} + \frac{10K}{\sqrt{Q_{0}t_{0}}} \leq 300m K^{2}.   \label{eqn:RJ15_6}
 \end{align}
 However, it follows from assumptions (\ref{eqn:RJ13_2}) and (\ref{eqn:RI06_3}) that 
 \begin{align*}
   F_{0} \geq \sqrt{\frac{\alpha_{m}A}{2}} \geq \frac{\alpha_{m}^{-2} K^{2}}{2}=5000m^{2}K^{2},
 \end{align*}
 which contradicts (\ref{eqn:RJ15_6}). Therefore, we obtain (\ref{eqn:MJ19_3}) and finish the proof of Step 1. \\

 \noindent
 \textit{Step 2. (\ref{eqn:RI06_3}) implies  (\ref{eqn:MJ19_1}).}\\

 By step 1, we know $s_0 \leq t_{0}-Q_{0}^{-1}$. Therefore, (\ref{eqn:MJ19_1}) follows directly from (\ref{eqn:RJ15_7}). 

 \end{proof}

\begin{remark}
  The method we applied to find the well-behaved space-time domain $\mathcal{H}$ in Proposition~\ref{prn:ML23_1} can be naturally regarded as a parabolic version of
  the one in Theorem~\ref{thm:RH27_10}, as illustrated by Figure~\ref{fig:pointselecting_ball} and Figure~\ref{fig:welldomain}. 
  In Perelman~\cite{Pe1}, such choice was achieved through an iterating point-selecting method.
  Here we follow a more traditional route to achieve a similar purpose, by choosing a proper auxiliary function and selecting maximum value point.
  Furthermore, the domain $\mathcal{H}$ obtained here is always in
  the ``central'' part, which property is not available in Perelman's original proof and is important for our improvement of Perelman's pseudo-locality theorem.
\label{rmk:RJ04_1}
\end{remark}

The reason we choose constant $\frac{\alpha_{m}A}{2}$ in condition (\ref{eqn:RI06_3}) is for the purpose of maximum principle and contradiction argument.
Roughly speaking, under the condition $F^{2} \leq \alpha_{m}A$ in a small domain, we can apply the entropy gap property to improve the estimate to $F^{2} \leq \frac{\alpha_{m}A}{2}$ in a large domain.
Following the notation in Proposition~\ref{prn:MJ24_1} , by setting  $\alpha=0.5$ and $H=L=4$, we define
\begin{align}
  &\delta_m \coloneqq \delta(0.5,m,4,4),   \label{eqn:RJ13_4}\\
  &A_{m} \coloneqq \max \left\{ \alpha_{m}^{-5}, \alpha_{m}^{-4} \delta_m^{-2} \right\}.   \label{eqn:MJ25_1}
\end{align}

\begin{proposition}[\textbf{Improvement of later curvature by almost non-negative initial functional value}]
Suppose that $\left\{(M^m, g(t)),  0 \leq t \leq r_0^2 \right\}$ is a Ricci flow solution satisfying
\begin{align}
 &t \cdot Rc(x,t) \leq (m-1) \alpha_{m} A, \quad \forall \; x \in B_{g(t)} \left(x_0, \sqrt{t} \right), \; t \in (0, r_0^2]; \label{eqn:MJ23_1}\\
 &\inf_{0<t \leq r_{0}^{2}} \boldsymbol{\mu}\left(B_{g(0)}\left(x_0, 20A\sqrt{t} \right), g(0), t\right)  
   \geq -A^{-2} \label{eqn:RJ15_8}  
\end{align}
for a constant
\begin{align}
  A \geq A_{m}K^2.   \quad \label{eqn:RJ16_2}
\end{align}
Here $K \geq 1$ is a parameter and $A_{m}$ is the constant defined in (\ref{eqn:MJ25_1}). 
Then we have curvature estimate
\begin{align}
  |Rm|(x,t)\left(2K\sqrt{t}-d_{g(t)}(x,x_{0}) \right)^{2}< \frac{\alpha_{m}A}{2}, \quad \forall \; x \in B_{g(t)} \left(x_0, 2K\sqrt{t} \right), \; t \in (0, r_0^2].  \label{eqn:RJ15_18}
\end{align}
In particular, we have
\begin{align}
t|Rm|(x,t) < \frac{\alpha_{m}A}{2K^{2}}, \quad \forall \; x \in B_{g(t)} \left(x_0, K\sqrt{t} \right), \; t \in (0, r_0^2].   \label{eqn:MJ23_2}  
\end{align}

\label{prn:CA04_1}
\end{proposition}

\begin{proof}
It is clear that (\ref{eqn:MJ23_2}) follows directly from (\ref{eqn:RJ15_18}).
Therefore, we shall only prove (\ref{eqn:RJ15_18}). 
Define $\Omega_{t}$, $\mathcal{D}$ and $F$ as in (\ref{eqn:MJ22_4}), (\ref{eqn:MJ20_0}) and (\ref{eqn:MJ19_0}) in the proof of Proposition~\ref{prn:ML23_1}.
Then (\ref{eqn:RJ15_18}) is equivalent to
\begin{align}
  F^{2}(x,t)<\frac{\alpha_{m}A}{2}.  \label{eqn:RJ15_9}
\end{align}
We argue by contradiction. For otherwise, we have 
\begin{align}
  F_{0}^{2}=F^{2}(y_{0},t_{0})=\sup_{\mathcal{D}} F^{2}  \geq \frac{\alpha_{m}A}{2}.  \label{eqn:RJ14_2}
\end{align}
Then Proposition~\ref{prn:ML23_1} applies.  Around the maximum value point $(y_{0}, t_{0})$,  the conditions (\ref{eqn:MJ19_3}) and (\ref{eqn:MJ19_1}) now read as 
 \begin{align}
  &B_{g(t)}\left(y_0, \frac{1}{5}F_0 Q_0^{-\frac{1}{2}} \right)  \subset B_{g(t)}\left(x_0,  2K\sqrt{t} \right),
  \quad \forall \; t \in \left[t_0-Q_0^{-1}, t_0 \right];    \label{eqn:RJ15_10}\\
  &|Rm|(x,t) \leq 4 Q_0, \quad \forall \; x \in B_{g(t)} \left(y_0, \frac{1}{5}F_0 Q_0^{-\frac{1}{2}} \right), \; t \in \left[t_0-Q_0^{-1}, t_0 \right]. \label{eqn:RJ15_11}   
 \end{align}
 By their choices, it is clear that $2K \leq 8A$ and $F_{0}>5\delta_{m}^{-1}$. Thus we have
 \begin{align*}
   B_{g(t)}\left(y_0, \delta_{m}^{-1}Q_0^{-\frac{1}{2}} \right) \subset B_{g(t)}\left(y_0, \frac{1}{5} F_0 Q_0^{-\frac{1}{2}} \right)
   \subset B_{g(t)}\left(x_0,  2K\sqrt{t} \right) \subset B_{g(t)}\left(x_0,  8A\sqrt{t} \right) 
 \end{align*}
 for each $t \in [t_{0}-Q_{0}^{-1}, t_{0}]$. Consequently, we have
 \begin{align}
   \boldsymbol{\mu}\left(B_{g(t)}\left(x_0, \delta_{m}^{-1} Q_{0}^{-\frac{1}{2}}\right), g(t), Q_{0}^{-1} \right) 
   \geq 
   \boldsymbol{\mu}\left(B_{g(t)}\left(x_0, 8A\sqrt{t} \right), g(t), Q_{0}^{-1} \right) 
 \label{eqn:RJ17_6} 
 \end{align}
 for each $t \in [t_{0}-Q_{0}^{-1}, t_{0}]$. 
 In inequality (\ref{eqn:MJ17_9}) of Theorem~\ref{thm:CA02_3}, by setting $\tau_T=Q_{0}^{-1}$ and $T=t_{0}' \coloneqq t_{0}-Q_{0}^{-1}$, we  obtain
\begin{align*}
  &\quad \boldsymbol{\mu}\left(B_{g(t_{0}')}\left(x_0, 8A\sqrt{t_{0}'} \right), g(t_{0}'), Q_{0}^{-1} \right) 
   \geq -A^{-2} + \boldsymbol{\mu}\left(B_{g(0)}\left(x_0, 20A\sqrt{t_{0}'} \right), g(0), t_{0} \right)\\
  &\geq -A^{-2} + \boldsymbol{\mu}\left(B_{g(0)}\left(x_0, 20A\sqrt{t_{0}} \right), g(0), t_{0} \right)
   \geq -A^{-2} +\inf_{0<t \leq r_{0}^{2}} \boldsymbol{\mu}\left(B_{g(0)}\left(x_0, 20A\sqrt{t} \right), g(0), t\right). 
\end{align*}
Plugging (\ref{eqn:RJ15_8}) into the above inequality yields
\begin{align}
  \boldsymbol{\mu}\left(B_{g(t_{0}')}\left(x_0, 8A\sqrt{t_{0}'} \right), g(t_{0}'), Q_{0}^{-1} \right) \geq -2A^{-2}.   \label{eqn:RJ17_7} 
\end{align}
It follows from the combination of (\ref{eqn:RJ17_7}) and (\ref{eqn:RJ17_6}) that
\begin{align}
  \boldsymbol{\mu}\left(B_{g(t_{0}')}\left(x_0, \delta_{m}^{-1} Q_{0}^{-\frac{1}{2}}\right), g(t_{0}'), Q_{0}^{-1} \right) \geq -2A^{-2}.   \label{eqn:RJ17_8}
\end{align}
On the other hand, the curvature estimate (\ref{eqn:RJ15_11}) implies that
 \begin{align*}
   |Rm|(x,t) \leq 4 Q_0 =4|Rm|(y_0,t_0), \quad \forall \; x \in B_{g(t)} \left(y_0, \delta_{m}^{-1} Q_0^{-\frac{1}{2}} \right), \; t \in \left[t_0-Q_0^{-1}, t_0 \right]. 
 \end{align*} 
Let $\tilde{g}(t)=Q_{0}g\left(Q_{0}^{-1}t + t_{0}-Q_{0}^{-1} \right)$. Then $\left\{ (M^{m}, y_{0}, \tilde{g}(t)), 0 \leq t \leq 1 \right\}$ is a Ricci flow solution satisfying
\begin{align*}
  &|Rm|_{\tilde{g}}(x,t) \leq 4, \quad \forall \; x \in B_{\tilde{g}(t)}(y_{0}, \delta_{m}^{-1}), \; t \in [0, 1];\\
  &|Rm|_{\tilde{g}}(y_{0},1)=1>0.5.
\end{align*}
Therefore, we can apply Proposition~\ref{prn:MJ24_1} to obtain
\begin{align}
  \boldsymbol{\mu} \left( B_{g(t_0')}\left(y_0, \delta_{m}^{-1} Q_0^{-\frac{1}{2}} \right), g(t_0'), Q_0^{-1} \right)
  =\boldsymbol{\mu} \left( B_{\tilde{g}(0)}\left(y_0, \delta_{m}^{-1} \right), \tilde{g}(0), 1\right)  \leq - \delta_{m}.
  \label{eqn:RJ16_3}
\end{align}
Recall that $A$ satisfies (\ref{eqn:RJ16_2}),  $\delta_{m}$ and $A_{m}$ are defined in (\ref{eqn:RJ13_4}) and (\ref{eqn:MJ25_1}). 
It follows from the combination of (\ref{eqn:RJ17_8}) and (\ref{eqn:RJ16_3}) that $\delta_{m} \leq A^{-2}$, which is impossible by the choice of
$\delta_{m}$ and $A$. This contradiction establishes the proof of (\ref{eqn:RJ15_9}) and thus the proof of (\ref{eqn:MJ23_2}).  
\end{proof}

We are now ready to prove a rough pseudo-locality theorem.

\begin{theorem}[\textbf{Pseudo-locality with rough curvature estimate}]
Suppose $\left\{(M^m, g(t)),  0 \leq t \leq T\right\}$ is a Ricci flow solution satisfying
\begin{align}
 \inf_{0<t \leq T} \boldsymbol{\mu}\left(B_{g(0)}\left(x_0, 20A\sqrt{t} \right), g(0), t \right) \geq -A^{-2}
\label{eqn:MJ23_13}  
\end{align}
for some $A \geq A_{m}$, which is a large dimensional constant defined in (\ref{eqn:MJ25_1}). 
Then we have
\begin{align}
t|Rm|(x,t) \leq \alpha_{m}A, \quad \forall \; x \in B_{g(t)}\left(x_0, \sqrt{t} \right), \; t \in (0, T].  \label{eqn:MJ23_14}
\end{align}
\label{thm:MJ22_1}
\end{theorem}

\begin{proof}
Since we are working on a smooth Ricci flow solution, it is clear that (\ref{eqn:MJ23_14}) holds when $t$ is very small.
Therefore, if (\ref{eqn:MJ23_14}) fails at some $(x,t)$, then there must exist a first time $t=r_0^2 \in (0,T)$ and a point $z_0 \in B_{g(r_0^2)}(x_0, r_0)$ such that
 (\ref{eqn:MJ23_14}) holds for all $(x, t)$ satisfying $x \in B_{g(t)}\left(x_0, \sqrt{t} \right)$ and $t \in \left(0, r_0^2 \right]$. Moreover, we have
\begin{align}
  r_0^2 |Rm|(z_0, r_0^2)=\alpha_{m}A.    \label{eqn:MJ23_21}
\end{align}
Since $|Rc| \leq (m-1)|Rm|$, the above curvature condition guarantees that Proposition~\ref{prn:CA04_1} can be applied on the flow $\{(M^m, x_{0}, g(t)), 0 \leq t \leq r_0^2\}$. 
It follows from (\ref{eqn:MJ23_13}) that
\begin{align*}
  \inf_{0<t \leq r_{0}^{2}}  \boldsymbol{\mu}\left(B_{g(0)}\left(x_0, 20A\sqrt{t} \right), g(0), t\right) 
 \geq \inf_{0<t \leq T}  \boldsymbol{\mu}\left(B_{g(0)}\left(x_0, 20A\sqrt{t} \right), g(0), t\right) 
 \geq -A^{-2}. 
\end{align*}
Therefore, we can apply (\ref{eqn:MJ23_2}) in Proposition~\ref{prn:CA04_1} with $K=1$ to obtain
\begin{align*}
  |Rm|(z_0, r_0^2) \leq \frac{\alpha_{m}A}{2 r_0^2},
\end{align*}
which contradicts (\ref{eqn:MJ23_21}).  This contradiction establishes the proof of (\ref{eqn:MJ23_14}). 

\end{proof}

Based on Proposition~\ref{prn:MJ24_1} and Proposition~\ref{prn:CA04_1}, we can improve Theorem~\ref{thm:MJ22_1} to the following version.

\begin{theorem}[\textbf{Pseudo-locality with refined coefficients}]

  For each $\alpha \in (0, \alpha_{m})$,  there exists a large constant $A_{\alpha}$ depending on $\alpha$ and $m$ with the following properties. 

Suppose $\left\{(M^m, g(t)),  0 \leq t \leq T\right\}$ is a Ricci flow solution satisfying
\begin{align}
  \inf_{0<t \leq T} \boldsymbol{\mu}\left(B_{g(0)}\left(x_0, 20A \sqrt{t} \right), g(0), t \right) \geq -A^{-2}
\label{eqn:MJ24_2}  
\end{align}
for some $A \geq A_{\alpha}$. 
Then for each $t \in (0, T]$ and $x \in B_{g(t)}\left(x_0, \alpha^{-1}\sqrt{t} \right)$, we have
\begin{align}
\begin{cases}
 & t|Rm|(x,t) \leq \alpha,  \\
 &  \inf_{\rho \in (0, \alpha^{-1} \sqrt{t})} \rho^{-m}\left|B_{g(t)}\left(x, \rho \right) \right|_{dv_{g(t)}} \geq (1-\alpha) \omega_m,\\
 &t^{-\frac{1}{2}} \cdot  inj(x,t) \geq \alpha^{-1}. 
\end{cases} 
\label{eqn:MJ24_3}
\end{align}
\label{thm:MJ24_1}
\end{theorem}

\begin{proof}
Using the notation in (\ref{eqn:MJ25_1}) and Proposition~\ref{prn:MJ24_1}, we define 
\begin{align}
  &\hat{\delta}_{\alpha} \coloneqq \delta(0.1 \alpha,m,A_{m}, 10 \alpha^{-1}), \label{eqn:RJ17_2} \\
  &K_{\alpha} \coloneqq 10 \max\left\{  \hat{\delta}_{\alpha}^{-1}, \alpha^{-1} \right\},  \label{eqn:RJ17_3}\\
  &A_{\alpha} \coloneqq \max\left\{ A_{m}, \hat{\delta}_{\alpha}^{-1} \right\} K_{\alpha}^{2}.  \label{eqn:RJ17_4}
\end{align}
After we fix the notations as above, we are going to prove (\ref{eqn:MJ24_3}) by a contradiction argument. 
If (\ref{eqn:MJ24_3}) were wrong,  we could find some $t=r_0^2<T$ such that (\ref{eqn:MJ24_3}) fails the first time at $t$.
Up to rescaling, we can assume $r_{0}=1 \leq T$ without loss of generality. 
Therefore, we can find a point $z_0 \in B_{g(1)} \left(x_0, \alpha^{-1} \right)$ such
that either
\begin{align}
  |Rm|(z_0, 1)=\alpha,   \label{eqn:MJ24_10}
\end{align}
or 
\begin{align}
\rho_0^{-m} |B_{g(1)}(z_0, \rho_0)|=(1-\alpha) \omega_m, \quad \textrm{for some} \quad \rho_0 \in (0, \alpha^{-1}],   \label{eqn:MJ24_11}
\end{align}
or
\begin{align}
  inj(z_0, 1) = \alpha^{-1}. 
  \label{eqn:RH27_5}
\end{align}

According to the choice of $A_{\alpha}$ and $K_{\alpha}$ in (\ref{eqn:RJ17_3}) and (\ref{eqn:RJ17_4}), it is clear that $A_{\alpha} \geq A_{m}K_{\alpha}^{2}$ and $K_{\alpha} \geq 1$.
Thus (\ref{eqn:MJ24_2}) implies that 
\begin{align}
  \inf_{0<t \leq 1} \boldsymbol{\mu}\left(B_{g(0)}\left(x_0, 20A_{m}K_{\alpha}^{2} \sqrt{t} \right), g(0), t \right)  \geq -\left( A_{m} K_{\alpha}^{2} \right)^{-2}. 
\label{eqn:RJ17_1}  
\end{align}
Therefore, Theorem~\ref{thm:MJ22_1} applies and the following condition holds:
\begin{align*}
Rc(x,t) \leq (m-1)\alpha_{m}A, \quad \forall \; x \in B_{g(t)}(x_{0}, \sqrt{t}), \; t \in (0, 1].
\end{align*}
Consequently, Proposition~\ref{prn:CA04_1} applies. Since $\hat{\delta}_{\alpha}^{-1} \leq K_{\alpha}$ by (\ref{eqn:RJ17_3}),  it follows from (\ref{eqn:MJ23_2}) that 
\begin{align}
t|Rm| < \frac{\alpha_{m}A_{m}}{2}<0.1A_{m}, \quad \forall \; x \in B_{g(t)}(x_{0},  \hat{\delta}_{\alpha}^{-1} \sqrt{t}), \; t \in (0, 1]. \label{eqn:RJ17_5}
\end{align}
Note that $A_{\alpha} \geq \hat{\delta}_{\alpha}^{-1}$ by (\ref{eqn:RJ17_3}) and (\ref{eqn:RJ17_4}). 
Following the deduction of (\ref{eqn:RJ17_8}) in Proposition~\ref{prn:CA04_1}, it follows from (\ref{eqn:MJ24_2}) that 
\begin{align}
  \boldsymbol{\mu}\left(B_{g(0.5)}\left(x_0, 5 \hat{\delta}_{\alpha}^{-1}  \right), g(0.5), 0.5 \right) 
  \geq \inf_{0<t \leq T} \boldsymbol{\mu}\left(B_{g(0)}\left(x_0, 20A \sqrt{t} \right), g(0), t \right) - \hat{\delta}_{\alpha}^{2}  \geq -2\hat{\delta}_{\alpha}^{2}.
  \label{eqn:MJ24_8} 
\end{align}
Similar to the argument around (\ref{eqn:RJ16_3}) in the proof of Proposition~\ref{prn:CA04_1}, we can apply proper rescaling and time-shifting
so that we can apply Proposition~\ref{prn:MJ24_1} here.  
It follows from (\ref{eqn:RJ17_5}) and (\ref{eqn:MJ24_8}) and the definition of $\delta_{0}$ in (\ref{eqn:RJ17_2}) that 

\begin{align*}
    &|Rm|(x,1) \leq 0.5 \alpha, \\
    &\inf_{\rho \in (0, 2\alpha^{-1})} \rho^{-m} \left|B_{g(1)}(x,\rho) \right|_{dv_{g(1)}}  \geq (1-0.5 \alpha) \omega_{m}, \\
    &inj(x,1) \geq 2\alpha^{-1}, 
\end{align*}
for all $x \in B_{g(1)}(x_0, 2 \alpha^{-1})$. 
Therefore, the above inequalities imply that the identities (\ref{eqn:MJ24_10})-(\ref{eqn:RH27_5}) all fail at $(z_0, 1)$. 
This contradicts our choice of $(z_0, 1)$. 
The proof of Theorem~\ref{thm:MJ24_1} is complete. 
\end{proof}

Note that Theorem~\ref{thm:MJ22_1} is an intermediate step for Theorem~\ref{thm:MJ24_1}, the pseudo-locality theorem with refined coefficients.   
This step is technically necessary for the application of Proposition~\ref{prn:CA04_1}, where one can improve curvature estimate in a ``big" domain whenever there is
a ``large" curvature point appears in the ``central" space-time.    If we try to prove  (\ref{eqn:MJ24_3}) firstly by contradiction, then Proposition~\ref{prn:CA04_1} cannot be applied to obtain a priori curvature bound
in a ``big" domain, which prevents the further application of  Proposition~\ref{prn:MJ24_1}.    

\begin{corollary}[\textbf{Pseudo-locality in terms of $\boldsymbol{\nu}$}]
Suppose $\left\{(M^m, g(t)),  0 \leq t \leq T\right\}$ is a Ricci flow solution satisfying
\begin{align}
  \boldsymbol{\nu}\left(B_{g(0)}\left(x_0, 20 A \sqrt{T} \right), g(0), T \right)  \geq -A^{-2}, 
\label{eqn:ML13_3}  
\end{align}
where  $A \geq A_{\alpha}$(cf. (\ref{eqn:RJ17_4})). 
Then (\ref{eqn:MJ24_3}) holds.   
\label{cly:RJ18_0}
\end{corollary}

\begin{proof}
  For each $0<t \leq T$,  it is clear that
  \begin{align*}
      &\quad \boldsymbol{\mu}\left(B_{g(0)}\left(x_0, 20A\sqrt{t} \right), g(0), t\right) \geq \boldsymbol{\nu}\left(B_{g(0)}\left(x_0, 20A\sqrt{t} \right), g(0), t\right)\\
      &\geq \boldsymbol{\nu}\left(B_{g(0)}\left(x_0, 20A\sqrt{t} \right), g(0), T\right) \geq \boldsymbol{\nu}\left(B_{g(0)}\left(x_0, 20A\sqrt{T} \right), g(0), T\right). 
  \end{align*}
  Taking infimum on the two ends of the above inequalities and applying (\ref{eqn:ML13_3}), we obtain
  \begin{align*}
  \inf_{t \in (0,T]} \boldsymbol{\mu}\left(B_{g(0)}\left(x_0, 20A\sqrt{t} \right), g(0), t\right)
  \geq \boldsymbol{\nu}\left(B_{g(0)}\left(x_0, 20A\sqrt{T} \right), g(0), T\right) \geq -A^{-2}. 
  \end{align*}
  Therefore,  we can apply Theorem~\ref{thm:MJ24_1} to obtain (\ref{eqn:MJ24_3}). 
\end{proof}

\begin{corollary}[\textbf{$\boldsymbol{\nu}$-value and the existence time of flow}]
  Suppose $(M^{m},g)$ is a complete Riemannian manifold with bounded curvature satisfying
  \begin{align}
    \inf_{x \in M} \boldsymbol{\nu}(B(x,r),g,r^{2}) \geq -A^{-2}  \label{eqn:RJ17_10}
  \end{align}
  for some $r>0$ and $A \geq A_{\alpha}$(cf. (\ref{eqn:RJ17_4})).  Then there exists a Ricci flow solution  initiated from $g(0)=g$ and satisfies (\ref{eqn:MJ24_3}) on $M \times [0, T)$ with
  \begin{align}
     T \geq \frac{r^{2}}{400A^{2}}.  \label{eqn:RJ18_2}
  \end{align}
   \label{cly:RJ17_9}
\end{corollary}

\begin{proof}
  By the work of Shi(cf.~Theorem 1.1 of~\cite{Shi}), there exists a Ricci flow $\{(M,g(t)), 0 \leq t <T\}$ with  $T>0$ being the maximal existence time.  
  We now show (\ref{eqn:RJ18_2}) by a contradiction argument. For otherwise, we have $T<\frac{r^{2}}{400A^{2}}$. 
  It follows from (\ref{eqn:RJ17_10}) that 
  \begin{align*}
    &\quad \boldsymbol{\nu}\left(B\left(x,20A\sqrt{T} \right),g,T \right) \geq \boldsymbol{\nu}\left(B\left(x,20A\sqrt{T} \right),g,\frac{r^{2}}{400A^{2}} \right)\\ 
    &\geq  \boldsymbol{\nu}\left(B(x,r),g,\frac{r^{2}}{400A^{2}} \right) \geq \boldsymbol{\nu}(B(x,r),g,r^{2}) \geq -A^{-2}.   
  \end{align*}
  Therefore, Corollary~\ref{cly:RJ18_0} applies and inequality (\ref{eqn:MJ24_3}) holds until time $T$. In particular, we have
  \begin{align*}
    \lim_{t \to T_{-}}  \sup_{x \in M} |Rm|(x,t) < \frac{\alpha}{T}<\infty,
  \end{align*}
  which contradicts the assumption that $T$ is the maximal existence time(cf.~\cite{Ha82}). 
\end{proof}

\begin{remark}
One can regard Corollary~\ref{cly:RJ18_0} as the ``almost local" version, or the quantitative version,  of the rigidity property Proposition~\ref{prn:RB06_1},
in terms of the $\boldsymbol{\nu}$-functional.   Theorem~\ref{thm:MJ24_1} is the further localization of Corollary~\ref{cly:RJ18_0} and it is the ``almost local" version of Proposition~\ref{prn:RH24_1} in terms of $\boldsymbol{\mu}$-functional. 
In short, the pseudo-locality theorems can be illustrated as quantitative versions of rigidity properties. 
\label{rmk:RH25_1}
\end{remark}

\begin{remark}
In Theorem~\ref{thm:MJ24_1}, Corollary~\ref{cly:RJ18_0} and Corollary~\ref{cly:RJ17_9}, for each  small $\alpha$, we can choose $A$ large enough such that (\ref{eqn:MJ24_3}) holds for the given $\alpha$. 
For simplicity of notation, we shall replace $\alpha$ by $\psi(A^{-1}|m)$ in future discussion. 
\label{rmk:RJ18_8}
\end{remark}

The condition (\ref{eqn:MJ24_2}) looks clumsy. However, we shall see from the following lemma that it is natural under the condition of Ricci curvature being uniformly bounded from below.

\begin{lemma}[\textbf{$\boldsymbol{\mu}$ by almost Euclidean volume ratio and Ricci curvature lower bound}]
  For each pair of positive numbers $(\eta, A)$, there exists a constant $\xi=\xi(\eta, m, A)$ with the following property.

Suppose 
\begin{align}
    &Rc(x, 0) \geq -(m-1)\xi,  \quad \forall \; x \in B\left(x_0,  \xi^{-1} \right); \label{eqn:MJ23_18}\\
    &\xi^{m} \left|B \left(x_0, \xi^{-1} \right) \right|  \geq (1-\xi) \omega_m.  \label{eqn:MJ23_19}
  \end{align}
 Then we have
 \begin{align}
   &\bar{\boldsymbol{\mu}} (B(x_0,20 A), g, 1) \geq -0.5 \eta,     \label{eqn:MJ23_20}\\
   &\boldsymbol{\mu} (B(x_0,20 A), g, 1) \geq - \eta.    \label{eqn:MJ25_2}
  \end{align}
\label{lma:MJ23_1}
\end{lemma}

\begin{proof}
 
 We shall show (\ref{eqn:MJ23_20}) first and then proceed to show (\ref{eqn:MJ25_2}). 

 The proof of (\ref{eqn:MJ23_20}) basically follows from the proof of Proposition 3.1 of Tian-Wang~\cite{TiWa}. We repeat the key steps here, using the local functionals, for the convenience of the readers. 

 For otherwise, we can choose a violating sequence for $\xi_i \to 0$ satisfying (\ref{eqn:MJ23_18}) and (\ref{eqn:MJ23_19}) for $\xi_i$.
 However, (\ref{eqn:MJ23_20}) is violated.  
 By Cheeger-Colding theory, we can assume that
 \begin{align*}
    (M_i, x_i, d_i)  \longright{pointed-Gromov-Hausdorff}  (M_{\infty}, x_{\infty}, d_{\infty}). 
 \end{align*}
 By volume convergence and Bishop-Gromov volume comparison, it is not hard to see that the asymptotic volume ratio of $M_{\infty}$ is $\omega_m$.
 This force that $M_{\infty}$ is actually isometric to $(\R^m, g_{E})$. 
 Suppose $\varphi_i$ is a minimizer function for $\bar{\boldsymbol{\mu}}(B(x_i,20A), g_i, 1)$. Then $\varphi_i$ satisfies the Euler-Lagrange equation
 \begin{align}
   -2\Delta \varphi_i -2 \varphi_i \log \varphi_i -m \left( 1+ \log \sqrt{4\pi} \right) \varphi_i= \bar{\boldsymbol{\mu}}_i \varphi_i,
   \label{eqn:RI06_1}  
 \end{align}
 for  $\bar{\boldsymbol{\mu}}_i=\bar{\boldsymbol{\mu}}(B(x_i, 20A), g_i, 1) \leq -\eta$. 
 Note that the Sobolev constant estimate and the Moser iteration argument guarantee that  $\frac{1}{C} \leq \varphi_{i, max} \leq C$ and $\bar{\boldsymbol{\mu}}_i \geq -C$ uniformly. 
 Furthermore, there are uniform interior Lipschitz estimate and uniform boundary  H\"{o}lder estimate, as applications of Cheng-Yau's estimate
 on the above equation.    Then we can show that $\varphi_i$ converges to a limit function $\varphi_{\infty}$ which is Lipschitz in each compact subset of $B(x_{\infty}, 20A)$ and vanishes on
 $\partial B(x_{\infty}, 20A)$.  Furthermore, by taking limit of (\ref{eqn:RI06_1}), we can show that $\varphi_{\infty}$ satisfies the equation
 \begin{align*}
   -2 \Delta \varphi_{\infty} -2 \varphi_{\infty}  \log \varphi_{\infty}-m \left( 1+ \log \sqrt{4\pi} \right) \varphi_{\infty}= \bar{\boldsymbol{\mu}}_{\infty} \varphi_{\infty},
 \end{align*}
 for some $-\infty<\bar{\boldsymbol{\mu}}_{\infty} \leq -\eta$ in the distribution sense on $B(x_{\infty}, 20A)$.   
 However, the above equation contradicts the Logarithmic Sobolev inequality in $(\R^m, g_{E})$ since
 \begin{align*}
    \varphi_{\infty} \in W_0^{1,2}(B(x_{\infty}, 20A)) \subset W_0^{1,2}(\R^{m}).
 \end{align*}
 Therefore, we finish the proof of (\ref{eqn:MJ23_20}). 
 
 We continue to show (\ref{eqn:MJ25_2}). Actually, it follows from (\ref{eqn:MJ16_6}) in Lemma~\ref{lma:MJ16_1} that
 \begin{align}
     \boldsymbol{\mu} (B(x_0,20 A), g, 1) - \bar{\boldsymbol{\mu}} (B(x_0,20 A), g, 1) \geq \underline{R} \geq -m(m-1)\xi.     \label{eqn:MJ25_3}
 \end{align}
 The $\underline{R}$ above means the lower bound of the scalar curvature.
 Note that $\xi$ is very small such that $B(x_0, 20A) \subset B(x_0, \xi^{-1})$, which fact together with (\ref{eqn:MJ23_18}) 
 are used in the deduction of the above inequality.
 By choosing $\xi$ even smaller, we can also assume that $2m(m-1)\xi \leq 0.5 \eta$.  Plugging this inequality into (\ref{eqn:MJ25_3}), we obtain
 \begin{align*}
   \boldsymbol{\mu} (B(x_0,20 A), g, 1)  \geq \bar{\boldsymbol{\mu}}  (B(x_0,20 A), g, 1) -0.5\eta. 
 \end{align*}
 Therefore, (\ref{eqn:MJ25_2}) follows from the combination of the above inequality with (\ref{eqn:MJ23_20}).
 The proof the lemma is complete. 
\end{proof}

Note that in Lemma~\ref{lma:MJ23_1}, we prove (\ref{eqn:MJ23_20}) first and then prove (\ref{eqn:MJ25_2}). This order is important.
Technically, one need uniform Lipschitz or H\"older continuity of minimizer functions to obtain a limit function.
To achieve such bounds, $\bar{\boldsymbol{\mu}}$ is more convenient than $\boldsymbol{\mu}$ as it does not contain the scalar curvature term $R$.
Actually, in order to develop Lipschitz or H\"older continuity for minimizers of $\boldsymbol{\mu}$,
one needs extra information of $\nabla R$, which is hard to obtain in general.

\begin{corollary}[\textbf{Perelman's pseudo-locality theorem}, cf. Theorem 10.1 and  Corollary 10.3 of~\cite{Pe1}]
For every $\alpha \in (0, \alpha_{m})$, there exist $\delta=\delta(\alpha,m)>0, \epsilon=\epsilon(\alpha,m)>0$ with the following properties.

Suppose $\left\{(M^m, g(t)),  0 \leq t \leq r_0^2\right\}$ is a Ricci flow solution satisfying the isoperimetric constant estimate and the scalar curvature estimate at time $t=0$:
\begin{align}
  &R(x,0) \geq -r_0^{-2}, \quad \forall \; x \in B_{g(0)}(x_0,r_0);  \label{eqn:ML13_2}\\
  &\mathbf{I}(B_{g(0)}(x_0, r_0), g(0)) \geq (1-\delta) \mathbf{I}(\R^m).  \label{eqn:ML13_1}
\end{align}
Then for each $t \in (0, (\epsilon r_0)^2]$ and $x \in B_{g(t)}\left(x_0, \epsilon r_0 \right)$, we have
\begin{align}
&|Rm|(x,t) \leq \alpha t^{-1}+ (\epsilon r_0)^{-2}, \label{eqn:MJ23_16}\\
&\left| B_{g(t)} \left(x, \sqrt{t} \right) \right|_{dv_{g(t)}} \geq \kappa(m) t^{\frac{m}{2}}, \label{eqn:MJ23_17}
\end{align}
where $\kappa(m)$ is a dimensional constant. 
\label{cly:MJ23_1}
\end{corollary}

\begin{proof}
  Up to rescaling, we may assume $r_{0}=1$ without loss of generality.
  We shall show that (\ref{eqn:MJ23_16}) and (\ref{eqn:MJ23_17}) hold by choosing
  \begin{align}
    \epsilon^{2}=\delta=\alpha_{m}^{2}A^{-2}, 
    \label{eqn:RJ18_6}
  \end{align}
  where $A=A_{\alpha}$ is defined in (\ref{eqn:RJ17_4}), $\alpha_{m}$ is the dimensional constant defined in (\ref{eqn:RJ13_3}). 

  Plugging the assumption (\ref{eqn:ML13_1}) into (\ref{eqn:MJ26_9}) of Lemma~\ref{lma:MJ25_1}, we can estimate $\boldsymbol{\nu}$ in terms of isoperimetric constant as follows: 
  \begin{align*}
    \boldsymbol{\nu} \left(B_{g(0)} \left(x_{0}, 1\right), g(0), \epsilon^{2}  \right)
    \geq m \log (1- \delta)-\epsilon^{2}>-2m \delta -\epsilon^{2}=-(2m+1)\alpha_{m}^{2}A^{-2}>A^{-2}.  
  \end{align*}
  It is clear that $20A\epsilon =20\alpha_{m}=\frac{1}{5m}<1$. Thus the triangle inequality implies the relationship
  \begin{align*}
    B_{g(0)}\left(y_0, 20A\epsilon \right) \subset B_{g(0)}(x_0, 1), \quad \forall \; y_0 \in B_{g(0)}(x_0, 0.5). 
  \end{align*}
  Consequently, we have 
  \begin{align*}
    \boldsymbol{\nu} \left(B_{g(0)} \left(y_{0}, 20A\epsilon\right), g(0), \epsilon^{2}\right)
    \geq \boldsymbol{\nu} \left(B_{g(0)} \left(x_{0}, 1\right), g(0), \epsilon^{2}  \right) > -A^{-2}. 
  \end{align*}
  Then Corollary~\ref{cly:RJ18_0} applies and gives
  \begin{align}
  &\sup_{x \in B_{g(t)}(y_{0}, \sqrt{t})} t|Rm|(x,t) \leq \alpha, \label{eqn:ML14_7} \\
  &\inf_{x \in B_{g(t)}(y_{0}, \sqrt{t})}\inf_{\rho \in (0, \alpha^{-1} \sqrt{t})} \rho^{-m}\left|B_{g(t)}\left(x, \rho \right) \right|_{dv_{g(t)}} \geq (1-\alpha) \omega_m,  \label{eqn:ML14_8}
  \end{align}
  for every $y_{0} \in B_{g(0)}(x_{0}, 0.5)$ and $t \in [0, \epsilon^{2}]$. 
  Comparing (\ref{eqn:MJ23_16}) with (\ref{eqn:ML14_7}), (\ref{eqn:MJ23_17}) with (\ref{eqn:ML14_8}), 
  it is clear that  (\ref{eqn:MJ23_16}) and (\ref{eqn:MJ23_17}) follow from  (\ref{eqn:ML14_7}) and (\ref{eqn:ML14_8}) respectively,
  if we can prove the following containment relationship: 
  \begin{align}
  B_{g(t)}\left(x_0,  \epsilon \right) \subset B_{g(0)} \left(x_0, 0.2 \right)  \label{eqn:ML13_5}
  \end{align}
  for each $t \in [0, \epsilon^{2}]$.
  Following the discussion below (\ref{eqn:MJ23_6}) in the proof of Proposition~\ref{prn:ML23_1}, 
  this can be achieved by the standard application of  Hamilton-Perelman distance estimate(cf. Lemma 8.3(b) of~\cite{Pe1}), via the curvature control (\ref{eqn:ML14_7}).
  Since the proof is almost the same as the one in Proposition~\ref{prn:ML23_1}, we omit it here and leave the details to the interested readers. 
  After the proof of (\ref{eqn:ML13_5}), it is clear that (\ref{eqn:MJ23_16}) and (\ref{eqn:MJ23_17}) follow from  (\ref{eqn:ML14_7}) and (\ref{eqn:ML14_8}) respectively, as discussed previously.  
 The proof of Corollary~\ref{cly:MJ23_1} is complete. 
\end{proof}

\begin{corollary}[\textbf{Tian-Wang's version of pseudo-locality theorem}, cf. Proposition 3.1 and Theorem 3.1 of~\cite{TiWa}]
  For each $\alpha \in (0, \alpha_{m})$, there exist constants $\delta=\delta(\alpha, m), \epsilon=\epsilon(\alpha, m)$ with the following properties.

  Suppose $\left\{ (X^{m}, g(t)), 0 \leq t \leq 1 \right\}$ is a Ricci flow solution, $x_0 \in X$. Suppose
  \begin{align}
    &Rc(x, 0) \geq -(m-1)\delta^4,  \quad \forall \; x \in B_{g(0)}\left(x_0,  \delta^{-1} \right). \label{eqn:pseudocondition_1}\\
    &\delta^{m} \left|B_{g(0)}\left(x_0, \delta^{-1} \right) \right|_{dv_{g(0)}} \geq (1-\delta) \omega_m.  \label{eqn:pseudocondition_2}
  \end{align}
  Then for each $t \in (0, \epsilon^2]$ and $x \in B_{g(t)} \left(x_0, \sqrt{t} \right)$, we have
  \begin{align}
    &|Rm|(x, t) \leq \alpha t^{-1} +\epsilon^{-2},    \label{eqn:pseudo_ricci}\\
    &\left| B_{g(t)} \left(x, \sqrt{t} \right) \right|_{dv_{g(t)}} \geq \kappa' t^{\frac{m}{2}}, \label{eqn:pseudo_kappa}
  \end{align}
  where $\kappa'=\kappa'(m)$ is a universal constant.
\label{cly:MJ23_2}
\end{corollary}

\begin{proof}
  Fix $A=A_{\alpha}$ and $\xi=\xi(A^{-2}, m, A)$ as in Lemma~\ref{lma:MJ23_1}.  
   By choosing $\delta$ sufficiently small and applying Bishop-Gromov volume comparison, (\ref{eqn:pseudocondition_1}) and (\ref{eqn:pseudocondition_2}) yield that
   \begin{align*}
    & Rc(x,0) \geq -(m-1) \xi r^{-2}, \quad \xi^m r^{-m} |B(x_0, \xi^{-1} r)| \geq (1-\xi) \omega_m
   \end{align*}
 for each $r \in (0, 1]$. By rescaling each $r=\sqrt{t}$ to $1$, it follows from Lemma~\ref{lma:MJ23_1} that
   \begin{align*}
     \inf_{0<t \leq 1} \boldsymbol{\mu}\left(B_{g(0)}\left(x_0, 20A\sqrt{t} \right), g(0), t \right)  \geq -A^{-2}. 
   \end{align*}
 Consequently,  Theorem~\ref{thm:MJ24_1} can be applied to obtain (\ref{eqn:MJ24_3}), whence (\ref{eqn:pseudo_kappa}) and (\ref{eqn:pseudo_ricci}).
 The proof of Corollary~\ref{cly:MJ23_2} is complete. 
\end{proof}

By Corollary~\ref{cly:MJ23_1} and Corollary~\ref{cly:MJ23_2}, we see that the pseudo-locality of Perelman and Tian-Wang can be unified in terms of the local entropy,
without using the deep analysis from RCD theory and mass transportation by Cavalletti-Mondino~\cite{CaMo}. 

We remind the readers that Theorem~\ref{thmin:ML14_2} was already proved.

\begin{proof}[Proof of Theorem~\ref{thmin:ML14_2}]
  By setting $\delta=\alpha_{m}A_{\alpha}^{-1}$ and adjusting the balls if necessary, Theorem~\ref{thmin:ML14_2} is implied by Theorem~\ref{thm:MJ24_1} and Corollary~\ref{cly:RJ18_0}.  
\end{proof}

\section{Distance distortion and topological stability}
\label{sec:topstable}

In this section, we localize the distance distortion estimate in section 4 of Tian-Wang~\cite{TiWa} and provide topological applications.
The basic idea is that when the volume element is decreasing and the distance is expanding, we shall have a rough distance distortion estimate along the flow, if the initial volume ratio has an upper bound.
This estimate can be refined if we further have the almost Einstein condition, i.e.,
$\int_{0}^{1}\int_{\Omega} |R| dv dt$ is very small for some fixed domain $\Omega$.
Under the condition of Tian-Wang's version of the pseudo-locality theorem(cf. Corollary~\ref{cly:MJ23_2}), if the initial time slice has almost Euclidean volume ratio, then the later time slice is almost Euclidean too.
Fix a unit geodesic ball $\Omega$ at the initial time. On the one hand,  the loss of its volume along the flow is measured by $\int_{0}^{1}\int_{\Omega} |R| dv dt$. On the other hand, the ``size'' of the set $\Omega$ 
is expanding along the flow, which together with the almost Euclidean condition at both time-slices make the volume almost increasing. 
The combination of these two effects force $\int_{0}^{1}\int_{\Omega} |R| dv dt$ to be very small.
Namely, the almost Einstein condition holds locally on $\Omega$.  Therefore, the distance distortion estimate follows directly from section 4 of~\cite{TiWa}.
Although no new idea beyond~\cite{TiWa} is needed, here we shall provide more details under more general conditions,  which will be important for the topological application(cf. Corollary~\ref{cly:RH27_1}) and further improvement in the K\"ahler setting(cf.~Theorem~\ref{thm:RJ08_0}). 

We briefly overview the structure of this section.  We first develop a local maximum principle in Theorem~\ref{thm:RD10_1}, 
which guarantees the almost non-negativity of scalar curvature in the setting of the pseudo-locality theorem. This step is necessary for showing the almost-decreasing property of volume element.
Then we follow the idea described above to obtain the locally almost Einstein condition and the distance distortion estimate(cf. Proposition~\ref{prn:RE04_3} and Theorem~\ref{thm:RE05_5}).
Combining these with Hochard's method, we obtain Ricci flow solutions with life span estimate and distance distortion estimate(cf.~Corollary~\ref{cly:CK21_1}). 
In particular, starting from a complete Riemannian manifold with non-negative Ricci curvature and almost Euclidean asymptotic volume ratio, we obtain an immortal Ricci flow solution(cf.~Corollary~\ref{cly:RH27_1}), 
which is used to construct a diffeomorphism from $M$ to $\R^{m}$ in Theorem~\ref{thm:CK21_2}. Such a flow can also be used to show global rigidities 
of $(M, g)$ (cf.~Theorem~\ref{thm:RC12_1}) and local estimates of manifolds with Ricci curvature bounded below(cf.~Theorem~\ref{thm:RH25_1}).
Finally, we analyze the asymptotic behavior of the immortal solution and show it must diverge at infinity(cf. Proposition~\ref{prn:RE08_4}, Proposition~\ref{prn:RE14_12} and Corollary~\ref{cly:RE14_14}),
justifying that the distance distortion estimate is optimal, at least on non-compact manifolds.
Near the discussions on rigidities, we also define some regularity radii(cf. Definition~\ref{dfn:RH26_3}) for later applications.\\

The following local maximum principle is motivated by Theorem 3.1 of B.L.Chen~\cite{CBL07}. It will be used in this section to show the almost non-negativity of the scalar curvature
and later in Section~\ref{sec:kpseudo} to show the almost non-increasing of the maximum value of the scalar curvature. 

\begin{theorem}[\textbf{Localized maximum principle}]
For each large constant $A \geq 1000$, the following properties hold. 

  Suppose $\left\{ (M, g(t)), 0 \leq t \leq r^2 \right\}$ is a Ricci flow solution satisfying
  \begin{align}
    t \cdot Rc(x,t) \leq \alpha_{m} \cdot (m-1) A, \quad \forall \; x \in B_{g(t)}(x_0, \sqrt{t}), \; t \in [0, r^2],   \label{eqn:RJ04_4} 
  \end{align}
  where $\alpha_{m}$ is the constant defined in (\ref{eqn:RJ13_3}). 
  Suppose  $f$ is a smooth function satisfying
  \begin{align}
    &\square f \leq 0, \quad \forall \; x \in B_{g(t)}(x_0, 10 Ar), \; t \in [0,r^2];  \label{eqn:RD11_4}\\
    &t \cdot f(x,t)<A, \quad \forall \; x \in B_{g(t)}(x_0, 10 Ar), \; t \in [0,r^2]; \label{eqn:RD28_1}\\
    &\sigma \coloneqq  \sup_{x \in B_{g(0)}(x_{0},10Ar)} r^2 \cdot f_{+}(x,0) \leq  A.  \label{eqn:RD28_2}
  \end{align}
  Here $\square=\frac{\partial}{\partial t} -\Delta$ and $f_{+}=\max\left\{ f,0 \right\}$. 
  Then we have
  \begin{align}
    r^2 f(x,t) \leq (1+A^{-1}t) \sigma + A^{-4} t,  \quad \forall \;x \in B_{g(t)}(x_0, A r), \;  t \in [0, r^2].  
    \label{eqn:RD11_13}  
  \end{align}

\label{thm:RD10_1}
\end{theorem}

\begin{proof}
  Replacing $f$ by $f_{+}$ if necessary and noting that $\square f_{+} \leq 0$ in the distribution sense, we may assume that $f \geq 0$.  
  Up to rescaling of the space-time and the function $f$, we may further assume that   
  \begin{align}
    r=1, \quad  \sup_{y \in B_{g(0)}(x_0, 10A)} f(y,0) =\sigma \leq A.   \label{eqn:RD11_12}
  \end{align}
  Then the estimate (\ref{eqn:RD11_13}) reads as 
  \begin{align}
    f(x,t) \leq  \sigma +  A^{-1}(\sigma +A^{-3})t, \quad \forall \;x \in B_{g(t)}(x_0, A), \;  t \in [0, 1]. 
    \label{eqn:RD11_14}
  \end{align}
  We shall prove (\ref{eqn:RD11_14}) in two steps. 

  \textit{Step 1.  The rough estimate holds:
  \begin{align}
    f(x,t) \leq 10 A \left\{ 1-\frac{d(x,t)}{10A} \right\}^{-2}, \quad \forall \; x \in B_{g(t)}(x_0, 10A), \; t \in [0, 1], 
       \label{eqn:RD17_4}
  \end{align}
  where  $d(x,t) \coloneqq d_{g(t)}(x,x_0)$. 
  }

  For simplicity of notation, we define 
   \begin{align}
   \rho \coloneqq 10 A, \quad   
   \bar{Q} \coloneqq f(\bar{x}, \bar{t}), \quad \bar{d} \coloneqq d(\bar{x}, \bar{t}), 
  \label{eqn:RJ06_1}
 \end{align}
  where $(\bar{x}, \bar{t})$ is the space-time point where $\displaystyle \sup_{x \in B_{g(t)}(x_0, \rho),\;t\in [0,1]} f(x,t) \{\rho-d(x,t)\}^2$ is achieved. 
  We also define
  \begin{align*}
    \bar{G} \coloneqq f(\bar{x}, \bar{t})\left\{\rho-d(\bar{x}, \bar{t}) \right\}^2. 
  \end{align*} 
  In order to prove (\ref{eqn:RD17_4}), it suffices to show $\bar{G} \leq 1000 A^{3}$. 
  We shall prove this by contradiction.  For otherwise, we have 
  \begin{align}
   \bar{G}=\bar{Q}(\rho-\bar{d})^2>1000 A^{3}.
   \label{eqn:RD17_2}
  \end{align}
 Let $\eta$ be the cutoff function such that $\eta \equiv 1$ on $[0,1]$ and $\eta \equiv 0$ on $[2,\infty)$.
 It decreases from $1$ to $0$ on the interval $[1,2]$ and satisfies the inequality
\begin{align}
  \max \left\{ 2\left( \eta' \right)^{2}, \;  -\eta'' \right\} \leq 10 \eta.     \label{eqn:RD19_1}
\end{align}
 Abusing notation, we denote
  \begin{align*}
    \eta(x,t) \coloneqq \eta\left( \frac{d(x,t)+A\sqrt{t}-\bar{d}-A\sqrt{\bar{t}}}{2A^{\frac{3}{2}} \bar{Q}^{-\frac{1}{2}}}  \right).
  \end{align*}
 In light of (\ref{eqn:RD19_1}) and Perelman's distance estimate(cf. Lemma 8.3(a) of~\cite{Pe1}),  the condition (\ref{eqn:RJ04_4}) then implies
  \begin{align}
    \square \eta =\left( \frac{d}{dt}-\Delta \right) \eta=\sqrt{\frac{\bar{Q}}{4A^{3}} } \eta' \left( \square d + \frac{A}{2\sqrt{t}} \right)
    +\frac{\bar{Q}}{4A^{3}} \eta'' \leq  \frac{\bar{Q}}{4A^{3}} \eta''    \label{eqn:RJ05_1}
  \end{align}
 on the support set of $\eta'$.  
 Exploiting (\ref{eqn:RD28_1}) and (\ref{eqn:RD17_2}), if $\rho-d(x,t) \leq \frac{1}{2}(\rho-\bar{d})$, we have 
 \begin{align*}
  \frac{d(x,t)+A\sqrt{t}-\bar{d}-A\sqrt{\bar{t}}}{2A^{\frac{3}{2}} \bar{Q}^{-\frac{1}{2}}} 
  \geq \frac{(\rho-\bar{d})-(\rho-d)+A\sqrt{t}-A\sqrt{\bar{t}}}{2A^{\frac{3}{2}} \bar{Q}^{-\frac{1}{2}}} 
  \geq \frac{1}{2} \left\{5\sqrt{10}-\sqrt{\frac{\bar{Q}\bar{t}}{A}} \right\} >2.   
 \end{align*}
 Consequently, $\eta(x,t)=0$ if $\rho-d(x,t)\leq \frac{1}{2} (\rho-\bar{d})$.  
 In other words, on the support of $\eta$, it holds that
 \begin{align*}
  \rho -d(x,t)>\frac{1}{2}(\rho-\bar{d}),
 \end{align*}
 whence 
 \begin{align}
  f(x,t) \leq f(\bar{x}, \bar{t}) \frac{(\rho-d(\bar{x}, \bar{t}))^2}{(\rho-d(x,t))^2} < 4 f(\bar{x}, \bar{t})=4\bar{Q}.  \label{eqn:RD17_1}
 \end{align}
 Define $F(t) \coloneqq \max_{x \in M} f(x,t)\eta(x,t)$. 
 In view of (\ref{eqn:RD17_1}) and the fact that $\eta(\bar{x}, \bar{t})=1$, we have
 \begin{align}
  \bar{Q}=f(\bar{x}, \bar{t}) \eta(\bar{x}, \bar{t}) \leq F(\bar{t})=\sup_{x \in supp \; \eta(\cdot, \bar{t})} f(x,\bar{t})\eta(x,\bar{t}) \leq \sup_{x \in supp \; \eta(\cdot, \bar{t})} f(x,\bar{t})
  \leq 4\bar{Q}.      \label{eqn:RD17_3}
 \end{align} 
 Suppose the maximum value of $f(\cdot, t)\eta(\cdot, t)$ is non-negative.
 At the maximum value point of $f\eta$, we have  $\nabla (f\eta)=0$.
 Applying also (\ref{eqn:RJ05_1}), we obtain 
 \begin{align*}
  \square \left( f\eta \right)&=\eta \square f + f \square \eta - 2\langle \nabla f, \nabla \eta \rangle\\
  &\leq \frac{1}{4} A^{-3}\bar{Q} \left\{  \eta'' f +  2(\eta')^{2} \eta^{-1} |\nabla d|^2 f \right\}
   < 10 A^{-3}\bar{Q} \left\{ \eta f + f \right\}. 
 \end{align*}
 Plugging (\ref{eqn:RD17_1}) into the above inequality, we can apply maximum principle to obtain
 \begin{align}
  F'(t) \leq 10A^{-3}\bar{Q} \{ F + f\} \leq 10A^{-3}\bar{Q}\{ F + 4 \bar{Q}\}.
  \label{eqn:RD17_6}  
 \end{align}
 Rewriting the above inequality in the form 
 \begin{align*}
  \left( e^{-10A^{-3}\bar{Q}t} F \right)' \leq  e^{-10A^{-3}\bar{Q}t} \cdot 40 A^{-3} \bar{Q}^{2},   \label{eqn:RJ05_2} 
 \end{align*}
 and integrating it over $[0, \bar{t}]$, we obtain
 \begin{align}
  F(\bar{t}) \leq e^{10A^{-3}\bar{Q}\bar{t}} F(0) + 40\bar{Q} \left\{ e^{10A^{-3}\bar{Q}\bar{t}} -1  \right\}. \label{eqn:RJ05_2} 
 \end{align}
 Since $A \geq 1000$, it follows from (\ref{eqn:RD28_1}) and (\ref{eqn:RJ06_1}) that 
 \begin{align*}
  e^{A^{-3}\bar{Q}\bar{t}}<e^{A^{-2}}<1+2A^{-2}. 
 \end{align*}
 Plugging (\ref{eqn:RD17_3}) and the above inequality into (\ref{eqn:RJ05_2}), we obtain
 \begin{align*}
  \bar{Q} \leq \left\{ 1+20A^{-2} \right\} F(0) + 80 A^{-2} \bar{Q}.   
 \end{align*}
 Combined with (\ref{eqn:RD17_2}), the above inequality implies that
 \begin{align*}
  1000 A^{3} (\rho-\bar{d})^{-2}  \leq \bar{Q} \leq  \frac{1+20A^{-2}}{1-80A^{-2}} F(0) \leq 4A.
 \end{align*}
 Consequently, $250 A^{2} \leq (\rho-\bar{d})^2$.  
 On the other hand, definition (\ref{eqn:RJ06_1}) implies $(\rho-\bar{d})^2 < 100 A^{2}$. 
 It follows from the combination of the previous two inequalities that $250A^{2}<100A^{2}$, which is impossible. 
 Therefore, (\ref{eqn:RD17_2}) fails to be true, which establishes the proof of (\ref{eqn:RD17_4}) and completes the proof of Step 1. \\ 

\textit{Step 2. The rough estimate (\ref{eqn:RD17_4}) can be improved to the refined estimate (\ref{eqn:RD11_14}) via iteration. }

We define
\begin{align*}
  &\eta_{k}(x,t) \coloneqq \eta\left( \frac{d(x,t)-(9-2k)A+A\sqrt{t}}{A} \right), \quad  k=1,2,3,4;\\
  &Q_k \coloneqq \sup_{x \in supp \; \eta_{k}(\cdot, t), \; 0 \leq t \leq 1}  f_{k}(x,t)  \leq  \sup_{x \in B_{g(t)}(x_0, (11-2k)A), \; 0 \leq t \leq 1}  f_{k}(x,t); \\
  &F_k(t) \coloneqq \max_{x \in M} \eta_{k}(x,t)f(x,t). 
\end{align*}
Similar to the deduction of (\ref{eqn:RD17_6}), the maximum principle implies that 
\begin{align*}
  F_k'(t) \leq 10 A^{-2} \left( F_{k} +f \right) \leq 10 A^{-2} \left( F_{k} + Q_k \right),
\end{align*}
whose integration over time yields that 
\begin{align*}
  F_{k}(t) &\leq e^{10A^{-2}t} F_{k}(0) + Q_{k} \left( e^{10A^{-2}t}-1 \right)
  \leq \sigma +20(\sigma +Q_k)A^{-2}t.  
\end{align*}
In particular, we have
\begin{align}
  \sup_{0 \leq t \leq 1}F_{k}(t) \leq \sigma +20(\sigma +Q_k)A^{-2}. \label{eqn:RJ06_2} 
\end{align}
Note that the support of $\eta_k(\cdot, t)$ is contained in $B_{g(t)}(x_0, (11-2k)A)$, and $\eta_k(\cdot, t) \equiv 1$ on the ball $B_{g(t)}(x_0, (9-2k)A)$.  
Therefore, $\eta_{k} \equiv 1$ on the support of $\eta_{k+1}$.
Then we have
\begin{align}
  f(x,t)-\sigma \leq F_{k}(t)-\sigma \leq 20(\sigma +Q_k)A^{-2}t, \quad \forall\; x \in B_{g(t)}(x_0, (9-2k)A), \; 0 \leq t \leq 1. 
  \label{eqn:RD19_2}  
\end{align}
It follows from the combination of (\ref{eqn:RD19_2}) and (\ref{eqn:RJ06_2}) that 
\begin{align*}
  Q_{k+1}-\sigma&=\sup_{x \in B_{g(t)}(x_0, 9-2k),0\leq t\leq 1} f(x,t) -\sigma \leq \sup_{0 \leq t \leq 1}F_{k}(t)-\sigma\\ 
  &\leq  20(\sigma +Q_k)A^{-2}=20A^{-2}(Q_k-\sigma) +40\sigma A^{-2}. 
\end{align*}
Elementary deduction implies that 
\begin{align*}
  Q_{k+1}-\frac{A^2+20}{A^2-20}\sigma \leq 20A^{-2} \left\{ Q_{k}-\frac{A^2+20}{A^2-20}\sigma \right\}.
\end{align*}
Note that $A \geq \max\{1000, \sigma\}$. It follows from (\ref{eqn:RD17_4}) that $Q_1-\frac{A^2+20}{A^2-20}\sigma <1000A$.
Consequently, we have
\begin{align*}
 &Q_{2}-\frac{A^2+20}{A^2-20}\sigma \leq 20000A^{-1}<20,\\
 &Q_{3}-\frac{A^2+20}{A^2-20}\sigma \leq  400A^{-2}, \\
 &Q_{4}-\frac{A^2+20}{A^2-20}\sigma \leq 8000A^{-4}<8A^{-3}.
\end{align*}
In particular, we have
\begin{align*}
  Q_{4} \leq \left( 1+50A^{-2} \right) \sigma + 8A^{-3}.
\end{align*}
Plugging the above estimate and $k=4$ into (\ref{eqn:RD19_2}), we arrive at (\ref{eqn:RD11_14}) and finishes the proof of step 2.
The proof of the theorem is complete. 
\end{proof}

 Theorem~\ref{thm:RD10_1} is a localized version of the maximum principle for sub-heat-solution. 
 On a Ricci flow solution with Ricci curvature bounded, a bounded sub-heat-solution $f$ will satisfy the maximum principle:
 \begin{align*}
   f(x, t) \leq \sup_{M} f(\cdot, 0), \quad \forall \; x \in M.
 \end{align*}
 This can also be obtained from Theorem~\ref{thm:RD10_1} by adding $f$ by a constant and letting $A \to \infty$.

 We quote the next lemma from Tian-Wang~\cite{TiWa} for the convenience of the readers. 

 \begin{lemma}[\textbf{Lemma 4.4 of Tian-Wang~\cite{TiWa}}]
   Suppose $\left\{ (M^{m}, x_0, g(t)), 0 \leq t \leq 1 \right\}$ satisfies all the conditions in Corollary~\ref{cly:MJ23_2},
   $\delta_0=\delta_0(m)<0.1$ is a small positive constant. 

   Let $\Omega=B_{g(0)}(x_0, 1), \Omega'=B_{g(0)} \left(x_0, \frac{1}{2} \right)$. For every $l<\frac{1}{2}$, define
\begin{align}
  &A_{+,l}=\sup_{B_{g(0)}(x, r) \subset \Omega', 0<r\leq l} \omega_m^{-1}r^{-m} \left| B_{g(0)}(x, r)\right|_{dv_{g(0)}}, \label{eqn:PA05_1} \\
  &A_{-,l}=\inf_{B_{g(\delta_0)}(x, r) \subset \Omega', 0<r\leq l} \omega_m^{-1}r^{-m} \left| B_{g(\delta_0)}(x, r)\right|_{dv_{g(\delta_0)}}. \label{eqn:PA05_2}
\end{align}

If $x_1, x_2 \in \Omega''=B_{g(0)}\left(x_0, \frac{1}{4}\right)$, $l=d_{g(0)}(x_1, x_2) < \frac{1}{8}$, then we have
  \begin{align}
    l- C E^{\frac{1}{2(m+3)}} \leq   d_{g(\delta_0)}(x_1, x_2) \leq
  l+CA_{+,4l} \left\{ \left| \frac{A_{+,l}}{A_{-,l}} -1 \right|^{\frac{1}{m}} +  l^{-\frac{1}{m}}E^{\frac{1}{2m(m+3)}}  \right\}l
    \label{eqn:metricequivalent}
  \end{align}
  whenever  $\displaystyle E=\int_0^{2 \delta_0} \int_{\Omega} |R| dv dt<<  l^{2(m+3)}$.
 \label{lma:PA05_1}
\end{lemma}

\begin{proposition}[\textbf{Distance distortion---short range}]
  Suppose $\left\{ (M^{m}, x_0, g(t)), 0 \leq t \leq r^2 \right\}$ is a Ricci flow solution.
  Suppose under the metric $g(0)$, the following conditions are satisfied 
  \begin{align}
    &R(\cdot, 0) \geq -r^{-2}, \quad \textrm{on} \; B_{g(0)}(x_0,r); \label{eqn:RE04_1a}\\
    &\overline{\boldsymbol{\nu}}(B_{g(0)}(x_0,r), g(0), r^2) \geq -\epsilon; \label{eqn:RE04_1b} \\
    &\omega_{m}^{-1} \rho^{-m} |B(y,\rho)|_{dv_{g(0)}} \leq 1+\epsilon, \quad \forall \; y \in B_{g(0)}(x_0, 0.5r), \quad \rho \in (0, 0.5r). 
    \label{eqn:RE04_1c}
  \end{align}
  Then we have the distance distortion estimate 
  \begin{align}
    \frac{\left|  d_{g(\epsilon^2 r^2)}(x, x_0) -d_{g(0)}(x, x_0)  \right|}{\epsilon r} \leq \psi(\epsilon|m), \quad \forall \; x \in B_{g(0)}(x_0, 0.1 \epsilon r) \backslash B_{g(0)}(x_0, 0.01\epsilon r).   
  \label{eqn:RE04_2}  
  \end{align}

  \label{prn:RE04_3}
\end{proposition}

\begin{proof}
 Up to rescaling, we may assume $\epsilon r=1$. Then (\ref{eqn:RE04_2}) can be simplified as 
  \begin{align}
    \left| d_{g(1)}(x,x_0)-d_{g(0)}(x,x_0) \right| \leq \psi(\epsilon|m), \quad \forall \; x \in B_{g(0)}(x_0, 0.1) \backslash B_{g(0)}(x_0, 0.01).    \label{eqn:RE05_1}
  \end{align}
  Same as in Lemma~\ref{lma:PA05_1}, we denote $B_{g(0)}(x_0,1)$ by $\Omega$ for simplicity of notation.
  Applying Theorem~\ref{thm:RD10_1} on $-R+\epsilon^2$, we obtain that
\begin{align}
  R(x,t)>-\psi(\epsilon|m), \quad \forall \; x \in \Omega, \; t \in (0, 1).      \label{eqn:RH25_3} 
\end{align}
Therefore, we have
\begin{align*}
     \int_{0}^{1} \int_{\Omega} |R| dv dt &\leq  \int_{0}^{1} \int_{\Omega} (R+\psi) dv dt + \psi \int_{0}^{1} \int_{\Omega} dv dt. 
\end{align*}
Along the Ricci flow, $\displaystyle \D{}{t} dv = - R dv$, which implies that 
\begin{align*}
  |\Omega|_{dv_{g(t)}} \leq e^{\psi t} |\Omega|_{dv_{g(0)}} \leq e^{\psi} |\Omega|_{dv_{g(0)}}. 
\end{align*}
Up to redefining $\psi$ if necessary, we have
\begin{align}
  \int_{0}^{1} \int_{\Omega} |R| dv dt \leq  \int_{0}^{1} \int_{\Omega} R dv dt + \psi \leq |\Omega|_{dv_{g(0)}} -|\Omega|_{dv_{g(1)}} + \psi.
  \label{eqn:RE05_2}  
\end{align}
Note that the metric almost expanding implies $\displaystyle \Omega \supset B_{g(1)} \left(x_0, 1- \psi \right)$. 
Therefore, we have
\begin{align}
  |\Omega|_{dv_{g(1)}} \geq \left| B_{g(1)} \left(x_0, 1- \psi \right) \right| \geq (1-\psi) \omega_{m},   \label{eqn:RE05_3} 
\end{align}
where we used the fact that the volume ratio of $B_{g(1)} \left(x_0, 1- \psi \right)$ is very close to $\omega_m$ by Corollary~\ref{cly:RJ18_0}. 
Also, in light of (\ref{eqn:RE04_1c}), we have
\begin{align}
  |\Omega|_{dv_{g(0)}}=|B(x_0,1)|_{dv_{g(0)}} \leq (1+\epsilon) \omega_m.   \label{eqn:RE05_4}
\end{align}
Plugging (\ref{eqn:RE05_3}) and (\ref{eqn:RE05_4}) into (\ref{eqn:RE05_2}), and absorbing $\epsilon$ into $\psi=\psi(\epsilon|m)$, we arrive at
\begin{align*}
    \int_{0}^{1} \int_{\Omega} |R| dv dt \leq \psi(\epsilon|m).
\end{align*}
Now we are able to apply Lemma~\ref{lma:PA05_1} by choosing $(x_1,x_2)=(x_0,x)$ and $l=d_{g(0)}(x_0, x) \in (0.01, 0.1)$. It follows that 
\begin{align*}
  |d_{g(2\delta_0)}(x_0, x)-d_{g(0)}(x_0, x)|<\psi(\epsilon|m).  
\end{align*}
On the time interval $[2\delta_0, 1]$, since the flow is almost flat by Corollary~\ref{cly:RJ18_0}, we obtain that
\begin{align*}
  |d_{g(2\delta_0)}(x_0, x)-d_{g(1)}(x_0, x)|<\psi(\epsilon|m).
\end{align*}
Then (\ref{eqn:RE05_2}) follows from the combination of the above two inequalities.   
\end{proof}

Proposition~\ref{prn:RE04_3} can be applied to show the following distance distortion in the long range.  

\begin{theorem}[\textbf{Distance distortion---long range}]
  Same conditions as in Proposition~\ref{prn:RE04_3}. 
  Then we have the distance distortion estimate
  \begin{align}
    &\left| \log \frac{d_{g(t)}(x,y)}{d_{g(0)}(x,y)} \right|<\psi \left\{ 1+ \log_{+} \frac{\sqrt{t}}{d_{g(0)}(x,y)} \right\}, 
    \quad  \forall \; t \in (0, \epsilon^2 r^2), \; x,y \in B_{g(0)}(x_0, 0.5 r),    \label{eqn:RE05_6}
  \end{align}
  where $\psi=\psi(\epsilon|m)$, $\log_{+} x=\max\left\{ 0, \log x \right\}$. 

\label{thm:RE05_5}
\end{theorem}

\begin{proof}
 By shrinking $\epsilon$ and changing base point if necessary, we may assume that $x=x_0$ and apply Corollary~\ref{cly:RJ18_0}.  
 Up to rescaling, we assume $\epsilon r=1$.  
 We shall prove (\ref{eqn:RE05_6}) in two cases. 

 \textit{Case 1. $d_{g(0)}(x_{0},y)<\sqrt{t}<1$.}

 We denote $d_{g(0)}(x_{0},y)$ by $\rho$ for simplicity,  it follows from entropy monotonicity and the distance distortion that
 \begin{align*}
   \left| \log \frac{d_{g(\rho^2)}(x_0,y)}{d_{g(0)}(x_{0},y)} \right|<\psi.
 \end{align*}
 On the other hand, on the time period $[\rho^2, t]$, we have $|Rc|t<\psi$ in the concerned space-time domain, it follows from the evolution of distance along Ricci flow that
 \begin{align*}
   \left| \log \frac{d_{g(t)}(x_0,y)}{d_{g(\rho^2)}(x_{0},y)} \right|<\psi \log \frac{t}{\rho^2}=2\psi \log \frac{\sqrt{t}}{d_{g(0)}(x_{0},y)}.  
 \end{align*}
 Up to modifying $\psi$,  (\ref{eqn:RE05_6}) follows from the combination of the previous two inequalities.

 \textit{Case 2. $\sqrt{t}<d_{g(0)}(x_0,y)<\epsilon^{-2}$.}

 Choose $\gamma$ as a shortest geodesic connecting $x_0$ and $y$ under the metric $g(0)$. Let $N$ be the integer part of $t^{-\frac{1}{2}}d_{g(0)}(x,y)$.
 Let $\rho=\frac{d_{g(0)}(x,y)}{2(N+1)}$. 
 Let $x_k=\gamma(k \rho), 0 \leq k \leq 2(N+1)$.  By calculating the length variation of each geodesic length, we have
 \begin{align*}
   d_{g(t)}(x_0,y)\leq \sum_{k=0}^{2N+1} d_{g(t)}(x_{k}, x_{k+1})<e^{\psi} \cdot \left( \sum_{k=0}^{2N+1} d_{g(0)}(x_{k}, x_{k+1})\right)=e^{\psi} d_{g(0)}(x,y), 
 \end{align*}
 which implies that
 \begin{align}
   \log \frac{d_{g(t)}(x_0,y)}{d_{g(0)}(x_0,y)}<\psi.  \label{eqn:RE06_1}
 \end{align}
 On the other hand, by Corollary~\ref{cly:RJ18_0}, we have
 \begin{align*}
    s|Rc|(x,s)<\psi(\epsilon|m), \quad \forall \; x \in  B_{g(s)}\left(x_0, \psi^{-1}\sqrt{s} \right) \cup B_{g(s)}\left(y, \psi^{-1} \sqrt{s} \right), \; s \in (0,t). 
 \end{align*}
 Then it follows from Hamilton-Perelman type distance estimate(cf. Lemma 8.3(b) of~\cite{Pe1}) that
 \begin{align*}
   \frac{d}{ds} d_{g(s)}(x_0,y)>-C(m) s^{-\frac{1}{2}},
 \end{align*}
 whose integration along time yields that 
 \begin{align*}
   d_{g(t)}(x_0,y)>d_{g(0)}(x_0,y) - \psi \sqrt{t}.
 \end{align*}
 Since $\sqrt{t}<d_{g(0)}(x_0,y)$ in the current situation, we have 
 \begin{align}
   \log \frac{d_{g(t)}(x_0,y)}{d_{g(0)}(x_0,y)}> \log \left( 1-\psi \frac{\sqrt{t}}{d_{g(0)}(x,y)} \right)>-\psi.   \label{eqn:RE06_2}
 \end{align}
 As $\log_{+} \frac{\sqrt{t}}{d_{g(0)}(x_0,y)}=0$ now, it is clear that (\ref{eqn:RE05_6}) follows from the combination of (\ref{eqn:RE06_1}) and (\ref{eqn:RE06_2}). 

\end{proof}

We notice that  distance distortion estimates similar to Proposition~\ref{prn:RE04_3} and Theorem~\ref{thm:RE05_5} were already obtained in Chen-Rong-Xu~\cite{CRX} and Huang-Kong-Rong-Xu~\cite{HKRX}.

\begin{corollary}[\textbf{$\bar{\boldsymbol{\nu}}$-entropy and lifespan of the Ricci flow}]
Suppose $M^{m}$ is a complete manifold satisfying 
\begin{align}
  \min \left\{ \boldsymbol{\bar{\nu}}(M, g, T), mT  Rc_{min}(x) \right\} \geq -0.1\epsilon, \quad \forall \; x \in M, 
  \label{eqn:RE06_3}
\end{align}
where $Rc_{min}(x)$ is the minimum eigenvalue of $Rc(x)$, $T$ is a positive number, $\epsilon<\bar{\epsilon}(m)$ is a sufficiently small positive number.
Let $\psi=\psi(\epsilon|m)$. 
Then there exists a Ricci flow solution $\left\{ (M, g(t))| 0 \leq t \leq T \right\}$ initiated from $g(0)=g$ with the following estimates: 
\begin{align}
  t|Rm|(x,t) \leq \psi, \quad  \left| \log  \left\{ \omega_{m}^{-1} t^{-\frac{m}{2}}  \left|B_{g(t)}\left(x, \sqrt{t} \right) \right|_{dv_{g(t)}} \right\}  \right| \leq \psi, 
  \quad inj(x,t) \geq \psi^{-1} \sqrt{t}. 
  \label{eqn:RE08_2}
\end{align}
Furthermore, the following distortion estimates holds:
  \begin{align}
    &\left| \log \frac{d_{g(t)}(x,y)}{d_{g(0)}(x,y)} \right|<\psi \left\{ 1+ \log_{+} \frac{\sqrt{t}}{d_{g(0)}(x,y)} \right\}, 
    \quad  \forall \; t \in (0, T), \; x,y \in M.     \label{eqn:RE08_1}\\
    &\left|  d_{g(0)}(x, y) -d_{g(t)}(x,y) \right| <  \psi \sqrt{t}, \quad \forall \; t \in (0, T), \; x,y \in M, \; d_{g(0)}(x,y) \leq \sqrt{t}.    \label{eqn:PI28_2} 
  \end{align}
\label{cly:CK21_1}
\end{corollary}

\begin{proof}
Note that
\begin{align}
  \boldsymbol{\nu}(M,g,T) \geq  \left\{ \inf_{x \in M} R(x) \right\} T + \overline{\boldsymbol{\nu}}(M,g,T)
  \geq \left\{ \inf_{x \in M} Rc_{min}(x) \right\} mT + \overline{\boldsymbol{\nu}}(M,g,T)
  \geq -0.2\epsilon. 
  \label{eqn:RE07_1}  
\end{align}
If $(M,g)$ has bounded curvature, we may start the Ricci flow for a short time period. 
In light of (\ref{eqn:RE07_1}), we can apply Theorem~\ref{thmin:ML14_2} through the condition (\ref{eqn:RH17_4}). 
This pseudo-locality theorem implies that this flow existence time must be at least $T$.
The bounded curvature condition can be dropped by exploiting the method of Hochard~\cite{Hochard}. 
See also~\cite{HeFei} and~\cite{HuangWang20B} for more details. 
Actually, we can find a sequence of cutoff functions $0 \leq \eta_k=e^{-2f_{k}} \leq 1$ whose support set is $U_{k}=B(x_0, 2k)$ for some fixed $x_0 \in M$ and $\eta_k \equiv 1$ on $D_{k}=B(x_0,k)$.
Moreover, we can ensure that $\eta_{k}^{-1}$ satisfies some uniform bounds(see (2.5) of~\cite{HuangWang20B}). 
Clearly, we have
\begin{align}
  D_{k} \Subset U_{k}, \quad  M=\cup_{k=1}^{\infty} D_{k}.  \label{eqn:RI23_9}
\end{align}
Let $h_k \coloneqq \eta_{k}^{-1} g=e^{2f_{k}}g$.
Then $(U_{k}, h_{k})$ becomes a complete Riemannian manifold with $|Rm|$ bounded by $C_k$, a constant depending on $(U_{k},g)$. 
The scalar curvature of $(U_{k}, h_{k})$ is uniformly bounded from below 
\begin{align}
  R_{h_{k}} \geq -\Lambda,    \label{eqn:RI23_12}
\end{align}
which is independent of $k$(see (2.2) of~\cite{HeFei} or (2.27) of~\cite{HuangWang20B}).
Furthermore, fixing $\xi$ very small, there exists an $r_0$ independent of $k$ such that 
\begin{align}
  e^{2f_k(y)-\xi} g  \leq  h_{k} \leq e^{2f_k(y) +\xi} g, \quad \textrm{on} \; B_{h_{k}}(y,r_0).   \label{eqn:RI22_2}
\end{align}
Namely, the metric $h_k$ and $e^{2f_{k}(y)}g$ are quasi-isometric to each other in the ball $B_{h_k}(y,r_0)$. 
In~\cite{HeFei} and~\cite{HuangWang20B}, we are dealing with almost Euclidean isoperimetric condition for Perelman's pseudo-locality. Suppose each $r_0$-ball under metric $g$ has 
almost Euclidean isoperimetric constant, then the quasi-isometric condition (\ref{eqn:RI22_2}) is sufficient
to guarantee that each $r_0$-ball under $h_{k}$ has almost Euclidean isoperimetric constant(see (2.31) of~\cite{HuangWang20B}).
However, here we only have a weaker condition (\ref{eqn:RE06_3}). 
Then Theorem~\ref{thmin:ML14_2}, the improved pseudo-locality theorem shall take charge through the $\boldsymbol{\mu}$-functional condition (\ref{eqn:MK30_1}).
Actually,  it follows from (\ref{eqn:RI22_2}) that
\begin{align}
B_{h_{k}}(y,r) \subset B_{e^{2f_{k}(y)}g}(y,2r)=B_{g}(y, 2e^{-f_{k}(y)}r) \subset B_{g}(y, 2r), \quad \forall \; r \in (0, r_0].
  \label{eqn:RI23_7}  
\end{align}
Fix $r \in (0, r_{0}]$ where $r_0<0.1 \sqrt{T}$ is a uniform small constant to be determined later, through the requirement of many inequalities.  
It follows from (\ref{eqn:RE06_3}) that  
\begin{align}
  &\bar{\boldsymbol{\nu}}(B_{g}(y, 10r), g, 100r^2) \geq \bar{\boldsymbol{\nu}}(M,g,T) \geq -0.1\epsilon, \label{eqn:RI23_10} \\
  & Rc_{g} \geq -\frac{\epsilon}{mT}.  \label{eqn:RI23_11}
\end{align}
By Theorem~\ref{thm:CF21_3}, the inequality (\ref{eqn:RI23_10}) implies that the $r$-ball has volume ratio lower bound.
In view of (\ref{eqn:RI23_11}), we can apply Bishop-Gromov volume comparison. Using Lemma 4.1 and Lemma 4.2 of Wang~\cite{BWang08e}, 
we obtain a uniform isoperimetric constant bound for $\left( B_{g}(y,2r), g \right)$, which combined with the quasi-isometric condition (\ref{eqn:RI22_2}) in turn implies 
uniform isoperimetric constant estimate and consequently uniform $L^2$-Sobolev constant $C_{S}$ for $\left( B_{h_{k}}(y,r), h_{k} \right)$.
Then standard argument(cf. Section 3.1 of~\cite{SSHuang20} and references therein) shows that the log-Sobolev constant of  $\left( B_{h_{k}}(y,r), h_{k} \right)$ is uniformly bounded from below. Namely, we have
\begin{align}
  \bar{\boldsymbol{\nu}}(B_{h_{k}}(y, r), h_{k}, r^2) \geq -C_{LS}(m).  \label{eqn:RI23_5} 
\end{align}
The Bishop-Gromov volume comparison for $g$ and quasi-isometric condition (\ref{eqn:RI22_2}) together imply that
\begin{align}
  r^{-m}|B_{h_{k}}(y,r)|_{dv_{h_{k}}} \leq 2 \omega_{m}.  \label{eqn:RI23_6}
\end{align}
In Proposition~\ref{prn:RG08_1},  we let $\Omega=B_{h_{k}}(y,r)$ and $\delta=1$.
Using definition (\ref{eqn:RI10_5}), we have
\begin{align*}
  \Gamma(\Omega, h_k, r^2, r^2)
  =2\bar{\boldsymbol{\nu}}(\Omega,h, r^{2})-2\log_{+} \frac{|\Omega|_{dv_{h}}}{r^{m}} + \frac{m}{2} \log (\pi e^2)-\frac{2}{e}
  \geq -C_{\Gamma}(m). 
\end{align*}
Since $\bar{\boldsymbol{\nu}} \leq 0$ by Proposition~\ref{prn:CA01_3}, it follows from (\ref{eqn:RI10_7}) that 
\begin{align*}
  \bar{\boldsymbol{\mu}}\left( B_{h_{k}}(y,r), h_{k}, r^{2} \right)
  \geq \bar{\boldsymbol{\nu}}\left( B_{h_{k}}(y,r), e^{2f_{k}(y)} g, r^{2} \right)
  -C_{\Gamma} \cdot  (1-e^{-m\xi}) -\frac{m(m+2)}{2} e^{-m\xi} \xi.
\end{align*}
Fix $\epsilon>0$, by choosing $\xi$ sufficiently small correspondingly, we can make the last two terms of the above inequality being greater than $-0.1\epsilon$. 
On the other hand, by the containment condition (\ref{eqn:RI23_7}), we have
\begin{align*}
  \bar{\boldsymbol{\nu}}\left( B_{h_{k}}(y,r), e^{2f_{k}(y)} g, r^{2} \right)
  &\geq \bar{\boldsymbol{\nu}}\left(   B_{e^{2f_{k}(y)}g}(y,2r), e^{2f_{k}(y)} g, r^{2} \right)
  =\bar{\boldsymbol{\nu}}\left(   B_{g}\left(y,2e^{-f_{k}(y)}r \right), g, e^{-2f_{k}(y)}r^{2} \right)\\
  &\geq \bar{\boldsymbol{\nu}}\left( M, g, T \right) 
   \geq -0.1\epsilon.
\end{align*}
Combining the previous inequalities together and fixing $s=r$, we obtain 
\begin{align*}
  \bar{\boldsymbol{\mu}}\left( B_{h_{k}}(y,r), h_{k}, r^{2} \right) \geq -0.2\epsilon. 
\end{align*}
By the lower bound of scalar curvature (\ref{eqn:RI23_12}) and the fact $r \in (0, r_{0}]$ with $r_0$ very small, we have
\begin{align*}
  \boldsymbol{\mu} \left( B_{h_{k}}(y,r), h_{k}, r^{2} \right) 
  \geq -\Lambda r^{2}+ \bar{\boldsymbol{\mu}}\left( B_{h_{k}}(y,r), h_{k}, r^{2} \right)>-\epsilon. 
\end{align*}
Since $r \in (0, r_0]$ and $y \in U_{k}$ is arbitrary, we have 
\begin{align}
\inf_{r \in (0, r_0]} \boldsymbol{\mu}\left( B_{h_{k}}(y,r), h_{k}, r^{2} \right) \geq -\epsilon, \quad \forall \; y \in U_{k}.
\label{eqn:RI23_8}
\end{align}
Since $(U_{k}, h_{k})$ is a complete manifold with bounded curvature, the Ricci flow initiated from $(U_{k}, h_{k})$ may exist for a while.
The condition (\ref{eqn:RI23_8}) assures us to apply Theorem~\ref{thmin:ML14_2}. 
Following the same argument as that in the proof of Corollary~\ref{cly:RJ17_9}, we know the existence time of the Ricci flow is uniformly bounded from below by $\delta^{2} r_{0}^{2}$ for some $\delta$ independent of $k$.
Recall that $h_{k}=g$ on $\Omega_{k}=B(x_0, k)$.
In view of the exhaustion condition (\ref{eqn:RI23_9}), the curvature and volume estimates (\ref{eqn:ML28_1}) and (\ref{eqn:ML28_2}) in Theorem~\ref{thmin:ML14_2}, 
and the improved curvature estimate by B.L. Chen(cf. Theorem 3.1 of~\cite{CBL07}), 
we can take limit of the flows $\left\{(U_{k}, x_0, h_{k}(t)), 0 \leq t \leq  \delta^2 r_{0}^{2} \right\}$ in pointed-Cheeger-Gromov topology to obtain a limit Ricci flow 
$\left\{(M, x_{0}, g(t)), 0 \leq t \leq \delta^2 r_{0}^{2} \right\}$ satisfying $g(0)=g$. 
For each time $t \in (0, \delta^2 r_{0}^2)$, the manifold $(M, g(t))$ has bounded curvature via the pseudo-locality theorem.
In light of (\ref{eqn:RE07_1}),  we can apply the pseudo-locality theorem again to obtain the existence of $g(t)$ is at least $T$,
and the estimates in (\ref{eqn:RE08_2}) hold along this flow. 

The distance distortion estimate (\ref{eqn:RE08_1}) follows directly from Theorem~\ref{thm:RE05_5}.  
Then (\ref{eqn:PI28_2}) can be deduced from (\ref{eqn:RE08_1}). 
We denote $d_{g(t)}=d_{g(t)}(x,y)$ and $d_{g(0)}=d_{g(0)}(x,y)$. Since $d_{g(0)}\leq \sqrt{t}$, we have 
\begin{align*}
  \frac{\left|  d_{g(0)} -d_{g(t)} \right|}{\sqrt{t}}&= \frac{d_{g(0)}}{\sqrt{t}}  \cdot \left|  1- \frac{d_{g(t)}}{d_{g(0)}} \right|
  \leq  \frac{d_{g(0)}}{\sqrt{t}}  \cdot \max \left\{  \left|  1- e^{\psi} \left( \frac{d_{g(0)}}{\sqrt{t}} \right)^{-\psi} \right|,  \left|  1- e^{-\psi} \left( \frac{d_{g(0)}}{\sqrt{t}} \right)^{\psi} \right| \right\}\\
  &\leq \max \left\{ \sup_{s \in [0,1]} |s-e^{\psi}s^{1-\psi}|, \; \sup_{s \in [0,1]} |s-e^{-\psi} s^{1+\psi}| \right\} \leq 10\psi.
\end{align*}
Therefore, we obtain (\ref{eqn:PI28_2}), up to modifying $\psi$. 
\end{proof}

A particularly interesting case of Corollary~\ref{cly:CK21_1} is $T=\infty$ and $Rc \geq 0$. 
\begin{corollary}[\textbf{$\bar{\boldsymbol{\nu}}$-entropy and the immortal Ricci flow}] 
  Suppose $(M,g)$ is a complete Riemannian manifold satisfying
  \begin{align*}
    Rc \geq 0, \quad \bar{\boldsymbol{\nu}}(M,g) \geq -\epsilon \geq -\bar{\epsilon}(m), 
  \end{align*}
  then there exists an immortal Ricci flow solution $\left\{ (M, g(t)), 0 \leq t<\infty \right\}$ initiated from $g$ 
  satisfying (\ref{eqn:RE08_2}), (\ref{eqn:RE08_1}) and (\ref{eqn:PI28_2}) with $T=\infty$.  
\label{cly:RH27_1}
\end{corollary}

The immortal Ricci flow solution in Corollary~\ref{cly:RH27_1} provides a new perspective to understand the geometry and topology of the initial Riemannian manifold $(M, g)$. 
For example, we have the following topological consequence, which is known by Cheeger-Colding~\cite{CC}.

\begin{theorem}[\textbf{Construction of diffeomorphisms to $\R^{m}$ via immortal Ricci flow solution}]
  There exists $\delta_{0}=\delta_{0}(m)$ with the following properties.

  Suppose $(M, g)$ is a complete Riemannian manifold with non-negative Ricci curvature and $AVR(M,g) \geq 1-\delta$ for some $\delta \in (0, \delta_{0}(m))$.
  Then there exists a canonical diffeomorphism map from $M$ to $\R^{m}$. 
  \label{thm:CK21_2}
\end{theorem}

\begin{proof}
  By the condition $Rc \geq 0$ and the isoperimetric constant estimate of Cavalletti-Mondino~\cite{CaMo}, we know that
\begin{align*}
  \boldsymbol{\nu}(M,g)\geq \overline{\boldsymbol{\nu}}(M, g) \geq 1- \psi(\delta|m).
\end{align*}
Therefore, if $\delta$ is sufficiently small,  we can apply Corollary~\ref{cly:RH27_1} to obtain an immortal Ricci flow solution $\left\{ (M, g(t)), 0 \leq t <\infty \right\}$
initiated from $g(0)=g$. Then estimates (\ref{eqn:RE08_2}) hold for any $x \in M$ and $t>0$.  
Define 
  \begin{align*}
    M' \coloneqq \left\{ (x,t) \left| d_{g(t)}(x,x_0)=\sqrt{t} \right.  \right\}.
  \end{align*}
By (\ref{eqn:RE08_2}), the distance function $d_{g(t)}(\cdot, x_0)$ is a smooth function nearby the scale $\sqrt{t}$. It is clear that
  $M'=\bigcup_{t \geq 0} \partial B_{g(t)}(x_0, \sqrt{t})$ is a smooth manifold.
We now define a space-time exponential map as follows
  \begin{align*}
    Exp: \R^{m} \times [0, \infty) &\mapsto M \times [0, \infty), \\
      (v, t)               &\mapsto  (exp_{x_0, g(t)} v, t). 
  \end{align*}
In view of (\ref{eqn:RE08_2}), the map $Exp|_{\Omega}$ is a diffeomorphism, where $\Omega=\left\{ (x,t)|d_{g(t)}(x,x_0) \leq 2\sqrt{t} \right\}$. 
In particular, $M'$ is diffeomorphic to $Exp^{-1}(M')=\left\{ (v,t) \left||v|=\sqrt{t} \right. \right\}$ via the diffeomorphism map $Exp^{-1}$. 
Note that $\left\{ (v,t) \left||v|=\sqrt{t} \right. \right\}$ is diffeomorphic to $\R^{m}$ through the obvious projection $P'$ defined as
\begin{align*}
  P'(v,t) \coloneqq  v \in \R^{m}. 
\end{align*}
In short, we have already obtained the diffeomorphism 
  \begin{align}
    M'   \longright{P' \circ Exp^{-1}}  \R^{m}.  \label{eqn:RC12_2}   
  \end{align}
Let $P$ be the obvious projection map from $M'$ to $M$ defined by
  \begin{align*}
     P(x,t) \coloneqq x, \quad \forall \; (x,t) \in M'.
  \end{align*}
The map $P$ is obviously smooth and non-degenerate everywhere. 
We claim that $P$ is both injective and surjective.
Actually,  fix each $x \in M \backslash \left\{ x_0 \right\}$, the function $\frac{d_{g(t)}(x,x_0)}{\sqrt{t}}$ is a continuous function of $t$ and 
  \begin{align*}
    \lim_{t \to 0}  \frac{d_{g(t)}(x,x_0)}{\sqrt{t}}=\infty, \quad \lim_{t \to \infty} \frac{d_{g(t)}(x,x_0)}{\sqrt{t}}=0, 
  \end{align*}
where we used the distance distortion estimate (\ref{eqn:RE08_1}) in the last equation. By continuity, there exists a time $t>0$ such that
  \begin{align*}
    \frac{d_{g(t)}(x,x_0)}{\sqrt{t}}=1,
  \end{align*}
which implies that $(x,t) \in M'$ and $x=P(x,t)$.  Clearly, $P(x_0,0)=x_0$.  Therefore, the surjectivity of $P$ is obtained.  We proceed to show that $P$ is 
  injective.  Suppose $P(x,t_1)=P(y, t_2)$ for some $t_1<t_2$.  It is clear that $x=y$ and we have
  \begin{align*}
    \frac{d_{g(t_1)}(x,x_0)}{\sqrt{t_1}}= \frac{d_{g(t_2)}(x,x_0)}{\sqrt{t_2}}=1.
  \end{align*}
By distance distortion estimate (\ref{eqn:PI28_2}), we know that
  \begin{align*}
    e^{-\psi} < \frac{d_{g(t_{i})}(x,x_0)}{d_{g(0)}(x,x_0)}<e^{\psi}, \quad  i=1,2.  
  \end{align*}
  Combining the above two steps, we have
  \begin{align}
    1 < \frac{\sqrt{t_2}}{\sqrt{t_1}}=\frac{d_{g(t_2)}(x,x_0)}{d_{g(t_1)}(x,x_0)}<e^{2\psi}.    \label{eqn:RC12_3} 
  \end{align}
  In particular, we have $t_2 \in (t_1, e^{4\psi} t_1)$.  Note that for $t \in [t_1, t_2]$, we can apply (\ref{eqn:RE08_2}) 
  to estimate the curvature nearby the shortest geodesic connecting $x$ and $x_0$. Then we apply the geodesic length  evolution equation along the Ricci flow to deduce that
  \begin{align*}
    \left| \frac{d}{dt} \log  d_{g(t)} \right|= \frac{1}{|\gamma|}\int_{\gamma} |Rc|(\dot{\gamma}, \dot{\gamma}) d\theta \leq \frac{\psi}{t}. 
  \end{align*}
  It follows that
  \begin{align}
    \frac{d_{g(t_2)}(x,x_0)}{d_{g(t_1)}(x,x_0)} \leq \left( \frac{t_2}{t_1} \right)^{\psi}.
    \label{eqn:RC12_4}  
  \end{align}
  Putting (\ref{eqn:RC12_4}) into (\ref{eqn:RC12_3}), we arrive at
  \begin{align*}
    \left( \frac{t_2}{t_1} \right)^{\frac{1}{2}-\psi} \leq 1, 
  \end{align*}
  which contradicts the assumption that $t_2>t_1$ and $\psi<<1$. In conclusion, $P(x,t_2)=P(y, t_1)$ implies that $(x,t_2)=(y,t_1)$. Namely, we showed that $P$ is injective. 
  In summary, we obtain $P: M' \to M$ is a smooth diffeomorphism. 
  Combining this diffeomorphism with the one in (\ref{eqn:RC12_2}), we obtain the desired canonical diffeomorphism map $P' \circ Exp^{-1} \circ P^{-1}$: 
  \begin{align*}
    M \longright{P^{-1}}  M'   \longright{P' \circ Exp^{-1}}  \R^{m}. 
  \end{align*}
  In particular, $M$ is diffeomorphic to $\R^{m}$. The proof of the theorem is complete. 
\end{proof}

The immortal Ricci flow solution in Corollary~\ref{cly:RH27_1} can also help us to show geometry rigidity properties related to $(M,g)$.

\begin{theorem}[\textbf{Equivalence of several rigidities}]
  Suppose $(M,g)$ is a complete Riemannian manifold with $Rc \geq 0$. Then we have the equivalence relationship
  \begin{align}
    \bar{\boldsymbol{\nu}}(M,g)=0   \Leftrightarrow \mathbf{I}(M,g)=\mathbf{I}(\R^{m},g_{E}) \Leftrightarrow AVR(M,g)=\omega_{m}.
    \label{eqn:RC12_5}
  \end{align}
  Furthermore, one of the identities hold in (\ref{eqn:RC12_5}) if and only if $(M,g)$ is isometric to $(\R^{m}, g_{E})$. 
  \label{thm:RC12_1}
\end{theorem}

\begin{proof}
  It suffices to show each identity is equivalent to the manifold $(M^{m}, g)$ being isometric to the Euclidean space.
  The asymptotic volume ratio version and isoperimetric constant version are well-known. Therefore, we only need to show 
  that $\bar{\boldsymbol{\nu}}(M,g)=0$ implies $(M, g)$ is isometric to the Euclidean metric.  By Corollary~\ref{cly:RH27_1}, 
  there exists an immortal solution initiated from $(M, g)$ with curvature estimate (\ref{eqn:RE08_2}).  
  Applying the local monotonicity formula for $\boldsymbol{\nu}$-functional(cf.Theorem~\ref{thm:CA02_3}), for each $t>0$, we have
  \begin{align*}
    \boldsymbol{\nu}(M, g(t)) \geq \boldsymbol{\nu}(M, g(0)) =\boldsymbol{\nu}(M, g)  \geq \bar{\boldsymbol{\nu}}(M, g)=0, 
  \end{align*}
  where we used the fact $R_{g} \geq 0$ in the last inequality. 
  In light of Proposition~\ref{prn:RB06_1}, $(M, g(t))$ is isometric to the Euclidean metric. 
  Let $t \to 0$, the continuity of metrics $g(t)$ then yields that $(M, g)$ is isometric to $(\R^{m}, g_{E})$. 
\end{proof} 

The quantitive version of Theorem~\ref{thm:RC12_1} is the following theorem.

\begin{theorem}[\textbf{Equivalence of several quantitive rigidities}]
  Suppose $(M^{m}, g)$ is a complete Riemannian manifold and $B(x_0,1) \subset M$ is a geodesic ball with $Rc \geq -2(m-1)$.
  \begin{itemize}
    \item[(a)] For each small $\epsilon$, there exists $\delta=\delta(\epsilon,m)$ such that if $\log \frac{|B(x_0,r)|}{r^{m}}\geq -\delta$ for some $r \in (0, \delta)$, 
      then $\mathbf{I}(B(x_0, \delta r)) \geq (1-\epsilon) \mathbf{I}(\R^{m})$. 
    \item[(b)] For each small $\delta$,  there exists $\xi=\xi(\delta,m)$ such that if  $\mathbf{I}(B(x_0, r)) \geq (1-\xi) \mathbf{I}(\R^{m})$ for some $r \in (0, \xi)$, then $\bar{\boldsymbol{\nu}}(B(x_0, r)) \geq -\delta$.   
    \item[(c)] For each small $\xi$, there exists $\epsilon=\epsilon(\xi,m)$ such that if $\bar{\boldsymbol{\nu}}(B(x_0, r), r^2) \geq -\epsilon$ for some $r \in (0, \epsilon)$,
          then $\log \frac{|B(x_0,\rho)|}{\rho^{m}}\geq -\xi$ for each $\rho \in (0,\xi r)$.  
  \end{itemize} 
  \label{thm:RH25_1}
\end{theorem}

\begin{proof}
  It is clear that (a) follows from the work of Cavalletti-Mondino~\cite{CaMo}, and (b) follows from Lemma~\ref{lma:MJ25_1}. 
We now focus on the proof of (c).  Suppose $\bar{\boldsymbol{\nu}}(B(x_0, r), r^2) \geq -\epsilon$ for some $r \in (0, \epsilon)$.
Let $\tilde{g}= (\xi r)^{-2}g$. Then in the ball $B_{\tilde{g}}(x_0, \xi^{-1}) \subset M$, we have
\begin{align}
  &Rc_{\tilde{g}} \geq -2(m-1)\xi^2 r^2 \geq -2(m-1)\xi^2 \epsilon^2, \label{eqn:RH26_1}\\
  &\bar{\boldsymbol{\nu}}(B_{\tilde{g}}(x_0, \xi^{-1}), \tilde{g}, \xi^{-2}) \geq -\epsilon.  \label{eqn:RH26_2}
\end{align}
Similar to the proof of Corollary~\ref{cly:CK21_1}, we can follow Hochard~\cite{Hochard} to construct a complete Riemannian manifold 
$(\hat{M}, \hat{g})$ with bounded curvature. The point is to choose a cutoff function $\eta$ which is identically $1$ inside $B_{\tilde{g}}(x_0,2)$ and vanishes outside $B_{\tilde{g}}(x_0, 4)$.
Then we set $\hat{M}=\left\{ x \in M| \eta(x)<\infty \right\}$ and $\hat{g}=e^{\frac{1}{\eta}}\tilde{g}$.  
Under this construction, it is clear that 
\begin{align}
  \hat{g}=\tilde{g}, \quad \textrm{on} \; B_{\tilde{g}}(x_0, 2). 
  \label{eqn:RJ04_1}
\end{align}
Explicit calculation shows that $(\hat{M}, \hat{g})$ is complete, has bounded curvature and satisfies 
\begin{align}
  R_{\hat{g}} \geq -\psi(\epsilon|m,\xi).
  \label{eqn:RH27_7}
\end{align}
Furthermore, by adjusting the parameter in cutoff function, we can assume that each unit ball in $(\hat{M}, \hat{g})$ is $\psi(\epsilon|\xi,m)$-close to a small ball in $B_{\tilde{g}}(x_0, 4)$
up to rescaling(cf. (\ref{eqn:RI22_2})).
Using the same argument as that in the proof of Corollary~\ref{cly:CK21_1}, we can apply Proposition~\ref{prn:RG08_1} to obtain that
\begin{align}
\inf_{r \in (0, \xi^{-1}]} \bar{\boldsymbol{\mu}}\left(B_{\hat{g}}(y, r), \hat{g}, r^{2} \right) \geq -\psi(\epsilon|m,\xi), \quad \forall \; y \in \hat{M}.  \label{eqn:RH27_8} 
\end{align}
Since $(\hat{M}, \hat{g})$ has bounded curvature, the Ricci flow initiated from $\hat{g}$ should exist for a short time period.
In light of (\ref{eqn:RH27_7}) and (\ref{eqn:RH27_8}), we can apply the pseudo-locality theorem to guarantee that the existence time is uniformly bounded below.
Without loss of generality, we may assume $\left\{ (\hat{M}, \hat{g}(t)), 0 \leq t \leq 1  \right\}$ is the Ricci flow initiated from
$\hat{g}(0)=\hat{g}$.  Along this flow, we have the curvature estimate $t|Rm|<\psi(\epsilon|m,\xi)$.
Thanks to (\ref{eqn:RH27_7}), the maximum principle on scalar curvature implies that $R \geq -\psi(\epsilon|m,\xi)$ along this flow.
Therefore the volume element is almost decreasing along the flow, in the sense that
\begin{align}
  \frac{dv_{\hat{g}(t)}}{dv_{\hat{g}}} \leq e^{\psi}, \quad \forall \; x \in \hat{M}, \; t \in [0,1].  \label{eqn:RJ04_2}
\end{align}
Applying the distance almost expanding estimate (cf. Lemma 8.3(b) of~\cite{Pe1}), we have
\begin{align}
  B_{\hat{g}(1)}(x_0, 1-\psi) \subset B_{\hat{g}(0)}(x_0, 1). \label{eqn:RJ04_3}
\end{align}
Since $g$ is the default metric, we have $B(x_0, \xi r)=B_{g}(x,\xi r)=B_{\tilde{g}}(x_0,1)$.
By (\ref{eqn:RJ04_1}), we know $\hat{g}(0)=\hat{g}=\tilde{g}$ on $B_{\tilde{g}}(x_0,2)$.
Then it follows from (\ref{eqn:RJ04_2}), (\ref{eqn:RJ04_3}) and the almost Euclidean volume ratio estimate (\ref{eqn:ML28_2}) that
\begin{align*}
  (\xi r)^{-m}|B(x_0,\xi r)|&=|B_{\tilde{g}}(x_0,1)|_{dv_{\tilde{g}}}=|B_{\hat{g}(0)}(x_0,1)|_{dv_{\hat{g}(0)}}\\
  &\geq |B_{\hat{g}(1)}(x_0,1-\psi)|_{dv_{\hat{g}(0)}} \geq e^{-\psi}|B_{\hat{g}(1)}(x_0,1-\psi)|_{dv_{\hat{g}(1)}}  \geq \omega_{m} -\psi.  
\end{align*}
In light of (\ref{eqn:RH26_1}), a further step of Bishop-Gromov volume comparison then yields that
\begin{align*}
  \log \frac{|B(x_0,\rho)|}{\rho^{m}}\geq -\psi(\epsilon|m,\xi) \geq -\xi, \quad \forall \rho \in (0,\xi r),
\end{align*}
by choosing $\epsilon$ sufficiently small. Therefore, we finish the proof of (c).  
\end{proof}

Because of Theorem~\ref{thm:RH25_1}, we are ready to define some radii for the convenience of further study. 

\begin{definition}
  Suppose $(M,g)$ is a closed manifold and $\epsilon_0=\epsilon_0(m)$ is sufficiently small.  Then we define
  \begin{align*}
    &\mathbf{vr}_{\epsilon_0}(x)\coloneqq \inf \left\{ r>0|r^{-m}|B(x,r)|<(1-\epsilon_0) \omega_{m} \right\}, \\
    &\mathbf{er}_{\epsilon_0}(x)\coloneqq \inf \left\{ r>0|\bar{\boldsymbol{\nu}}(B(x,r), r^2)< -\epsilon_0 \right\},\\
    &\mathbf{Ir}_{\epsilon_0}(x)\coloneqq \inf \left\{ r>0|\mathbf{I}(B(x,r))<(1-\epsilon_0) \mathbf{I}(\R^{m}) \right\}. 
  \end{align*}
  They are called volume radius, entropy radius and isoperimetric radius respectively. 
  Accordingly, we define
  \begin{align*}
    &\mathbf{vr}_{\epsilon_0}(M,g) \coloneqq  \inf_{x \in M}\mathbf{vr}_{\epsilon_0}(x,g), \\
    &\mathbf{er}_{\epsilon_0}(M,g) \coloneqq  \inf_{x \in M}\mathbf{er}_{\epsilon_0}(x,g), \\
    &\mathbf{Ir}_{\epsilon_0}(M,g) \coloneqq  \inf_{x \in M}\mathbf{Ir}_{\epsilon_0}(x,g). 
  \end{align*}
  When no confusion is caused, we may omit $\epsilon_0$ and $g$. 
  \label{dfn:RH26_3}
\end{definition}

\begin{remark}
Suppose $(M, g)$ is an Einstein manifold with $|Rc| \leq m-1$, by the regularity improvement of metrics under elliptic PDE, we can apply Theorem~\ref{thm:RH25_1} to obtain 
\begin{align*}
  \mathbf{vr}_{\epsilon} \leq C(m) \mathbf{Ir}_{\delta} \leq C(m) \mathbf{er}_{\xi} \leq C(m) \mathbf{vr}_{\epsilon'} \leq C(m) \mathbf{vr}_{\epsilon}.  
\end{align*}
Therefore, the different radii defined above are all equivalent(cf. Proposition 2.49 of~\cite{CW17A} for more information). 
Similar conclusions also hold on Ricci shrinkers(cf. Theorem 4.10 of~\cite{LiLiWang18}).
However, for general Riemannian manifolds, $ \mathbf{vr}_{\epsilon}$ is hard to apply, due to the lack of Bishop-Gromov volume comparison and a gap lemma of Anderson type(cf.~\cite{Anderson90}).
It becomes natural to exploit the entropy radius $\mathbf{er}_{\xi}$ instead to study the compactness of other critical metrics with bounded scalar curvature, e.g. the cscK metrics and extK(extremal K\"ahler) metrics. 
\label{rmk:RH26_4}
\end{remark}

Theorem~\ref{thm:RH25_1} shows that  quantitive rigidities hold for initial manifold.
A natural question is whether a quantitive rigidity holds for the flow initiated from $(M, g)$.  For example, shall the immortal flow in Corollary~\ref{cly:RH27_1} converge back to a Euclidean metric?
Unfortunately, we shall show that the flow must fail to converge.

\begin{proposition}[\textbf{Scalar curvature estimate by volume ratio}]
 For each $\delta \in (0, \delta_0(m))$, there is a small positive constant $\xi$ with the following property.

 Suppose $(M, x_0, g)$ is a complete manifold satisfying
 \begin{align}
   & Rc(x) \geq  -(m-1)\xi r^{-2}, \forall \; x \in B(x_0, \xi^{-1} r); \label{eqn:RE14_8} \\
   &1-\delta \leq  \omega_{m}^{-1} r^{-m}|B(x_0,r)| \leq 1- 0.5 \delta, \quad \forall\;  r \in [1, \xi^{-1}]. \label{eqn:RE14_9}
 \end{align}
 Let $\left\{ (M, g(t)), 0 \leq t \leq 1 \right\}$ be the Ricci flow initiated from $g(0)=g$.  Then 
 \begin{align}
   tR(x_0,t) \geq \xi, \quad \forall \; t \in [0.5 r^2, r^2].   \label{eqn:RE14_10}
 \end{align}
\label{prn:RE08_4}
\end{proposition}

\begin{proof}
  Without loss of generality, we can assume $r=1$ up to rescaling. 

  We argue by contradiction.  If the statement were false, we could find a sequence of Ricci flows $\left\{ (M_i, g_i(t)), 0 \leq t \leq 1 \right\}$ with the following properties
  \begin{align}
    & 1-\delta \leq  \omega_{m}^{-1} \rho^{-m}|B(x_i,\rho)|_{g_i(0)} \leq 1- 0.5 \delta, \quad \forall\;  \rho \in [1, \xi_i^{-1}]; \label{eqn:RE14_1}\\
    & Rc(x, 0) \geq -(m-1)\xi_i, \quad \forall \; x \in B(x_i, \xi_{i}^{-1}); \label{eqn:RE14_2} \\
    & t_i R(x_i, t_i)  \leq \xi_{i}, \quad t_i \in [0.5, 1].   \label{eqn:RE14_3}
  \end{align}
  Note that the pseudo-locality theorem can be applied in the current setting. Therefore, the compactness theorem of Hamilton implies that
  \begin{align}
    \left\{ (M_i, x_i, g_i(t)), 0 <t \leq 1 \right\} \longright{C^{\infty}-Cheeger-Gromov} \left\{ (M_{\infty}, x_{\infty}, g_{\infty}(t)), 0<t \leq 1 \right\},
    \label{eqn:RE14_5}  
  \end{align}
  where $M_{\infty}$ is a smooth Riemannian manifold. 

  \begin{claim}
    The limit flow in (\ref{eqn:RE14_5}) is a static flow on Euclidean space in the sense that $g(t) \equiv g(1)$ and $(M_{\infty}, g(1))$ is isometric to $(\R^{m}, g_{E})$. 
    \label{clm:RE14_11}
  \end{claim}

  Note that the local maximum principle for $R$ implies that $R$ is almost non-negative on $B(x_i, L) \times [L^{-1}, 1]$ for each $L$.
  Combining this fact with the upper bound (\ref{eqn:RE14_3}) we have
  \begin{align*}
    0=R(x_{\infty}, t_{\infty})=\inf_{x \in M_{\infty}, 0<t \leq 1} R(x, t).
  \end{align*}
  Since $t_{\infty} \in [0.5, 1]$ by the condition (\ref{eqn:RE14_3}), the strong maximum principle implies that $R \equiv 0$,
  which in turn implies that $Rc \equiv 0$ on the limit flow $\left\{ (M_{\infty}, g_{\infty}(t)), 0<t \leq 1 \right\}$.
  By the effective local entropy monotonicity formula(cf. Theorem~\ref{thm:CA02_3}) and the smooth convergence, we have
  \begin{align*}
    \boldsymbol{\nu}(B(x_{\infty}, L), g_{\infty}(1), L^2)&=\lim_{i \to \infty} \boldsymbol{\nu}(B(x_i, L), g_i(1), L^2) \geq \lim_{i \to \infty} \left\{ \boldsymbol{\nu}(M_i, g_i(0), L^2+1) -\psi(L^{-1}|m) \right\}\\
    &=\lim_{i \to \infty} \left\{ \boldsymbol{\nu}(M_i, g_i(0), \xi_{i}^{-1}) -\psi(L^{-1}|m) \right\}\\
    &\geq -\psi(\delta|m) -\psi(L^{-1}|m), 
  \end{align*}
  where $\lim$ should be understood as $\liminf$ if necessary. The same convention will be used in the following argument. 
  For each fixed number $K>0$, taking limit of the above inequality and letting $L \to \infty$,  we obtain  
  \begin{align*}
    \boldsymbol{\nu}(M_{\infty}, g_{\infty}(1), K)=\lim_{L \to \infty} \boldsymbol{\nu}(B(x_{\infty}, L), g_{\infty}(1), K) \geq \lim_{L \to \infty} \boldsymbol{\nu}(B(x_{\infty}, L), g_{\infty}(1), L^2)  \geq  -\psi(\delta|m).
  \end{align*}
  By the arbitrary choice of $K$, we know that  $(M_{\infty}, g_{\infty}(1))$ is a Ricci-flat manifold satisfying
  \begin{align*}
    \boldsymbol{\nu}(M_{\infty}, g_{\infty}(1))=\inf_{K>0} \boldsymbol{\nu}(M_{\infty}, g_{\infty}(1), K) \geq - \psi(\delta|m).
  \end{align*}
  Then the gap theorem of Anderson type(cf.~Theorem~\ref{thm:RH27_10}) implies that the Ricci-flat manifold $(M_{\infty}, g_{\infty}(1))$ is isometric to the Euclidean space $(\R^{m}, g_{E})$.
  Since the limit flow is static in the sense that $Rc \equiv 0$, it is clear that $(M_{\infty}, g_{\infty}(t))$ is isometric to $(\R^{m}, g_{E})$ for each $t \in (0,1]$.
  The proof of Claim~\ref{clm:RE14_11} is complete. 

  We proceed our proof. In light of the almost Euclidean property of $(M_i, x_i, g_i(t))$ and the volume comparison at $t=0$,  via the deduction before (\ref{eqn:RE05_2}), we obtain 
  \begin{align*}
    \lim_{i \to \infty} \int_{0}^{1} \int_{B_{g_i(0)}(x_i, L)} |R|dv dt =0, \quad \forall \; L>0. 
  \end{align*}
  Applying Lemma~\ref{lma:PA05_1}, we can derive as in the proof of Theorem~\ref{thm:RE05_5} to obtain that
  \begin{align}
  \left| \log \frac{d_{g_i(t)}(x,y)}{d_{g_i(0)}(x,y)} \right|< \psi(\xi_{i}|m) \left\{ 1+ \log_{+} \frac{\sqrt{t}}{d_{g_i(0)}(x,y)} \right\}, \quad \forall \; x,y \in B_{g_i(0)}(x_i, L).
  \label{eqn:RE14_4}  
  \end{align}
  Using Cheeger-Colding theory, we obtain that
  \begin{align}
    (M_i, x_i, g_i(0)) \longright{pointed-Gromov-Hausdorff} \left( X_{\infty}, x_{\infty}, d_{\infty} \right). \label{eqn:RE14_6}
  \end{align}
  For each pair of points $z,w \in X_{\infty}$, we can find $z_i, w_i \in (M_i, g_i(0))$ such that $z_i \to z, w_i \to w$. Then it follows from (\ref{eqn:RE14_4}) and the Ricci-flatness of the limit flow  that
  \begin{align*}
    d(z,w)=\lim_{i\to \infty} d_{g_i(0)}(z_i, w_i)=\lim_{i \to \infty} d_{g_i(1)}(z_i, w_i).
  \end{align*}
  Therefore, the limit flow can be extended at the initial time.  We can combine (\ref{eqn:RE14_5}) and (\ref{eqn:RE14_6}) to obtain
  \begin{align}
    (M_i, x_i, g_i(t)) \longright{pointed-Gromov-Hausdorff} \left( \R^{m}, 0, g_{E} \right), \quad \forall\; t \in [0, 1],  \label{eqn:RE14_7}
  \end{align}
  with natural commutativity between limit process and identity map between different time slices.  In particular, we have 
  \begin{align*}
    \lim_{i \to \infty} d_{GH} \left\{ (B(0,1), g_{E}), \left( B(x_i,1), g_i(0) \right) \right\}=0.
  \end{align*}
  Then it follows from the volume continuity theorem of Colding(cf.~\cite{Co}) that
  \begin{align*}
    \lim_{i \to \infty} \omega_{m}^{-1}|B(x_i,1)|_{dv_{g_i(0)}}= 1, 
  \end{align*}
  which contradicts (\ref{eqn:RE14_1}) for $\rho=1$.  This contradiction establishes the proof of the proposition. 
\end{proof}

\begin{proposition}[\textbf{Asymptotic scalar curvature behavior}]
  Suppose $(M,g)$ is a non-flat complete Riemannian manifold with non-negative Ricci curvature and  $AVR(M,g)=(1-\delta) \omega_{m} \geq (1-\delta_0) \omega_{m}$. 
  Then we have 
  \begin{align}
     \liminf_{t \to \infty}  R(x,t)t  \geq \psi(\delta|m),  \label{eqn:RE14_13}
  \end{align}
  and the limit is independent the choice of $x$. 
  \label{prn:RE14_12}
\end{proposition}

\begin{proof}
  The inequality (\ref{eqn:RE14_13}) follows directly from Proposition~\ref{prn:RE08_4}. 
  It remains to show that the limit in (\ref{eqn:RE14_13}) is independent of $x$. 
  Applying the pseudo-locality theorem and Shi's estimate(cf.~\cite{Shi}) on the flow $g(t)$, we have
  \begin{align*}
    t|Rm|(\cdot,t) + t^{\frac{3}{2}} |\nabla Rm|(\cdot, t)< \psi(\delta|m).
  \end{align*}
  In particular, we have
  \begin{align*}
    t^{\frac{3}{2}} |\nabla R| < \psi(\delta|m).
  \end{align*}
  Fix $x,y \in M$. For $t$ very large,  it follows from the distance distortion estimate Theorem~\ref{thm:RE05_5} and the above inequality that 
  \begin{align*}
    |R(x,t)-R(y,t)| &\leq \int_{\gamma} |\nabla R|(\gamma(\theta)) d\theta \leq \left\{ t^{-\frac{3}{2}} \psi(\delta|m) \right\} \cdot d_{g(t)}(x,y)\\
    &\leq 2 \left\{ t^{-\frac{3}{2}} \psi(\delta|m) \right\} \cdot d_{g(0)}^{1-\psi}(x,y) \left(\sqrt{t}\right)^{\psi},
  \end{align*}
  where $\gamma$ is a shortest geodesic connecting $x,y$ under the metric $g(t)$.  It follows that
  \begin{align*}
    \lim_{t \to \infty} |tR(x,t)-tR(y,t)| \leq \lim_{t \to \infty} 2\psi(\delta|m) \left( \sqrt{t} \right)^{-1+\psi} d_{g(0)}^{1-\psi}(x,y)=0. 
  \end{align*}
  Consequently, $\displaystyle \lim_{t \to \infty} tR(x,t)$ does not depend on $x$. 
\end{proof}

\begin{corollary}[\textbf{Divergence of the immortal Ricci flow}]
  Same conditions as in Proposition~\ref{prn:RE14_12}.  Define
  \begin{align*}
     \lambda_1(x,t) \coloneqq \inf_{0 \neq V \in T_{x}M} \frac{\langle V,V\rangle_{g(t)}}{\langle V,V\rangle_{g(0)}}.
  \end{align*}
  Then we have 
  \begin{align}
    \lim_{t \to \infty} \lambda_1(x,t)=0   \label{eqn:RE14_15}
  \end{align}
  for each $x \in M$. 
  \label{cly:RE14_14}
\end{corollary}

\begin{proof}
  Fix $x \in M$ and $t>0$. The linear space $T_{x} M$ has two metrics $g(0)$ and $g(t)$.
  We can choose orthonormal basis $\left\{ e_{i} \right\}_{i=1}^{m} \subset T_{x}M$ with respect to $g(0)$ such that
  \begin{align*}
    \langle e_i, e_{i}\rangle_{g(t)}=\lambda_i(t), \quad 0<\lambda_1 \leq \lambda_2 \leq \cdots \leq \lambda_m.
  \end{align*}
  Then we have 
   \begin{align*}
     \lambda_1^{\frac{m}{2}}(t) \leq \sqrt{\lambda_{1}(t) \lambda_2(t) \cdots \lambda_{m}(t)}=\frac{dv_{t}}{dv_{0}}=e^{-\int_{0}^{t} R(x,s)ds}.
   \end{align*}
   It follows that
   \begin{align*}
     \lambda_1(t) \leq e^{-\frac{2}{m} \int_{0}^{t} R(x,s)ds}. 
   \end{align*}
   In light of Proposition~\ref{prn:RE14_12}, we have
   \begin{align*}
     0\leq \lim_{t \to \infty} \lambda_1(t) \leq e^{\frac{-2\lim_{t \to \infty} \int_{0}^{t}R(x,s)ds}{m}}=0,
   \end{align*}
   which is nothing but (\ref{eqn:RE14_15}). 
\end{proof}

In view of Corollary~\ref{cly:RE14_14}, the immortal Ricci flow solution initiated from $(M, g)$ diverges at time infinity. 
In particular, for each $x \in M$ and constant $C$, we can find $T=T(x,C)$ such that
\begin{align}
  \frac{1}{C} g(x,0) < g(x,t)   \label{eqn:RH27_3} 
\end{align}
fails for each $t>T$. 
In Corollary~\ref{cly:RH27_1}, for fixed $x,y \in M$ and large $t$, we have the bi-H\"older equivalence between $d_{g(0)}(x,y)$ and $d_t(x,y)$ as
\begin{align*}
  e^{-\psi} d_{g(0)}^{1+\psi}(x,y) \left(\sqrt{t} \right)^{-\psi}
  <d_{g(t)}(x,y)< e^{\psi} d_{g(0)}^{1-\psi}(x,y) \left(\sqrt{t} \right)^{\psi},
\end{align*}
which follows from (\ref{eqn:RE08_1}). 
One may wonder whether we can improve the bi-H\"older equivalence to the bi-Lipschitz equivalence
\begin{align*}
  \frac{1}{C(m)} d_{g(0)}(x,y) < d_{g(t)}(x,y)< C(m) d_{g(0)}(x,y).
\end{align*}
By arbitrary choice of $x,y$ and $t>d_{g(0)}(x,y)$, such estimate will imply that
\begin{align}
  \frac{1}{C} g_{0} < g(t) < C g_{0},  \quad \forall \; t>1,     \label{eqn:RE08_3}   
\end{align}
which contradicts the failure of (\ref{eqn:RH27_3}).  Therefore, the distance distortion estimate (\ref{eqn:RE08_1}) can hardly be improved. 
In fact, we shall construct examples on closed K\"ahler manifolds in Section~\ref{sec:globalstability} to show that (\ref{eqn:RE08_1}) can be achieved.   Therefore, (\ref{eqn:RE08_1}) can be
regarded as optimal. 

\section{Continuous dependence on the initial data}
\label{sec:cdepend}

The purpose of this section is to show Theorem~\ref{thm:RG05_3}, which says that the Ricci flow solution depends on the initial data continuously in the Gromov-Hausdorff topology.
We shall first show this for unnormalized Ricci flow in Theorem~\ref{thm:PB03_1}, then we show the uniform volume continuity(cf. Lemma~\ref{lma:RB01_1}) along the Ricci flow and 
prove a version for normalized Ricci flow in Theorem~\ref{thm:RF24_2}.

\begin{theorem}[\textbf{Continuous dependence in the $C^{\infty}$-Cheeger-Gromov topology}]
  Suppose that  $\mathcal{N}=\{(N^{m}, h(t)), 0 \leq t \leq T\}$ is a Ricci flow solution initiated from $h=h(0)$, a smooth Riemannian metric on a closed manifold $N^{m}$. 
  For each $\epsilon$ small, there exists $\delta=\delta(\mathcal{N}, \epsilon)$ with the following properties. 

  Suppose $(M^{m}, g)$ is a Riemannian manifold satisfying
  \begin{align}
    Rc>-(m-1)\epsilon^{-1},  \quad d_{GH}\left\{ (M,g), (N,h) \right\}<\delta. 
    \label{eqn:PB03_3}
  \end{align}
  Then the Ricci flow initiated from $(M,g)$ exists on $[0, T]$. 
  Furthermore, there exists a family of diffeomorphisms $\left\{ \Phi_{t}: N \to M, \; \epsilon \leq t \leq T \right\}$ such that
 \begin{align}
   \sup_{t \in [\epsilon, T]} \norm{\Phi_{t}^{*}g(t)-h(t)}{C^{[\epsilon^{-1}]}(N,h(t))} < \epsilon.   \label{eqn:RF22_1}
 \end{align}
 \label{thm:PB03_1}
\end{theorem}

In view of  Theorem~\ref{thm:PB03_1}, 
if we take a sequence $(M_i^{m}, g_i)$ converging to $(N^{m},h)$ in the Gromov-Hausdorff topology while keeping Ricci curvature  bounded from below,
then the existence time of the Ricci flow initiated from $(M_i, g_i)$ will be at least  $T$ and 
 \begin{align}
     \left( M_i, g_i(t) \right) \longright{C^{\infty}-Cheeger-Gromov} \left( M, g(t) \right) 
     \label{eqn:PI27_10}  
   \end{align}
 for each $t \in (0,T]$, by taking subsequence if necessary.  

The proof of Theorem~\ref{thm:PB03_1} needs three preliminary steps.  
We first construct diffeomorphism between locally almost flat manifolds with rough Gromov-Hausdorff approximations.
This step has nothing to do with flow.
In step two, we show that the almost flatness and rough Gromov-Hausdorff approximation conditions in Lemma~\ref{lma:PB01_1} are naturally realized
by the curvature-injectivity-radius estimate in the improved pseudo-locality.  Therefore, we are able to construct a diffeomorphism map from $N$ to $M$.
Via this diffeomorphism map and time shifting, we can regard $g(\xi)$ as a $C^{0}$-approximation of $h$ for very small $\xi$. 
This second step will be done in Lemma~\ref{lma:PI27_1}.
In the third step, we use a fact(cf.~Lemma~\ref{lma:RF24_1}) that the $C^{0}$-approximation is preserved for a long period of time along the Ricci-DeTurck flow 
with given background metric $h$.  With these three steps,  one can finish the proof of Theorem~\ref{thm:PB03_1} by the standard method of coupling Ricci-flow
with harmonic map flow.

\begin{lemma}[\textbf{Construction of diffeomorphisms}]
For each small $\epsilon>0$, there exists a constant $\delta$ with the following properties.

Suppose $(M^{m}, g)$ and $(N^{m}, h)$ are two complete Riemannian manifolds satisfying
\begin{align}
  |Rm| \leq \delta^{2}, \quad inj \geq \delta^{-1}.   \label{eqn:PB03_2}
\end{align}
Suppose there exist  (not necessarily continuous) maps $F: M \to N$ and $G: N \to M$ such that 
\begin{align}
    &\left|  d(x,y)-d(F(x), F(y)) \right| < \delta, \quad \forall \; x, y \in M \; \textrm{satisfying} \; d(x,y)<100;   \label{eqn:PI27_1}\\
    &\left| d(x,y) - d(G(x), G(y)) \right|< \delta, \quad \forall \; x,y \in N \; \textrm{satisfying} \; d(x,y)<100;   \label{eqn:PI27_2}\\
    &\sup_{x \in M} d(x, G \circ F(x)) + \sup_{y \in N} d(y, F \circ G(y))<\delta.   \label{eqn:PI27_5}
\end{align}
Then we have  diffeomorphisms $\tilde{F}: M \to N$ and $\tilde{G}: N \to M$ such that
\begin{align}
 \sup_{x \in M} \left|\tilde{F}^{*} h -g \right|_{g}(x)  + \sup_{y \in N} \left| \tilde{G}^{*}g-h \right|_{h}(y)< \epsilon.   \label{eqn:PB03_1}
\end{align}

 \label{lma:PB01_1}
\end{lemma}

\begin{proof}
   The proof originates from the celebrated work of Cheeger~\cite{Cheeger70}.  
   The diffeomorphisms $\tilde{F}$ and $\tilde{G}$ satisfying (\ref{eqn:PB03_1}) can be constructed explicitly.
   We shall basically follow the construction in Peters~\cite{SPeters} and Green-Wu~\cite{GreenWu}, where a uniform diameter of $M$ and $N$ is assumed.
   Checking the proofs there carefully, it is not hard to see that the diameter bound is not essentially used
   for deriving diffeomorphisms satisfying $C^{0}$-norm estimate (\ref{eqn:PB03_1}).

   Let $\left\{ x_1, x_2, x_3 \cdots, x_{k}, \cdots \right\}$ be points in $M$ such that
   $\{B(x_i, 1)\}_{i=1}^{\infty}$ are disjoint and $M \subset \cup_{i=1}^{\infty} B(x_i, 2)$.
   Denote $F(x_i)$ by $y_{i}$.  By condition (\ref{eqn:PI27_1}), we know
   \begin{align*}
     |d(y_i, y_j) -d(x_i, x_j)|<\delta, \quad \textrm{whenever}\;  d(x_i, x_j)<10. 
   \end{align*}
   By the choice of $x_i$ and the condition (\ref{eqn:PI27_2}), it is clear that  $B(y_i, 1-\delta)$ are disjoint and $B(y_i, 2+\delta)$ is a covering of $N$. 
   For each $i$, let $A(i)$ be the collection of $j$ such that $B(x_j, 2) \cap B(x_i, 2) \neq \emptyset$. 
   By volume comparison, it is easy to see that $\sharp\{A(i)\}<5^{m}$.

   Following Peters~\cite{SPeters}, we let $u_i: \R^{m} \to T_{x_i}M$ and $\bar{u}_i: \R^{m} \to T_{y_{i}}N$ be linear isometries.
   Then we set
   \begin{align*}
     &\phi_{i} \coloneqq exp_{x_{i}} \circ u_{i}: B(0,10) \subset \R^{m} \to B(x_{i}, 10) \subset M, \\
     &\bar{\phi}_{i} \coloneqq exp_{y_{i}} \circ \bar{u}_{i}: B(0,10) \subset \R^{m} \to B(y_{i}, 10) \subset N.  
   \end{align*}
   For each $j \in A(i)$, we denote the the parallel translation along the shortest geodesic connecting $x_i$ and $x_j$ by $P_{ij}$.
   Similarly, the parallel translation along the shortest geodesic connecting $y_i$ and $y_j$ is denoted by $\bar{P}_{ij}$. 
   In light of the distance distortion estimate (\ref{eqn:PI27_1}) and the curvature-injectivity-radius estimate (\ref{eqn:PB03_2}), for each fixed $i$, we could adjust $\phi_{j}$ for all $j \in A(i)$ such that(cf. Lemma of~\cite{SPeters}) 
   \begin{align*}
    &d(\phi_{j}^{-1}\phi_{i}, \bar{\phi}_{j}^{-1} \bar{\phi}_{i})<\psi(\delta|m), \\
    &\norm{u_{j}^{-1}P_{ij}u_{i}-\bar{u}_{j}^{-1}\bar{P}_{ij}\bar{u}_{i}}{}<\psi(\delta|m),
   \end{align*}
   for each $j \in A(i)$. 
   After these preparations, on each $B(x_i, 8)$, we define
   \begin{align*}
     F_{i} \coloneqq \bar{\phi}_{i} \phi_{i}^{-1}, \quad B(x_i, 8) \mapsto B(y_i, 8). 
   \end{align*}
   Clearly, $F_{i}$ is a locally defined diffeomorphism satisfying
   \begin{align*}
     F_{i}(x_i)=y_i, \quad  
     \inf_{x \in B(x_i, 8)} d(F(x), F_{i}(x))<\psi(\delta|m). 
   \end{align*}
   Let $\eta: \R \to \R^{+} \cup \left\{ 0 \right\}$ be a cutoff function such that $\eta \equiv 1$ on $(-\infty, 1)$ and $\eta \equiv 0$ on $(2, \infty)$.
   Furthermore, we have $\eta' \leq 0$ and $|\eta'| \leq 2$.   
   Let $f_i=\eta \left( d(\cdot, x_i) \right)$. It is clear that $f_i$ is supported on $B(x_i,2)$ and equals $1$ on $B(x_i,1)$. 
   Define
   \begin{align*}
     \Phi_{i}(x) \coloneqq \frac{f_i(x)}{\sum_{j} f_j(x)}=\frac{f_i(x)}{\sum_{j \in A(i)} f_j(x)}. 
   \end{align*}
   We see that $\left\{ \Phi_{i} \right\}_{i=1}^{\infty}$ is a partition of unity of $M$, with respect to the covering $\left\{ B(x_i,2) \right\}_{i=1}^{\infty}$.

   Fix an arbitrary point $x \in M$.  Without loss of generality, we may assume $x \in B(x_i,2)$.   Then $\Phi_j(x) \neq 0$ if and only if $j \in A(i)$.
   We define $\tilde{F}(x)$ to be the mass center of $F_{j}(x)$ with respect to the weight $\Phi_j(x)$. In other words, $\tilde{F}(x)$ is the unique point
   $y \in \cup_{j \in A(i)} B(y_j, 2+\delta) \subset B(y_i,5)$
   satisfying
   \begin{align*}
     \sum_{j \in A(i)}  \Phi_{j}(x) Exp_{y}^{-1} F_{j}(x)=0.
   \end{align*}
   The existence and uniqueness of $y$ are guaranteed by the implicit function theorem.   Note that in $B(y_i,5)$, the curvature is very small, and the injectivity
   radius is very large, provided by the condition (\ref{eqn:PB03_2}).  Further calculation implies that
   \begin{align}
     1-\psi(\delta|m) \leq \frac{|d\tilde{F}(\vec{v})|}{|\vec{v}|} \leq 1+\psi(\delta|m), \quad \forall \; \vec{v} \in T_x M \backslash \left\{ 0 \right\}. \label{eqn:PB05_1} 
   \end{align}
   By the arbitrary choice of $x$, the above inequalities imply that 
   \begin{align}
       \sup_{x \in M}  \left|\tilde{F}^* h -g \right|_{g}(x) < \psi.   \label{eqn:PI27_3}
   \end{align}
   By similar construction, we can perturb $G$ to a smooth map $\tilde{G}: N \to M$ such that
   \begin{align}
       \sup_{y \in N}  \left|\tilde{G}^* g -h \right|_{h}(x) < \psi.   \label{eqn:PI27_4}
   \end{align}
   Combining (\ref{eqn:PI27_3}) and (\ref{eqn:PI27_4}), by choosing $\delta$ sufficiently small, we arrive at (\ref{eqn:PB03_1}).

   It remains to show that both $\tilde{F}$ and $\tilde{G}$ are diffeomorphisms. 
   By symmetry, it suffices to show that $\tilde{F}: M \to N$ is a diffeomorphism.
   This can be realized by a standard argument(cf.~\cite{SPeters},~\cite{BuKa} and the references therein). 
   We write down a few more words for the convenience of the readers. 
   The center of mass construction guarantees that $\tilde{F}, \tilde{G}$ are just perturbations of $F$ and $G$.  Then it follows from (\ref{eqn:PI27_5}) that
   \begin{align*}
     \sup_{x \in M} d\left(x, \tilde{G} \circ \tilde{F}(x) \right)<\psi, \quad  \sup_{y \in N} d\left(y, \tilde{F} \circ \tilde{G}(y) \right) < \psi
   \end{align*}
   for some $\psi=\psi(\delta|m)$.
   In other words, both $\tilde{F} \circ \tilde{G}$ and $\tilde{G} \circ \tilde{F}$ are close to the identity map in the $C^0$-sense. 
   By condition (\ref{eqn:PI27_1}) and (\ref{eqn:PI27_2}), $\tilde{G} \circ \tilde{F}(x)$ locates in a small neighborhood of $x$ where the geometry is almost Euclidean and strong convexity holds.
   Therefore, for each $x \in M$,  there exists a unique shortest geodesic connecting $x$ to $\tilde{G} \circ \tilde{F}(x)$ for each $x \in M$.  
   Therefore, $\tilde{G} \circ \tilde{F}(x)$ is in the trivial homotopy class and can be smoothly deformed to $Id_{M \to M}$. 
   In particular, the degree of $\tilde{G} \circ \tilde{F}$ is $1$.  Note that $\tilde{G} \circ \tilde{F}$ is non-degenerate by (\ref{eqn:PI27_3}) and (\ref{eqn:PI27_4}). 
   This forces that for every point $y \in N$, the preimage $\tilde{F}^{-1}(y)$ contains only one point.
   Since $\tilde{F}: M \to N$ is a smooth  covering map by (\ref{eqn:PB05_1}), it is clear that $\tilde{F}$ is
   a diffeomorphism from $M$ to $N$.  By reversing the role of $\tilde{F}$ and $\tilde{G}$, we obtain that $\tilde{G}$ is a diffeomorphism from $N$ to $M$. 
 \end{proof}

 \begin{lemma}[\textbf{Construction of $C^{0}$-approximation}]
   Suppose $(N^{m},h)$ is a closed smooth Riemannian manifold. 
   For each small $\xi$, there exists $\delta=\delta(\xi|N,h)$ such that the following properties hold.

   If $(M^{m}, g)$ is a Riemannian manifold satisfying
   \begin{align}
     Rc \geq -(m-1)\xi^{-1}, \quad d_{GH}\left\{(M,g),  (N,h) \right\}<\delta,  \label{eqn:PI27_6}
   \end{align}
   then there exist a diffeomorphism map $\tilde{G}: N \to M$ and a time $t_0 \in (0, \xi)$ such that
   \begin{align}
     \left| \log \frac{\langle \vec{v}, \vec{v}\rangle_{\tilde{G}^{*}g(t_0)}}{\langle \vec{v}, \vec{v} \rangle_{h}} \right|<0.5\xi,
      \quad \forall \;  \vec{v} \in T_{x}N \backslash \left\{ 0 \right\}, \; x \in N,
      \label{eqn:PI28_1}   
   \end{align}
   where $g(t_0)$ is the time-$t_0$ slice of the Ricci flow solution started from $g(0)=g$.  In particular, we have
   \begin{align}
     \norm{\tilde{G}^{*}g(t_0)-h}{h}<\xi.
     \label{eqn:RF16_1}
   \end{align}
   \label{lma:PI27_1}
 \end{lemma}

 \begin{proof}
 Since $(N,h)$ has bounded geometry, for $\xi<1$ sufficiently small, we have $|Rm|_{h}\leq \frac{1}{100m^2 \xi}$.
 Then the maximum principle argument for the curvature evolution equation implies that the Ricci flow started from $(N, h)$ exists and preserves $|Rm|_{h(t)} \leq \frac{1}{2m\xi}$ on the time interval $[0, \xi]$. 
 Fix another small number $\beta$,  we can choose $r_0=r_0(\beta,\xi)<\xi$ sufficiently small such that the following estimates hold for every $x \in N$:
 \begin{align*}
 1-0.5\beta \leq \omega_{m}^{-1} r^{-m}|B(x,r)| \leq 1 + 0.5 \beta, \quad \forall \; r \in (0, r_{0}]. 
 \end{align*}
 Furthermore, in light of Bishop-Gromov volume comparison and Colding's volume continuity theorem(cf.~\cite{Co}), 
 the above inequalities are almost stable under the Gromov-Hausdorff convergence while keeping Ricci curvature bounded from below. 
 Namely, if $(M,g)$ satisfies (\ref{eqn:PI27_6}) and $\delta$ is sufficiently small,  we have the following estimates for every $x \in M$:
 \begin{align*}
  1-\beta \leq \omega_{m}^{-1} r^{-m}|B(x,r)| \leq 1 + \beta, \quad \forall \; r \in (0, r_{0}]. 
 \end{align*}
 In view of Theorem~\ref{thm:RH25_1}, we can apply Theorem~\ref{thm:RE05_5} to obtain the distance distortion estimate
   \begin{align}
     \left|\log \frac{d_{(\xi r_0)^2}(x,y)}{d_{g(0)}(x,y)} \right|
     <\psi(\beta|m, \xi), \quad \forall \; x,y \in M, \; 0.1 r_0 <d_{g(0)}(x,y)<10 r_0.
     \label{eqn:RF15_1}  
   \end{align}
   Let $\left\{ x_{i} \right\}$ be an $r_0$-net of $(M, g)$.
   Since $(M,g)$ satisfies (\ref{eqn:PI27_6}), we can find a $\delta$-Gromov-Hausdorff approximation map 
   $F:(M,g) \to (N,h)$ and $G: (N,h) \to (M,g)$ such that $y_i=F(x_i), x_i=G(y_i)$.  
   Clearly, $\left\{ y_i \right\}$ is an $(r_0+2\delta)$-net in $(N,h)$.   
   In light of the distance distortion estimate (\ref{eqn:RF15_1}) and its correspondence form for $N$, we know that
   both $F:(M, g(\xi^2r_0^2)) \to (N, h(\xi^2 r_0^2))$ and $G:(N, h(\xi^2 r_0^2)) \to (M, g(\xi^2 r_0^2))$ are
   $\psi(\beta|\xi,m) r_0$-Gromov-Hausdorff approximation, whenever $\delta$ is sufficiently small.
   Applying rescaling, we see that 
   \begin{align*}
     &F: (M, \xi^{-2} r_{0}^{-2} g(\xi^{2}r_{0}^{2})) \to (N, \xi^{-2} r_{0}^{-2} h(\xi^{2}r_{0}^{2})),\\ 
     &G: (N, \xi^{-2} r_{0}^{-2} h(\xi^{2}r_{0}^{2})) \to (M, \xi^{-2} r_{0}^{-2} g(\xi^{2}r_{0}^{2})) 
   \end{align*}
   are $\psi(\beta|\xi,m)$ approximations. 
   By the curvature estimate (\ref{eqn:ML28_1}) and injectivity radius estimate (\ref{eqn:RH27_4}) in the improved pseudo-locality theorem,
   we see that both $(M, \xi^{-2} r_{0}^{-2} g(\xi^{2}r_{0}^{2}))$ and $(N, \xi^{-2} r_{0}^{-2} h(\xi^{2}r_{0}^{2}))$ are locally almost flat on the scale $100$. 
   Therefore, we can apply Lemma~\ref{lma:PB01_1} here to deform $G$ to a diffeomorphism $\tilde{G}: N \to M$ satisfying
   \begin{align}
     \left| \log \frac{\langle \vec{v}, \vec{v}\rangle_{\tilde{G}^{*}g(\xi^2 r_0^2)}}{\langle \vec{v}, \vec{v}\rangle_{h(\xi^2 r_0^2)}} \right|
     =\left| \log \frac{\langle \vec{v}, \vec{v}\rangle_{\tilde{G}^{*}\left(\xi^{-2}r_0^{-2}g(\xi^2 r_0^2)\right)}}{\langle \vec{v}, \vec{v}\rangle_{\xi^{-2}r_0^{-2}h(\xi^2 r_0^2)}} \right|
     <\psi(\beta|\xi,m),
      \quad \forall \;  \vec{v} \in T_{x}N \backslash \left\{ 0 \right\}, \; x \in N. 
   \label{eqn:RF15_2}
   \end{align}
   In the above inequalities, we used the scaling invariance of the $C^{0}$-equivalence of metrics. 
   Since the curvature of $\{(N, h(t)), 0 \leq t \leq \xi \}$ is bounded by $\frac{1}{2m\xi}$, we have
   \begin{align}
    \left| \log \frac{\langle \vec{v}, \vec{v}\rangle_{h(\xi^2 r_0^2)}}{\langle \vec{v},\vec{v}\rangle_{h}} \right|<\xi r_0^2<\xi^{3},
      \quad \forall \;  \vec{v} \in T_{x}N \backslash \left\{ 0 \right\}, \; x \in N. 
    \label{eqn:RF15_3}
   \end{align}
   It follows from the combination of  (\ref{eqn:RF15_2}) and (\ref{eqn:RF15_3}) that
   \begin{align*}
    \left| \log \frac{\langle \vec{v}, \vec{v}\rangle_{\tilde{G}^* g(\xi^2 r_0^2)}}{\langle \vec{v},\vec{v}\rangle_{h}} \right|<\psi(\beta|\xi,m)+\xi^{3},
      \quad \forall \;  \vec{v} \in T_{x}N \backslash \left\{ 0 \right\}, \; x \in N. 
  \end{align*}
  For fixed $\xi$, we can choose $\beta$ sufficiently small such that $\psi(\beta|\xi, m)<\xi^{3}$.   
  Since $\xi<<1$, it is clear that (\ref{eqn:PI28_1}) follows from the above inequality by letting $t_0=\xi^2 r_0^2<\xi$. 
  Then (\ref{eqn:RF16_1}) is a direct consequence of (\ref{eqn:PI28_1}). 
 \end{proof}

The stability of Ricci-DeTurck flow under $C^{0}$-perturbation was first studied by M. Simon(cf. Theorem 1.1. of~\cite{MSimon}). 
For a given Ricci-DeTurck flow, the linearized equation of Ricci-DeTurck flow is strictly parabolic.
Applying the heat kernel of the given Ricci-Deturk space-time on linearized equation, one can estimate more precisely how does 
the evolving metrics depend on the initial metric.  This was observed and achieved by Koch-Lamm (cf. Theorem 4.3 of~\cite{KoLa}) when
the background metric is Euclidean. 
By further localization technique, this idea can be realized on Ricci-DeTurck flow on manifolds with bounded curvature(cf. Corollary 3.4 of Burkhardt-Guim~\cite{PBG}).
The following lemma is a special case of the results of ~\cite{MSimon} or~\cite{PBG}. 

\begin{lemma}[\textbf{Preservation of $C^{0}$-approximation}]
  Suppose $\left\{ (N, h(t)), 0 \leq t \leq T \right\}$ is a smooth Ricci-DeTurck flow solution on a closed manifold $N$, with background metric $\tilde{h}$. 
  Namely, we have
  \begin{align}
    \partial_t h_{ij}= -2R_{ij} + \nabla_i W_j + \nabla_j W_i, 
    \label{eqn:PB12_1}
  \end{align}
  where $W_i=h_{ij}\tilde{h}^{pq}\left( \Gamma_{pq}^{j} -\tilde{\Gamma}_{pq}^{j} \right)$.
  Then there exist  a small constant $\xi$ and a large constant $L$, both depending on the given Ricci-DeTurck flow,  with the following properties.

  Suppose $\left\{ (N, g(t)), 0 \leq t \leq T' \right\}$ is another Ricci-DeTurck flow with the same background metric $\tilde{h}$ and maximal existence time $T'$.  
  Suppose $\norm{g(0)-h(0)}{\tilde{h}} <\xi$, then we have $T' \geq T$ and 
  \begin{align}
    \sup_{t \in [0,T]} \norm{g(t)-h(t)}{\tilde{h}} < L \norm{g(0)-h(0)}{\tilde{h}}.   \label{eqn:PI28_3}
  \end{align}

  \label{lma:RF24_1}
\end{lemma}

Combining Lemma~\ref{lma:PB01_1}, Lemma~\ref{lma:PI27_1} and Lemma~\ref{lma:RF24_1},  we are now ready for the proof of Theorem~\ref{thm:PB03_1}. 

\begin{proof}[Proof of Theorem~\ref{thm:PB03_1}:]

  By compactness of the flow $\left\{ (N,h(t)), 0 \leq t \leq T \right\}$, we can find a small positive constant $\eta$ such that
  the flow can be extended to the time interval $[0, T+\eta]$. 

  Fix $\epsilon<\eta$. For each small positive constant $\xi<<\epsilon$ such that we can apply Lemma~\ref{lma:PI27_1} for constant $\xi$.
  By choose $\delta$ sufficiently small, we can find $t_0 \in [0,\xi]$ such that (\ref{eqn:PI28_1}) holds. As $\xi$ is very small, with respect to the background metric $h=h(0)$, we can apply Lemma~\ref{lma:RF24_1}
  to obtain a Ricci-DeTurck flow  $\left\{ (N, \bar{g}(t)), \; t_0 \leq t \leq T+t_{0} \right\}$, which starts from $\bar{g}(t_0) \coloneqq \tilde{G}^* g(t_0)$ and satisfies the estimate
  \begin{align}
    \sup_{t \in [0,T]} \norm{\bar{g}(t+t_0)-\tilde{h}(t)}{\tilde{h}(t)} < L \norm{\bar{g}(t_0)-h}{h}<L\xi.   \label{eqn:RF22_2} 
  \end{align}
  Here the constant $L$ depends on the compact space-time $\left\{ (N, h(t)), 0 \leq t \leq T \right\}$. 
  Via diffeomorphism $\tilde{G}: N \to M$, it is clear that $\tilde{g}(t) \coloneqq (\tilde{G}^{-1})^{*} \bar{g}(t)$ is a Ricci-DeTurck flow on $M$ with background metric $( \tilde{G}^{-1})^{*} h$ 
  and it starts from initial metric $\tilde{g}(t_0)=(\tilde{G}^{-1})^{*} \bar{g}(t_0)=g(t_0)$.
  Since the Ricci flow can be achieved as the Ricci DeTurck flow coupled with the harmonic map flow of diffeomorphisms,
  we can find time dependent diffeomorphisms $\left\{ \tilde{\Phi}_{t}: M \to M, t_0 \leq t \leq T+t_{0} \right\}$ with $\tilde{\Phi}_{t_0}=Id$ such that
  $\left\{\left(M, \left(\tilde{\Phi}_{t}^{-1}\right)^{*} \tilde{g}(t) \right), t_0 \leq t \leq T+t_0 \right\}$ is the Ricci flow solution initiated from $g(t_0)$. 
  Concatenating the Ricci flow constructed above with the Ricci flow $\left\{ (M, g(t)), 0 \leq t \leq t_0 \right\}$, we naturally  obtain a Ricci flow solution $\left\{ (M, g(t)), 0 \leq t \leq T+t_0 \right\}$. 

\begin{align}
\begin{split}
\begin{xy}
%row first, column next; left to right, bottom to top.  
  <-11em,5em>*+{\textcolor{black}{(N,h(t))}}="u",
<11em,5em>*+{(M,g(t+t_0))}="v",
<-11em, 0em>*+{\textcolor{black}{(N,\tilde{h}(t))}}="w",
<0em,0em>*+{(N,\bar{g}(t+t_0))}="x", 
<11em,0em>*+{(M, \tilde{g}(t+t_0))}="y",
<-11em,-5em>*+{\textcolor{black}{(N,h)}}="z",
<0em,-5em>*+{(N, \bar{g}(t_0))}="m", 
<11em,-5em>*+{(M,g(t_0))}="n",
 "w";"x" **@{-} ?>*@{>}?<>(.5)*!/_0.5em/{\scriptstyle Id},
 "x";"y" **@{-} ?>*@{>} ?<>(.5)*!/_0.5em/{\scriptstyle \tilde{G}},
 "z";"m" **@{-} ?>*@{>} ?<>(.5)*!/_0.5em/{\scriptstyle Id},
 "m";"n" **@{-} ?>*@{>} ?<>(.5)*!/_0.5em/{\scriptstyle \tilde{G}},
 "z";"w" **@{-} ?>*@{>} ?<>(.5)*!/_0.5em/{\scriptstyle \textcolor{black}{Id}},
 "w";"u" **@{-} ?>*@{>} ?<>(.5)*!/_0.5em/{\scriptstyle \textcolor{black}{\tilde{\Psi}_{t}}}, 
 "m";"x" **@{-} ?>*@{>} ?<>(.5)*!/_0.5em/{\scriptstyle Id},
 "n";"y" **@{-} ?>*@{>} ?<>(.5)*!/_0.5em/{\scriptstyle Id}, 
 "y";"v" **@{-} ?>*@{>} ?<>(.5)*!/_0.5em/{\scriptstyle \tilde{\Phi}_{t+t_0}}, 
\end{xy}
\end{split}
\label{eqn:diffeomorphisms}
\end{align}

  We proceed to prove (\ref{eqn:RF22_1}).  Let $\tilde{h}(t)$ be the Ricci-DeTurck flow started from $h$, with background metric $h$. 
  Then there are time-dependent diffeomorphism $\tilde{\Psi}_{t}:N \to N$
  such that $\tilde{h}(t)=\tilde{\Psi}_{t}^{*}(h(t))$. 
  Plugging the facts $g(t)=\left(\tilde{G}^{-1} \circ \tilde{\Phi}_{t}^{-1} \right)^{*} \bar{g}(t)$ and $h(t)=(\tilde{\Psi}_{t}^{-1})^{*}\left(\tilde{h}(t) \right)$ into (\ref{eqn:RF22_2}), it is clear that 
  \begin{align*}
    &\sup_{t \in [0,T]} \norm{ \left( \tilde{\Phi}_{t+t_0} \circ \tilde{G} \circ \tilde{\Psi}_{t}^{-1} \right)^{*}  g(t+t_0)- h(t)}{h(t)} 
    =\sup_{t \in [0,T]} \norm{\left(\tilde{\Phi}_{t+t_0} \circ \tilde{G} \right)^{*} g(t+t_0)-\tilde{\Psi}_{t}^{*} h(t)}{\Psi_{t}^{*}(h(t))}\\ 
    &\quad=\sup_{t \in [0,T]} \norm{\bar{g}(t+t_0)-\tilde{h}(t)}{\tilde{h}(t)}< L\xi. 
  \end{align*}
  Define $\Phi_{t+t_0} \coloneqq \tilde{\Phi}_{t+t_0} \circ \tilde{G} \circ \tilde{\Psi}_{t}^{-1}$,  we have
  \begin{align}
    \sup_{t \in [0,T]} \norm{\Phi_{t+t_0}^{*}  g(t+t_0)- h(t)}{h(t)}<L\xi<\epsilon   \label{eqn:RF23_1}
  \end{align}
  since we have chosen $\xi << \epsilon$. Recall that $0<t_0<\xi<<\epsilon$. 
  The flow $\left\{ (N,h(t)), 0 \leq t \leq T \right\}$ has bounded geometry by smooth compactness.
  Thus the isoperimetric radius is uniformly bounded from below along this flow.
  By adjusting the gap constant $\epsilon_{0}$ in definition if necessary, it follows from (\ref{eqn:RF23_1})
  that the isoperimetric radius(cf. Definition~\ref{dfn:RH26_3}) of $g(s)$ for each $s \in [t_0, T+t_0]$ is uniformly bounded from below. 
  Via the pseudo-locality theorem and Shi's curvature estimate, we obtain
  \begin{align}
    \sup_{t \in [0.5\epsilon, T]}\norm{\nabla^{k} Rm}{g(t)}<C(\epsilon,k,\mathcal{N})
  \label{eqn:RF23_3}
  \end{align}
  for each non-negative integer $k$. In particular, we have
  $|Rc|_{g(t)}$ is uniformly bounded.  As metrics evolve by $-2Rc$ along Ricci flow, it follows that
  \begin{align}
    \sup_{t \in [0.5 \epsilon,T]} \norm{\Phi_{t}^{*}  g(t)- \Phi_{t+t_0}^{*}  g(t+t_0)}{\Phi_{t+t_0}^{*}  g(t+t_0)}<L \xi \label{eqn:RF23_2} 
  \end{align}
  by adjusting $L$ if necessary.  Combining (\ref{eqn:RF23_1}) and (\ref{eqn:RF23_2}), by adjusting $L$ again, we obtain 
  \begin{align}
    \sup_{t \in [\epsilon,T]} \norm{\Phi_{t}^{*}  g(t)- h(t)}{h(t)}<L\xi.  \label{eqn:RF23_4} 
  \end{align}
  For each integer $k \in [1, \epsilon^{-1}+1]$, we show that
  \begin{align}
    \sup_{t \in [\epsilon,T]} \norm{\Phi_{t}^{*}  g(t)- h(t)}{C^{k}(N,h(t))}<\epsilon  \label{eqn:RF23_5}
  \end{align}
  whenever $d_{GH}\left( (M,g), (N,h) \right)<\delta$ sufficiently small. 
  For otherwise, we can find $\delta_i \to 0$ and $t_i \in [\epsilon, T]$ such that 
  \begin{align}
    \norm{\Phi_{t_{i}}^{*}  g(t_{i})- h(t_{i})}{C^{k}(N,h(t_{i}))} \geq \epsilon.   \label{eqn:RF23_6}
  \end{align}
  By (\ref{eqn:RF23_3}) and (\ref{eqn:RF23_4}), we can take smooth limits of $\Phi_{t_{i}}^{*}  g(t_{i})$ and $h(t_i)$ in each fixed normal coordinate chart of $(N, h(T))$. 
  Let $g_{\infty}(t_{\infty})$ and $h_{\infty}(t_{\infty})$ be the limits respectively. 
  Since $\delta_i \to 0$, we can take $\xi_{i} \to 0$, then it follows from (\ref{eqn:RF23_4}) that 
  \begin{align*}
    g_{\infty}(t_{\infty}) - h_{\infty}(t_{\infty}) \equiv 0,
  \end{align*}
  which contradicts the smooth limit of (\ref{eqn:RF23_6}).  Therefore, the contradiction argument establishes the proof of (\ref{eqn:RF23_5}).
  Since $k$ is an arbitrary positive constant less that $\epsilon^{-1}+1$, it is clear that (\ref{eqn:RF22_1}) follows from the combination of (\ref{eqn:RF23_5}) and (\ref{eqn:RF23_4}).  
  The proof of the theorem is complete.   
\end{proof}

Then we proceed to generalize Theorem~\ref{thm:PB03_1} to normalized Ricci flow.  The following lemma is the key for our generalization.

\begin{lemma}[\textbf{Uniform volume continuity}]
   Same conditions as in Theorem~\ref{thm:PB03_1}. Suppose $\epsilon$ is small enough such that $|Rm|_{h} \leq \epsilon^{-1}$. 

   For each small positive number $\xi<\frac{\epsilon}{100}$, there is a $\delta=\delta(\xi, \epsilon, h)$ with the following property:

   If $d_{GH} \left(  (M,g), (N,h) \right)<\delta$, then 
   \begin{align}
     \left| \log   \frac{|M|_{dv_{g(t)}}}{|M|_{dv_{g(0)}}} \right|  < 4m^{2} \epsilon^{-1}  \xi, \quad \forall \; t \in (0, \xi).     \label{eqn:RB01_1}
   \end{align}

   \label{lma:RB01_1}
\end{lemma}

\begin{proof}
  By Theorem~\ref{thm:PB03_1}, it is clear that (\ref{eqn:RB01_1}) holds at $t=0$ and $t=\xi$ if we choose $\delta<\delta_{0}(N,h)$ sufficiently small. 
  The key of (\ref{eqn:RB01_1}) is that the inequality holds for arbitrary $t \in (0,\xi)$.  We shall exploit the almost monotonicity of volume element to prove it.

  On initial manifold $(M, g)$, the condition $Rc \geq -(m-1)\epsilon^{-1}$ implies that $R \geq -m(m-1)\epsilon^{-1}$, which is preserved under the Ricci flow by maximum principle.
  Therefore, the volume element is almost decreasing along the Ricci flow and we have 
  \begin{align}
    e^{-m(m-1)\epsilon^{-1}(\xi-t)} |M|_{dv_{g(\xi)}} \leq  |M|_{dv_{g(t)}} \leq e^{m(m-1)\epsilon^{-1}t} |M|_{dv_{g(0)}}, \quad \forall \; t \in (0, \xi).  \label{eqn:RB01_2}
   \end{align}
   If we choose $\delta$ sufficiently small, then by Colding's volume continuity theorem and Theorem~\ref{thm:PB03_1} respectively, we have
   \begin{align}
     \left| \log \frac{|M|_{dv_{g(0)}}}{|N|_{dv_{h(0)}}} \right|  \leq \psi, \quad   \left| \log \frac{|M|_{dv_{g(\xi)}}}{|N|_{dv_{h(\xi)}}} \right| \leq \psi, \label{eqn:RB01_3}
   \end{align}
   where $\psi=\psi(\delta|\xi,\epsilon,m)$.  
   Note that $\{(N, h(t)), 0 \leq t \leq \xi\}$ is a smooth, compact space-time. 
   Applying maximum principle on Riemannian curvature, we obtain  
   \begin{align*}
   |Rm|_{h(t)}(x) \leq 2 \sup_{z \in N}|Rm|_{h(0)}(z) \leq 2\epsilon^{-1}, \quad \forall \; x \in N, \; t \in (0, \xi]. 
   \end{align*}
   Therefore,  the evolution of volume element implies that
   \begin{align}
     \frac{|N|_{dv_{h(\xi)}}}{|N|_{dv_{h(0)}}} \geq e^{-2m(m-1)\epsilon^{-1}\xi}.  \label{eqn:RB01_4}
   \end{align}
   Plugging (\ref{eqn:RB01_3}) and (\ref{eqn:RB01_4}) into (\ref{eqn:RB01_2}), we derive that
   \begin{align*}
     |M|_{dv_{g(t)}} &\geq e^{-2m(m-1)\epsilon^{-1}(\xi-t)} |M|_{dv_{g(\xi)}}\\
     &\geq  e^{-\psi-2m(m-1)\epsilon^{-1}(\xi-t)} |N|_{dv_{h(\xi)}} \\
     &\geq  e^{-\psi-2m(m-1)\epsilon^{-1}(\xi-t) -2m(m-1)\epsilon^{-1}\xi} |N|_{dv_{h(0)}}\\
     &\geq  e^{-\psi-4m(m-1)\epsilon^{-1} \xi}  |M|_{dv_{g(0)}},
   \end{align*} 
   where we adjust $\psi$ in the last step. 
   Combining the above inequality with (\ref{eqn:RB01_2}) yields that 
   \begin{align*}
     e^{-\psi-4m(m-1)\epsilon^{-1}\xi} \leq \frac{|M|_{dv_{g(t)}}}{|M|_{dv_{g(0)}}} \leq e^{2m(m-1) \epsilon^{-1} \xi}.
   \end{align*}
   Consequently, by choosing $\delta$ sufficiently small,  we obtain
   \begin{align*}
     \left| \log \frac{|M|_{dv_{g(t)}}}{|M|_{dv_{g(0)}}} \right| \leq \psi + 4m(m-1)\epsilon^{-1}\xi \leq 4m^2 \epsilon^{-1} \xi, 
   \end{align*}
   which is exactly (\ref{eqn:RB01_1}). 
\end{proof}

We are ready to show the continuity dependence of the normalized Ricci flow.

\begin{theorem}[\textbf{Continuous dependence in the case of normalized Ricci flow}]
  Suppose that  $\mathcal{N}=\{(N^{m}, \tilde{h}(\tilde{t})), 0 \leq \tilde{t} \leq  \tilde{T}\}$ is a normalized Ricci flow solution initiated from $\tilde{h}(0)=h$, a smooth Riemannian metric on a closed manifold $N^{m}$. 
  For each $\epsilon$ small, there exists $\delta=\delta(\mathcal{N}, \epsilon)$ with the following properties. 

  Suppose $(M^{m}, g)$ is a Riemannian manifold satisfying
  \begin{align}
    Rc>-(m-1)\epsilon^{-1},  \quad d_{GH}\left\{ (M,g), (N,h) \right\}<\delta. 
    \label{eqn:RF24_3}
  \end{align}
  Then the normalized Ricci flow initiated from $(M,g)$ exists on $[0, \tilde{T}]$. 
  Furthermore, there exists a family of diffeomorphisms $\left\{ \tilde{\Phi}_{\tilde{t}}: N \to M, \; \epsilon \leq \tilde{t} \leq \tilde{T} \right\}$ such that
 \begin{align}
   \sup_{\tilde{t} \in [\epsilon, \tilde{T}]} \norm{\tilde{\Phi}_{\tilde{t}}^{*} \tilde{g}(\tilde{t})-\tilde{h}(\tilde{t})}{C^{[\epsilon^{-1}]}(N,\tilde{h}(\tilde{t}))} < \epsilon.   \label{eqn:RF24_4}
 \end{align}
 \label{thm:RF24_2}
\end{theorem}

\begin{proof}
  
Suppose $\left\{ (N, h(t)), 0 \leq t \leq T \right\}$ is the corresponding unnormalized Ricci flow solution on closed manifold $N$. 
Define
\begin{align}
  \lambda_{N}(t) \coloneqq \left( \frac{|N|_{dv_{h(t)}}}{|N|_{dv_{h(0)}}} \right)^{-\frac{2}{m}}, \quad  \tilde{t} \coloneqq \int_{0}^{t} \lambda_{N}(\theta) d\theta, 
  \quad \tilde{h}(\tilde{t}) \coloneqq \lambda_{N}(t) h(t). 
  \label{eqn:RB02_1}
\end{align}
Then direct calculation shows that $|N|_{dv_{\tilde{h}(\tilde{t})}} \equiv |N|_{dv_{h(0)}}$.  The evolving metrics $(N, \tilde{h}(\tilde{t}))$ satisfy the volume normalized Ricci flow equation. 
In particular, we have
\begin{align}
  \tilde{T}=\int_{0}^{T} \lambda_{N}(\theta) d \theta.   \label{eqn:RG05_1}
\end{align}
Similarly, we can define $\tilde{g}, \lambda_{M}$.

Since $\left\{ (N, h(t)), 0 \leq t \leq T \right\}$ is a compact smooth space-time, up to a slight extension, we could assume this flow exists on $[0, T+\epsilon']$. 
By Theorem~\ref{thm:PB03_1} and  Lemma~\ref{lma:RB01_1},  for each small $\xi$,  the Ricci flow started from $(M, g)$ shall exist on $[0, T+\epsilon']$ and satisfy
\begin{align*}
  \sup_{t \in [0,T]} |\lambda_{M}(t)-\lambda_{N}(t)| < \xi
\end{align*}
whenever $d_{GH}\left( (M,g), (N,h) \right)<\delta$ and $Rc_{g} \geq -(m-1) \epsilon^{-1}$.   
As $\delta \to 0$, we have $\lambda_{M}(t) \to \lambda_{N}(t)$ uniformly on $[0, T+\epsilon']$. 
The normalized Ricci flow initiated from $\tilde{g}(0)=g$ shall exist for the time 
\begin{align*}
  \int_{0}^{T+\epsilon'} \lambda_{M}(\theta) d\theta \to \int_{0}^{T+\epsilon'} \lambda_{N}(\theta) d\theta> \int_{0}^{T} \lambda_{N}(\theta) d\theta = \tilde{T}.  
\end{align*}
Replacing $\epsilon$ by $C(\mathcal{N})^{-1} \epsilon$ if necessary, we know that (\ref{eqn:RF24_4}) is a direct consequence of (\ref{eqn:RF22_1}), via the composition of
the parametrization functions $\lambda_{M}$ and $\lambda_{N}$. 
\end{proof}

We conclude this section by the proof of Theorem~\ref{thm:RG05_3}.

\begin{proof}[Proof of Theorem~\ref{thm:RG05_3}:]
  The unnormalized Ricci flow case  and the volume normalized Ricci flow case follow from Theorem~\ref{thm:PB03_1} and Theorem~\ref{thm:RF24_2} respectively. 
  It remains to show the $c$-normalized Ricci flow case. However, if we denote by $\left\{ (M, \tilde{g}(\tilde{t})), 0 \leq t \leq \tilde{T} \right\}$ the $c$-normalized Ricci flow corresponding
  to the unnormalized Ricci flow $\left\{ (M, g(t)), 0 \leq t \leq T \right\}$.  Then 
  \begin{align*}
    \tilde{t}=-\frac{1}{c} \log \left( 1-c t \right),  \quad \tilde{g}(\tilde{t})= \frac{1}{1-ct}g(t). 
  \end{align*}
  Following the argument in the proof of Theorem~\ref{thm:RF24_2}, all the desired estimates can be obtained from Theorem~\ref{thm:PB03_1} through a reparametrization. 
  Therefore,  the proof of Theorem~\ref{thm:RG05_3} is complete. 
\end{proof}

\section{Stability of  immortal Ricci flow solution}
\label{sec:globalstability}

In this section, we shall study the global behavior of the normalized Ricci flow near a given immortal solution initiated from $(N, h)$. 
For simplicity of notation, up to rescaling,  we always assume that  $(N, h)$ satisfies
\begin{align}
  Rc_{h} > -2(m-1).      \label{eqn:RG06_1}
\end{align}
The reason we use $-2(m-1)$ rather than $-(m-1)$ is to reserve rooms to perturb metrics nearby space form metrics of sectional curvature $-1$, whose Ricci curvature is $-(m-1)$. 

\begin{definition}
  A closed  Einstein metric $(N,h_{E})$ is called weakly stable if there is an $\epsilon=\epsilon(N, h_{E})$ with the following properties.

  For every smooth Riemannian metric $h$ satisfying
  \begin{align}
    |N|_{dv_{h}}=|N|_{dv_{h_{E}}}, \quad   \norm{h-h_{E}}{C^{[\epsilon^{-1}]}(h_{E})}<\epsilon,   \label{eqn:RG06_2}
  \end{align}
  the normalized Ricci flow initiated from $(N, h)$ exists immortally and converges(in smooth topology) to an Einstein metric $(N, h_{E}')$. 

  $(N,h_{E})$ is called strictly stable if each $h_{E}'$ is isometric to $h_{E}$ for some sufficiently small $\epsilon$. 
  \label{dfn:RB05_1}
\end{definition}

\begin{theorem}[\textbf{Stability nearby a stable Ricci flow}]
  Suppose $(N, h)$ is the initial metric of an immortal solution of normalized Ricci flow which converges to a weakly stable Einstein manifold $(N, h_{E})$.
  Then there exists a small constant $\epsilon=\epsilon(N,h)$ with the following properties. 

  If $(M,g)$ is a Riemannian manifold satisfying
  \begin{align}
    Rc_{g} > -2(m-1), \quad d_{GH} \left( (M,g), (N,h) \right)<\epsilon, 
    \label{eqn:RF12_1}
  \end{align}
  then the normalized Ricci flow solution initiated from $(M, g)$ exists immortally and converges to an Einstein manifold $(N, h_{E}')$.
  Furthermore, $h_{E}'$ is homothetic to $h_{E}$ if $h_{E}$ is strictly stable. 
  \label{thm:RF12_2}
\end{theorem}

\begin{proof}
  Since the normalized Ricci flow initiated from $(N,h)$ exists immortally and converges to $(N, h_{E})$, for each small $\epsilon$,  we may choose $T=T(\epsilon)$ such that 
  \begin{align}
    \norm{h(T)-h_{E}}{C^{[\epsilon^{-1}]}(h(T))}<\epsilon.  \label{eqn:RF13_4} 
  \end{align}
  By the continuous dependence of the normalized Ricci flow on the initial data(cf. Theorem~\ref{thm:RF24_2}), we can choose $\delta=\delta(\epsilon)$ such that if $(M,g)$ satisfies
  \begin{align*}
    d_{GH}\left( (M,g),(N,h) \right)<\delta, \quad Rc_{g} > -2(m-1),
  \end{align*}
  then the normalized Ricci flow initiated from $g(0)=g$ exists on time interval $[0,T]$ and
  \begin{align}
    \norm{h(T)-\Phi^{*}g(T)}{C^{[\epsilon^{-1}]}(h(T))}<\epsilon   \label{eqn:RF13_5}
  \end{align}
  for some diffeomorphism $\Phi:N \to M$.  Combining (\ref{eqn:RF13_4}) and (\ref{eqn:RF13_5}), we know that $\Phi^{*}g(T)$ locates in a $2\epsilon$-neighborhood of $h_{E}$.
  Since $h_{E}$ is weakly stable and $\epsilon$ is sufficiently small, it follows from Definition~\ref{dfn:RB05_1} that the normalized Ricci flow initiated from $\Phi^{*}g(T)$ exists immortally and converges to some Einstein metric $h_{E}'$ nearby $h_{E}$. 
  If $h_{E}$ is strictly stable, then we have $h_{E}'=c h_{E}$ for some positive constant $c$ nearby $1$ such that $|M|_{dv_{ch_{E}}}=|M|_{dv_{h}}$.   
  Concatenating the flows in an obvious way, we obtain the convergence of the normalized flow initiated from $(M,g)$. 
\end{proof}

In Theorem~\ref{thm:RF12_2}, if $(N,h)$ itself is a weakly stable Einstein manifold, then we immediately have the following corollary. 

\begin{corollary}[\textbf{Stability nearby a stable Einstein manifold}]
  Each weakly stable Einstein manifold $(N,h)$ has a Gromov-Hausdorff neighborhood such that every Riemannian metric in this neighborhood with $Rc > -2(m-1)$ 
  can be smoothly deformed to an Einstein metric via the normalized Ricci flow. Furthermore, the limit metric is homothetic to $(N,h)$ if $(N,h)$ is strictly stable.
  \label{cly:RF13_2}
\end{corollary}

There are many cases where we can apply Theorem~\ref{thm:RF12_2} and Corollary~\ref{cly:RF13_2}.   
For example, by the work of Guenther-Isenberg-Knopf~\cite{GIK}, each compact flat manifold is weakly stable.
Therefore,  starting from a Riemannian manifold nearby a closed flat manifold in the Gromov-Hausdorff topology  with $Rc > -2(m-1)$, the normalized Ricci flow
exists immortally and converges to a flat manifold in the smooth topology.   Note that the limit flat manifold may not be homothetic to the original model flat manifold,
as the flat manifold is only weakly stable.   On the other hand, each compact space form $(N,h)$ of curvature $\pm 1$ is strictly stable.
Therefore,  given any nearby metric $(M,g)$ in the Gromov-Hausdorff topology satisfying $Rc > -2(m-1)$ and $|M|_{dv_{g}}=|N|_{dv_{h}}$, the Ricci flow initiated from $(M,g)$ exists
forever and converges to $(N, h)$. For more results concerning the convergence of Ricci flow nearby closed Einstein manifold, see also \v{S}e\v{s}um~\cite{Sesum}, R.G. Ye~\cite{RYe}, etc.
A comprehensive review for the stability of Ricci flow can be found in Chapter 35 of the book~\cite{CCGGIIKLLN}. 
The case when $K=1$ is even more interesting, as we can combine Theorem~\ref{thm:RF12_2} with the great development in the last decade by B\"ohm-Wilking~\cite{BoWil}, Brendle-Schoen~\cite{BrSch} and Brendle~\cite{SB}. 
In particular, we have the following theorem.

\begin{theorem}[\textbf{Stability nearby a $PIC_{1}$-metric}]
  Suppose $(N,h)$ is a closed Riemannian manifold with $PIC_1$ curvature condition.
  Namely, $(N,h) \times (\R, g_{E})$ has positive isotropic curvature. 
  Then there exists a small constant $\epsilon=\epsilon(N,h)$ with the following properties. 

  If $(M,g)$ is a Riemannian manifold satisfying (\ref{eqn:RF12_1}),  
  then the normalized Ricci flow initiated from $(M,g)$ exists immortally and converges to a space form metric with positive curvature.  
  \label{thm:RF13_1}
\end{theorem}

\begin{proof}
  By the work of Brendle~\cite{SB}, the normalized Ricci flow initiated from $(N,h)$ has immortal existence and convergence.
  For $\epsilon$ sufficiently small, we know the normalized Ricci flow initiated from $g$ exists for at least time $T=1$ and $\tilde{g}(1)$ is very close to $\tilde{h}(1)$ in $C^{5}$-topology,
  via Theorem~\ref{thm:RF24_2}.  As $PIC_1$ condition is an open condition  preserved by the Ricci flow, we know $\tilde{h}(1)$ and $\tilde{g}(1)$ also satisfy $PIC_1$-condition.
  Applying~\cite{SB} again on $\tilde{g}(1)$, we obtain the desired immortal existence and convergence.  
\end{proof}

Theorem~\ref{thm:RF12_2} can also be applied to study the stability of the K\"ahler Ricci flows. For example, we have the following theorem. 

\begin{theorem}[\textbf{Stability nearby a K\"ahler Ricci flow of vanishing $c_{1}$}] 
  Suppose $(N^{n},h,J)$ is a closed K\"ahler manifold of complex dimension $n$ satisfying $c_1(N, J)=0$.
  There exists a small constant $\epsilon=\epsilon(N,h,J)$ with the following properties.

  If $(M,g)$ is a Riemannian manifold satisfying (\ref{eqn:RF12_1}) with $m=2n$.  
  Then the normalized Ricci flow initiated from $(M,g)$ exists immortally and converges to a Calabi-Yau metric $(N^{n}, h', J')$. 

  \label{thm:RF13_3}
\end{theorem}

\begin{proof}
  In light of the work of Cao~\cite{HDC}, the normalized Ricci flow initiated from $(N,h,J)$ converges to a Calabi-Yau metric $(N,h_{KE}, J)$. 
  By Bogomolov-Tian-Todorov theorem(cf.~\cite{Bo78},~\cite{T86} and~\cite{To89}), the complex structure $J$ is integrable. 
  Consequently, the Einstein manifold $(N, h_{KE})$ is weakly stable by Theorem 1.7 of~\cite{DWW}. 
  Therefore, we conclude the immortal existence and convergence by Theorem~\ref{thm:RF12_2}. 
\end{proof}

In all the above applications of Theorem~\ref{thm:RF12_2}, we \textit{a priori} assume the existence of a normalized Ricci flow solution with immortal existence and convergence.
It is desirable to drop such conditions and obtain existence and convergence only by initial geometric conditions like curvature, volume, etc.
Such expectation is partially achieved by the following theorem. 

\begin{theorem}[\textbf{$=$Theorem~\ref{thm:RG26_3}}]
  For each $m \geq 3$, there is a constant $\delta_{0}=\delta_{0}(m)$ with the following property.

  Suppose $(M^{m}, g)$ is a Riemannian manifold satisfying 
  \begin{align}
    \begin{cases}
    &Rc \geq (m-1)(1-\delta), \\
    &|M|_{dv_{g}}=(m+1)\omega_{m+1}, 
    \end{cases}
    \label{eqn:RG06_4}  
  \end{align}
  where $\omega_{m+1}$ is the volume of unit ball in $\R^{m+1}$, and $\delta \in (0, \delta_{0})$. 
  Then the normalized Ricci flow initiated from $(M^{m}, g)$ exists immortally and converges to a round metric $g_{\infty}$ exponentially fast.
  The limit metric $(M, g_{\infty})$ has constant sectional curvature $1$. 

  Furthermore, for each pair of points $x,y \in M$ satisfying $d_{g(0)}(x,y) \leq 1$, the following distance bi-H\"older estimate holds:
 \begin{align}
    e^{-\psi} d_{g(0)}^{1+\psi}(x,y) \leq d_{\infty}(x,y) \leq e^{\psi} d_{g(0)}^{1-\psi}(x,y),    \label{eqn:RF13_6} 
 \end{align}
 where $d_{g(0)}=d_{g}, \; d_{\infty}=d_{g_{\infty}}$, and $\psi=\psi(\delta|m)$.    
  \label{thm:RB04_1}
\end{theorem}

\begin{proof}
 Let $g_{round}$ be the space form metric on $S^{m}$ with constant sectional curvature $1$. Clearly, $|S^{m}|_{dv_{g_{round}}}=(m+1)\omega_{m+1}$. 
 By Colding's fundamental results~\cite{Co1},~\cite{Co2}, we have
 \begin{align}
  d_{GH}( (M,g), (S^{m}, g_{round}))<\psi(\delta|m).
  \label{eqn:RF13_7}
 \end{align}
  Note that $(S^{m}, g_{round})$ is a fixed point of the normalized Ricci flow. Then the immortal existence and convergence
  of the normalized Ricci flow initiated from $g$ is a consequence of Corollary~\ref{cly:RF13_2}. 
  The exponential convergence is due to Hamilton~\cite{Ha82},~\cite{Ha86}. Note that the flow limit metric $g_{\infty}$ is isometric to $g_{round}$. 

  We proceed to prove (\ref{eqn:RF13_6}).  Suppose $\delta$ is small enough, by the Bishop-Gromov volume comparison and the L\'{e}vy-Gromov isoperimetric constant estimate(cf.~\cite{Gmv86}),
  one can find a fixed scale $r_0 \in (0, 1)$ such that
  $\boldsymbol{\nu}(M,g, r_0^{2})$ is almost non-negative such that Corollary~\ref{cly:CK21_1} can be applied. It follows that 
  \begin{align}
    e^{-\psi} d_{g(0)}^{1+\psi}(x,y) \leq d_{g(\epsilon^{2} r_0^{2})}(x,y) \leq e^{\psi} d_{g(0)}^{1-\psi}(x,y), \quad \forall\; x, y \in M, \; d_{g(0)}(x,y)<1. 
  \label{eqn:RI27_1}  
  \end{align}
  Here $\psi=\psi(\delta|m)$ and $\epsilon, r_0$ are independent of $\delta$.
  In light of Theorem~\ref{thm:RF24_2}, the flow $g(t)$ converges to the static round sphere normalized Ricci flow on the time period $[(\epsilon r_0)^{2}, 1]$, as $\delta \to 0$.  
  Consequently, the following estimates hold:
  \begin{align}
    e^{-\psi} d_{g(\epsilon^{2} r_0^{2})}(x,y) \leq d_{g(1)}(x,y) \leq e^{\psi} d_{g(\epsilon^{2} r_0^{2})}(x,y),  \quad \forall\; x, y \in M.  
    \label{eqn:RI27_2}  
  \end{align}
  By exponential convergence of $g(t)$ to $g(\infty)$ and the fact that $g(1)$ is already very close to $g_{round}$ in the smooth topology, we have 
  \begin{align}
    e^{-\psi}  d_{g(1)}(x,y) \leq d_{\infty}(x,y) \leq e^{\psi} d_{g(1)}(x,y),  \quad \forall\; x, y \in M.  
    \label{eqn:RI27_3}  
  \end{align}
  Combining (\ref{eqn:RI27_1}), (\ref{eqn:RI27_2}) and (\ref{eqn:RI27_3}), up to modifying $\psi$ if necessary, we obtain (\ref{eqn:RF13_6}). 
\end{proof}

In Riemannian geometry, the round metrics on unit sphere $S^{m}$ are the most well-known classical metrics of positive curvature. 
Their counter parts in K\"ahler geometry are the Fubini-Study metrics on $\CP^{n}$, whose sectional curvatures vary from $1$ to $4$.
Since the Ricci curvature of the Fubini-Study metrics is $2(n+1)$, 
the following theorem can be regarded as the K\"ahler version of Theorem~\ref{thm:RB04_1}.

\begin{theorem}[\textbf{$=$Theorem~\ref{thm:RG26_5}} ]
  There is a constant $\delta_{0}=\delta_{0}(n)$ with the following properties.

Suppose $(M^{n},g,J)$ is a Fano manifold whose metric form is proportional to  $c_1(M,J)$ and satisfies
 \begin{align}
    \begin{cases}
    &Rc \geq 2(n+1)(1-\delta), \\
    &|M|_{dv_{g}}=\Omega_n,
    \end{cases}
    \label{eqn:RE15_2}  
  \end{align}
where $\Omega_{n}$ is the volume of $(\CP^{n}, g_{FS})$, and $\delta \in (0, \delta_{0})$.  
Then the normalized Ricci flow initiated from $(M, g)$ exists immortally and converges smoothly(without further diffeomorphisms) to a metric $g_{\infty}$  such that
$(M, g_{\infty}, J)$ is biholomorphic-isometric to $(\CP^{n}, g_{FS}, J_{FS})$.
Furthermore, for each pair of points $x,y \in M$ satisfying $d_{g(0)}(x,y) \leq 1$, the distance bi-H\"older estimate (\ref{eqn:RF13_6}) holds. 
\label{thm:RE14_15}
\end{theorem}

\begin{proof}
  By the result of Zhang~\cite{KZhang}, there exists a dimensional constant $\bar{\delta}(n)$ such that if $\delta <\bar{\delta}(n)$, 
  then $(M,J)$ is biholomorphic to $\CP^{n}$, which implies that $c_{1}(M,J)=c_1(\CP^{n},J_{FS})$.  According to our assumption, we have $[\omega]=2\pi \lambda  c_{1}(M,J)$,
  which together with (\ref{eqn:RE15_2}) forces that $\lambda=2(n+1)$.  Therefore, the normalized Ricci flow
  \begin{align}
    \frac{d}{dt} g_{ij}=-R_{ij} + 2(n+1) g_{ij}   \label{eqn:RE17_1}
  \end{align}
  preserves the K\"ahler class $2(n+1) \cdot 2\pi c_1(M,J)$. As volume is determined by K\"ahler class, the volume along the above flow is fixed. 
  Consequently, the flow (\ref{eqn:RE17_1}) is both normalized flow and $2(n+1)$-normalized flow.
  Therefore, the immortal existence of (\ref{eqn:RE17_1}) is guaranteed by the fundamental result of Cao~\cite{HDC}. 
  It suffices to show the convergence and distance bi-H\"older equivalence of the flow (\ref{eqn:RE17_1}).

  A modification of Perelman's argument by Jiang(cf.\cite{Jiang}) implies that 
  \begin{align}
    \diam(M,g(t))+ \norm{R}{C^{0}(M, g(t))}+ \norm{\nabla \dot{\varphi}}{C^{0}(M, g(t))}<C(n),  \quad \forall t \in [0.5,\infty),  \label{eqn:RE17_2}
  \end{align}
  where $\dot{\varphi}$ is the Ricci potential function.  Along the normalized flow (\ref{eqn:RE17_1}), the scalar curvature $R$ satisfies the evolution equation
  \begin{align*}
    \frac{\partial}{\partial t} \left( R-4n(n+1) \right) =\frac{1}{2} \Delta R +  \left| R_{ij}-2(n+1)g_{ij} \right|^2 + (R-4n(n+1)). 
  \end{align*}
  Then maximum principle implies that
  \begin{align}
    R_{min}(t)-4n(n+1)\geq \left( R_{min}(0)-4n(n+1) \right) e^{t} \geq -4n(n+1) e^{t} \delta. \label{eqn:RG07_1}
  \end{align}
  Note that $|M|_{dv_{g(t)}} \equiv \Omega_{n}$ and $\int_{M} (R-4n(n+1))dv \equiv 0$ along the flow (\ref{eqn:RE17_1}),  it follows from the above inequality that 
 \begin{align} 
   &\quad\int_{0}^{1} \int_{M} |R-4n(n+1)| dv dt \notag\\
   &\leq \int_{0}^{1} \int_{M} \left\{ |R-4n(n+1)+4n(n+1)\delta e^{t}| + 4n(n+1)\delta e^{t}  \right\} dvdt \notag \\
   &=\int_{0}^{1} \int_{M} \left\{ \left\{R-4n(n+1)+4n(n+1)\delta e^{t}\right\} + 4n(n+1)\delta e^{t}  \right\} dvdt \notag \\
   &=8n(n+1)\delta \int_{0}^{1} \int_{M} e^{t} dvdt \notag\\
   &=8n(n+1) \Omega_{n} (e-1) \cdot \delta.    \label{eqn:RE17_3}  
  \end{align}
  In light of (\ref{eqn:RE17_2}) and (\ref{eqn:RE17_3}), the K\"ahler condition implies 
  \begin{align}
    &\quad\int_{0.5}^{1} \int_{M} |R_{ij}-2(n+1)g_{ij}|^2 dv dt=\frac{1}{2} \int_{0.5}^{1} \int_{M} |R-4n(n+1)|^2 dv dt \notag\\
    &\leq C(n) \int_{0.5}^{1} \int_{M} |R-4n(n+1)|dv dt \leq C(n) \delta.   \label{eqn:RF05_1} 
  \end{align}

  \begin{claim}
    There exists $\delta_{a}=\delta_{a}(n)$ with the following properties. 
    For each $(M,g)$ satisfies (\ref{eqn:RE15_2}) for some $\delta \in (0, \delta_{a})$, we have
  \begin{align}
     d_{GH} \left( (M, g(1)), (\CP^{n}, g_{FS}) \right)<\psi(\delta|n).
  \label{eqn:RI27_5}
  \end{align}
    \label{clm:RE17_4}
  \end{claim}

  We argue by contradiction. 
  For otherwise, there exist $\delta_i \to 0$ and flows starting from $g_i(0)=g_i$ with $Rc \geq 2(n+1)(1-\delta_i) g_{i}$ and $|M|_{dv_{g_i}}=\Omega_{n}$
  such that 
  \begin{align}
    d_{GH}\left( (M_{i}, g_{i}), (\CP^{n}, g_{FS})  \right) \geq \xi>0. 
  \label{eqn:RI27_6}
  \end{align}
  In light of estimate (\ref{eqn:RF05_1}) and (\ref{eqn:RE17_2}), we obtain a sequence of polarized K\"ahler Ricci flows where Theorem 1.2 of Chen-Wang~\cite{CW17B} applies. 
  By taking subsequence if necessary, we have
  \begin{align}
    (M_i, g_i(1))  \longright{\hat{C}^{\infty}-Cheeger-Gromov} (M_{\infty}, g_{\infty}(1)),  \label{eqn:RI27_7}
  \end{align}
  where the limit is a K\"ahler Einstein conifold in the sense of Chen-Wang(cf. Definition 1.2 of~\cite{CW17A}). 
  The definition of $\hat{C}^{\infty}$-Cheeger-Gromov convergence can be found on page 2 of~\cite{CW17A}. 
  In view of the partial-$C^{0}$-estimate(cf. Theorem 1.10 of~\cite{CW17B}), we can construct a test-configuration which achieves a degeneration from $(M_i, g_i(1), J_i)$ to a $Q$-Fano K\"ahler-Einstein space
  $(M_{\infty}, g_{\infty}(1), J_{\infty})$.   Recall that each $M_i$ is already biholomorphic to $\CP^n$ and $g_i(0)$ corresponds to a metric form $\omega_{i}(0)$ in the pluri-anti-canonical class $4(n+1)\pi c_1(M_i, J_i)$.  
  By the stability theorem of Chen-Donaldson-Sun(cf.~\cite{CDS1},~\cite{CDS2} and~\cite{CDS3}), 
  we know that $(M_{\infty}, g_{\infty}(1),  J_{\infty})$ must be biholomorphic-isometric to the standard projective space $(\CP^{n}, g_{FS}, J_{FS})$.
  In particular, the limit space $(M_{\infty}, g_{\infty}(1),  J_{\infty})$ is smooth. Therefore, we can apply the backward pseudo-locality theorem(cf. Theorem 4.7 of~\cite{CW17B}) to improve the convergence (\ref{eqn:RI27_7}) as: 
  \begin{align}
    (M_i, g_i(1))  \longright{C^{\infty}-Cheeger-Gromov} (\CP^{n}, g_{FS}).   \label{eqn:RF09_4}
  \end{align}
  In particular, $d_{GH} \left( (M_{i}, g_{i}(1)), (\CP^{n}, g_{FS}) \right)$ could be arbitrarily small if we choose $i$ sufficiently large. 
  This contradicts our choice of $g_i$ in (\ref{eqn:RI27_6}).   The proof of Claim~\ref{clm:RE17_4} is complete.

  \begin{claim}
    There exists $\delta_{b}=\delta_{b}(n, \epsilon_{0})$ with the following properties. 
    For each $(M,g)$ satisfies (\ref{eqn:RE15_2}) for some $\delta \in (0, \delta_{b})$, 
    we have 
    \begin{align}
      2\mathbf{vr}_{\epsilon_0}(M,g) > \mathbf{vr}_{\epsilon_0}(\CP^{n}, g_{FS}) =:r_0, \label{eqn:RI27_8} 
    \end{align}
    where $\mathbf{vr}_{\epsilon_0}$ for some sufficiently small $\epsilon_0=\epsilon_0(2n)$ is defined in Definition~\ref{dfn:RH26_3}. 
    \label{clm:RE17_5}
  \end{claim}

  For simplicity of notation, we denote $\mathbf{vr}_{\epsilon_0}$ by $\mathbf{vr}$. 
  We argue by contradiction. For otherwise, we can find manifolds $(M_i, g_i)$ violating the statement with $\delta_i \to 0$.
  Therefore,  we can find points $x_i \in M_{i}$ such that  $\mathbf{vr}(x_i, g_i)=\mathbf{vr}(M_i, g_i) <\frac{1}{2} r_0$. 
  By Cheeger-Colding theory, we can assume that
  \begin{align}
    (M_i, g_i) \longright{Gromov-Hausdorff}  (\bar{M}, \bar{d}),    \label{eqn:RF09_5}
  \end{align}
  where $(\bar{M}, \bar{d})$ is a length space with a regular-singular decomposition $\bar{M}=\mathcal{R}(\bar{M}) \cup \mathcal{S}(\bar{M})$.
  Let $\bar{x}$ be the limit point of $x_i$, then volume continuity implies that 
  \begin{align}
    \mathbf{vr}(\bar{x}, \bar{d}) = \lim_{i \to \infty} \mathbf{vr}(x_i, g_i) \leq 0.5 r_0.  \label{eqn:RF11_1} 
  \end{align}
  We decompose each space, $\bar{M}$ or $M_i$, by setting $\mathcal{R}_{r}=\left\{ x|\mathbf{vr}(x) \geq r \right\}$ and $\mathcal{S}_{r}=\left\{ x|\mathbf{vr}(x)<r \right\}$. 
  Then we have
  \begin{align*}
    \mathcal{R}=\cup_{r>0} \mathcal{R}_{r}, \quad \mathcal{S}=\cap_{r>0} \mathcal{S}_{r}.
  \end{align*}
  For each $i$, as $M_i$ is a smooth manifold, it is clear that $\mathcal{S}(M_i)=\emptyset$.
  However, $\mathcal{S}(\bar{M}) \neq \emptyset$ in general. For simplicity of notation, we regard the $m$-dimensional Hausdorff measure as the default measure on 
  the limit space $\bar{M}$, and we regard $dv_{g_{i}}=dv_{g_i(0)}$ as the default measure on $M_{i}$. 
  Then it is known(cf.~\cite{CC}) that $\mathcal{S}(\bar{M})$ is a measure-zero set.  Consequently,  for each $\epsilon$ small, we can find $\rho$ such that
  \begin{align*}
    |\mathcal{S}_{\rho}(\bar{M})|=|\cap_{r \geq \rho} \mathcal{S}_{r}(\bar{M})|<\epsilon. 
  \end{align*}
  By volume continuity, we have
  \begin{align}
    |\mathcal{S}_{\rho}(M_i)|<\epsilon      \label{eqn:RF11_2}
  \end{align}
  for large $i$.

  Fix $\bar{y} \in \mathcal{R}(\bar{M})$. We can assume $y_{i} \in M_{i}$ such that $y_i \to \bar{y}$ along the convergence (\ref{eqn:RF09_5}).  

  Thanks to (\ref{eqn:RE17_3}), the sequence $(M_i, g_i)$ is an almost Einstein sequence in the sense of Tian-Wang(cf. Definition 1 of~\cite{TiWa}). 
  In light of (\ref{eqn:RE17_2}), (\ref{eqn:RE17_3}) and (\ref{eqn:RF09_4}), we can apply the compactness
  theorem of Chen-Wang(cf. Theorem 1.6 of~\cite{CW17B}) to obtain that
  \begin{align}
    \left\{ \left( M_i, g_i(t) \right), 0<t \leq 2 \right\} \longright{C^{\infty}-Cheeger-Gromov} \left\{ \left( \CP^{n}, g_{FS} \right), 0<t \leq 2 \right\}.
    \label{eqn:RF09_6}  
  \end{align}
  Combining the above convergence with the distance distortion estimate of Hamilton-Perelman(cf. Lemma 8.3(b) of~\cite{Pe1}) implies that
  \begin{align*}
    B_{g_i(1)}(y_i,r) \backslash \mathcal{S}_{\rho}(M_i)  \subset B_{g_{i}(0)}(y_i,(1+\psi) r)=B(y_i,(1+\psi) r)
  \end{align*}
  where $\psi=\psi(\delta_i|n, \mathbf{vr}(\bar{y},\bar{d}),\rho)$.   In view of (\ref{eqn:RG07_1}), the volume element is almost decreasing along the normalized Ricci flow (\ref{eqn:RE17_1}).
  Recalling that $g_i=g_{i}(0)$ is the default metric and $dv_{g_{i}}=dv_{g_i(0)}$ is the default measure,  we consequently have  
 \begin{align*}
   |B_{g_i(1)}(y_i,r)|_{dv_{g_i(1)}}&=|B_{g_i(1)}(y_i,r) \backslash \mathcal{S}_{\rho}(M_i)|_{dv_{g_i(1)}}+|B_{g_i(1)}(y_i,r) \cap \mathcal{S}_{\rho}(M_i)|_{dv_{g_i(1)}}\\
   &\leq |B_{g_{i}(0)}(y_{i}, (1+\psi)r)|_{dv_{g_i(1)}} + |B_{g_i(1)}(y_i,r) \cap \mathcal{S}_{\rho}(M_i)|_{dv_{g_i(1)}}\\
   &\leq e^{\psi}\left\{ |B_{g_{i}(0)}(y_{i}, (1+\psi)r)|_{dv_{g_i(0)}} + |B_{g_i(1)}(y_i,r) \cap \mathcal{S}_{\rho}(M_i)|_{dv_{g_i(0)}} \right\}\\
   &=e^{\psi}\left\{ |B(y_{i}, (1+\psi)r)| + |B_{g_i(1)}(y_i,r) \cap \mathcal{S}_{\rho}(M_i)| \right\}\\
   &\leq e^{\psi} \left\{ |B(y_i, (1+\psi) r)| + |\mathcal{S}_{\rho}(M_i)| \right\}. 
 \end{align*}
 It follows that
 \begin{align*}
   |B(y_i, (1+\psi) r)| \geq e^{-\psi}|B_{g_i(1)}(y_i,r)|_{dv_{g_i(1)}} - |\mathcal{S}_{\rho}(M_i)|.
 \end{align*}
 By (\ref{eqn:RF09_6}), for each small fixed $t$, the limit of $(M_i, g_i(t))$ is $(\CP^{n}, g_{FS})$. 
 Plugging (\ref{eqn:RF11_2}) into the above inequality and letting $i \to \infty$, we obtain
 \begin{align*}
   |B(\bar{y},r)| \geq |B_{g_{FS}}(\bar{y},r)|_{dv_{g_{FS}}} -\epsilon.
 \end{align*}
 By the arbitrary choice of $\epsilon$, we have $ |B(\bar{y},r)| \geq |B_{g_{FS}}(\bar{y},r)|_{dv_{g_{FS}}}$  for each $r \in (0, r_0]$.
 In particular, we know $\mathbf{vr}(\bar{y}, \bar{d}) \geq r_0$. By the arbitrary choice of $\bar{y} \in \mathcal{R}(\bar{M})$, we have proved that
 \begin{align*}
   \inf_{\bar{y} \in \mathcal{R}(\bar{M})} \mathbf{vr}(\bar{y}, \bar{d}) \geq r_0.
 \end{align*}
 Since $\mathcal{R}(\bar{M})$ is a dense subset of $\bar{M}$, we can choose $\bar{y} \in \mathcal{R}(\bar{M})$ such that $\bar{y} \to \bar{x}$. 
 Then the volume continuity implies that $\mathbf{vr}(\bar{x}) \geq r_0$. This contradicts (\ref{eqn:RF11_1}). 
 The proof of Claim~\ref{clm:RE17_5} is complete.\\

 Thanks to Claim~\ref{clm:RE17_5}, we obtain
 \begin{align*}
      \mathbf{vr}(M,g) > 0.5 r_0
 \end{align*}
 when $(M,g)$ satisfies (\ref{eqn:RE15_2}) for sufficiently small $\delta$.
 In light of Theorem~\ref{thm:RH25_1},  the above inequality implies uniform lower bound of entropy radius (cf. Definition~\ref{dfn:RH26_3} up to modifying $\epsilon_0$).
 Following the same route as in the proof of Theorem~\ref{thm:RB04_1}, we can apply Corollary~\ref{cly:CK21_1} to derive (\ref{eqn:RI27_1}) and (\ref{eqn:RI27_2}).
 Combining them together, we obtain
  \begin{align}
    e^{-\psi} d_{g(0)}^{1+\psi}(x,y) \leq d_{g(1)}(x,y) \leq e^{\psi} d_{g(0)}^{1-\psi}(x,y),  \quad \forall\; x, y \in M \;\textrm{satisfying} \; d_{g(0)}(x,y)<1.
  \label{eqn:RI27_4}  
  \end{align}
  The Bishop-Gromov volume comparison implies that the diameter of $(M,g)$ is uniformly bounded. Therefore, it follows from (\ref{eqn:RI27_4}) that
  \begin{align*}
    d_{GH}\left\{ (M, g(0)), (M, g(1)) \right\}<\psi(\delta|n),
  \end{align*}
  which combined with (\ref{eqn:RI27_5}) yields 
  \begin{align*}
    d_{GH}\left\{ (M, g(0)), (\CP^{n}, g_{FS}) \right\}<\psi(\delta|n).
  \end{align*}
 Therefore, $(M,g)$ converges in Gromov-Hausdorff topology to $(\CP^{n}, g_{FS})$ if $\delta \to 0$.

 In (\ref{eqn:RF09_6}), we can restrict our attention on the time slice $t=1$ and also consider the degeneration of $J$.
 As we already know that each $(M_{i}, J_{i})$ is biholomorphic to $(\CP^{n}, J_{FS})$, we have 
  \begin{align}
    \left( \CP^{n}, \bar{g}_{i}, J_{FS} \right)
   \cong \left( M_i, g_i(1),J_{i} \right) \longright{C^{\infty}-Cheeger-Gromov} \left( \bar{M}, \bar{g}, \bar{J} \right) \cong \left( \CP^{n}, g_{FS}, J_{FS} \right),
   \label{eqn:RI28_2} 
  \end{align}
  where ``$\cong$" means being biholomorphic-isometric.  For example, the last equality in (\ref{eqn:RI28_2}) means that there exists a diffeomorphism map $\Phi:\CP^{n} \mapsto \bar{M}$
  such that
  \begin{align*}
    \Phi^{*}(\bar{g})=g_{FS}, \quad \Phi^{*}(\bar{J})=J_{FS}.
  \end{align*}
  Following the definition of Cheeger-Gromov convergence, thanks to (\ref{eqn:RI28_2}), we can find diffeomorphims $\varphi_{i}: \CP^{n} \to \CP^{n}$ such that
  \begin{align}
    h_i \coloneqq \varphi_{i}^{*}(\bar{g}_{i}) \longright{C^{\infty}} g_{FS}. \label{eqn:RI28_1} 
  \end{align}
  It follows that
  \begin{align*}
    \sum_{k=0}^{5} \norm{\nabla^{k}Rm_{h_{i}} }{h_{i}} \leq C_{m}, 
    \quad 
    \sum_{k=0}^{5} \norm{\nabla^{k}\left\{Rc_{h_{i}}-2(n+1)h_{i}\right\}}{h_{i}} \to 0. 
  \end{align*}
  By the invariant property of curvature derivatives under diffeomorphism actions, we have
  \begin{align*}
    \sum_{k=0}^{5} \norm{\nabla^{k}Rm_{\bar{g}_{i}} }{\bar{g}_{i}} \leq C_{m}, 
    \quad 
    \sum_{k=0}^{5} \norm{\nabla^{k}\left\{Rc_{\bar{g}_{i}}-2(n+1)\bar{g}_{i}\right\}}{\bar{g}_{i}} \to 0. 
  \end{align*}
  Recall the definition of $\bar{g}_{i}$ in (\ref{eqn:RI28_2}), the above equation actually means that
  \begin{align}
    \sum_{k=0}^{5} \norm{\nabla^{k}Rm_{g(1)} }{g(1)} \leq C_{m}, 
    \quad 
    \sum_{k=0}^{5} \norm{\nabla^{k}\left\{Rc_{g(1)}-2(n+1)g(1)\right\}}{g(1)}<\psi(\delta|n), 
    \label{eqn:RI29_1}
  \end{align}
  whenever $g=g(0)$ satisfies (\ref{eqn:RE15_2}) for some $\delta \in (0, \delta_{c}(n))$. 
  Recall that the normalized Ricci flow solution  (\ref{eqn:RE17_1}) initiated from $g(1)$ stays in the K\"ahler class $4(n+1)\pi c_1(\CP^{n}, J_{FS})$.
  Then the condition (\ref{eqn:RI29_1}) guarantees that the flow converges exponentially fast to a metric $g_{\infty}=\Phi^{*}(g_{FS})$ for some $\Phi \in Aut(\CP^{n}, J_{FS})$.
  In particular, we have
  \begin{align}
    \{1-\psi(\delta|n)\} g_{\infty}\leq g(1) \leq \{1+\psi(\delta|n)\} g_{\infty}.   \label{eqn:RI29_4}
  \end{align}
  Recall the definitions of gaps at the beginning of this proof, in Claim~\ref{clm:RE17_4}, Claim~\ref{clm:RE17_5} and the one after (\ref{eqn:RI29_1}).  
  Now we set $\delta_{0}(n)=\min\left\{\bar{\delta}, \delta_{a}, \delta_{b}, \delta_{c} \right\}$. 
  Combining (\ref{eqn:RI29_4}) with (\ref{eqn:RI27_4}), we obtain the bi-H\"{o}lder distance estimate (\ref{eqn:RF13_6}) for each pair of points $x,y \in M$ satisfying
  $d_{g}(x,y)\leq 1$. 
 \end{proof}

 The following corollary is the direct consequence of Theorem~\ref{thm:RE14_15}, which seems to be unknown without exploiting the normalized Ricci flow. 

\begin{corollary}[\textbf{Equivalence of the volume closeness and the Gromov-Hausdorff closeness}]
  Suppose $(M^{n}, g, J)$ is a Fano manifold whose metric form class is proportional to $c_1(M, J)$ and $Rc \geq 2(n+1) g$.  
  Then $d_{GH}\left\{ (M,g), (\CP^{n}, g_{FS}) \right\}$ is close to $0$ if and only if  volume of $(M, g)$ is close to $\Omega_{n}$. 
  \label{cly:RF09_1}
\end{corollary}

\begin{remark}
  Suppose $(M,J)$ is a Fano manifold with a K\"ahler Einstein metric $g_{KE}$ in the anti-canonical class $2\pi c_1(M,J)$. 
  If $g$ is an arbitrary metric in the K\"ahler class $2\pi c_1(M,J)$, then  
  \begin{align*}
    0=\int_{M} (R-2n) dv= \int_{M} \Delta f dv
  \end{align*}
  where $f$ is the Ricci potential function. Therefore, if $R \geq 2n$, then we must have $R \equiv 2n$.
  Along the normalized Ricci flow initiated from $g$, the maximum principle on $R$ implies that $R \geq 2n$ is preserved.
  Then we know $R \equiv 2n$ on the whole flow solution.  Evolution equation of scalar curvature along the flow 
  then implies $Rc-g \equiv 0$. Consequently, by the uniqueness of K\"ahler Einstein metrics, we obtain $g=\Phi^{*}(g_{KE})$
  for some $\Phi \in Aut(M,J)$. Since $Rc \geq 1$ implies $R \geq 2n$, we have the rigidity property:
  \begin{align*}
    Rc \geq 1 \;  \Leftrightarrow  d_{GH}\left\{ (M,g), (M,g_{KE}) \right\}=0. 
  \end{align*}
  Following the proof of Theorem~\ref{thm:RE14_15}, one
  can prove a quantative version of this property:
  \begin{align*}
    Rc \geq 1-\delta   \Leftrightarrow   d_{GH}\left\{ (M,g), (M,g_{KE}) \right\}<\psi(\delta|M,J).
  \end{align*}
  We believe that $Rc \geq 1-\delta$ in the above inequality could be replaced by $R \geq 2n(1-\delta)$. 
  \label{rmk:RI29_3}
\end{remark}

In the remainder of this section, we shall apply the tools we developed in this and the previous section to construct examples of immortal solutions of normalized Ricci flow.  
Our examples shall show that the bi-H\"older distance distortion estimate in Corollary~\ref{cly:CK21_1} is optimal.  Namely, it cannot be improved to bi-Lipschitz estimate. 

Let $(\CP^{1},g_{FS},J)$ be the standard Fubini-Study metric on $\CP^{1} \cong S^2$. Then both the Ricci curvature and the sectional curvature are equal to the constant $4$. 
Let $\omega_{FS}$ and $\rho_{FS}$ be the metric form and Ricci form of $g_{FS}$ respectively.  Then we have
\begin{align}
  Rc_{FS}=4g_{FS}, \quad \rho_{FS}=4\omega_{FS}.   \label{eqn:RG08_1}
\end{align}
Let $[u:v]$ be the homogeneous coordinate of $\CP^{1}$.   Let $U=\{ [u:v]|u \neq 0, |\frac{v}{u}|<1\}$. 
Let $z=\frac{v}{u}=x+\sqrt{-1}y$.
Then the Fubini-Study metric can be written explicitly as
\begin{align}
  g_{FS}= \frac{Re(dz \otimes d\bar{z})}{(1+r^2)^2}=\frac{dx^2+dy^2}{(1+r^2)^2}, \; \textrm{where} \; r=|z|.   \label{eqn:RG08_2}
\end{align}
We shall adjust $g_{FS}$ to a nearby metric with non-negative Ricci curvature, by adding a ``pimple" to $g_{FS}$. 
Let $\eta$ be a non-increasing smooth function on $\R$ satisfying $-2 \leq \eta'' \leq 0$ and 
\begin{align*}
  \eta(\theta)=
  \begin{cases}
    1, & \textrm{if} \; \theta <-2;\\
    -\theta, & \textrm{if} \; \theta>0. 
  \end{cases}
\end{align*}
Let $\tilde{\eta}(\theta)=-\eta(-\theta)+1$. Then $0 \leq \tilde{\eta}'' \leq 2$ and  
\begin{align*}
  \tilde{\eta}(\theta)=
  \begin{cases}
    -\theta+1, & \textrm{if} \; \theta<0;\\
    0, & \textrm{if} \; \theta>2.
  \end{cases}
\end{align*}
In the coordinate $U=\left\{ z| r=|z|<1 \right\}$, we set $s \coloneqq \log r$ and  define function 
\begin{align}
  F_{\epsilon}(s) \coloneqq
  \begin{cases}
    \epsilon^{-2},  & \textrm{if} \; s < -\epsilon^{-2}-2; \\
    \epsilon^{-2}-1+\eta(s+\epsilon^{-2}), & \textrm{if} \; -\epsilon^{-2}-2 <s<-\epsilon^{-2};\\
    -s-1, & \textrm{if} \; -\epsilon^{-2}<s<-2;\\
    \tilde{\eta}(s+2), & \textrm{if} \; -2<s<0. 
  \end{cases}
  \label{eqn:RG13_1}  
\end{align}
Then we define 
\begin{align}
  \varphi_{\epsilon} \coloneqq
\begin{cases}
  -\epsilon F_{\epsilon}, & \textrm{on} \; U;\\
  0, & \textrm{on} \; S^{2} \backslash U. 
\end{cases}
\label{eqn:RG13_5}
\end{align}

\begin{claim}
  The metric $g=g_{\epsilon} \coloneqq e^{-\varphi_{\epsilon}}g_{FS}$ satisfies 
  \begin{align}
    Rc_{g} \geq 0, \quad \lim_{\epsilon \to 0} d_{GH} \left( (S^{2}, g), (S^{2}, g_{FS}) \right)=0. 
    \label{eqn:RG13_3}
  \end{align}
  \label{clm:RG13_2}
\end{claim}

\begin{proof}
The Ricci form $\rho$ of the metric $g$ can be written as 
\begin{align*}
  \rho=-\sqrt{-1} \partial \bar{\partial} \log \det g=\rho_{FS} +\sqrt{-1} \partial \bar{\partial} \varphi_{\epsilon}=\omega_{FS}+\sqrt{-1} \partial \bar{\partial}  \varphi_{\epsilon}. 
\end{align*}
Note that $\sqrt{-1}\partial \bar{\partial}s=0$ and $\sqrt{-1}\partial s \wedge \bar{\partial}s=\frac{\sqrt{-1}}{4r^2}dz\wedge d\bar{z} \geq 0$.
It follows from direct computation that
\begin{align*} 
  \rho=
 \begin{cases}
   \omega_{FS}, &\textrm{if} \; s<-\epsilon^{-2}-2;\\
   \omega_{FS} + (-\epsilon \eta'') \sqrt{-1} \partial s \wedge \bar{\partial} s \geq \omega_{FS}, &\textrm{if} \; -\epsilon^{-2}-2<s<-\epsilon^{-2};\\
   \omega_{FS}, &\textrm{if} \; -\epsilon^{-2}<s<-2;\\
   \omega_{FS}+ (-\epsilon \tilde{\eta}'') \sqrt{-1} \partial s \wedge \bar{\partial} s \geq (1-100\epsilon) \omega_{FS}, &\textrm{if} \; -2<s<0. 
 \end{cases}
\end{align*}
In the last step, we used the facts that $|\tilde{\eta}''|\leq 2$ and
\begin{align*}
  0 \leq \sqrt{-1}\partial s \wedge \bar{\partial}s=\frac{\sqrt{-1}}{4r^2} dz \wedge d\bar{z} \leq \frac{\sqrt{-1}}{2} dz \wedge d\bar{z} \leq 50 \omega_{FS}.
\end{align*}
Since $\epsilon$ is sufficiently small, it is clear that $\rho \geq 0$ on $U$. 
On $S^{2} \backslash U$, since $\omega \equiv \omega_{FS}$, we have $\rho=\rho_{FS}=\omega_{FS} \geq 0$.
In summary, $\rho \geq 0$ on $S^{2}$. In other words, we prove the first inequality in (\ref{eqn:RG13_3}).  

We proceed to prove the second equality in (\ref{eqn:RG13_3}).
In light of the construction of $\eta$, $\tilde{\eta}$ and the fact that $s=\log r$,  we can apply (\ref{eqn:RG13_1}) to bound the metric $g$ on $U$ as follows
\begin{align*}
  g_{FS} \leq g \leq 
  \begin{cases}
    e^{\epsilon^{-1}} g_{FS}, & \textrm{if} \; 0<r<e^{-\epsilon^{-2}}; \\
    r^{-\epsilon} e^{-\epsilon}g_{FS}, &\textrm{if} \; e^{-\epsilon^{-2}}<r<e^{-2};\\
    e^{\epsilon} g_{FS}, &\textrm{if} \; e^{-2}<r<1.
  \end{cases}
\end{align*}
Fix $r_0<<e^{-2}<1$. Let $\gamma$ be a shortest geodesic (under metric $dx^2+dy^2$) connecting $0$ and a point $z$ with $|z|=r_0$. It follows from direct calculation that 
\begin{align}
  |\gamma|_{g} &\leq \int_{0}^{e^{-\epsilon^{-2}}} e^{0.5 \epsilon^{-1}} \frac{dr}{1+r^2} + \int_{e^{-\epsilon^{-2}}}^{r_0} r^{-0.5\epsilon} e^{-0.5\epsilon} \frac{dr}{1+r^2} \notag\\ 
  &\leq \int_{0}^{e^{-\epsilon^{-2}}} e^{0.5 \epsilon^{-1}} dr + \int_{e^{-\epsilon^{-2}}}^{r_0} r^{-0.5\epsilon} dr
  \leq e^{-0.5 \epsilon^{-2}} + \frac{r_0^{1-0.5\epsilon}}{1-0.5\epsilon}.   \label{eqn:RG13_9}
\end{align}
Choosing $\epsilon<<r_0$, it follows from the above inequalities that $ |\gamma|_{g} \leq  \sqrt{r_0}$. 
Note that $\gamma$ is also a shortest geodesic under metric $g_{FS}$ and $0.5 r_0 \leq |\gamma|_{g_{FS}}\leq r_0$. Therefore, we have 
\begin{align}
  B_{g}(0, \sqrt{r_0}) \supset B_{g_{FS}}(0, 0.5 r_0).
  \label{eqn:RG13_6}  
\end{align}
On $S^{2} \backslash B_{g_{FS}}(0, 0.5 r_0)$, it is clear that
\begin{align}
  g=e^{-\varphi_{\epsilon}} g_{FS}= e^{-\epsilon F_{\epsilon}} g_{FS} \to g_{FS}, \quad \textrm{as} \; \epsilon \to 0.  
  \label{eqn:RG13_7}  
\end{align}
Then the second equality in (\ref{eqn:RG13_3}) follows from the combination of (\ref{eqn:RG13_6}) and (\ref{eqn:RG13_7}). 
The proof of Claim~\ref{clm:RG13_2} is complete. 
\end{proof}

The metric $g=g_{\epsilon}$ in Claim~\ref{clm:RG13_2} should be modified further to satisfy the normalization condition. 
In light of (\ref{eqn:RG13_6}), we have 
\begin{align*}
  |S^{2}|_{dv_{g}}=|S^{2} \backslash B_{g_{FS}}(0, 0.5 r_0)|_{dv_{g}} + |B_{g_{FS}}(0, 0.5 r_0)|_{dv_{g}}
  \leq |S^{2} \backslash B_{g_{FS}}(0, 0.5 r_0)|_{dv_{g}} + |B_{g}(0,\sqrt{r_0})|_{dv_{g}}.
\end{align*}
Because of (\ref{eqn:RG13_3}), we can apply Bishop-Gromov volume comparison on $B_{g}(0, \sqrt{r_{0}})$.  Thus it follows from (\ref{eqn:RG13_7}) and the above inequality that 
\begin{align*}
  |S^{2}|_{dv_{g}} &\leq e^{\psi}  |S^{2} \backslash B_{g_{FS}}(0, 0.5 r_0)|_{dv_{g_{FS}}} + \pi r_0
  \leq e^{\psi} \cdot |S^{2}|_{dv_{g_{FS}}} + \pi r_0,
\end{align*}
where $\psi=\psi(\epsilon)$ satisfies $\displaystyle \lim_{\epsilon \to 0} \psi(\epsilon)=0$.  
Recall that $r_0<<1$ and $\epsilon <<r_0$.  We can temporarily fix $r_0$, let $\epsilon \to 0$, and then let $r_0 \to 0$.
It turns out that 
\begin{align}
  \lim_{\epsilon \to 0} |S^{2}|_{dv_{g_{\epsilon}}}=\pi=|S^{2}|_{dv_{g_{FS}}}.   \label{eqn:RG08_3}
\end{align}
Let $\tilde{g}_{\epsilon}=e^{c_{\epsilon}} g_{\epsilon}$ where $c_{\epsilon}$ is a constant such that 
\begin{align*}
  |S^{2}|_{dv_{\tilde{g}_{\epsilon}}}=|S^{2}|_{dv_{g_{FS}}}=\pi. 
\end{align*}
It follows from (\ref{eqn:RG08_3}) that $c_{\epsilon} \to 0$ as $\epsilon \to 0$. 
Then we can adjust (\ref{eqn:RG13_3}) to obtain
  \begin{align}
    Rc_{\tilde{g}_{\epsilon}} \geq 0, \quad \lim_{\epsilon \to 0} d_{GH} \left( (S^{2}, \tilde{g}_{\epsilon}), (S^{2}, g_{FS}) \right)=0. 
    \label{eqn:RG13_8}
  \end{align}
Since $\tilde{g}_{\epsilon}$ is conformal to $g_{FS}$ and $ |S^{2}|_{dv_{\tilde{g}_{\epsilon}}}=|S^{2}|_{dv_{g_{FS}}}$, we know that $\tilde{\omega}_{\epsilon}$ is in the same K\"ahler class of $\omega_{FS}$.
Therefore, the normalized Ricci flow initiated from $\tilde{g}_{\epsilon}(0)=\tilde{g}_{\epsilon}$ is also the K\"ahler Ricci flow preserving the K\"ahler class $\frac{\pi}{2} c_1(S^{2})$. 
In light of the result of Hamilton~\cite{Ha88} and Chow~\cite{BChow}, this flow converges to a round K\"ahler metric $(S^{2}, \tilde{g}_{\epsilon}(\infty), J)$.   
Since $(S^{2}, g_{FS}, J)$ is also a round K\"ahler metric in the same class $\frac{\pi}{2} c_{1}(S^{2})$, it follows from the uniqueness of K\"ahler Einstein metrics that $g_{FS}$ and $\tilde{g}_{\epsilon}(\infty)$
are identical up to a biholomorphic map. Namely, there exists a map $\Phi \in Aut(S^{2}, J)$ such that
\begin{align}
  \Phi^{*} g_{FS}=\tilde{g}_{\epsilon}(\infty), \quad \Phi_{*} J=J.    \label{eqn:RG08_4}
\end{align}
It is well known that $Aut(S^{2}, J)=PGL(2,\C)$.  Therefore, we can find a non-degenerate $2 \times 2$ matrix $A$ such that
\begin{align*}
  \Phi([u:v])=A.[u:v]=[du+cv:bu+av]
\end{align*}
in the homogeneous coordinate of $\CP^{1} \cong S^{2}$.  Here 
\begin{align}
  A=
  \left(
  \begin{matrix}
    d&c\\
    b&a
  \end{matrix}
  \right), \quad
  \det A=1. 
  \label{eqn:RG12_3}  
\end{align}
Note that the Fubini-Study metric here is $U(2)$-invariant.  Namely, if $\Psi \in Aut(\CP^{1}, J)$ with representative matrix $B \in U(2)$, then 
(\ref{eqn:RG08_4}) implies that $(\Psi \circ \Phi)^{*}(\omega_{FS})=\tilde{\omega}_{\epsilon}(\infty)$. 
Therefore, up to multiplying $A$ with a unitary matrix $B$ from left, we can simplify $A$ in (\ref{eqn:RG12_3}) to the following form
\begin{align}
  A=
  \left(
  \begin{matrix}
    1&0\\
    b&1
  \end{matrix}
  \right). 
  \label{eqn:RG12_4}  
\end{align}
Then we have $ \Phi(z)=z+b$ for each $z \in U$. Consequently, it follows from (\ref{eqn:RG08_4}) that
\begin{align}
  \tilde{\omega}_{\epsilon}(\infty)=\sqrt{-1}\frac{\partial \Phi \wedge \bar{\partial} \bar{\Phi}}{(1+|\Phi|^2)^{2}}
  =\sqrt{-1} \frac{dz \wedge d\bar{z}}{(1+|z+b|^2)^2}.  \label{eqn:RG12_1} 
\end{align}
Recall that
\begin{align}
  \omega_{FS}=\sqrt{-1} \frac{ dz \wedge d\bar{z}}{(1+|z|^2)^2}.  \label{eqn:RG12_2}
\end{align}
In light of the strong pseudo-locality theorem of Perelman(cf. Theorem 10.3 of~\cite{Pe1} and Theorem 3.1 of~\cite{CBL07}), we know that $\tilde{\omega}_{\epsilon}(\infty)$ converges to $\omega_{FS}$ uniformly on the circle 
$\left\{ z| |z|= 0.5 \right\}$ as $\epsilon \to 0$.  
Since $b$ is a constant depending on $\epsilon$, it follows from (\ref{eqn:RG12_1}) and (\ref{eqn:RG12_2}) that $\displaystyle \lim_{\epsilon \to 0} b=0$.  
Exploiting (\ref{eqn:RG12_1}) and (\ref{eqn:RG12_2}) again, it is not hard to see that  $\tilde{\omega}_{\epsilon}(\infty)$ converges to $\omega_{FS}$ uniformly on $S^{2}$.
Therefore, for each $\xi$ small, we can choose $\epsilon$ sufficiently small such that 
\begin{align*}
  (1-\xi) g_{FS}<\tilde{g}_{\epsilon}(\infty)<(1+\xi) g_{FS}. 
\end{align*}
By exponential convergence of the normalized flow, it is clear that
\begin{align*}
  (1-\xi) \tilde{g}_{\epsilon}(\infty)<\tilde{g}_{\epsilon}(1) < (1+\xi) \tilde{g}_{\epsilon}(\infty). 
\end{align*}
Combining the previous two steps yields that
\begin{align}
  (1-3\xi) g_{FS}<\tilde{g}_{\epsilon}(1)<(1+3\xi) g_{FS}.   \label{eqn:RG12_8} 
\end{align}
However, according its construction, we have 
\begin{align}
  \tilde{g}_{\epsilon}=\tilde{g}_{\epsilon}(0)=e^{\epsilon^{-1}+c_{\epsilon}} g_{FS}, \; \textrm{if} \; |z|< e^{-\epsilon^{-2}-2}  
  \label{eqn:RG13_10}  
\end{align}
with $c_{\epsilon} \to 0$.  Therefore, along the normalized Ricci flow initiated from $\tilde{g}_{\epsilon}(0)=\tilde{g}_{\epsilon}$, the metric $\tilde{g}_{\epsilon}(1)$ cannot be uniformly
bi-Lipschitz to $\tilde{g}_{\epsilon}$.

We shall further show that bi-H\"older estimate can be achieved. 
Fix $\epsilon$ small and choose $r_{0}=\sqrt{\epsilon}$. It is clear that $\epsilon << r_{0} << 1$. 
Then we can choose $\gamma: [0, r_{0}] \to U$ such that $\gamma(t)=tz$.
In view of the rotational symmetry, it is clear that the image of $\gamma$ is the support set of the shortest geodesics connecting 
$0$ and $z$ under both metrics $g_{FS}$ and $\tilde{g}_{\epsilon}$. 
On the one hand, it follows from (\ref{eqn:RG08_2}) that 
\begin{align*}
  \frac{r_0}{1+r_0^2} \leq |\gamma|_{g_{FS}} \leq r_0.  
\end{align*}
On the other hand, by the calculation nearby  (\ref{eqn:RG13_9}) and the fact $c_{\epsilon} \to 0$, it is clear that the major part of $|\gamma|_{\tilde{g}_{\epsilon}}$ is
$\int_{e^{-\epsilon^{-2}}}^{r_0} r^{-0.5\epsilon} e^{-0.5\epsilon} \frac{dr}{1+r^2}$. Consequently, we have
\begin{align*}
  (1-\psi(\epsilon)) r_0^{1-0.5\epsilon} \leq |\gamma|_{\tilde{g}_{\epsilon}} \leq (1+\psi(\epsilon)) r_0^{1-0.5\epsilon},   
\end{align*}
where $\psi(\epsilon)$ satisfies $\displaystyle \lim_{\epsilon \to 0} \psi(\epsilon)=0$. 
Combining the previous two steps yields that 
\begin{align*}
 \{1-\psi(\epsilon)\} |\gamma|_{g_{FS}}^{1-0.5\epsilon} \leq  |\gamma|_{\tilde{g}_{\epsilon}} \leq  \{1+\psi(\epsilon)\} |\gamma|_{g_{FS}}^{1-0.5\epsilon}.  
\end{align*}
In light of (\ref{eqn:RG12_8}), we have
\begin{align*}
 \{1-\psi(\epsilon)\} d_{\tilde{g}_{\epsilon}(1)}^{1-0.5\epsilon}(0, z) \leq d_{\tilde{g}_{\epsilon}(1)}(0, z) 
 \leq  d_{\tilde{g}_{\epsilon}(0)}(0, z) \leq \{1+\psi(\epsilon)\} d_{\tilde{g}_{\epsilon}(1)}^{1-0.5\epsilon}(0, z).   
\end{align*}
This means that the bi-H\"older distortion estimate in Corollary~\ref{cly:CK21_1} can be achieved. 
Since the bi-Lipschitz estimate fails by the combination of  (\ref{eqn:RG12_8}) and (\ref{eqn:RG13_10}), 
we conclude that the bi-H\"older estimate in Corollary~\ref{cly:CK21_1} is optimal. 

The construction above can be used to develop high dimensional examples.   
For simplicity, we only point out that uniform bi-Lipschitz estimate fails along the normalized Ricci flow.

\begin{proposition}[\textbf{The failure of bi-Lipschitz equivalence}]
  For each positive integer $n$ and small positive numbers $\xi, \epsilon$, there exists a Fano manifold $(M^{n}, g, J)$ satisfying
\begin{align}
  \mathbf{vr}(M,g) \geq 1, \quad Rc_{g} \geq 0,   
  \label{eqn:RG12_6}
\end{align}
 with the following properties.

  There exists a point $x_0 \in M$ and $\vec{v} \in T_{p}M \backslash \left\{ 0 \right\}$ such that along the normalized Ricci flow initiated from $g(0)=g$, we have
  \begin{align}
    \frac{\langle \vec{v},\vec{v}\rangle_{g(\xi)}}{\langle \vec{v},\vec{v}\rangle_{g(0)}}< \epsilon.
    \label{eqn:RG12_7}  
  \end{align}
  Consequently, (\ref{eqn:RG12_6}) cannot imply uniform bi-Lipschitz equivalence between $g(0)$ and $g(\xi)$, no matter how small $\xi$ is. 
  \label{prn:RG12_5}
\end{proposition}

\begin{proof}
  Let $\tilde{g}_{\epsilon}$ be the metric constructed above on $\CP^1$.   Then let $g=\tilde{g}_{\epsilon} \times g_{FS} \times \cdots \times g_{FS}$ on the product complex manifold
  $M=\CP^{1} \times \cdots \times \CP^{1}$, which is Fano. The metric form is proportional to the anti-canonical class.
  Therefore, the normalized K\"ahler Ricci flow initiated from $g$ will preserve the K\"ahler class and exist immortally.
  Moreover, the flow converges to the K\"ahler Einstein metric $g_{KE}=g_{FS} \times g_{FS} \times \cdots \times g_{FS}$.
  Note that $(M,g)$ satisfies (\ref{eqn:RG12_6}) up to a rescaling of a fixed scale. 
  Since $d_{GH}\left( (M,g), (M, g_{KE}) \right) \to 0$ as $\epsilon \to 0$, we can apply Theorem~\ref{thm:RG05_3}.  
  Therefore, by the same deduction as before, we know that $g_{KE}$, $g(\xi)$ and $g(\infty)$ are all bi-Lipschitz equivalent among each other with Lipschitz constant $1+\xi$, whenever $\epsilon$ is small enough.
  However, by choosing $p=\left( [1:0], [1:0], \cdots, [1:0] \right)$ and $\vec{v} \in T_{[1:0]} \CP^{1} \times \left\{ 0 \right\} \times \cdots \times \left\{ 0 \right\} \subset T_{p}M$, 
  it follows from the construction that 
  \begin{align*}
    \frac{\langle \vec{v},\vec{v}\rangle_{g(\xi)}}{\langle \vec{v},\vec{v}\rangle_{g(0)}}
    \leq (1+\xi) \frac{\langle \vec{v},\vec{v}\rangle_{g_{KE}}}{\langle \vec{v},\vec{v}\rangle_{g(0)}} < 2 e^{-\epsilon^{-1}}<\epsilon, 
  \end{align*}
  which is nothing but (\ref{eqn:RG12_7}). 
\end{proof}

Note that normalization in Proposition~\ref{prn:RG12_5} is not important, since the logarithm of volume is bounded along the unnormalized Ricci flow for a fixed amount of time 
by condition (\ref{eqn:RG12_6}).  In the study of Riemannian manifolds of Ricci curvature bounded below,  Ricci flow smoothing is a convenient method.  
Proposition~\ref{prn:RG12_5} indicates that there is no uniform control of the time $\xi$ such that the metric $g(\xi)$ and $g(0)$ are uniformly bi-Lipschitz equivalent, even if we have
the extra K\"ahler condition. Since K\"ahler geometry has independent interest, it is desirable to find conditions to guarantee the lower bound of $\xi$ along the K\"ahler Ricci flow. 
It turns out that scalar curvature being locally bounded is a sufficient condition, which will be studied in the next section.

\section{The pseudo-locality theorem in the K\"ahler Ricci flow}
\label{sec:kpseudo}

A K\"ahler manifold $(M^{n}, g, J)$ is a Riemannian manifold with real dimension $m=2n$ and an almost complex structure $J$ which is parallel under the Levi-Civita connection, i.e. $\nabla_{g}J \equiv 0$. 
If we run the Ricci flow starting from a closed  K\"ahler manifold $(M^{n}, g, J)$,  then it is well known that the K\"ahler condition is preserved by the flow.
Namely, for each $t$, we have $\nabla_{g(t)}J \equiv 0$. 
The K\"ahler Ricci flow admits some amazing rigidities. 
For example, the identity map between different time slices of a K\"ahler Ricci flow is automatically 
biholomorphic and therefore harmonic.  In this section, we shall take advantage of the special properties of the K\"ahler Ricci flow to further improve the pseudo-locality theorems.

This section is structured as follows.
First we show a key improvement from the Riemannian Ricci flow:  the metric bi-Lipschitz equivalence. 
Roughly speaking,  when the initial metric is locally almost Euclidean,  then near the central part along the flow, all $g(\cdot,t)$ are quasi-isometric to $g(\cdot,0)$. 
Note that here almost Euclidean condition also requires scalar curvature being bounded.
This bi-Lipschitz equivalence is achieved by Theorem~\ref{thm:RJ08_0} and Corollary~\ref{cly:RD29_4}. 
Then we construct local holomorphic charts and study the K\"ahler Ricci flow in the local charts.
It turns out that the localization in the K\"ahler Ricci flow is much more straightforward than the one in the Riemannian Ricci flow.
Besides $-R$, we have another quantity $(R+|\nabla \dot{\varphi}|^2)$ as a local sub-heat-solution.
Here $\dot{\varphi}$ is local Ricci potential function, which can be adjusted by adding linear functions.   
Exploiting the local maximum principle for them, we obtain the growth estimates for the scalar curvature bounds in Proposition~\ref{prn:RC13_1}.
Combining the scalar curvature bounds and the metric bi-Lipschitz equivalence, we then deduce $C^{1,\alpha}$-estimate on the metrics via complex Monge-Amp\`{e}re  equations. 
This is done in Theorem~\ref{thm:RC10_1}.  
As applications, in Theorem~\ref{thm:RB21_12}, we show a $C^{1,\alpha}$-compactness theorem for K\"ahler manifolds whose scalar curvatures are uniformly 
bounded and whose entropy radii are uniformly bounded from below.  If the limit is a smooth K\"ahler manifold, we also obtain the lower and upper bound of scalar curvatures being preserved under the $C^{1,\alpha}$-convergence.
In Corollary~\ref{cly:RD29_3}, we show that the assumption of the smooth K\"ahler manifold limit is natural under some circumstances.  
Finally, we finish this section by proving Theorem~\ref{thm:RG26_8} and Theorem~\ref{thm:RH16_1}.

\begin{theorem}(\textbf{Metric bi-Lipschitz equivalence at the center point})
For each small constant $\xi$, there exists $\epsilon=\epsilon(n, \xi)$ with the following property. 
  
Suppose $\{(M^n,g(t),J), 0 \leq t \leq 1\}$ is a K\"ahler Ricci flow solution.
Suppose for  $x_0 \in M$ and $r \in (0,1]$, we have 
    \begin{align}
       \begin{cases}
	 &\boldsymbol{\bar{\nu}}(B_{g(0)}(x_0, r), g(0), r^2) \geq -\epsilon;\\
	 &|R|_{g(0)} \leq r^{-2}, \; on \;  B_{g(0)}(x_0, r).    
       \end{cases}
   \label{eqn:RB16_1}
   \end{align}
  Then the following metric equivalence holds:
  \begin{align}
    e^{-\xi} g(x_0,0) \leq g(x_0,t) \leq e^{\xi} g(x_0,0), \quad \forall \; t \in [0, (\epsilon r)^{2}].
    \label{eqn:RB16_3}  
  \end{align}
  \label{thm:RJ08_0}
\end{theorem}

\begin{proof}
Without loss of generality, we may assume $r=1$. 
We argue by contradiction.  
If this proposition were wrong,  for every small $\epsilon$, we were able to find a K\"ahler Ricci flow solution
$\{(M, x_0, g(t), J),  0 \leq t \leq 1\}$ satisfying
\begin{align}
  \boldsymbol{\bar{\nu}}(B_{g(0)}(x_0, 1), g(0), 1) \geq -\epsilon;  \quad |R|_{g(0)} \leq 1 \; on \;  B_{g(0)}(x_0, 1).    \label{eqn:RB15_4}
\end{align}
However, (\ref{eqn:RB16_3}) fails. 
We shall finish the contradiction argument in the following steps.\\ 

\noindent
\textit{Step 1. Shift the base point via a maximum principle argument.}\\

For each point $y \in B_{g(0)}(x_0, 1)$, we define $I_{y}$ to be the first time such that $e^{\xi}$-equivalence with initial metric starts to fail. 
Namely, we have
\begin{align*}
 &e^{-\xi} g(y,0) \leq g(y,t) \leq e^{\xi} g(y,0), \quad \forall \; t \in [0, I_{y}];\\
 &\left| \log \frac{\langle \vec{v}_i,\vec{v}_i\rangle_{g(t_i)}}{\langle \vec{v}_i,\vec{v}_i\rangle_{g(0)}} \right|>\xi,  \quad \textrm{for a sequence of}\; \vec{v}_i \in T_{y}M\;  \textrm{and} \; t_i  \to (I_{y})_{+}. 
\end{align*}
According to our choice,  it is clear that $I_{x_0}<\epsilon^2$.    The function $\frac{I_{y}}{d_{g(0)}^2(y, \partial B_{g(0)}(x_0,1))}$ 
is clearly a scaling invariant lower semi-continuous function.
Since this function takes value $\infty$ on $\partial B_{g(0)}(x_0, 1)$, we can assume its minimum value is achieved at some interior point, say $y_0$. 
Then we know
\begin{align*}
  \frac{I_{y_0}}{d_{g(0)}^2(y_0, \partial B_{g(0)}(x_0,1))} \leq  \frac{I_{x_0}}{d_{g(0)}^2(x_0, \partial B_{g(0)}(x_0, 1))}=I_{x_0} <\epsilon^2. 
\end{align*}
For each point $y \in B_{g(0)}\left(y_0, 0.5 \epsilon^{-1} \sqrt{I_{y_0}} \right)$, the triangle inequality then implies that
\begin{align*}
  I_{y} \geq \frac{d_{g(0)}^2(y, \partial B_{g(0)}(x_0,1))}{d_{g(0)}^2(y_0, \partial B_{g(0)}(x_0,1))} \cdot I_{y_0} \geq \frac{1}{4} I_{y_0}. 
\end{align*}
Now we rescale the flow $g$ to $\tilde{g}$ by setting $\tilde{g}(\cdot,t)=I_{y_0}^{-1} g(\cdot, I_{y_0} t)$.   
For simplicity of notation, we denote the rescaled flow again by $g(t)$, and denote $g(0)$ by $g$, which is reserved as the default metric. 
Therefore, the subscript $g$ will be omitted, unless we need to emphasize its importance and reduce confusions.  

Under these notational conventions, we then have the following properties:
\begin{align}
  &\boldsymbol{\bar{\nu}}(B(y_0, 0.5 \epsilon^{-1}), g, 1) \geq \boldsymbol{\bar{\nu}}(B(x_0, I_{y_0}^{-\frac12}), g, 1) \geq -\epsilon;  \label{eqn:RB19_1}\\
  &|R| \leq I_{y_0} \leq \epsilon^2 \; on \;  B(y_0,  0.5 \epsilon^{-1});   \label{eqn:RB19_2}\\
  &e^{-\xi} g(y,0) \leq g(y,t) \leq e^{\xi} g(y,0), \quad \forall \; y \in B(y_0, 0.5 \epsilon^{-1}), \;  t \in [0, 0.25]. \label{eqn:RB19_3}
\end{align}
Furthermore, we have a nonzero vector $\vec{v} \in T_{y_0} M$ such that 
\begin{align}
  \left| \log \frac{\langle \vec{v},\vec{v} \rangle_{g(1)}}{\langle \vec{v}, \vec{v}\rangle_{g(0)}} \right|=\xi.   \label{eqn:RB19_4}
\end{align}
The condition (\ref{eqn:RB19_1}) assures us to apply the curvature estimate (\ref{eqn:ML28_1}) in the improved pseudo-locality theorem(cf. Figure~\ref{fig:threelayers} for intuition).
Therefore, on the ball $B(y_0, 0.1 \epsilon^{-1})$, 
the metric $g(t)$ is almost static on the time period $[0.25, 1]$. In particular, we can improve (\ref{eqn:RB19_3}) to 
\begin{align}
  e^{-2\xi} g(y,0) \leq g(y,t) \leq e^{2\xi} g(y,0), \quad \forall \; y \in B(y_0, 0.1 \epsilon^{-1}), \;  t \in [0, 1]. \label{eqn:RB19_5}
\end{align}

 \begin{figure}[H]
 \begin{center}
 \psfrag{A}[c][c]{$(y_0, 1)$}
 \psfrag{B}[c][c]{$(y_0, 0)$}
 \psfrag{C}[c][c]{\color{blue}{$B_{g(0)}(x_0, 0.5 \epsilon^{-1})$}}
 \psfrag{D}[c][c]{\color{green}{$B_{g(0)}(x_0, 0.1 \epsilon^{-1})$}}
 \psfrag{E}[c][c]{\color{blue}{$e^{-\xi}g(\cdot, 0) \leq g(\cdot, t) \leq e^{\xi}g(\cdot, 0)$}}
 \psfrag{F}[c][c]{\color{green}{$t|Rm|(\cdot, t)<\psi(\epsilon|n)$}}
 \psfrag{G}[c][c]{\color{red}{$B_{g(1)}(y_0, L)$}}
 \psfrag{H}[c][c]{$t=0$}
 \psfrag{I}[c][c]{$t=0.25$}
 \psfrag{J}[c][c]{$t=1$}
 \includegraphics[width=0.8 \columnwidth]{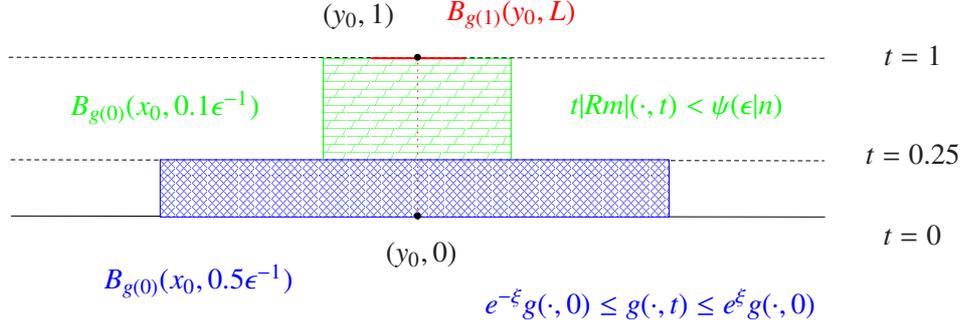}
 \caption{Shift the base point to a place with good space-time neighborhood}
 \label{fig:threelayers}
 \end{center}
 \end{figure}

\noindent
\textit{Step 2. Restrict the K\"ahler Ricci flow to a local chart.}\\

For simplicity of notation, we denote $g(0)$ by $g$ and $g(1)$ by $h$. 
Fix a large number $L$. In light of (\ref{eqn:ML28_1}), it is clear that 
\begin{align*}
  |Rm|_{h}+|\nabla Rm|_{h}<\psi(\epsilon|n)<<L^{-3},   
\end{align*}
in the geodesic ball $B_{h}(y_{0}, 2L)$. Clearly, the ball $(B_{h}(y_{0}, 2L), L^{-2}h)$ has radius $2$ and satisfies
\begin{align}
  |Rm|_{L^{-2}h}+|\nabla Rm|_{L^{-2}h}<<1.     \label{eqn:RJ08_1}
\end{align}
Following a standard application of H\"ormander's $L^{2}$-estimate(cf. Proposition 1.2 of~\cite{TianYau1} and Lemma 2.2 of~\cite{LiWaZh}),
the above condition assures us to construct a biholomorphic map
\begin{align}
  \Phi: B_{h}(y_{0}, L) \to U \in \C^n
  \label{eqn:RJ07_0}
\end{align}
satisfying the following estimates
\begin{align}
   &  B(0, C^{-1} L) \subset U \subset B(0, C L);  \label{eqn:RB23_4a}\\
   &(\Phi_* h)_{\alpha \bar{\beta}}(0)=\delta_{\alpha \bar{\beta}}(0);  \label{eqn:RB23_4b}\\
   &\frac{1}{C}\delta_{\alpha \bar{\beta}} <\Phi_* h< C \delta_{\alpha \bar{\beta}}, \quad \textrm{on} \; U; \label{eqn:RB23_4}\\
   &\norm{\Phi_{*}h}{C^{2,\frac{1}{2}}(U)}<C L^{2+\frac{1}{2}}. \label{eqn:RJ08_20}
\end{align}
Here $C=C(n)>1$. 
Note that we can first construct $\Phi$ for the unit ball with respect to metric $L^{-2}h$, which is guaranteed by (\ref{eqn:RJ08_1}),  and then rescale $L^{-2}h$ to $h$.
This is the reason why the constant $C$ in (\ref{eqn:RB23_4a}) and (\ref{eqn:RB23_4}) is independent of $L$. 
Then we define 
\begin{align}
  V \coloneqq \Phi(B_{h}(y_{0},  0.1 C^{-1}L)) \Subset U \subset \C^{n}  \label{eqn:RJ07_1}
\end{align}
for the exact $C$ in (\ref{eqn:RB23_4a}).  Similar to (\ref{eqn:RB23_4a}), we have
\begin{align}
     B(0, 0.01 C^{-1}L) \subset V \subset B(0, C^{-1}L).     \label{eqn:RJ11_1}
\end{align}

 \begin{figure}[H]
 \begin{center}
 \psfrag{A}[c][c]{\color{blue}{$B_{g(0)}(y_0,L)$}}
 \psfrag{B}[c][c]{$\Phi$}
 \psfrag{C}[c][c]{\color{blue}{$U$}}
 \psfrag{D}[c][c]{$B_{g(0)}(y_0, 0.1 L)$}
 \psfrag{E}[c][c]{$V$}
 \psfrag{F}[c][c]{$y_0$}
 \psfrag{G}[c][c]{$0$}
 \psfrag{H}[c][c]{\color{green}$B(0,0.01C^{-1}L) \subset B(0, C^{-1}L) \subset B(0, CL)$}
 \psfrag{I}[c][c]{\color{blue}{$U=\Phi(B_{g(0)}(y_0, L))$}}
 \psfrag{J}[c][c]{$V=\Phi(B_{g(0)}(y_0, 0.1L))$}
 \includegraphics[width=0.8 \columnwidth]{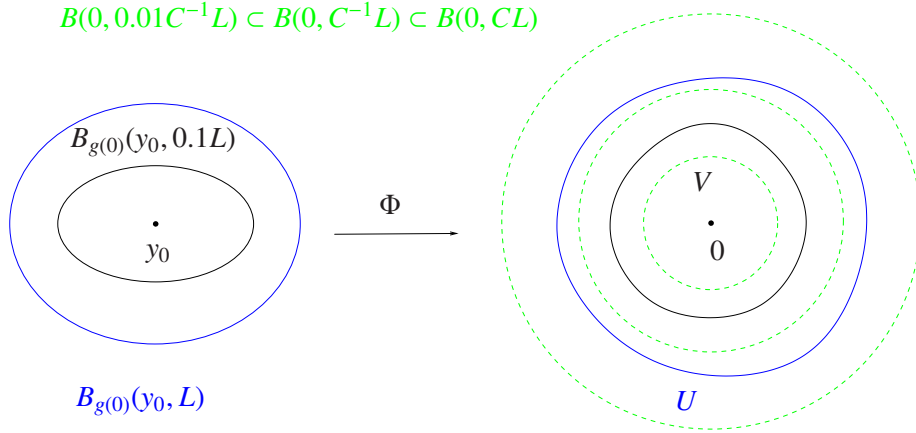}
 \caption{Construction of holomorphic coordinate chart}
 \label{fig:holomorphicmap}
 \end{center}
 \end{figure}

The identity map between different time slices of a K\"ahler Ricci flow is always biholomorphic.
One can apply $\Phi$ to obtain a local K\"ahler Ricci flow $\Phi_* g(t)$ on $U$ as follows.  
We denote the flat metric form on $U$ by $\omega_{E}$, and the metric form of $\Phi_* g(t)$ by $\omega_t$. 
By the construction of $\Phi$, there exists a smooth function $\varphi_1$ such that $\omega_{1}=\omega_{E}+\sqrt{-1} \partial \bar{\partial} \varphi_1$ is the metric form of $\Phi_{*}h=\Phi_{*} g(1)$.
Note that the choice of $\varphi_{1}$ is not unique,  as $\partial \bar{\partial}f=0$ for every linear function $f$. 
We make the choice by requiring
\begin{align}
  \varphi_{1}(0)=\frac{\partial \varphi_{1}}{\partial z^{\alpha}}(0)=\frac{\partial \varphi_{1}}{\partial z^{\bar{\alpha}}}(0)=0, \quad \forall \; \alpha \in \left\{ 1, 2, \cdots, n \right\}.
  \label{eqn:RJ08_18}  
\end{align}
By metric equivalence (\ref{eqn:RB23_4}), we have $ \frac{1}{C} \leq n+\Delta \varphi_{1}= tr_{g_{E}} g_{\varphi_{1}} \leq C$. 
Therefore, $\Delta \varphi_{1}$ is uniformly bounded. By adjusting $U$ at the beginning if necessary, we can apply Schauder estimate 
to obtain uniform gradient bound of $\varphi_{1}$ on $U$.  
In particular, it follows from (\ref{eqn:RB23_4}), (\ref{eqn:RJ08_18}) and path integration that 
\begin{align}
 \norm{\varphi_1}{C^0(U)}<C(n,L). 
 \label{eqn:RJ06_6}
\end{align}
Note that (\ref{eqn:RB23_4b}) is equivalent to 
\begin{align}
  \omega_{E}(0)=\omega_{1}(0).   \label{eqn:RJ08_12}
\end{align}
On $U$, we define
\begin{align}
  \varphi(\cdot, t) \coloneqq \varphi_1 + \int_{1}^{t }\log \frac{\omega_s^n}{\omega_{E}^n} ds.    \label{eqn:RB20_1}
\end{align}
Direct calculation shows that 
\begin{align*}
  \frac{d}{dt} \left( \omega_{E}+\sqrt{-1} \partial \bar{\partial} \varphi \right)=\sqrt{-1} \partial \bar{\partial} \dot \varphi=\sqrt{-1}  \partial \bar{\partial} \log \frac{\omega_t^n}{\omega_{E}^n}
  =-\rho(\omega_t)=\frac{d}{dt} \omega_t,  
\end{align*}
where $\rho(\omega_t)$ is the Ricci form of $\omega_t$.
Since $\omega_{1}$ is chosen already, by uniqueness of ODE solution with given initial value on each world-line $\left\{ (x,t)| t \in [0,1] \right\}$,
it is clear that
\begin{align*}
  \omega_t(x) \equiv \omega_{\varphi(\cdot, t)}(x), \quad \forall \; x \in U, \; t \in [0,1]. 
\end{align*}
Therefore, we have naturally localized the global K\"ahler Ricci flow to be a local K\"ahler Ricci flow on $U$, with $\varphi(\cdot, t)$ being the local K\"ahler potential.  
The geometry of $(U, g(t))$ is then completely determined by the potential function $\varphi(\cdot, t)$. 
Direct calculation shows that $\varphi$ satisfies the following equation
\begin{align}
  &\dot{\varphi}= \log \frac{\omega_{\varphi}^{n}}{\omega_{E}^{n}},  \label{eqn:RJ08_3}\\
  &\ddot{\varphi}= \Delta_{\varphi} \dot{\varphi}=-R. \label{eqn:RJ08_4}
\end{align}
Here $\Delta_{\varphi}$ means the Laplacian operator with respect to the metric form $\omega_{\varphi}=\omega_{E}+\sqrt{-1}\partial \bar{\partial} \varphi$.
Namely, if we denote the metric tensor compatible with $\omega_{\varphi}$ 
by $g_{\varphi}$ such that $(g_{\varphi})_{\alpha \bar{\beta}}=\delta_{\alpha \bar{\beta}} + \frac{\partial^{2} \varphi_{0}}{\partial z^{\alpha} \partial z^{\bar{\beta}}}$
and denote the inverse of $(g_{\varphi})_{\alpha \bar{\beta}}$ by $g_{\varphi}^{\alpha \bar{\beta}}$,  then
\begin{align}
  \Delta_{\varphi} \coloneqq g_{\varphi}^{\alpha\bar{\beta}} \frac{\partial^2 }{\partial z^{\alpha} \partial z^{\bar{\beta}}}.  
  \label{eqn:RJ08_2}
\end{align}
Similar to the global K\"ahler Ricci flow, we shall develop estimates for $\varphi_{t}=\varphi(\cdot, t)$ and $\dot{\varphi}_{t}=\dot{\varphi}(\cdot, t)$ to determine the geometry of $g(t)$.
The time derivatives of $\varphi$ at $t=0,1$ are of great importance, so we denote them by particular notation
\begin{align}
  F \coloneqq \dot{\varphi}_{0}, \quad H \coloneqq \dot{\varphi}_{1}.  \label{eqn:RJ08_19}
\end{align}\\

\noindent
\textit{Step 3. Develop a priori estimates for local potential functions at time $t=1$. More specifically,  let $V \Subset U$ be defined in (\ref{eqn:RJ07_1}), then we have
 \begin{align}
    &\norm{H}{C^{0}(V)}< C(n),    \label{eqn:RJ08_28}\\
    &\norm{\varphi_{1}}{C^{k,\frac{1}{2}}(V)} + \norm{H}{C^{k,\frac{1}{2}}(V)}< C_{k}.   \label{eqn:RJ08_24} 
 \end{align}
 }

Note that (\ref{eqn:RJ08_28}) follows directly from the combination of (\ref{eqn:RB23_4}), (\ref{eqn:RJ08_3}) and (\ref{eqn:RJ08_19}).  
We focus on the proof of (\ref{eqn:RJ08_24}). 
Direct calculation shows that
\begin{align}
  &\det \left( \delta_{\alpha \bar{\beta}} + (\varphi_{1})_{\alpha \bar{\beta}} \right)=e^{H}, \label{eqn:RJ08_21} \\
  &-\Delta_{\varphi_{1}} H=R_{g(1)}.  \label{eqn:RJ08_22}
\end{align}
In light of the improved curvature estiamte (\ref{eqn:ML28_1}) and Shi's estimates of curvature derivatives, we know $|\nabla_{g(1)}^{k} R_{g(1)}|_{g(1)}<C_{k}$ on $U$ for each non-negative integer $k$.
Then we can start bootstrapping as follows.
By (\ref{eqn:RJ08_20}) and uniform ellipticity (\ref{eqn:RB23_4}), we apply Schauder estimate on (\ref{eqn:RJ08_22}) to obtain $H$ has $C^{4,\frac{1}{2}}(U')$-norm for each  $U' \Subset U$.
Taking derivatives on (\ref{eqn:RJ08_21}) with respect to space directions, we can apply Schauder estimate on the resulting equation to obtain that $\varphi_{1}$ has $C^{6,\frac{1}{2}}(U')$-norm, by shrinking $U'$ if necessary. 
Then Schauder estimate on (\ref{eqn:RJ08_22}) again implies $H$ has uniform $C^{6,\frac{1}{2}}(U')$-norm, and so on. In each step we shall slightly shrink $U'$. 
For each $k$, we shall arrive (\ref{eqn:RJ08_24}) after finite steps of bootstrapping. \\

\noindent
\textit{Step 4. Develop a priori estimates for local potential functions at time $t=0$. More precisely, there exists a constant $C=C(n,L)$ such that
  \begin{align}
    \norm{\varphi_{0}}{C^{3,\frac{1}{2}}(V)} + \norm{F}{C^{1,\frac{1}{2}}(V)}<C.
    \label{eqn:RJ08_25}
  \end{align}
}

We shall finish the proof of (\ref{eqn:RJ08_25}) in Claim~\ref{clm:RJ06_10} and Claim~\ref{clm:RJ06_3}. 

\begin{claim}
  The following $C^{0}$-estimates of $\varphi_{0}$ and $F$ hold:
\begin{align}
 &\norm{\varphi_0}{C^0(U)}<C(n,L),   \label{eqn:RJ06_8} \\
 &\norm{F}{C^{0}(U)}<C(n). \label{eqn:RJ08_5}
\end{align}
Furthermore, $\varphi_{0}$ satisfies the following $C^{1,1}$-estimate on $U$:
\begin{align}
  \frac{1}{C(n)}\delta_{\alpha \bar{\beta}} \leq  \delta_{\alpha\bar{\beta}} + \frac{\partial^2 \varphi_{0}}{\partial z^{\alpha} \partial z^{\bar{\beta}}}  \leq C(n)  \delta_{\alpha \bar{\beta}}.  \label{eqn:RJ06_9}
\end{align}
  \label{clm:RJ06_10}
\end{claim}

We first show (\ref{eqn:RJ06_9}). 
The metric equivalence (\ref{eqn:RB19_5})  implies that $e^{-4\xi} \omega_{1} \leq \omega_t \leq e^{4\xi} \omega_{1}$ for all $t \in [0,1]$. 
On the other hand,  (\ref{eqn:RB23_4}) implies $\frac{1}{C(n)} \omega_{E} < \omega_{1} < C(n) \omega_{E}$. 
Therefore, by adjusting $C$ if necessary, we have
\begin{align}
  \frac{1}{C(n)} \omega_{E} \leq \omega_t \leq C(n) \omega_{E},  \quad \forall \; t \in [0,1]. \label{eqn:RJ07_14}
\end{align}
which is the same as (\ref{eqn:RJ06_9}). 
Then we move on to prove (\ref{eqn:RJ06_8}). 
Note that (\ref{eqn:RJ07_14}) implies that 
\begin{align}
  \frac{1}{C(n)} \leq \frac{\omega_{t}^{n}}{\omega_{E}^{n}} \leq C(n), \quad \forall \; t \in [0,1].  \label{eqn:RJ06_7} 
\end{align}
Plugging (\ref{eqn:RJ06_7}) into  (\ref{eqn:RB20_1}), we obtain
\begin{align*}
  |\varphi_{t}-\varphi_{1}| \leq \left| \int_{1}^{t} \log \frac{\omega_{s}^{n}}{\omega_{E}^{n}} ds  \right| \leq C (1-t) \leq C(n), \quad \forall \; t \in [0,1].  
\end{align*}
Then (\ref{eqn:RJ06_8}) follows from the combination of the above inequality and (\ref{eqn:RJ06_6}).
In (\ref{eqn:RJ06_7}), by setting $t=0$, we obtain the uniform bound of $|\log \frac{\omega_{1}^{n}}{\omega_{E}^{n}}|=|\log F|$, which shows  (\ref{eqn:RJ08_5}). 
The proof of Claim~\ref{clm:RJ06_10} is complete. \\

Then we improve the regularity of $\varphi_{0}$ under the help of complex Monge-Amp\`{e}re equation.

\begin{claim}
  On each $U' \Subset U$, there is a constant $C=C(n,U,U')$ such that
\begin{align}
  &\norm{\varphi_{0}}{C^{3,\frac{1}{2}}(U')}<C,  \label{eqn:RB21_1}\\
  &\norm{F}{C^{1,\frac{1}{2}}(U')}<C.   \label{eqn:RJ06_5}
\end{align}
  \label{clm:RJ06_3}
\end{claim}

The proof of this claim is basically a local version of Proposition 1.2 of~\cite{ChenCheng1}. 
We say a few words here for the convenience of the readers. 
We first show (\ref{eqn:RJ06_5}). 
On the open manifold $(U, \omega_{\varphi_{0}})$, the scalar curvature satisfies
\begin{align}
  -\Delta_{\varphi_0} \log \det \left( \delta_{\alpha \bar{\beta}} + (\varphi_0)_{\alpha \bar{\beta}} \right)=R.    \label{eqn:RB20_3} 
\end{align}
As usual(cf.~\cite{ChenXX}), we decompose the above equation into two equations
\begin{align}
  &\det \left( \delta_{\alpha \bar{\beta}} + (\varphi_0)_{\alpha \bar{\beta}} \right)=e^{F}, \label{eqn:RG19_1}\\
  &-\Delta_{\varphi_0} F=R.  \label{eqn:RJ06_4}
\end{align}
In light of (\ref{eqn:RB19_2}), the right hand side of equation (\ref{eqn:RJ06_4}) is  uniformly bounded by $\epsilon^2$.
If the coefficients $g_{\varphi_{0}}^{\alpha\bar{\beta}}$ were $\delta^{\alpha\bar{\beta}}$, we could apply the standard Schauder estimate 
to obtain $F$ is uniformly $C^{1,\frac{1}{2}}$ in each compact subset of $U$. 
However, now there exists only a uniform quasi-isometry between $g_{\varphi_{0}}^{\alpha \bar{\beta}}$ and $\delta^{\alpha \bar{\beta}}$ on $U$, the direct application of Schauder estimate breaks down. 
Fortunately, under the help of the extra equation (\ref{eqn:RG19_1}), we can draw the same conclusion. 
Actually, in view of (\ref{eqn:RJ06_9}), which is the same as $e^{-2\xi}\delta^{\alpha \bar{\beta}} \leq g_{\varphi_{0}}^{\alpha \bar{\beta}} \leq e^{2\xi} \delta^{\alpha \bar{\beta}}$,
we can apply the De Giorgi-Nash-Moser theory(cf. Theorem 8.22 of~\cite{GT}) on (\ref{eqn:RJ06_4}).
Therefore, for each domain $U' \Subset U$,  we obtain $\norm{F}{C^{\frac{2}{3}}(U')}<C'$ for some $C'$ depending on $U'$.
Plugging this estimate into (\ref{eqn:RG19_1}), by shrinking $U'$ if necessary, we then apply the theory of Caffarelli-Evans-Krylov(cf.~\cite{WY} and references therein) to obtain
uniform bound of $\norm{\varphi_{0}}{C^{2,\frac{1}{2}}(U')}$. This means that the metric form $\omega_{\varphi_{0}}$, or the metric tensor $g_{\varphi_{0}}$ in $U'$ has uniform $C^{\frac{1}{2}}(U')$-norm.
Consequently, in equation  (\ref{eqn:RJ06_4}),  the coefficient term $g_{\varphi_{0}}^{\alpha\bar{\beta}}$ has uniform $C^{\frac{1}{2}}(U')$-norm.
Therefore, we now can apply the standard Schauder theory(cf. Theorem 8.32 of~\cite{GT}) on (\ref{eqn:RJ06_4}) to obtain (\ref{eqn:RJ06_5}).  
Now we plugging (\ref{eqn:RJ06_5})  into (\ref{eqn:RG19_1}). 
By taking space derivatives, one can apply standard Schauder estimate for linear elliptic equations to obtain (\ref{eqn:RB21_1}), by shrinking $U'$ if necessary. 
The proof of Claim~\ref{clm:RJ06_3} is complete. 

Clearly, (\ref{eqn:RJ08_25}) follows from the combination of Claim~\ref{clm:RJ06_10} and Claim~\ref{clm:RJ06_3}.\\

\noindent
\textit{Step 5. Derive the local almost Einstein condition, which is used to connect $g(0)$ to $g(1)$. More precisely, we shall show the following curvature and distance estimates:
  \begin{align}
    &\int_{0}^{1} \int_{V} |R| dvdt<\psi(\epsilon|n,L),     \label{eqn:RB21_7}\\
    &\sup_{x,y \in V} \left\vert d_{g(0)}(x,y) -d_{g(1)}(x,y) \right\vert < \psi(\epsilon|n,L).   \label{eqn:RB21_11} 
  \end{align}
}

Now we change our point of view from PDE to Riemannian geometry. 
The uniform regularity estimate (\ref{eqn:RB21_1}) implies the following intrinsic volume ratio estimate and entropy radius estimate for the initial metric. 
\begin{claim}
  There is a constant $r_0=r_0(n,L)$ such that 
  \begin{align}
    & \sup_{r \in (0,r_0)}  \omega_{2n}^{-1}r^{-2n}\left| B_{g}(y,r)\right|_{dv_{g}} \leq   1+ \sqrt{r},   \label{eqn:RB21_6}\\
    & \mathbf{er}_{\epsilon_{0}}(x) > r_{0},  \label{eqn:RJ07_16}
  \end{align}
  for each $y \in V$, where $V$ is a subset of $U$ defined in (\ref{eqn:RJ07_1}), and $\mathbf{er}_{\epsilon_{0}}$ is the entropy radius defined in (\ref{dfn:RH26_3}).

  \label{clm:RB21_5}
\end{claim}

Recall that $\omega_{\varphi_{0}}$ is the metric form of the initial metric $g=g(0)$.  
We can regard $\omega_{\varphi_0}$ as matrix-valued function. 
For each pair of points $x,y \in V$, using the metric equivalence (\ref{eqn:RJ06_9}) between $\omega_{\varphi_0}$ and $\omega_{E}$, 
and the $C^{3,\frac{1}{2}}$-regularity estimate of potential function (\ref{eqn:RB21_1}), we have 
\begin{align}
  |\omega_{\varphi_0}(x)-\omega_{\varphi_0}(y)|<C|x-y|<C d_{\omega_{\varphi_0}}(x,y)    \label{eqn:RB21_2} 
\end{align}
for some $C=C(n,L)$. 
Let $G$ be the matrix of $g(y)$.   
For each $y \in V$ and $r \in (0, r_0)$, we define
\begin{align*}
  D_{r}(y) \coloneqq \left\{ x \in U | (x-y)^{\tau} G (x-y) \leq  r^2 \right\},
\end{align*}
where $\tau$ means transpose.
If $g$ is a constant matrix valued function, then $D_{r}(y)$ is nothing but an $r$-geodesic ball centered at $y$.
Estimate (\ref{eqn:RB21_2}) implies that $D_{r}(y)$ is almost an $r$-geodesic ball. 
In fact, by combining (\ref{eqn:RB19_5}), (\ref{eqn:RB21_1}) and (\ref{eqn:RB21_2}) together, we obtain 
\begin{align*}
  B_{g}(y, (1-Cr) r) \subset  D_{r}(y) \subset B_{g}(y, (1+Cr) r) \subset D_{r(1+Cr)}(y) \subset D_{2r}(y). 
\end{align*}
The determinant of $g$ on $D_{2r}(y)$ is bounded above by $\left( 1+Cr \right)\det G$. 
Therefore, we have
\begin{align*}
  |B_{g}(y,r)|_{dv_{g}} \leq |D_{r(1+Cr)}(y)|_{dv_{g}} \leq (1+Cr) \omega_{2n} r^{2n}. 
\end{align*}
Since $r<r_0$ very small, $Cr^{\frac{1}{2}}<1$. Therefore, the estimate (\ref{eqn:RB21_6}) follows from the above inequality.
Similarly, through the metric quasi-isometry (\ref{eqn:RB21_2}), we can follow the above argument to show that the isoperimetric constant of $B_{g}(y,r)$ is very close to that of a standard Euclidean ball, 
then we obtain (\ref{eqn:RJ07_16}) through the scalar curvature bound (\ref{eqn:RB19_2}) and the entropy estimate in Lemma~\ref{lma:MJ25_1}. 
The proof of Claim~\ref{clm:RB21_5} is complete. \\

We further put the flow structure into consideration and finish the proof of (\ref{eqn:RB21_7}) and (\ref{eqn:RB21_11}).

Recall the terminology(cf. (\ref{eqn:PA05_1}) and (\ref{eqn:PA05_2})) of $A_{+.r}$ and $A_{-,r}$ in Lemma~\ref{lma:PA05_1}.
On the one hand, $A_{+,r}$ is bounded from above by $1+\sqrt{r}$ for each $r \in (0, r_{0})$, as indicated by (\ref{eqn:RB21_6}). 
On the other hand, the curvature estimate (\ref{eqn:ML28_1}) in the improved pseudo-locality theorem implies that $A_{-,r} \geq 1-\psi(\epsilon|n)$.
Combining them together, we obtain 
\begin{align}
  A_{+,r}- A_{-,r} \leq  \sqrt{r} + \psi(\epsilon|n),  \label{eqn:RB21_9} 
\end{align}
which in turn yields that
\begin{align}
  r^{-2n} \int_{0}^{r^2}\int_{B(y,r)} |R| dvdt  \leq \psi(\epsilon|n) + \omega_{2n} (A_{+,r}-A_{-,r}) \leq \psi(\epsilon|n) +  \sqrt{r}.  \label{eqn:RB21_10} 
\end{align}
Now we cover $V$ by geodesic balls $B_{g}(z_i, r)$ such that $B_{g}(z_i,0.5r)$ disjoint to each other. Because of (\ref{eqn:RB21_6}), the total number of such balls is bounded by $C(n,L)r^{-2n}$.  
Therefore, we have 
\begin{align*}
  \int_{0}^{r^2} \int_{V} |R| dvdt \leq \sum_{i} \int_{0}^{r^2} \int_{B(z_i, r)} |R| dvdt 
  \leq C(n,L) \cdot r^{-2n}  \cdot r^{2n} \cdot \left\{ \psi(\epsilon|n) + \sqrt{r} \right\}.  
\end{align*}
Recall that (\ref{eqn:ML28_1}) can be rewritten as $t|Rm|(x,t)<\psi(\epsilon|n)$ in the current setting. 
Therefore, the almost flatness of $(V, g(t))$ for $t \in [r^2, 1]$ then implies that 
\begin{align*}
  \int_{r^2}^{1} \int_{V} |R| dvdt \leq \psi(\epsilon|n,L) \log r^{-2}.  
\end{align*}
Combining the previous two steps, we obtain
\begin{align*}
  \int_{0}^{1} \int_{V} |R| dvdt \leq \psi(\epsilon|n,L) \log r^{-2} + \sqrt{r}. 
\end{align*}
Note that the above estimate holds for arbitrary $r \in (0, r_0)$. In particular, we can choose $r=\psi^2(\epsilon|n,L)$. It follows that
\begin{align}
  \int_{0}^{1} \int_{V} |R| dvdt \leq 2 \psi |\log \psi|. \label{eqn:RJ07_15} 
\end{align}
Since $\psi |\log \psi| \to 0$ as $\psi \to 0$, we arrive at (\ref{eqn:RB21_7}) by setting $2\psi |\log \psi|$ as the new $\psi$. 

The estimate (\ref{eqn:RB21_11}) basically follows from the application of Lemma~\ref{lma:PA05_1}.
By Claim~\ref{clm:RB21_5}, we know $\mathbf{er}_{\epsilon_{0}}(y)>r_{0}$.
Fix $r \in (0,r_{0})$. It follows from Lemma~\ref{lma:PA05_1} that 
\begin{align*}
  \sup_{x,y \in V} \left|d_{g(0)}(x,y)-d_{g(r^2)}(x,y) \right| < \psi(\epsilon|r,n,L).
\end{align*}
The improved curvature estimate (\ref{eqn:ML28_1}) then implies that
\begin{align*}
  \sup_{x,y \in V} \left|d_{g(r^{2})}(x,y)-d_{g(1)}(x,y) \right| < \psi(\epsilon|n,L) |\log r|. 
\end{align*}
Similar to the deduction of (\ref{eqn:RJ07_15}), by combining the previous two steps, we can appropriately choose $r$ 
and modify $\psi(\epsilon|n,L)$ to obtain (\ref{eqn:RB21_11}).
In summary, we have already finished the proof of both (\ref{eqn:RB21_7}) and (\ref{eqn:RB21_11}).\\

Up to now, the scale $L$ is fixed. 
In (\ref{eqn:RJ07_0}) and (\ref{eqn:RJ07_1}), the definitions of $U$ and $V$ depends on the scale $L$.
For completeness of information, we now denote them by $U_{L}$ and $V_{L}$.
We now let $\epsilon \to 0$. Note that we can always fix $L$ first and then let $L \to \infty$.
Then we shall use the estimates developed in the previous steps to guarantee the convergence of the metric tensors in proper topology.\\

\noindent
\textit{Step 6. As $\epsilon \to 0$, we have
  \begin{align}
    (\varphi_{1}, H) \longright{C_{local}^{\infty}} (\hat{\varphi}_{1}, \hat{H})=(0,0).    
    \label{eqn:RJ09_1}  
  \end{align}
}

The existence of smooth convergence and smooth limit is guaranteed by estimate (\ref{eqn:RJ08_24}).
We only need to show the limit is zero.
As $\epsilon \to 0$, it follows from the improved curvature estimate (\ref{eqn:ML28_1}) and injectivity radius estimate (\ref{eqn:RH19_3}) that $g_{\hat{\varphi}_{1}}$ is a smooth flat Euclidean metric.
Furthermore, $g_{\hat{\varphi}_{1}}$ is uniformly quasi-isometric to the background Euclidean metric(cf. (\ref{eqn:RB23_4})).
Taking smooth limits of equations (\ref{eqn:RJ08_21}) and (\ref{eqn:RJ08_22}), we arrive at 
\begin{align}
  &\det \left( \delta_{\alpha \bar{\beta}} + (\hat{\varphi}_{1})_{\alpha \bar{\beta}} \right)=e^{\hat{H}}, \label{eqn:RJ08_26} \\
  &-\Delta_{\hat{\varphi}_{1}} \hat{H}=0.  \label{eqn:RJ08_27}
\end{align}
As $\hat{H}$ is a bounded smooth harmonic function(cf.~(\ref{eqn:RJ08_28}) and (\ref{eqn:RJ08_27})) on the Euclidean space $(\C^{n}, g_{\hat{\varphi}_{1}})$, 
it follows from classical Liouville theorem that $\hat{H} \equiv \hat{H}(0)=0$, thanks to (\ref{eqn:RJ08_12}). 
Then (\ref{eqn:RJ08_26}) is simplified as
\begin{align*}
  \det \left( \delta_{\alpha \bar{\beta}} + (\hat{\varphi}_{1})_{\alpha \bar{\beta}} \right)=1. \label{eqn:RJ08_29}
\end{align*}
Since identity map induces uniform quasi-isometry(cf. (\ref{eqn:RB23_4})) between $g_{\hat{\varphi}_1}$ and $g_{E}$, we know(cf. (3.16) in the proof of Lemma 3.5 of Chen-Wang~\cite{CW17B}) that $\hat{\varphi}_{1}$ is a quadratic function. 
Consequently, we have
\begin{align}
  \omega_{\hat{\varphi}_{1}} \equiv \omega_{\hat{\varphi}_{1}}(0)=\omega_{E}.  \label{eqn:RJ08_30}
\end{align}
It follows from smooth convergence and the restrictions (\ref{eqn:RJ08_18}) that $\hat{\varphi}_{1} \equiv 0$.
Thus we finish the proof of (\ref{eqn:RJ09_1}).\\

\noindent
\textit{Step 7. As $\epsilon \to 0$, we have
  \begin{align}
    \varphi_{0} \longright{C_{local}^{3, \frac{1}{3}}} \hat{\varphi}_{0},  \quad  F \longright{C_{local}^{1, \frac{1}{3}}} \hat{F} \equiv 0,     
    \label{eqn:RJ09_2}  
  \end{align}
  where $\hat{\varphi}_{0}$ is a smooth function. 
}

The existence of proper convergence topology and limits is guaranteed by (\ref{eqn:RB21_1}) and  (\ref{eqn:RJ06_5}).
We only need to prove $\hat{\varphi}_{0}$ is smooth and $\hat{F} \equiv 0$. 
Since the convergence of $\varphi_{0}$ happens in $C^{3,\frac{1}{3}}$-topology,  the $C^{1,1}$-estimate (\ref{eqn:RJ06_9}) also holds on the limit.
\begin{align}
  \frac{1}{C(n)}\delta_{\alpha \bar{\beta}} \leq  \delta_{\alpha\bar{\beta}} + \frac{\partial^2 \hat{\varphi}_{0}}{\partial z^{\alpha} \partial z^{\bar{\beta}}}  \leq C(n)  \delta_{\alpha \bar{\beta}}.  \label{eqn:RJ08_9}
\end{align}
Since $\epsilon \to 0$ as $L \to \infty$, it follows from (\ref{eqn:RB19_2}) that the scalar curvature tends to zero. 
By taking limits of (\ref{eqn:RG19_1}) and (\ref{eqn:RJ06_4}), we know $\hat{F}$ and $\hat{\varphi}_{0}$ satisfy
\begin{align}
  &\det \left( \delta_{\alpha \bar{\beta}} + (\hat{\varphi}_{0})_{\alpha \bar{\beta}} \right)=e^{\hat{F}}, \label{eqn:RJ07_17} \\
  &-\Delta_{\hat{\varphi}_{0}} \hat{F}=0,  \label{eqn:RJ07_18}
\end{align}
in the distribution sense.  In light of the uniform quasi-isometry of $g_{\hat{\varphi}_{0}}$ and $g_{E}$ by (\ref{eqn:RJ08_9}),
both (\ref{eqn:RJ07_17}) and (\ref{eqn:RJ07_18}) are uniformly elliptic. 
Also note that $\hat{F}$ is a $C_{local}^{1,\frac{1}{3}}$-function, $\hat{\varphi}_{0}$ is a $C_{local}^{3,\frac{1}{3}}$-function.
Alternative applying (\ref{eqn:RJ07_18}) and (\ref{eqn:RJ07_17}), we can improve the regularity of $\hat{F}$ and $\hat{\varphi}_{0}$ 
step by step(cf.~Proposition 1.2 of~\cite{ChenCheng1} and Lemma 4.11 of Li-Li-Wang~\cite{LiLiWang18} for similar arguments.)
It turns out that both $\hat{F}$ and $\hat{\varphi}_{0}$ must be smooth functions. 

We now show $\hat{F} \equiv 0$. Note that
\begin{align*}
  \int_{0}^{1} R(x,t)dt=\log \frac{dv_{g(0)}}{dv_{g(1)}}= \log \frac{\omega_{0}^{n}}{\omega_{1}^{n}}=\log \left\{ \frac{\omega_{0}^{n}}{\omega_{E}^{n}} \cdot \frac{\omega_{E}^{n}}{\omega_{1}^{n}} \right\}
  =F-H.
\end{align*}
By the uniform volume element equivalence (\ref{eqn:RJ06_7}), the almost Einstein condition (\ref{eqn:RB21_7}) then implies that 
\begin{align*}
  \int_{V_{L}} |F-H| \omega_{E}^{n} < \psi(\epsilon|n,L).
\end{align*}
Taking limit of the above equation and noting that $\hat{H} \equiv 0$, we obtain
\begin{align*}
 \int_{\C^{n}} |\hat{F}| \omega_{E}^{n}=\int_{\C^{n}} |\hat{F}-\hat{H}| \omega_{E}^{n}=0, 
\end{align*}
which forces $\hat{F} \equiv 0$ as $\hat{F} \in C_{local}^{1,\frac{1}{3}}$.\\

\noindent
\textit{Step 8. As $\epsilon \to 0$, we have
  \begin{align}
    g_{\varphi_{0}} \longright{C_{local}^{1, \frac{1}{3}}} g_{E}.    
    \label{eqn:RJ09_3}  
  \end{align}
}

Since $\hat{F} \equiv 0$, the equation (\ref{eqn:RJ07_17}) can be simplified as 
\begin{align}
  \det \left( \delta_{\alpha \bar{\beta}} + (\hat{\varphi}_{0})_{\alpha \bar{\beta}} \right)=1. \label{eqn:RJ08_15} 
\end{align}
Since we already know $\hat{\varphi}_{0}$ is smooth, we can apply the argument as the one nearby (\ref{eqn:RJ08_30}) to obtain $\omega_{\hat{\varphi}_{0}} \equiv \omega_{\hat{\varphi}_{0}}(0)$.
This means that $g_{\hat{\varphi}_{0}} \equiv G$ where $G$ is a constant matrix satisfying $\det G=1$. 
In the remainder of this step, we aim to show that $G$ is the identity matrix. 
For this purpose, we fix an arbitrary vector $\vec{v} \in \C^{n}$ such that $|\vec{v}|=1$. Let $q=\vec{v}$, 
$\gamma$ be the unit speed shortest geodesic(under metric $g_{E}$) connecting $0$ and $q$ such that $\gamma(0)=0$ and $\gamma(1)=q$.
Then $\gamma'(0)=\vec{v}$. It follows that
\begin{align}
  d_{g(0)}\left( 0, q \right) &\leq |\gamma|_{g(0)}=\int_{0}^{1} \sqrt{ \langle \gamma'(\theta), \gamma'(\theta)\rangle_{g(0)}} d\theta =\int_{0}^{1} \sqrt{\langle \vec{v}, \vec{v} \rangle_{g(0)}} d\theta
  \to \sqrt{G(\vec{v}, \vec{v})}. 
  \label{eqn:RJ09_4}  
\end{align}
As $g_{\varphi_{1}} \longright{C^{\infty}} g_{\hat{\varphi}_{1}}=g_{E}$,  it is clear that $ d_{g(1)}\left( 0, q \right) \to 1$, which implies $d_{g(0)}(0,q) \to 1$, 
via the metric distortion estimate (\ref{eqn:RB21_11}).
Then it follows from (\ref{eqn:RJ09_4}) that $G(\vec{v}, \vec{v}) \geq 1$. By the arbitrary choice of the unit vector $\vec{v}$, we obtain that $G_{\alpha \bar{\beta}} \geq \delta_{\alpha \bar{\beta}}$.
This inequality forces $G=\delta_{\alpha \bar{\beta}}$ as $\det G=1$.
Namely, we have $G_{\alpha \bar{\beta}}=\delta_{\alpha \bar{\beta}}$, which implies $g_{\hat{\varphi}_{0}} \equiv g_{E}$.
Consequently, (\ref{eqn:RJ09_3}) follows from (\ref{eqn:RJ09_2}).\\ 

\noindent
\textit{Step 9. Derive the desired contradiction.}\\

From previous steps, we already know that $g_{\varphi_{0}} \longright{C_{local}^{1,\frac{1}{3}}}  g_{E}$ and $g_{\varphi_{1}} \longright{C_{local}^{\infty}} g_{E}$.
Therefore, $g_{\varphi_{0}}(0)$ and $g_{\varphi_{1}}(0)$ could be arbitrarily close in the sense that
\begin{align*}
  |g_{\varphi_{0}}-g_{\varphi_{1}}|_{g_{\varphi_1}}(0) < \psi(\epsilon|n) << \xi. 
\end{align*}
Since $0$ is the coordinate of the point $y_{0} \in M$, the above inequality contradicts our assumption (\ref{eqn:RB19_4}). 
This contradiction establishes the proof of the theorem. 
\end{proof}

\begin{corollary}(\textbf{Metric bi-Lipschitz equivalence near the center point})
  Same conditions as in Theorem~\ref{thm:RJ08_0}. By shrinking $\epsilon$ if necessary, we have 
  \begin{align}
     e^{-2\xi} g(y,0) \leq g(y, t) \leq  e^{2\xi} g(y, 0), \quad \forall \; y \in B(x_0, 0.5r), \quad t \in [0, (\epsilon r)^2]. 
     \label{eqn:RD29_7}   
  \end{align}
\label{cly:RD29_4}
\end{corollary}

\begin{proof}
 Without loss of generality,  we may assume $r=\epsilon^{-1}$ up to rescaling. 
 Note that the flow exists on $t \in [0, \epsilon^{-2}]$ and satisfies
    \begin{align}
       \begin{cases}
	 &\boldsymbol{\bar{\nu}}(B(x_0, \epsilon^{-1}), g, \epsilon^{-2}) \geq -\epsilon;\\
	 &|R| \leq  \epsilon^2, \quad on \;  B(x_0, \epsilon^{-1}).      
       \end{cases}
   \label{eqn:RD29_5}
   \end{align}
   For each point $y \in B(x_0, 0.5 \epsilon^{-1})$, the ball $B(y, 0.5 \epsilon^{-1})$ is a subset of $B(x_0, \epsilon^{-1})$. 
   By the monotonicity of the $\bar{\boldsymbol{\nu}}$-functional, we have 
   \begin{align}
       \begin{cases}
	 &\boldsymbol{\bar{\nu}}(B(y, (2 \epsilon)^{-1}), g, (2\epsilon)^{-2}) \geq -\epsilon;\\
	 &|R| \leq  (2 \epsilon)^2, \quad on \;  B(x_0, (2 \epsilon)^{-1}).   
       \end{cases}
   \label{eqn:RB29_6}
   \end{align}
   By choosing $\epsilon$ sufficiently small, we can apply Theorem~\ref{thm:RJ08_0} to obtain the following metric equivalence
   \begin{align*}
     e^{-\xi} g(y,0) \leq g(y, t) \leq  e^{\xi} g(y, 0), \quad \forall \; y \in B(x_0, (2\epsilon)^{-1}), \quad t \in [0, 0.25]. 
   \end{align*}
   Note that $|Rm|(y,t)<\psi(\epsilon|n)$ for each $y \in B(x_0, (2\epsilon)^{-1})$ and $t \in [0.25, 1]$.  We can combine the previous two  metric equivalence estimates to obtain
   \begin{align*}
      e^{-2\xi} g(y,0) \leq g(y, t) \leq  e^{2\xi} g(y, 0), \quad \forall \; y \in B(x_0, (2\epsilon)^{-1}), \quad t \in [0, 1], 
   \end{align*}
   whence we arrive at (\ref{eqn:RD29_7}) as $r=\epsilon^{-1}$. 
\end{proof}

We have shown the metric bi-Lipschitz estimates in Theorem~\ref{thm:RJ08_0} and Corollary~\ref{cly:RD29_4}.
We now move on to show that the scalar curvature bounds are preserved for a short time period. 

\begin{proposition}(\textbf{Growth rate control of the scalar curvature})
  Same conditions as in Theorem~\ref{thm:RJ08_0}.  Then we have 
  \begin{align}
    &\left\{\inf_{y \in B(x_0,r)} R(y,0) \right\} r^2- C\left( r^{-2}t \right)^{\frac{1}{5}}
    <R(x_0,t) r^2<\left\{\sup_{y \in B(x_0,r)} R(y,0) \right\} r^2 + C\left( r^{-2}t \right)^{\frac{1}{5}}  \label{eqn:RD18_2} 
  \end{align}
  for each $t \in [0, (\epsilon r)^{2}]$.   Here $C$ can be chosen as $\epsilon^{-\frac{12}{5}}$. 
   
\label{prn:RC13_1}
\end{proposition}

\begin{proof}
  Since (\ref{eqn:RD18_2}) is scaling invariant, we may assume $r=\epsilon^{-1}$ up to rescaling. Define
\begin{align}
  \sigma_{a} \coloneqq \inf_{x \in B(x_0,\epsilon^{-1})} R(x,0), \quad \sigma_{b} \coloneqq \sup_{x \in B(x_0,\epsilon^{-1})} R(x,0).  
  \label{eqn:RD28_4}
\end{align}
It follows from assumption (\ref{eqn:RB16_1}) that
 \begin{align}
   |R|(x,0) \leq \max\left\{ |\sigma_{a}|, |\sigma_{b}| \right\} \leq \epsilon^{2}, \quad \forall \;  x \in B(x_0,\epsilon^{-1}). \label{eqn:RJ11_12}
 \end{align}
 Then (\ref{eqn:RD18_2}) reads as
 \begin{align}
   \sigma_{a}- t^{\frac{1}{5}}<R(x_0,t)<\sigma_{b}+ t^{\frac{1}{5}}.
   \label{eqn:RD28_3}
 \end{align}
 We first focus our attention on the second inequality of (\ref{eqn:RD28_3}).
 In light of the improved curvature estimate (\ref{eqn:ML28_1}) and the improved injectivity radius estimate (\ref{eqn:RH27_4}), 
 we can find a biholomorphic map $\Phi: B(x_0, L) \to U \subset \C^n$ satisfying the conditions (\ref{eqn:RB23_4a}), (\ref{eqn:RB23_4b}) and (\ref{eqn:RB23_4}), at time $t=1$. 
 Here $L>>10000$ is a large constant to be determined later. Then we set $V=\Phi(B(x_0, 0.1L))$ as done in (\ref{eqn:RJ07_1}). 
 By identifying $x_0$ with $0$, $x$ with $\Phi(x)$, $g(t)$ with $\Phi_{*}(g(t))$, we restrict the K\"ahler Ricci flow onto $U$ as done previously in step 2 of the proof of Theorem~\ref{thm:RJ08_0}.

  The local K\"ahler Ricci flow is completely determined by the local potential function $\varphi$, which can be adjusted by adding a linear function.  
  We define
  \begin{align}
    &a_{\alpha} \coloneqq  \frac{\partial \dot{\varphi}_{0}}{\partial z^{\alpha}}, \label{eqn:RD18_9}\\
    &\tilde{\varphi} \coloneqq \varphi -\varphi_{0}(0) - t \left(F(0)+a_{\alpha} z^{\alpha}+ \overline{a_{\alpha}z^{\alpha}}  \right),    \label{eqn:RD18_10}
  \end{align}
  where $F=\dot{\varphi}_{0}$ by (\ref{eqn:RJ08_19}).  
  In light of (\ref{eqn:RJ08_18}), it is clear that 
  \begin{align}
    \tilde{\varphi}_{0}(0)=\frac{\partial \tilde{\varphi}_{0}}{\partial z^{\alpha}}(0)=\frac{\partial \tilde{\varphi}_{0}}{\partial z^{\bar{\alpha}}}=0, 
     \quad \forall \; \alpha \in \left\{ 1, 2, \cdots, n \right\}.
  \label{eqn:RJ10_2}  
  \end{align}
  Taking time derivative of (\ref{eqn:RD18_10}) yields that
  \begin{align}
    \dot{\tilde{\varphi}}=\dot{\varphi} - \left( F(0)+a_{\alpha} z^{\alpha}+ \overline{a_{\alpha}z^{\alpha}} \right).   \label{eqn:RJ11_8}
  \end{align}
  By the choice of $a_{\alpha}$ in (\ref{eqn:RD18_9}), we obtain  
  \begin{align}
     \dot{\tilde{\varphi}}_{0}(0)=\frac{\partial \dot{\tilde{\varphi}}_{0}}{\partial z^{\alpha}}(0)=\frac{\partial \dot{\tilde{\varphi}}_{0}}{\partial z^{\bar{\alpha}}}=0,
     \quad \forall \; \alpha \in \left\{ 1, 2, \cdots, n \right\}.
     \label{eqn:RJ11_3}   
  \end{align}
  In view of (\ref{eqn:RJ09_2}), we have 
  \begin{align*}
    |F(0)|+ \sum_{\alpha=1}^{n} |a_{\alpha}| <\psi(\epsilon|n,L),
  \end{align*}
  which means the difference between $\tilde{\varphi}$ and $\varphi$ is very small.
  Combining it with (\ref{eqn:RJ06_5}) and (\ref{eqn:RJ09_2}), we have
  \begin{align*}
    \norm{\dot{\tilde{\varphi}}_{0}}{C^{1,\frac{1}{2}}(V)} < C(n,L), \quad \norm{\dot{\tilde{\varphi}}_{0}}{C^{1,\frac{1}{3}}(V)} < \psi(\epsilon|n,L).    
  \end{align*}
  Since $\nabla \dot{\tilde{\varphi}}_{0}(0)=0$ by (\ref{eqn:RJ11_3}), it follows from the above inequality that
  \begin{align*}
    |\nabla \dot{\tilde{\varphi}}_{0}|(x) \leq \min\left\{ \psi(\epsilon|n,L), C(n,L) \sqrt{|x|} \right\}, \quad \forall \; x \in V. 
   \end{align*}
  Consequently, we obtain
  \begin{align}
    |\nabla \dot{\tilde{\varphi}}_{0}|^2(x) \leq \psi(\epsilon|n,L) \cdot |\nabla \dot{\tilde{\varphi}}_{0}|(x)  \leq \psi(\epsilon|n,L) \sqrt{|x|}.  \label{eqn:RJ11_10} 
  \end{align}
  Elementary calculation shows that 
  \begin{align*}
    \square \left( R+|\nabla \dot{\tilde{\varphi}}|^2 \right)=-|\nabla \nabla \dot{\tilde{\varphi}}|^2-|\bar{\nabla} \bar{\nabla} \dot{\tilde{\varphi}}|^2 \leq 0, 
  \end{align*}
  which means that $R+|\nabla \dot{\tilde{\varphi}}|^2$ is locally a sub-heat-solution.
  In order to apply the localized maximum principle(cf. Theorem~\ref{thm:RD10_1}) on $R+|\nabla \dot{\tilde{\varphi}}|^2$, 
  we need to adjust $R+|\nabla \dot{\tilde{\varphi}}|^2$ by a constant and estimate its initial value and growth rate. 
  This leads to the following claim. 
   
   \begin{claim}
    The following estimates hold true: 
   \begin{align}
       &R(x,0)-\sigma_{b}+|\nabla \dot{\tilde{\varphi}}|^2(x,0) \leq \psi(\epsilon|n,L) \sqrt{|x|}, \quad \forall \; x \in V;    \label{eqn:RC13_1}\\
       &t\left(R(x,t)-\sigma_{b}+|\nabla \dot{\tilde{\varphi}}|^2(x,t) \right)<\psi(\epsilon|n,L), \quad  \forall\; x \in V, t \in (0,1).    \label{eqn:RD18_12}
   \end{align}
    \label{clm:RD18_11}
   \end{claim}

 Note that (\ref{eqn:RC13_1}) follows directly from (\ref{eqn:RJ11_10}) and (\ref{eqn:RJ11_12}).
 Thus it suffices to prove (\ref{eqn:RD18_12}). 
 As the linear term does not contribute to Laplacian, it follows from (\ref{eqn:RJ08_4}) and (\ref{eqn:RD18_10}) that 
   \begin{align*}
     \Delta_{g_{\tilde{\varphi}}} \dot{\tilde{\varphi}}=  \Delta_{\varphi} \dot{\varphi}=-R. 
   \end{align*}
 Rescaling the metric $g_{\tilde{\varphi}_{t}}$ by $t^{-1}$, we can apply the Schauder estimate on the above equation.
 It follows from the improved curvature estimate (\ref{eqn:ML28_1}) that $t|\nabla \dot{\tilde{\varphi}}|^2< \psi(\epsilon|n,L)$ on $V \times (0, 1]$.
 Applying (\ref{eqn:ML28_1}) again directly, we know $t|R|(x,t)<\psi(\epsilon|n,L)$. Also note that $|\sigma_{b}|<\epsilon^2$ by (\ref{eqn:RJ11_12}). 
 Combining the previous steps together, we arrive at (\ref{eqn:RD18_12}). Therefore, the proof of Claim~\ref{clm:RD18_11} is complete.

 Now we are ready to apply Theorem~\ref{thm:RD10_1} on the sub-heat-solution $f=R-\sigma_{b}+|\nabla \dot{\tilde{\varphi}}|^2$. 
 In Theorem~\ref{thm:RD10_1}, we choose $r=t^{\frac{2}{5}}$ and $A=1000$. By choosing $L>>100A$, we have
 \begin{align*}
   B_{g(0)}(x_0, 10A\rho) \subset B_{g(0)}(x_{0}, 10A) \subset  \Phi^{-1}(V).  
 \end{align*}
 In light of Claim~\ref{clm:RD18_11}, we see that $f=R-\sigma_{b}+|\nabla \dot{\tilde{\varphi}}|^2$ satisfies the desired initial condition and
 growth condition on the ball $B_{g(0)}(x_0, 10A\rho)$. It follows from Theorem~\ref{thm:RD10_1} that 
 \begin{align*}
     R(x_0,t)-\sigma_{b} +|\nabla \dot{\tilde{\varphi}}|^{2}(x_0, t)&< \left\{ \psi (1+A^{-1}t) + A^{-4} \right\} t^{\frac{1}{5}}< t^{\frac{1}{5}},
 \end{align*}
 and thus we arrive at the second inequality of (\ref{eqn:RD28_3}).

 It remains to show the first inequality of (\ref{eqn:RD28_3}). Actually, this part can be deduced in a similar but simpler way.
 Note that $\sigma_a-R \leq 0$ at $t=0$ and $\square \left( \sigma_a -R \right)=-2|Rc|^2 \leq 0$ along the flow. 
 Clearly, $\sigma_a-R$ satisfies the growth rate condition. Therefore, Theorem~\ref{thm:RD10_1} can be applied on $\sigma_a-R$ to obtain that  
 \begin{align*}
   \sigma_{a}-R(x_0,t) < A^{-4}t <t<t^{\frac{1}{5}}, 
 \end{align*}
 which implies the first inequality of (\ref{eqn:RD28_3}). 
 Therefore, we finish the proof of (\ref{eqn:RD28_3}).   The proof of the proposition is complete. 
 \end{proof}

 The scalar curvature bounds in Proposition~\ref{prn:RC13_1} can be used to improve the metric bi-Lipschitz estimates to $C^{1,\alpha}$-estimates.
 The key input is the application of the theory of complex Monge-Amp\`{e}re equation. 

 \begin{theorem}(\textbf{Metric $C^{1,\alpha}$-equivalence near the center point})
   For each  $\alpha \in (0,1)$ and $Q>1$, there is a constant $\bar{\epsilon}=\bar{\epsilon}(n,\alpha,Q)$ with the following property. 

  Suppose $\left\{ (M^{n}, g(t), J), 0 \leq t \leq 1 \right\}$ is a Ricci flow solution, $x_0 \in M$ and $r \in (0,1)$.
  Suppose (\ref{eqn:RB16_1}) hold for $\epsilon \in (0, \bar{\epsilon})$,  then there is a biholomorphic map $\Phi: B_{g(0)}(x_0, \epsilon r) \to U \subset  \C^n$ such that
  \begin{align}
    & B(0, Q^{-1}) \subset U \subset B(0, Q);  \label{eqn:RC15_10} \\
    & \norm{(\epsilon r)^{-2}\Phi_{*} g(t)}{C^{1,\alpha}(U)} < Q,  \;  \forall \; t \in [0, \epsilon^2 r^2].   \label{eqn:RC15_11}
  \end{align}
\label{thm:RC10_1}
\end{theorem}

\begin{proof}
  This theorem basically follows from the combination of Theorem~\ref{thm:RJ08_0} and Proposition~\ref{prn:RC13_1}. Since no new ingredients will be needed in the proof, we shall be sketchy. 
  We choose the biholomorphic map $\Phi$ as done nearby (\ref{eqn:RB23_4}) in the proof of Theorem~\ref{thm:RJ08_0}, with $\frac{1}{2}$ being replaced by any $\alpha \in (0,1)$ and $L=100$.
  By (\ref{eqn:RJ09_1}), (\ref{eqn:RJ08_30}) and (\ref{eqn:RJ09_3}), we see that the constant $C$ in (\ref{eqn:RB23_4}) could be replaced by an arbitrary number $Q>1$, whenever $\epsilon$ is sufficiently small.
  Therefore  (\ref{eqn:RC15_10}) is proved.   It follows from Proposition~\ref{prn:RC13_1} that the scalar curvature is uniformly bounded on $B(x_0, 2 \epsilon r) \times [0, (\epsilon r)^2]$.
  Following the argument in Claim~\ref{clm:RJ06_3} and in the proof of (\ref{eqn:RJ09_3}), we obtain (\ref{eqn:RC15_11}).   
\end{proof}

Then we discuss some applications of the improved pseudo-locality theorem in K\"ahler geometry. 

\begin{theorem}[\textbf{Compactness of moduli of K\"ahler manifolds with proper bounds}]
  Suppose $(M_i^{n}, g_i, J_i)$ is a sequence of K\"ahler manifolds of complex dimension $n$ satisfying 
  \begin{align}
    \mathbf{er}(M,g) \geq \iota, \quad \diam(M,g) \leq D, \quad  \sigma_a \leq R \leq \sigma_b.   \label{eqn:RD29_1}
  \end{align}
  Here $\mathbf{er}=\mathbf{er}_{\epsilon_0}$ is the entropy radius defined in Definition~\ref{dfn:RH26_3} for some sufficiently small $\epsilon_0=\epsilon_0(n)$. 
  For each $\alpha \in (0,1)$,  by taking subsequence if necessary, we have 
  \begin{align}
    \left( M_i, g_i, J_i \right) \longright{C^{1,\alpha}-Cheeger-Gromov} \left( M_{\infty}, g_{\infty}, J_{\infty} \right),
    \label{eqn:RC11_1}  
  \end{align}
  where $(M_{\infty}, g_{\infty}, J_{\infty})$ is a $C^{1,\alpha}$-K\"ahler manifold.  
  Furthermore, if we assume $(M_{\infty}, g_{\infty}, J_{\infty})$ is a smooth K\"ahler manifold, then it also satisfies estimate (\ref{eqn:RD29_1}). 
\label{thm:RB21_12}
\end{theorem}

\begin{proof}
  For each $t_0$ sufficiently small(determined by $n$ and $\iota$ in (\ref{eqn:RD29_1})), it follows from the geometric estimates (\ref{eqn:ML28_1}), (\ref{eqn:ML28_2}) in the pseudo-locality theorem
  and Shi's estimates of curvature derivatives that
  \begin{align*}
         \left( M_i, g_i(t_0), J_i \right) \longright{C^{\infty}-Cheeger-Gromov} \left( M_{\infty}, g_{\infty}(t_0), J_{\infty} \right).
  \end{align*}
  By definition of $C^{\infty}$-Cheeger-Gromov convergence, the above equation guarantees that there exist diffeomorphisms $\varphi_{i}: M_{\infty} \to M_{i}$ such that
  \begin{align}
    \varphi_{i}^{*}\left( g_{i}(t_{0}) \right) \overset{C^{\infty}}{\rightarrow} g_{\infty}(t_0), \quad 
    \varphi_{i}^{*}\left( J_{i} \right) \overset{C^{\infty}}{\rightarrow} J_{\infty}. 
    \label{eqn:RI30_1}
  \end{align}
  In view of Theorem~\ref{thm:RC10_1}, we know $\norm{g_i(0)-g_i(t_0)}{C^{1,\alpha}(g_i(t_0))}$ are uniformly bounded.
  Since $g_i(0)=g_{i}$, we consequently have
  \begin{align}
    \norm{\varphi_{i}^{*}(g_i)- \varphi_{i}^{*}(g_i(t_0))}{C^{1,\alpha}(\varphi_{i}^{*}(g_i(t_0)))}
    =\norm{\varphi_{i}^{*}(g_i(0))- \varphi_{i}^{*}(g_i(t_0))}{C^{1,\alpha}(\varphi_{i}^{*}(g_i(t_0)))}<C.   \label{eqn:RI30_2}
  \end{align}
  Combining (\ref{eqn:RI30_2}) with (\ref{eqn:RI30_1}), we obtain
  \begin{align*}
    \norm{\varphi_{i}^{*}(g_i)- g_{\infty}(t_{0})}{C^{1,\alpha}(g_{\infty}(t_{0}))}<C.    
  \end{align*}
  By shrinking $\alpha$ if necessary, we then obtain 
  \begin{align}
    \varphi_{i}^{*}\left( g_{i} \right) \overset{C^{1,\alpha}}{\rightarrow} g_{\infty}, \quad 
    \varphi_{i}^{*}\left( J_{i} \right) \overset{C^{\infty}}{\rightarrow} J_{\infty}. 
    \label{eqn:RI30_3}
  \end{align}
  In particular, we obtain (\ref{eqn:RC11_1}). 

  If we further assume that $(M_{\infty}, g_{\infty}, J_{\infty})$ is a smooth K\"ahler manifold, then we can run the K\"ahler Ricci flow starting from $(M_{\infty}, g_{\infty}, J_{\infty})$.
    The flow exists for a definite time period $[0, \delta_0]$.  It is important to observe that $\left\{ (M_{\infty}, g_{\infty}(t), J_{\infty}), 0 \leq t \leq \delta_0 \right\}$ is a smooth space-time.
    Fix each point $x \in M_{\infty}$ and $t \in (0, \delta_0]$,  the commutativity of Gromov-Hausdorff convergence and the Ricci flow implies that
    \begin{align*}
      \sigma_a-C t^{\frac{1}{5}} \leq R(x,t)=\lim_{i \to \infty} R(x_i,t) \leq \sigma_{b} + C t^{\frac{1}{5}}, 
    \end{align*}
    where $x_i \in M_i$ and $x_i \to x$ as $(M_i, g_i(t))$ converges to $(M_{\infty}, g_{\infty}(t))$. Then we let $t \to 0^{+}$ in the above inequality and obtain $\sigma_a \leq R(x,0) \leq \sigma_b$.
    By the arbitrary choice of $x$, we arrive at the scalar curvature bound of $(M_{\infty}, g_{\infty})$ in (\ref{eqn:RD29_1}). The entropy radius lower bound follows from the $C^{1,\alpha}$-convergence
    and Lemma~\ref{lma:MJ25_1}.  The diameter upper bound is obvious. Therefore, all the estimates in (\ref{eqn:RD29_1}) hold for $(M_{\infty}, g_{\infty}, J_{\infty})$. 
\end{proof}

In Theorem~\ref{thm:RB21_12}, the assumption that $(M_{\infty}, g_{\infty}, J_{\infty})$ is a smooth K\"ahler manifold seems strong.
The following corollary shows that if the scalar curvatures oscillate in  magnitudes tending to zero, then the limit is automatically a smooth K\"ahler manifold.

\begin{corollary}[\textbf{Smooth cscK limit spaces}]
  Same conditions as in Theorem~\ref{thm:RB21_12}. 
  If we further assume $K-\delta_i \leq R_{M_i} \leq K+\delta_i$ for $\delta_i \to 0$, 
  then $(M_{\infty}, g_{\infty}, J_{\infty})$ is a smooth K\"ahler manifold whose scalar curvature equals $K$ everywhere.  
  \label{cly:RD29_3}
\end{corollary}

\begin{proof}
    Fix each point $x \in M_{\infty}$, we can regard it as the limit point of $x_i \in M_i$.  Therefore, by assumption and Theorem~\ref{thm:RC10_1}, we can find uniform  $C^{1,\alpha}$-holomorphic chart
    near $x_i$.  In the holomorphic chart, we have 
    \begin{align*}
      -\Delta \log \det g_{\beta \bar{\gamma}} = R.
    \end{align*}
    Since $\norm{R-K}{C^0(M_i)} \to 0$, in the holomorphic coordinate chart around $x$,  the metric tensor $g_{\beta \bar{\gamma}}$ has uniform $C^{1,\alpha}$ estimate and satisfies the 
    equation
    \begin{align*}
        -\Delta \log \det g_{\beta \bar{\gamma}}-K=0 
    \end{align*}
    in the distribution sense. Then standard bootstrapping argument implies that $g_{\beta\bar{\gamma}}$ is a smooth metric tensor with constant scalar curvature $K$ nearby $x$.  
    By the arbitrary choice of $x \in M_{\infty}$, we know that $(M_{\infty}, g_{\infty}, J_{\infty})$ is a smooth K\"ahler manifold with constant scalar curvature $K$.  
\end{proof}

In view of Lemma~\ref{lma:MJ25_1}, the entropy radius is bounded from below by the isoperimetric radius(cf.~Remark~\ref{rmk:RH26_4}), up to multiplying a definite constant. 
Therefore, Theorem~\ref{thm:RB21_12} and Corollary~\ref{cly:RD29_3} can be regarded as generalizations of  Theorem~\ref{thm:RG23_1}, 
replacing cscK manifolds by K\"ahler manifolds with bounded scalar curvature.

We close this section by the proof of Theorem~\ref{thm:RG26_8} and Theorem~\ref{thm:RH16_1}.
Note that Theorem~\ref{thm:RH16_1} provides another condition to guarantee that the limit $(M_{\infty}, g_{\infty}, J_{\infty})$ in Theorem~\ref{thm:RB21_12} is 
a smooth K\"ahler manifold. 

\begin{proof}[Proof of Theorem~\ref{thm:RG26_8}:]
  It follows from the combination of Proposition~\ref{prn:RC13_1} and Theorem~\ref{thm:RC10_1}.  
\end{proof}

\begin{proof}[Proof of Theorem~\ref{thm:RH16_1}:]
  The preservation of both upper and lower bounds of scalar curvature were already proved in Theorem~\ref{thm:RB21_12}.
  The only thing we need to prove is the existence of a smooth K\"ahler structure compatible with the smooth metric $g$. 

  In light of Lemma~\ref{lma:MJ25_1}, the $C^{0}$-approximation condition (\ref{eqn:RH16_1}) implies that entropy radius is bounded from below.
  Therefore, (\ref{eqn:RD29_1}) holds and we can apply the pseudo-locality theorem.  
  Following the proof of Theorem~\ref{thm:PB03_1}, via the pseudo-locality theorem and the Ricci-DeTurck flow, we obtain the commutativity of taking $C^{0}$-limit and running the Ricci flow.
  For simplicity of notation, we denote $(M_{i}^{n}, f_{i}^{*}g_{i}, f_{i}^{*}J_{i})$ by $(M_{i}^{n}, g_i, J_i)$. Then we can rewrite (\ref{eqn:RH16_1}) as 
  \begin{align}
    g_i \overset{C^{0}}{\rightarrow} g.    \label{eqn:RI30_7}
  \end{align}
  For each small $t$ fixed, we have 
  \begin{align*}
    (M_i, g_{i}(t)) \longright{C^{\infty}-Cheeger-Gromov} (N, g(t)), 
  \end{align*}
  where $g(t)$ is the time-$t$-slice of the Ricci flow initiated from $g$.
  Fix $t_0$ sufficiently small and set $r=\epsilon^{-1}\sqrt{t_{0}}$, where $\epsilon$ is the small constant in Theorem~\ref{thm:RC10_1}. 
  Put the K\"ahler structures $J_{i}$ into consideration. 
  The above convergence can be improved to
  \begin{align}
    \left( M_i, g_i(t_{0}), J_i \right) \longright{C^{\infty}-Cheeger-Gromov} \left(N, g(t_0), J \right),  \label{eqn:RJ02_2}
  \end{align}
  which means that we can find diffeomorphisms $\varphi_{i}:N \to M_{i}$ such that
   \begin{align}
     h_{i} \coloneqq \varphi_{i}^{*}\left( g_{i}(t_{0}) \right) \overset{C^{\infty}}{\rightarrow} g(t_{0}), \quad
    \varphi_{i}^{*}\left( J_{i} \right) \overset{C^{\infty}}{\rightarrow} J.
    \label{eqn:RI30_8} 
   \end{align}
  Fix an arbitrary point $p \in N$.   In light of Theorem~\ref{thm:RC10_1}, for each $i$, one can find a holomorphic function
  \begin{align}
    \Phi_{i}: \left( B_{\varphi_{i}^{*}(g_{i}(t_{0}))}(p,10r), \;\varphi_{i}^{*}(g_{i}(t_{0})), \;\varphi_{i}^{*}(J_{i})\right) \to U_{i}' \subset \C^{n}.   \label{eqn:RJ02_1}
  \end{align}
  The choice of $t_{0}$ and the estimate (\ref{eqn:RC15_11}) guarantees that 
  \begin{align*}
    B_{h_{i}}(p, 2r)=B_{\varphi_{i}^{*}(g_i(0))}(p, 2r) \subset B_{\varphi_{i}^{*}(g_i(t_0))}(p,10r). 
  \end{align*}
  Since the K\"ahler Ricci flow preserves the K\"ahler structure, we can restrict $\Phi_{i}$ to a smaller set to obtain another holomorphic map:
  \begin{align*}
    \Phi_{i}: \left( B_{h_{i}}(p,2r), h_{i}, \varphi_{i}^{*}(J_{i})  \right) \to U_{i} \subset \C^{n}.
  \end{align*}
  Clearly, the holomorphicity implies harmonicity on a K\"ahler manifold.
  Therefore, it follows from (\ref{eqn:RC15_11}) that 
  \begin{align*}
    \left( U_{i}, (\Phi_{i})_{*} (h_{i}) \right)  \longright{Id}  \left( U_{i}, g_{E} \right)
  \end{align*}
  is a harmonic map such that $\norm{\Phi_{*} (h_{i})}{C^{1,\alpha}}$ is uniformly bounded.  
  By taking subsequence and shrinking $\alpha$ if necessary,
  we can apply the Arzela-Ascolli lemma to obtain
  \begin{align}
    (\Phi_{i})_{*} (h_{i}) \longright{C^{1,\alpha}} \bar{h}.  \label{eqn:RJ02_3}
  \end{align}
  It is not hard to see that the harmonicity passes to the limit.  Namely, we know
  \begin{align}
  \left( U, \bar{h} \right) \longright{Id}  \left( U, g_{E} \right)
  \label{eqn:RJ02_6}
  \end{align}
  is a harmonic map. 
  Abusing notation, we denote $\bar{h}$ by $h$.
  Through the identifications described above, we naturally regard $h$ as a $C^{1,\alpha}$-metric tensor in the coordinate $U$. 
  On the other hand, with respect to the underlying metric $g_{i}(t_{0})$, we can apply (\ref{eqn:RI30_8}) to take limit of (\ref{eqn:RJ02_1}).
  It turns out that we obtain a smooth harmonic map 
  \begin{align*}
    \Phi: \left( B_{g(t_{0})}(p,5r), g(t_{0})\right) \mapsto (U', g_{E}),
  \end{align*}
  where $U'$ satisfies $U \subset U' \subset \C^{n}$. This implies that $\Phi_{*}(g(t_{0}))$ is a smooth metric tensor on $U \subset U'$.
  Since $\left\{ (N, g(t)), 0 \leq t \leq t_{0} \right\}$ is a smooth Ricci flow solution. It is clear that $\Phi_{*}(g)=\Phi_{*}(g(0))$
  is a smooth metric tensor on $U$.
  Abusing notation again, we may denote $\Phi_{*}(g)$ by $g$. Then $g$ is a smooth metric tensor in the coordinate $U$.

  We can cover $N$ by $U=U_{p}$ and then select a finite covering $\left\{ U_{\beta} \right\}_{\beta \in I}$ for a finite index set $I$.
  In conclusion, we have obtained an atlas $\left\{ U_{\beta} \right\}_{\beta \in I}$ of $(N, g(t_{0}), J)$ such that $g$ is smooth and $h$ is $C^{1,\alpha}$
  with respect to this coordinate atlas. For each pair of smooth vector fields $V,W$ on $N$, we have
  \begin{align*}
    \langle \varphi_{i}^{*}(J_{i}) \cdot V, \varphi_{i}^{*}(J_{i}) \cdot W\rangle_{h_i}=\langle V, W\rangle_{h_{i}}. 
  \end{align*}
  Thanks to (\ref{eqn:RI30_8}) and (\ref{eqn:RJ02_3}), taking limit of the above equation yields
  \begin{align}
    \langle J V, JW \rangle_{h}=\langle V,W\rangle_{h},  \label{eqn:RJ02_4}
  \end{align}
  which means that $J$ is an almost complex structure compatible with $h$.    
  Similarly, we can take limit of the K\"ahler condition $\nabla_{h_{i}} \left( \varphi_{i}^{*} J_{i} \right) \equiv 0$
  to obtain 
  \begin{align}
    \nabla_{h} J \equiv 0.     \label{eqn:RJ02_5}
  \end{align}
  Therefore, we know that $(N,h,J)$ is a K\"ahler manifold such that $h$ is $C^{1,\alpha}$ with respect to a holomorphic coordinate chart atlas $\left\{ U_{\beta} \right\}_{\beta \in I}$ of $(N, J)$.
  On the other hand, the differential manifold $N$ admits a smooth Riemannian metric $g$, which is smooth with respect to $\left\{ U_{\beta} \right\}_{\beta \in I}$.
  However, $g$ may not be compatible with $J$. In the remainder of this proof, we shall improve the regularity of $h$ to be smooth in the atlas $\left\{ U_{\beta} \right\}_{\beta \in I}$.
  The essential reason we can do this is due to the fact that the metric in harmonic coordinate charts has the best regularity(cf.~\cite{DeTKaz}), 
  and the regularity improvement for isometries(cf.~\cite{CalaHart},~\cite{Taylor}).

  Since (\ref{eqn:RJ02_3}) holds in each $U_{\beta}$, following our conventions of notation we have
  \begin{align*}
    h_{i} \longright{C^{1,\alpha}} h,
  \end{align*}
  which combined with (\ref{eqn:RI30_7}) implies that $(N,g)$ and $(N,h)$ are isometric to each other. 
  Therefore, there exists an isometry $F:(N, g) \to (N, h)$. Namely, for each pair of points $x,y \in N$, we have
  \begin{align*}
    d_{g}(x,y)=d_{h}(F(x), F(y)). 
  \end{align*}

  For each $p \in N$, as discussed above, we can choose $r$ small enough and find $V \subset \C^{n}$ such that
  $(V,h)$ is isometric to a geodesic ball $\left( B_{h}(F(p),r), h \right)$ and 
  \begin{align}
    (V,h) \longright{Id} (V, g_{E})    \label{eqn:RJ02_7}
  \end{align}
  is a harmonic map, and $h$ is a $C^{1,\alpha}$-metric tensor on $V$. Also, we can find $U \subset \C^{n}$ such that $(U,g)$ is isometric to the geodesic ball $(B_{g}(p,r),g)$ such 
  that $g$ is a smooth metric tensor on $U$. 
  By the definition of $F$, it is clear that $F|_{(B_{g}(p,r),g)}$ is an isometry map from $(B_{g}(p,r),g)$ to $(B_{h}(F(p), r),h)$. 
  Consequently, abusing notation again, we have an isometry
  \begin{align}
    (U,g) \longright{F} (V,h). \label{eqn:RJ02_8}
  \end{align}
  Notice that both $U$ and $V$ are open sets in Euclidean space. 
  Now $g$ is a $C^{\infty}$-metric tensor field, while $h$ is only a $C^{1,\alpha}$-metric tensor field.
  The isometry property of $F$ clearly implies that $F$ is a bi-Lipschitz map. 
  The harmonicity in (\ref{eqn:RJ02_7}) and the isometry in (\ref{eqn:RJ02_8}) together then imply that each coordinate of $F$ is a continuous harmonic function(cf. Section 2 of~\cite{Taylor}).
  Let $F=(y^{1}, y^{2}, \cdots, y^{2n})$ with each $y^{k}=y^{k}(x^{1}, x^{2}, \cdots, x^{2n})$.
  The harmonicity then implies that
  \begin{align*}
    0=\Delta_{g} y^{k}
    =\left\{ g^{ij} \frac{\partial^{2}}{\partial x^{i} \partial x^{j}} - \frac{1}{\sqrt{g}} \frac{\partial}{\partial x^{i}}\left( \sqrt{g} g^{ij} \right)\frac{\partial}{\partial x^{j}} \right\}y^{k}
  \end{align*}
  for each $k \in \left\{ 1,2,\cdots, 2n \right\}$, in the distribution sense. 
  Since $g$ is a smooth metric tensor, standard regularity theory of elliptic PDE then implies that each $y^{k}$ is a smooth function. 
  Therefore, $F$ is a smooth map from $U$ to $V$. 
  Consequently, the isometry property in (\ref{eqn:RJ02_8}) forces $h=F_{*}g$ to be smooth metric tensor field on $V$.
  According to the choice of $V$, it is clear that $V$ is compatible with $\left\{ U_{\beta} \right\}_{\beta \in I}$.  
  This means that $h$ is a smooth metric tensor around point $F(p)$.   By the arbitrary choice of $p \in N$, we know that $h$ is a smooth metric tensor field with respect to $\left\{ U_{\beta} \right\}_{\beta \in I}$.
  Combining this smoothness with (\ref{eqn:RJ02_4}) and (\ref{eqn:RJ02_5}), we know that $(N,h,J)$ is a smooth K\"ahler manifold.
  By the arbitrary choice of $p$ and the local smoothness of isometry in (\ref{eqn:RJ02_8}), we know that the isometry $F:(N,g) \to (N,h)$ is a smooth map so that $h=F_{*}(g)$.
  This means that $(N,g,F^{*}(J))$ is a smooth K\"ahler manifold. 
  Replacing $F^{*}(J)$ by $J$ if necessary, we obtain the desired smooth K\"ahler structure compatible with $g$. 
  The proof of Theorem~\ref{thm:RH16_1} is complete. 
\end{proof}

\section{Further Questions}
\label{sec:further}

In our stability theorems, we only prove the existence of small gap constants.  It seems very natural to ask the following question.

 \begin{question}
   Explicitly calculate the sharp gap constant in Theorem~\ref{thm:RG26_3}, Theorem~\ref{thm:RG26_5} and Theorem~\ref{thm:RH19_5}. 
   In the particular case $m=4$, we conjecture the constant in Theorem~\ref{thm:RG26_3} is $\frac{1}{4}$, which represents the gap
   of volume ratio between $(S^{4}, g_{round})$ and $(\CP^{2}, g_{FS})$. 
  \label{qun:RG25_4}
 \end{question}

More boldly, we may even drop the initial Ricci curvature condition, or replace it with scalar curvature lower bound condition at least. 
Recall that the main theme of~\cite{BWang17local} and the current paper is: under Ricci flow, the local entropy functionals $\boldsymbol{\mu}$ and $\boldsymbol{\nu}$ are more natural than the volume.
Although the Ricci curvature lower bound and the volume lower bound are currently needed for many technical purposes, we believe that finally they can be essentially replaced by a lower bound
condition concerning only the local functional $\boldsymbol{\mu}$ or $\boldsymbol{\nu}$.
Therefore, the main theorems in this paper should have versions where only initial entropy conditions are assumed. 
Inspired by the recent development in the mean curvature flow through a series of work of Colding-Ilmanen-Minicozzi-White~\cite{CIMW}, Bernstein-Wang~\cite{BeWa}, Ketover-Zhou~\cite{KeZhou} and J.J. Zhu~\cite{JJZhu}, 
we make the following conjectures. 

\begin{conjecture}
  Among all Fano manifolds $(M^{n}, g, J)$ such that $[\omega]$ is proportional to $2\pi c_1(M,J)$, we have
  \begin{align*}
    \boldsymbol{\nu}(M,g) \leq \boldsymbol{\nu}(\CP^n, g_{FS}).
  \end{align*}
  The equality holds if and only if $(M,g,J)$ is biholomorphic-isometric to $(\CP^n,g_{FS},J_{FS})$.  
  Furthermore, there exists a constant $\epsilon=\epsilon(n)$ such that if 
  \begin{align}
    \boldsymbol{\nu}(M,g) > \boldsymbol{\nu}(\CP^n, g_{FS})-\epsilon,   \label{eqn:RG03_2}
  \end{align}
  then $(M,J)$ is biholomorphic to $(\CP^{n}, J_{FS})$ and the normalized  Ricci flow initiated from $(M,g)$ 
  will converge exponentially fast to a metric with constant positive holomorphic sectional curvature.  
  \label{cje:RH31_3}
 \end{conjecture}

 \begin{conjecture}
    Among all closed Riemannian manifold $(M,g)$ of real dimension $m$, we have
    \begin{align*}
      \boldsymbol{\nu}(M,g) \leq \boldsymbol{\nu}(S^m, g_{round}). 
    \end{align*}
    The equality holds if and only if $M$ is isometric to the round sphere $(S^{m}, g_{round})$.
    Furthermore, there exists a constant $\epsilon=\epsilon(m)$ such that if 
  \begin{align}
    \boldsymbol{\nu}(M,g) > \boldsymbol{\nu}(S^{m}, g_{round})-\epsilon,   \label{eqn:RG03_3}
  \end{align}
   then $M$ is diffeomorphic to $S^{m}$ and the normalized  Ricci flow initiated from $(M,g)$ will converge exponentially fast to a round sphere. 
 \label{cje:RH31_4}  
 \end{conjecture}
 
 Similar to Question~\ref{qun:RG25_4},  it is intriguing to sharply calculate the value of $\epsilon$ in (\ref{eqn:RG03_2}) and (\ref{eqn:RG03_3}).

%\appendixpage
%\addappheadtotoc
%\appendix

\vspace{0.5in}

Bing Wang, Institute of Geometry and Physics, and Key Laboratory of Wu Wen-Tsun Mathematics, School of Mathematical Sciences, University of Science and Technology of China, No. 96 Jinzhai Road, Hefei, Anhui Province, 230026, China; topspin@ustc.edu.cn.\\

\end{document}